\documentclass{article}
\usepackage{amsmath, amssymb, amsthm, graphicx, cite, comment}
\usepackage[arrow,curve,matrix,arc,2cell]{xy}
\UseAllTwocells
\DeclareFontFamily{U}{rsfs}{} \DeclareFontShape{U}{rsfs}{n}{it}{<->
rsfs10}{} \DeclareSymbolFont{mscr}{U}{rsfs}{n}{it}
\DeclareSymbolFontAlphabet{\scr}{mscr}
\def\mathscr{\scr}
\begin{document}
\def\e#1\e{\begin{equation}#1\end{equation}}
\def\ea#1\ea{\begin{align}#1\end{align}}
\def\eq#1{{\rm(\ref{#1})}}
\theoremstyle{plain}
\newtheorem{thm}{Theorem}[section]
\newtheorem{prop}[thm]{Proposition}
\newtheorem{lem}[thm]{Lemma}
\newtheorem{cor}[thm]{Corollary}
\newtheorem{quest}[thm]{Question}
\theoremstyle{definition}
\newtheorem{dfn}[thm]{Definition}
\newtheorem{ex}[thm]{Example}
\newtheorem{rem}[thm]{Remark}
\numberwithin{equation}{section}
\def\dim{\mathop{\rm dim}\nolimits}
\def\supp{\mathop{\rm supp}\nolimits}
\def\cosupp{\mathop{\rm cosupp}\nolimits}
\def\rank{\mathop{\rm rank}}
\def\coh{\mathop{\rm coh}\nolimits}
\def\id{\mathop{\rm id}\nolimits}
\def\Hess{\mathop{\rm Hess}\nolimits}
\def\Crit{\mathop{\rm Crit}}
\def\GL{\mathop{\rm GL}}
\def\SO{\mathop{\rm SO}\nolimits}
\def\Ker{\mathop{\rm Ker}}
\def\Im{\mathop{\rm Im}}
\def\Aut{\mathop{\rm Aut}}
\def\End{\mathop{\rm End}}
\def\Hom{\mathop{\rm Hom}\nolimits}
\def\Perv{\mathop{\rm Perv}\nolimits}
\def\Perf{\mathop{\rm Perf}\nolimits}
\def\Perfis{\mathop{\text{\rm Perf-is}}\nolimits}
\def\Spec{\mathop{\rm Spec}\nolimits}
\def\Lbis{\mathop{\text{\rm Lb-is}}\nolimits}
\def\red{{\rm red}}
\def\an{{\rm an}}
\def\alg{{\rm alg}}
\def\bs{\boldsymbol}
\def\ge{\geqslant}
\def\le{\leqslant\nobreak}
\def\bD{{\mathbin{\mathbb D}}}
\def\bH{{\mathbin{\mathbb H}}}
\def\bL{{\mathbin{\mathbb L}}}
\def\bP{{\mathbin{\mathbb P}}}
\def\A{{\mathbin{\mathcal A}}}
\def\PP{{\mathbin{\mathcal P}}}
\def\O{{\mathbin{\cal O}}}
\def\cA{{\mathbin{\cal A}}}
\def\cB{{\mathbin{\cal B}}}
\def\cC{{\mathbin{\cal C}}}
\def\cD{{\mathbin{\scr D}}}
\def\cE{{\mathbin{\cal E}}}
\def\cF{{\mathbin{\cal F}}}
\def\cG{{\mathbin{\cal G}}}
\def\cH{{\mathbin{\cal H}}}
\def\cL{{\mathbin{\cal L}}}
\def\cM{{\mathbin{\cal M}}}
\def\cO{{\mathbin{\cal O}}}
\def\cP{{\mathbin{\cal P}}}
\def\cS{{\mathbin{\cal S}}}
\def\cSz{{\mathbin{\cal S}\kern -0.1em}^{\kern .1em 0}}
\def\PV{{\mathbin{\cal{PV}}}}
\def\TS{{\mathbin{\cal{TS}}}}
\def\cQ{{\mathbin{\cal Q}}}
\def\cW{{\mathbin{\cal W}}}
\def\C{{\mathbin{\mathbb C}}}
\def\CP{{\mathbin{\mathbb{CP}}}}
\def\bA{{\mathbin{\mathbb A}}}
\def\K{{\mathbin{\mathbb K}}}
\def\Q{{\mathbin{\mathbb Q}}}
\def\R{{\mathbin{\mathbb R}}}
\def\Z{{\mathbin{\mathbb Z}}}
\def\al{\alpha}
\def\be{\beta}
\def\ga{\gamma}
\def\de{\delta}
\def\io{\iota}
\def\ep{\epsilon}
\def\la{\lambda}
\def\ka{\kappa}
\def\th{\theta}
\def\ze{\zeta}
\def\up{\upsilon}
\def\vp{\varphi}
\def\si{\sigma}
\def\om{\omega}
\def\De{\Delta}
\def\La{\Lambda}
\def\Si{\Sigma}
\def\Tau{{\rm T}}
\def\Th{\Theta}
\def\Om{\Omega}
\def\Ga{\Gamma}
\def\Up{\Upsilon}
\def\pd{\partial}
\def\db{{\bar\partial}}
\def\ts{\textstyle}
\def\st{\scriptstyle}
\def\sst{\scriptscriptstyle}
\def\w{\wedge}
\def\sm{\setminus}
\def\bu{\bullet}
\def\op{\oplus}
\def\ot{\otimes}
\def\otL{{\kern .1em\mathop{\otimes}\limits^{\sst L}\kern .1em}}
\def\boxtL{{\kern .2em\mathop{\boxtimes}\limits^{\sst L}\kern .2em}}
\def\boxtT{{\kern .1em\mathop{\boxtimes}\limits^{\sst T}\kern .1em}}
\def\ov{\overline}
\def\ul{\underline}
\def\bigop{\bigoplus}
\def\bigot{\bigotimes}
\def\iy{\infty}
\def\es{\emptyset}
\def\ra{\rightarrow}
\def\Ra{\Rightarrow}
\def\Longra{\Longrightarrow}
\def\ab{\allowbreak}
\def\longra{\longrightarrow}
\def\hookra{\hookrightarrow}
\def\dashra{\dashrightarrow}
\def\t{\times}
\def\ci{\circ}
\def\ti{\tilde}
\def\d{{\rm d}}
\def\ha{{\ts\frac{1}{2}}}
\def\md#1{\vert #1 \vert}
\def\ms#1{\vert #1 \vert^2}
\def\HM{\mathop{\rm HM}\nolimits}
\def\MHM{\mathop{\rm MHM}\nolimits}
\def\rat{\mathop{\bf rat}\nolimits}
\def\Rat{\mathop{\bf Rat}\nolimits}
\def\HV{{\mathbin{\cal{HV}}}}
\def\compact{{\rm c}}
\renewcommand{\Re}{\mathop{\rm Re}}
\title{Symmetries and stabilization \\ for sheaves of vanishing cycles}
\author{Christopher Brav, Vittoria Bussi, Delphine Dupont, \\
Dominic Joyce, and Bal\'azs Szendr\H oi, \\[6pt] 
with an Appendix by J\"org Sch\"urmann}
\date{}
\maketitle

\begin{abstract}
We study symmetries and stabilization properties of perverse sheaves
of vanishing cycles $\PV_{U,f}^\bu$ of a regular function $f:U\ra
\C$ on a smooth $\C$-scheme $U$, with critical locus $X=\Crit(f)$.
We prove four main results:

\smallskip

\noindent (a) If $\Phi:U\ra U$ is an isomorphism fixing $X$ and
compatible with $f$, then the action of $\Phi_*$ on $\PV_{U,f}^\bu$
is multiplication by $\det\bigl(\d\Phi\vert_{X^\red}\bigr)=\pm 1$.

\smallskip

\noindent (b) $\PV_{U,f}^\bu$ depends up to canonical isomorphism
only on $(X^{(3)},f^{(3)})$, for $X^{(3)}$ the third-order
thickening of $X$ in $U$, and
$f^{(3)}=f\vert_{X^{(3)}}:X^{(3)}\ra\C$.

\smallskip

\noindent (c) If $U,V$ are smooth $\C$-schemes, $f:U\ra\C$,
$g:V\ra\C$ are regular, $X=\Crit(f)$, $Y=\Crit(g)$, and $\Phi:U\ra
V$ is an embedding with $f=g\ci\Phi$ and $\Phi\vert_X:X\ra Y$ an
isomorphism, there is a natural isomorphism
$\Th_\Phi\!:\!\PV_{U,f}^\bu \!\ra\!\Phi\vert_X^*(\PV_{V,g}^\bu)
\!\ot_{\Z/2\Z}\!P_\Phi$, for $P_\Phi$ a principal $\Z/2\Z$-bundle
on~$X$.

\smallskip

\noindent (d) If $(X,s)$ is an {\it oriented d-critical locus\/} in
the sense of Joyce \cite{Joyc}, there is a natural perverse sheaf
$P_{X,s}^\bu$ on $X$, such that if $(X,s)$ is locally modelled on
$\Crit(f:U\ra\C)$ then $P_{X,s}^\bu$ is locally modelled
on~$\PV_{U,f}^\bu$.

\smallskip

We also generalize our results to replace $U,X$ by 
complex analytic spaces, and $\PV_{U,f}^\bu$ by $\cD$-modules or
mixed Hodge modules.

We discuss applications of (d) to categorifying Donaldson--Thomas
invariants of Calabi--Yau 3-folds, and to defining a `Fukaya
category' of Lagrangians in a complex symplectic manifold using
perverse sheaves.
\end{abstract}

\setcounter{tocdepth}{2}
\tableofcontents

\section{Introduction}
\label{sm1}

Let $U$ be a smooth $\C$-scheme and $f:U\ra\C$ a regular function,
and write $X=\Crit(f)$, as a $\C$-subscheme of $U$. Then one can
define the {\it perverse sheaf of vanishing cycles\/}
$\PV_{U,f}^\bu$ on $X$. Formally, $X=\coprod_{c\in f(X)}X_c$, where
$X_c\subseteq X$ is the open and closed $\C$-subscheme of points
$x\in X$ with $f(x)=c$, and $\PV_{U,f}^\bu\vert_{X_c}=
\phi_{f-c}^p(A_U[\dim U])\vert_{X_c}$ for each $c\in f(X)$, where
$A_U[\dim U]$ is the constant perverse sheaf on $U$ over a base ring
$A$, and $\phi_{f-c}^p:\Perv(U)\ra\Perv(f^{-1}(c))$ is the vanishing
cycle functor for $f-c:U\ra\C$. See \S\ref{sm2} for an introduction
to perverse sheaves, and an explanation of this notation.

This paper will prove four main results, Theorems \ref{sm3thm},
\ref{sm4thm}, \ref{sm5thm2} and \ref{sm6thm6}. The first three give
properties of the $\PV_{U,f}^\bu$, which we may summarize as
follows:
\begin{itemize}
\setlength{\itemsep}{0pt}
\setlength{\parsep}{0pt}
\item[(a)] Let $U,f,X$ be as above, and write $X^\red$ for the
reduced $\C$-subscheme of $X$. Suppose $\Phi:U\ra U$ is an
isomorphism with $f\ci\Phi=f$ and $\Phi\vert_X=\id_X$. Then
$\Phi$ induces a natural isomorphism $\Phi_*:\PV_{U,f}^\bu\ra
\PV_{U,f}^\bu$.

Theorem \ref{sm3thm} implies that
$\d\Phi\vert_{TU\vert_{X^\red}}:TU\vert_{X^\red}\ra
TU\vert_{X^\red}$ has determinant
$\det\bigl(\d\Phi\vert_{X^\red}\bigr):X^\red\ra\C\sm\{0\}$ which
is a locally constant map $X^\red\ra\{\pm 1\}$, and
$\Phi_*:\PV_{U,f}^\bu\ra\PV_{U,f}^\bu$ is multiplication
by~$\det\bigl(\d\Phi\vert_{X^\red}\bigr)$.

In fact Theorem \ref{sm3thm} proves a more complicated
statement, which only requires $\Phi$ to be defined \'etale
locally on $U$.

\item[(b)] Let $U,f,X$ be as above, and write $I_X\subseteq
\O_U$ for the sheaf of ideals of regular functions $U\ra\C$
vanishing on $X$. For each $k=1,2,\ldots,$ write $X^{(k)}$ for
the $k^{\rm th}$ {\it order thickening of\/ $X$ in\/} $U$, that
is, $X^{(k)}$ is the closed $\C$-subscheme of $U$ defined by the
vanishing of the sheaf of ideals $I_X^k$ in $\O_U$. Write
$f^{(k)}:=f\vert_{\smash{X^{(k)}}}:X^{(k)}\ra\C$.

Theorem \ref{sm4thm} says that the perverse sheaf
$\PV_{U,f}^\bu$ depends only on the third-order thickenings
$(X^{(3)},f^{(3)})$ up to canonical isomorphism.

As in Remark \ref{sm4rem}, \'etale locally, $\PV_{U,f}^\bu$
depends only on $(X^{(2)},f^{(2)})$ up to non-canonical
isomorphism, with isomorphisms natural up to sign.
\item[(c)] Let $U,V$ be smooth $\C$-schemes, $f:U\ra\C,$
$g:V\ra\C$ be regular, and $X=\Crit(f)$, $Y=\Crit(g)$ as
$\C$-subschemes of $U,V$. Let $\Phi:U\hookra V$ be a closed
embedding of $\C$-schemes with $f=g\ci\Phi:U\ra\C$, and suppose
$\Phi\vert_X:X\ra Y$ is an isomorphism. Then Theorem
\ref{sm5thm2} constructs a natural isomorphism of perverse
sheaves on $X$:
\e
\Th_\Phi:\PV_{U,f}^\bu\longra\Phi\vert_X^*\bigl(\PV_{V,g}^\bu\bigr)
\ot_{\Z/2\Z}P_\Phi,
\label{sm1eq1}
\e
where $\pi_\Phi:P_\Phi\ra X$ is a certain principal
$\Z/2\Z$-bundle on $X$. Writing $N_{\sst UV}$ for the normal
bundle of $U$ in $V$, then the Hessian $\Hess g$ induces a
nondegenerate quadratic form $q_{\sst UV}$ on $N_{\sst
UV}\vert_X$, and $P_\Phi$ parametrizes square roots
of~$\det(q_{\sst UV}):K_U^2\vert_X\ra \Phi\vert_X^*(K_V^2)$.

Theorem \ref{sm5thm2} also shows that the $\Th_\Phi$ in
\eq{sm1eq1} are functorial in a suitable sense under
compositions of embeddings $\Phi:U\hookra V$, $\Psi:V\hookra W$.
\end{itemize}
Here (c) is proved by showing that \'etale locally there exist
equivalences $V\simeq U\t\C^n$ identifying $\Phi(U)$ with $U\t\{0\}$
and $g:V\ra\C$ with $f\boxplus z_1^2+\cdots+z_n^2:U\t\C^n\ra\C$, and
applying \'etale local isomorphisms of perverse sheaves
\begin{equation*}
\PV_{U,f}^\bu\cong\PV_{U,f}^\bu\boxtL
\PV_{\C^n,z_1^2+\cdots+z_n^2}^\bu\cong\PV_{U\t\C^n,f\boxplus
z_1^2+\cdots+z_n^2}^\bu\cong \PV_{V,g}^\bu,
\end{equation*}
using $\PV_{\C^n,z_1^2+\cdots+z_n^2}^\bu\cong A_{\{0\}}$ in the
first step, and the Thom--Sebastiani Theorem for perverse sheaves in
the second.

Passing from $f:U\ra\C$ to $g=f\boxplus z_1^2+\cdots+
z_n^2:U\t\C^n\ra\C$ is an important idea in singularity theory, as
in Arnold et al.\ \cite{AGV} for instance. It is known as {\it
stabilization}, and $f$ and $g$ are called {\it stably equivalent}.
So, Theorem \ref{sm5thm2} concerns the behaviour of perverse sheaves
of vanishing cycles under stabilization.

Our fourth main result, Theorem \ref{sm6thm6}, concerns a new class
of geometric objects called {\it d-critical loci}, introduced in
Joyce \cite{Joyc}, and explained in \S\ref{sm61}. An (algebraic)
d-critical locus $(X,s)$ over $\C$ is a $\C$-scheme $X$ with a
section $s$ of a certain natural sheaf $\cSz_X$ on $X$. A d-critical
locus $(X,s)$ may be written Zariski locally as a critical locus
$\Crit(f:U\ra\C)$ of a regular function $f$ on a smooth $\C$-scheme
$U$, and $s$ records some information about $U,f$ (in the notation
of (b) above, $s$ remembers $f^{(2)}$). There is also a complex
analytic version.

Algebraic d-critical loci are classical truncations of the {\it
derived critical loci\/} (more precisely, $-1$-{\it shifted
symplectic derived schemes\/}) introduced in derived algebraic
geometry by Pantev, To\"en, Vaqui\'e and Vezzosi \cite{PTVV}.
Theorem \ref{sm6thm6} roughly says that if $(X,s)$ is an algebraic
d-critical locus over $\C$ with an `orientation', then we may define
a natural perverse sheaf $P_{X,s}^\bu$ on $X$, such that if $(X,s)$
is locally modelled on $\Crit(f:U\ra\C)$ then $P_{X,s}^\bu$ is
locally modelled on $\PV_{U,f}^\bu$. The proof uses
Theorem~\ref{sm5thm2}.

These results have exciting applications in the categorification of
Donaldson--Thomas theory on Calabi--Yau 3-folds, and in defining a
new kind of `Fukaya category' of complex Lagrangians in complex
symplectic manifold, which we will discuss at length in Remarks
\ref{sm6rem2} and~\ref{sm6rem3}.

Although we have explained our results only for $\C$-schemes and
perverse sheaves upon them, the proofs are quite general and work in
several contexts:
\begin{itemize}
\setlength{\itemsep}{0pt}
\setlength{\parsep}{0pt}
\item[(i)] Perverse sheaves on $\C$-schemes or complex analytic
spaces with coefficients in any well-behaved commutative ring
$A$, such as $\Z,\Q$ or $\C$.
\item[(ii)] $\cD$-modules on $\C$-schemes or complex analytic
spaces.
\item[(iii)] Saito's mixed Hodge modules on $\C$-schemes or
complex analytic spaces.
\end{itemize}
We discuss all these in \S\ref{sm2}, before proving our four main
results in \S\ref{sm3}--\S\ref{sm6}. Appendix \ref{smA}, by J\"org Sch\"urmann, proves two compatibility results between duality and Thom--Sebastiani type isomorphisms needed in the main text.

This is one of six linked papers \cite{BBBJ,BBJ,Buss,BJM,Joyc}, with
more to come. The best logical order is that the first is Joyce
\cite{Joyc} defining d-critical loci, and the second Bussi, Brav and
Joyce \cite{BBJ}, which proves Darboux-type theorems for the
$k$-shifted symplectic derived schemes of Pantev et al.\
\cite{PTVV}, and defines a truncation functor from $-1$-shifted
symplectic derived schemes to algebraic d-critical loci.

This paper is the third in the sequence. Combining our results with
\cite{Joyc,PTVV} gives new results on categorifying
Donaldson--Thomas invariants of Calabi--Yau 3-folds, as in Remark
\ref{sm6rem2}. In the fourth paper Bussi, Joyce and Meinhardt
\cite{BJM} will generalize the ideas of this paper to motivic Milnor
fibres (we explain the relationship between the motivic and cohomological approaches
below in Remark~\ref{sm6rem1}), and deduce new results on motivic Donaldson--Thomas
invariants using \cite{Joyc,PTVV}. In the fifth, Ben-Bassat, Brav,
Bussi and Joyce \cite{BBBJ} generalize \cite{BBJ,BJM} and this paper
from (derived) schemes to (derived) Artin stacks.

Sixthly, Bussi \cite{Buss} will show that if $(S,\om)$ is a complex
symplectic manifold, and $L,M$ are complex Lagrangians in $S$, then
the intersection $X=L\cap M,$ as a complex analytic subspace of $S$,
extends naturally to a complex analytic d-critical locus $(X,s)$. If
the canonical bundles $K_L,K_M$ have square roots $K_L^{1/2},
K_M^{1/2}$ then $(X,s)$ is oriented, and so Theorem \ref{sm6thm6}
below defines a perverse sheaf $P_{L,M}^\bu$ on $X$, which Bussi
also constructs directly. As in Remark \ref{sm6rem3}, we hope in
future work to define a `Fukaya category' of complex Lagrangians in
$(X,\om)$ in which~$\Hom(L,M)\cong\bH^{-n}(P_{L,M}^\bu)$.
\medskip

\noindent{\bf Conventions.} All $\C$-schemes are
assumed separated and of finite type. All complex analytic
spaces are Hausdorff and locally of finite type.
\medskip

\noindent{\bf Acknowledgements.} We would like to thank Oren Ben-Bassat, Alexandru Dimca, Young-Hoon Kiem, Jun Li, Kevin McGerty, Sven Meinhardt, Pierre Schapira, and especially Morihiko Saito and J\"org Sch\"urmann for useful conversations and correspondence, and J\"org Sch\"urmann for a very careful reading of our manuscript, leading to many improvements, as well as providing the Appendix. This research was supported by EPSRC Programme Grant EP/I033343/1.

\section{Background on perverse sheaves}
\label{sm2}

Perverse sheaves, and the related theories of $\cD$-modules and
mixed Hodge modules, make sense in several contexts, both algebraic
and complex analytic:
\begin{itemize}
\setlength{\itemsep}{0pt}
\setlength{\parsep}{0pt}
\item[(a)] Perverse sheaves on $\C$-schemes with coefficients
in a ring $A$ (usually $\Z,\Q$ or $\C$), as in Beilinson,
Bernstein and Deligne \cite{BBD} and Dimca \cite{Dimc}.
\item[(b)] Perverse sheaves on complex analytic spaces with
coefficients in a ring $A$ (usually $\Z,\Q$ or $\C$), as in
Dimca~\cite{Dimc}.
\item[(c)] $\cD$-modules on $\C$-schemes, as in Borel \cite{Bore} 
in the smooth case, and Saito \cite{Sait4} in general. 
\item[(d)] $\cD$-modules on complex manifolds as in
Bj\"ork \cite{Bjor}, and on complex analytic spaces as in Saito \cite{Sait4}.
\item[(e)] Mixed Hodge modules on $\C$-schemes, as in
Saito~\cite{Sait1,Sait3}.
\item[(f)] Mixed Hodge modules on complex analytic spaces, as in
Saito~\cite{Sait1,Sait3}.
\end{itemize}

All our main results and proofs work, with minor modifications, in
all six settings (a)--(f). As (a) is arguably the simplest and
most complete theory, we begin in \S\ref{sm21}--\S\ref{sm24} with a
general introduction to constructible complexes and perverse sheaves
on $\C$-schemes, the nearby and vanishing cycle functors, and
perverse sheaves of vanishing cycles $PV_{U,f}^\bu$ on $\C$-schemes,
following Dimca~\cite{Dimc}.

Several important properties of perverse sheaves in (a) either do
not work, or become more complicated, in settings (b)--(f). Section
\ref{sm25} lists the parts of \S\ref{sm21}--\S\ref{sm24} that we
will use in proofs in this paper, so the reader can check that they
do work in (b)--(f). Then \S\ref{sm26}--\S\ref{sm210} give brief
discussions of settings (b)--(f), focussing on the differences with
(a) in~\S\ref{sm21}--\S\ref{sm24}.

A good introductory reference on perverse sheaves on $\C$-schemes
and complex analytic spaces is Dimca \cite{Dimc}. Three other books
are Kashiwara and Schapira \cite{KaSc1}, Sch\"urmann \cite{Schu},
and Hotta, Tanisaki and Takeuchi \cite{HTT}. Massey \cite{Mass2} and
Rietsch \cite{Riet} are surveys on perverse sheaves, and Beilinson,
Bernstein and Deligne \cite{BBD} is an important primary source, who
cover both $\Q$-perverse sheaves on $\C$-schemes as in (a), and
$\Q_l$-perverse sheaves on $\K$-schemes as in~(g) below.

\begin{rem} Two further possible settings, in which not all the results
we need are available in the literature, are the following. 
\begin{itemize}
\setlength{\itemsep}{0pt}
\setlength{\parsep}{0pt}
\item[(g)] Perverse sheaves on $\K$-schemes with coefficients
in $\Z/l^n\Z$, $\Z_l$, $\Q_l$, or $\bar\Q_l$ for
$l\ne\mathop{\rm char}\K\ne 2$ a prime, as in Beilinson et
al.~\cite{BBD}.
\item[(h)] $\cD$-modules on $\K$-schemes for
$\K$ an algebraically closed field, as in Borel~\cite{Bore}.
\end{itemize}

The issue is that the Thom--Sebastiani theorem is not available in these contexts in the
generality we need it. Once an appropriate form of this result becomes available, our main theorems 
will hold also in these two contexts, sometimes under the further assumption that ${\rm char}\ \K=0$,
needed for the results quoted from \cite{PTVV,BBJ}. We leave the details to the interested
reader. 
\label{sm2rem1}
\end{rem}

\subsection{Constructible complexes on $\C$-schemes}
\label{sm21}

We begin by discussing constructible complexes, following
Dimca~\cite[\S 2--\S 4]{Dimc}.

\begin{dfn} Fix a well-behaved commutative base ring $A$
(where `well-behaved' means that we need assumptions on $A$ such as
$A$ is regular noetherian, of finite global dimension or finite
Krull dimension, a principal ideal domain, or a Dedekind domain, at
various points in the theory), to study sheaves of $A$-modules. For
some results $A$ must be a field. Usually we take $A=\Z,\Q$ or~$\C$.

Let $X$ be a $\C$-scheme, always assumed of finite type. Write
$X^\an$ for the set of $\C$-points of $X$ with the complex analytic
topology. Consider sheaves of $A$-modules $\cS$ on $X^\an$. A sheaf
$\cS$ is called ({\it algebraically\/}) {\it constructible\/} if all
the stalks $\cS_x$ for $x\in X^\an$ are finite type $A$-modules, and
there is a finite stratification $X^\an=\coprod_{j\in
J}X^\an_j$ of $X^\an$, where $X_j\subseteq X$ for $j\in J$ are
$\C$-subschemes of $X$ and $X_j^\an\subseteq X^\an$ the
corresponding subsets of $\C$-points, such that $\cS\vert_{X_j^\an}$
is an $A$-local system for all~$j\in J$.

Write $D(X)$ for the derived category of complexes $\cC^\bu$ of
sheaves of $A$-modules on $X^\an$. Write $D^b_c(X)$ for the full
subcategory of bounded complexes $\cC^\bu$ in $D(X)$ whose
cohomology sheaves $\cH^m(\cC^\bu)$ are constructible for all
$m\in\Z$. Then $D(X),D^b_c(X)$ are triangulated categories. An
example of a constructible complex on $X$ is the {\it constant
sheaf\/} $A_X$ on $X$ with fibre $A$ at each point.

Grothendieck's ``six operations on sheaves'' $f^*,f^!,Rf_*,Rf_!,
{\cal RH}om,\smash{\otL}$ act on $D(X)$ preserving the subcategory
$D^b_c(X)$. That is, if $f:X\ra Y$ is a morphism of $\C$-schemes,
then we have two different pullback functors $f^*,f^!:D(Y)\ra D(X)$,
which also map $D^b_c(Y)\ra D^b_c(X)$. Here $f^*$ is called the {\it
inverse image\/} \cite[\S 2.3]{Dimc}, and $f^!$ the {\it exceptional
inverse image\/} \cite[\S 3.2]{Dimc}.

We also have two different pushforward functors $Rf_*,Rf_!:D(X)\ra
D(Y)$ mapping $D^b_c(X)\ra D^b_c(Y)$, where $Rf_*$ is called the
{\it direct image\/} \cite[\S 2.3]{Dimc} and is right adjoint to
$f^*:D(Y)\ra D(X)$, and $Rf_!$ is called the {\it direct image with
proper supports\/} \cite[\S 2.3]{Dimc} and is left adjoint to
$f^!:D(Y)\ra D(X)$. We need the assumptions from \S\ref{sm1} that
$X,Y$ are separated and of finite type for $Rf_*,Rf_!:D^b_c(X)\ra
D^b_c(Y)$ to be defined for arbitrary morphisms~$f:X\ra Y$.

For $\cB^\bu,\cC^\bu$ in $D^b_c(X)$, we may form their {\it derived
Hom\/} ${\cal RH}om(\cB^\bu,\cC^\bu)$ \cite[\S 2.1]{Dimc}, and {\it
left derived tensor product\/} $\cB^\bu\otL\cC^\bu$ in $D^b_c(X)$,
\cite[\S 2.2]{Dimc}. Given $\cB^\bu\in D^b_c(X)$ and $\cC^\bu\in
D^b_c(Y)$, we define $\cB^\bu\smash{\boxtL}\cC^\bu=
\pi_X^*(\cB^\bu)\otL\pi_Y^*(\cC^\bu)$ in $D^b_c(X\t Y)$, where
$\pi_X:X\t Y\ra X$, $\pi_Y:X\t Y\ra Y$ are the projections.

If $X$ is a $\C$-scheme, there is a functor $\bD_X:D^b_c(X)\ra
D^b_c(X)^{\rm op}$ with $\bD_X\ci\bD_X\cong\id:D^b_c(X)\ra
D^b_c(X)$, called {\it Verdier duality}. It reverses shifts, that
is, $\bD_X\bigl(\cC^\bu[k]\bigr)=\bigl(\bD_X(\cC^\bu)\bigr)[-k]$ for
$\cC^\bu$ in $D^b_c(X)$ and~$k\in\Z$.
\label{sm2def1}
\end{dfn}

\begin{rem} Note how Definition \ref{sm2def1} mixes the complex
analytic and the complex algebraic: we consider sheaves on $X^\an$
in the {\it analytic\/} topology, which are constructible with
respect to an {\it algebraic\/} stratification~$X=\coprod_jX_j$.
\label{sm2rem2}
\end{rem}

Here are some properties of all these:

\begin{thm} In the following, $X,Y,Z$ are $\C$-schemes, and\/ $f,g$
are morphisms, and all isomorphisms\/ {\rm`$\cong$'} of functors or
objects are canonical.
\smallskip

\noindent{\bf(i)} For\/ $f:X\ra Y$ and\/ $g:Y\ra Z,$ there are
natural isomorphisms of functors
\begin{align*}
R(g\ci f)_*&\cong Rg_*\ci Rf_*,& R(g\ci f)_!&\cong Rg_!\ci Rf_!,\\
(g\ci f)^*&\cong f^*\ci g^*,& (g\ci f)^!&\cong f^!\ci g^!.
\end{align*}

\noindent{\bf(ii)} If\/ $f:X\ra Y$ is proper then~$Rf_*\cong Rf_!$.
\smallskip

\noindent{\bf(iii)} If\/ $f:X\ra Y$ is \'etale then\/~$f^*\cong
f^!$. More generally, if\/ $f:X\ra Y$ is smooth of relative
(complex) dimension\/ $d,$ then\/ $f^*[d]\cong f^![-d],$ where
$f^*[d], f^![-d]$ are the functors $f^*,f^!$ shifted by\/~$\pm d$.
\smallskip

\noindent{\bf(iv)} If\/ $f:X\ra Y$ then $Rf_!\cong \bD_Y\ci
Rf_*\ci\bD_X$ and~$f^!\cong \bD_X\ci f^*\ci\bD_Y$.
\smallskip

\noindent{\bf(v)} If\/ $U$ is a smooth\/ $\C$-scheme
then\/~$\bD_U(A_U)\cong A_U[2\dim U]$.
\label{sm2thm1}
\end{thm}

If $X$ is a $\C$-scheme and $\cC^\bu\in D^b_c(X)$, the {\it
hypercohomology\/} $\bH^*(\cC^\bu)$ and {\it compactly-supported
hypercohomology\/} $\bH^*_\compact(\cC^\bu)$, both graded
$A$-modules, are
\e
\bH^k(\cC^\bu)=H^k(R\pi_*(\cC^\bu))\quad\text{and}\quad \bH^k_\compact(\cC^\bu)=H^k(R\pi_!(\cC^\bu))\quad\text{for $k\in\Z$,}
\label{sm2eq1}
\e
where $\pi:X\ra *$ is projection to a point. If $X$ is proper then
$\bH^*(\cC^\bu)\cong\bH^*_\compact(\cC^\bu)$ by Theorem
\ref{sm2thm1}(ii). They are related to usual cohomology by
$\bH^k(A_X)\cong H^k(X;A)$ and $\bH^k_\compact(A_X)\cong H^k_\compact(X;A)$. 
If $A$ is a field then under Verdier duality we
have~$\bH^k(\cC^\bu)\cong \bH^{-k}_\compact(\bD_X(\cC^\bu))^*$.

\subsection{Perverse sheaves on $\C$-schemes}
\label{sm22}

Next we review perverse sheaves, following Dimca~\cite[\S 5]{Dimc}.

\begin{dfn} Let $X$ be a $\C$-scheme, and for each $x\in X^\an$,
let $i_x:*\ra X$ map $i_x:*\mapsto x$. If $\cC^\bu\in D^b_c(X)$,
then the {\it support\/} $\supp^m\cC^\bu$ and {\it cosupport\/}
$\cosupp^m\cC^\bu$ of $\cH^m(\cC^\bu)$ for $m\in\Z$ are
\begin{align*}
\supp^m\cC^\bu&=\overline{\bigl\{x\in X^\an:\cH^m(i_x^*(\cC^\bu))
\ne 0\bigr\}},\\
\cosupp^m\cC^\bu&=\overline{\bigl\{x\in X^\an:\cH^m(i_x^!(\cC^\bu))
\ne 0\bigr\}},
\end{align*}
where $\overline{\{\cdots\}}$ means the closure in $X^\an$. If $A$
is a field then $\cosupp^m\cC^\bu=\supp^{-m}\bD_X(\cC^\bu)$. We call
$\cC^\bu$ {\it perverse}, or a {\it perverse sheaf}, if
$\dim_{\sst\C}\supp^{-m}\cC^\bu\!\le\!m$ and
$\dim_{\sst\C}\cosupp^m\cC^\bu\!\le\! m$  for all $m\in\Z$, where by
convention $\dim_{\sst\C}\es=-\iy$. Write $\Perv(X)$ for the full
subcategory of perverse sheaves in $D^b_c(X)$. Then $\Perv(X)$ is an
abelian category, the heart of a t-structure on~$D^b_c(X)$.
\label{sm2def2}
\end{dfn}

Perverse sheaves have the following properties:

\begin{thm}{\bf(a)} If\/ $A$ is a field then $\Perv(X)$ is
noetherian and artinian.
\smallskip

\noindent{\bf(b)} If\/ $A$ is a field then $\bD_X:D^b_c(X)\ra
D^b_c(X)$ maps $\Perv(X)\ra\Perv(X)$.
\smallskip

\noindent{\bf(c)} If\/ $i:X\hookra Y$ is inclusion of a closed\/
$\C$-subscheme, then $Ri_*,Ri_!$ (which are naturally isomorphic)
map $\Perv(X)\ra\Perv(Y)$.

Write $\Perv(Y)_X$ for the full subcategory of objects in $\Perv(Y)$
supported on $X$. Then $Ri_*\cong Ri_!$ are equivalences of
categories $\Perv(X)\,{\buildrel\sim\over\longra}\, \Perv(Y)_X$. The
restrictions $i^*\vert_{\Perv(Y)_X},i^!\vert_{\Perv(Y)_X}$ map\/
$\Perv(Y)_X\ra \Perv(X),$ are naturally isomorphic, and are
quasi-inverses for\/~$Ri_*,Ri_!:\Perv(X)\ra\Perv(Y)_X$.
\smallskip

\noindent{\bf(d)} If\/ $f:X\ra Y$ is \'etale then $f^*$ and\/ $f^!$
(which are naturally isomorphic) map $\Perv(Y)\ra\Perv(X)$. More
generally, if\/ $f:X\ra Y$ is smooth of relative dimension\/ $d,$
then\/ $f^*[d]\cong f^![-d]$ map $\Perv(Y)\ra\Perv(X)$.
\smallskip

\noindent{\bf(e)} $\smash{\boxtL}:D^b_c(X)\!\t\! D^b_c(Y)\!\ra\!
D^b_c(X\!\t\! Y)$ maps\/~$\Perv(X)\!\t\!\Perv(Y)\!\ra\!\Perv(X\!\t\!
Y)$.

\smallskip

\noindent{\bf(f)} Let\/ $U$ be a smooth\/ $\C$-scheme. Then
$A_U[\dim U]$ is perverse, where $A_U$ is the constant sheaf on $U$
with fibre $A,$ and\/ $[\dim U]$ means shift by $\dim U$ in the
triangulated category $D^b_c(X)$. Note that Theorem\/
{\rm\ref{sm2thm1}(v)} gives a canonical
isomorphism\/~$\bD_U\bigl(A_U[\dim U]\bigr)\cong A_U[\dim U]$.
\label{sm2thm2}
\end{thm}

When $A=\Q$, so that $\Perv(X)$ is noetherian and artinian by
Theorem \ref{sm2thm2}(a), the simple objects in $\Perv(X)$ admit a
complete description: they are all isomorphic to intersection
cohomology complexes $IC_{\smash{\bar V}}(\cL)$ for $V\subseteq X$
a smooth locally closed $\C$-subscheme and $\cL\ra V$ an irreducible
$\Q$-local system, \cite[\S 5.4]{Dimc}. Furthermore, if $f:X\ra Y$
is a proper morphism of $\C$-schemes, then the Decomposition Theorem
\cite[6.2.5]{BBD}, \cite[Th.~5.4.10]{Dimc}, \cite[Cor.~3]{Sait1}
says that, in case $IC_{\smash{\bar V}}(\cL)$ is of geometric origin, 
$Rf_*(IC_{\smash{\bar V}}(\cL))$ is isomorphic
to a finite direct sum of shifts of simple objects $IC_{\smash{\bar
V'}}(\cL')$ in $\Perv(Y)$. 

The next theorem is proved by Beilinson et al.\ \cite[Cor.~2.1.23,
\S 2.2.19, \& Th.~3.2.4]{BBD}. The analogue for $D^b_c(X)$ or $D(X)$
rather than $\Perv(X)$ is false. One moral is that perverse sheaves
behave like sheaves, rather than like complexes.

Theorem \ref{sm2thm3}(i) will be used throughout
\S\ref{sm3}--\S\ref{sm6}. Theorem \ref{sm2thm3}(ii) will be used
only once, in the proof of Theorem \ref{sm6thm6} in \S\ref{sm63},
and we only need Theorem \ref{sm2thm3}(ii) to hold in the Zariski
topology, rather than the \'etale topology.

\begin{thm} Let\/ $X$ be a $\C$-scheme. Then perverse sheaves on $X$
form a \begin{bfseries}stack\end{bfseries} (a kind of sheaf of
categories) on $X$ in the \'etale topology.

Explicitly, this means the following. Let $\{u_i:U_i\ra X\}_{i\in
I}$ be an \begin{bfseries}\'etale open cover\end{bfseries} for $X,$
so that\/ $u_i:U_i\ra X$ is an \'etale morphism of\/ $\C$-schemes
for\/ $i\in I$ with\/ $\coprod_iu_i$ surjective. Write
$U_{ij}=U_i\t_{u_i,X,u_j}U_j$ for $i,j\in I$ with projections
\begin{equation*}
\pi_{ij}^i:U_{ij}\longra U_i,\quad \pi_{ij}^j:U_{ij}\longra U_j,\quad
u_{ij}\!=\!u_i\ci\pi_{ij}^i\!=\!u_j\ci\pi_{ij}^j:U_{ij}\!\longra\! X.
\end{equation*}
Similarly, write $U_{ijk}=U_i\!\t_X\!U_j\!\t_X\!U_k$ for $i,j,k\in
I$ with projections
\begin{gather*}
\pi_{ijk}^{ij}:U_{ijk}\longra U_{ij},\quad
\pi_{ijk}^{ik}:U_{ijk}\longra
U_{ik},\quad \pi_{ijk}^{jk}:U_{ijk}\longra U_{jk},\\
\pi_{ijk}^i:U_{ijk}\ra U_i,\;\> \pi_{ijk}^j:U_{ijk}\ra U_j,\;\>
\pi_{ijk}^k:U_{ijk}\ra U_k, \;\> u_{ijk}:U_{ijk}\ra X,
\end{gather*}
so that\/ $\pi_{ijk}^i=\pi_{ij}^i\ci\pi_{ijk}^{ij},$
$u_{ijk}=u_{ij}\ci\pi_{ijk}^{ij}=u_i\ci\pi_{ijk}^i,$ and so on. All
these morphisms $u_i,\pi_{ij}^i,\ldots,u_{ijk}$ are \'etale, so by
Theorem\/ {\rm\ref{sm2thm2}(d)} $u_i^*\cong u_i^!$ maps
$\Perv(X)\ra\Perv(U_i),$ and similarly for $\pi_{ij}^i,\ldots,
u_{ijk}$. With this notation:
\smallskip

\noindent{\bf(i)} Suppose $\cP^\bu,\cQ^\bu\in\Perv(X),$ and we are
given $\al_i:u_i^*(\cP^\bu)\ra u_i^*(\cQ^\bu)$ in $\Perv(U_i)$ for
all\/ $i\in I$ such that for all\/ $i,j\in I$ we have
\begin{equation*}
(\pi_{ij}^i)^*(\al_i)= (\pi_{ij}^i)^*(\al_j):u_{ij}^*(\cP^\bu)\longra
u_{ij}^*(\cQ^\bu).
\end{equation*}
Then there is a unique $\al:\cP^\bu\ra\cQ^\bu$ in $\Perv(X)$ with\/
$\al_i=u_i^*(\al)$ for all\/~$i\in I$.
\smallskip

\noindent{\bf(ii)} Suppose we are given objects\/ $\cP^\bu_i\in
\Perv(U_i)$ for all\/ $i\in I$ and isomorphisms
$\al_{ij}:(\pi_{ij}^i)^*(\cP^\bu_i)\ra (\pi_{ij}^j)^*(\cP^\bu_j)$ in
$\Perv(U_{ij})$ for all\/ $i,j\in I$ with\/ $\al_{ii}=\id$ and\/
\begin{equation*}
(\pi_{ijk}^{jk})^*(\al_{jk})\ci (\pi_{ijk}^{ij})^*(\al_{ij})=
(\pi_{ijk}^{ik})^*(\al_{ik}):(\pi_{ijk}^i)^*(\cP_i)\longra
(\pi_{ijk}^k)^*(\cP_k)
\end{equation*}
in $\Perv(U_{ijk})$ for all\/ $i,j,k\in I$. Then there exists\/
$\cP^\bu$ in $\Perv(X),$ unique up to canonical isomorphism, with
isomorphisms $\be_i:u_i^*(\cP^\bu)\ra\cP^\bu_i$ for each\/ $i\in I,$
satisfying $\al_{ij}\ci(\pi_{ij}^i)^*(\be_i)=(\pi_{ij}^j)^*(\be_j):
u_{ij}^*(\cP^\bu)\ra(\pi_{ij}^j)^*(\cP^\bu_j)$ for all\/~$i,j\in I$.
\label{sm2thm3}
\end{thm}

We will need the following proposition in \S\ref{sm33} to prove
Theorem \ref{sm3thm}(b). Most of it is setting up notation, only the
last part $\al\vert_{X'}=\be\vert_{X'}$ is nontrivial.

\begin{prop} Let\/ $W,X$ be $\C$-schemes, $x\in X,$ and\/
$\pi_\C:W\ra\C,$ $\pi_X:W\ra X,$ $\io:\C\ra W$ morphisms, such
that\/ $\pi_\C\t\pi_X:W\ra\C\t X$ is \'etale, and\/
$\pi_\C\ci\io=\id_\C:\C\ra\C,$ and\/ $\pi_X\ci\io(t)=x$ for all\/
$t\in\C$. Write $W_t=\pi_\C^{-1}(t)\subset W$ for each\/ $t\in\C,$
and\/ $j_t:W_t\hookra W$ for the inclusion. Then
$\pi_X\vert_{W_t}=\pi_X\ci j_t:W_t\ra X$ is \'etale, and\/
$\io(t)\in W_t$ with\/ $\pi_X\vert_{W_t}(\io(t))=x,$ so we may think
of\/ $W_t$ for $t\in\C$ as a $1$-parameter family of \'etale open
neighbourhoods of\/ $x$ in~$X$.

Let\/ $\cP^\bu,\cQ^\bu\in\Perv(X),$ so that by Theorem\/
{\rm\ref{sm2thm2}(d)} as $\pi_X$ is smooth of relative dimension $1$
and\/ $\pi_X\vert_{W_t}$ is \'etale, we have
$\pi_X^*[1](\cP^\bu)\in\Perv(W)$ and\/ $\pi_X\vert_{W_t}^*(\cP^\bu)
=j_t^*[-1]\bigl(\pi_X^*[1](\cP^\bu)\bigr)\in\Perv(W_t),$ and
similarly for\/~$\cQ^\bu$.

Suppose $\al,\be:\cP^\bu\ra\cQ^\bu$ in $\Perv(X)$ and\/
$\ga:\pi_X^*[1](\cP^\bu)\ra\pi_X^*[1](\cQ^\bu)$ in $\Perv(W)$ are
morphisms such that\/ $\pi_X\vert_{W_0}^*(\al)=j_0^*[-1](\ga)$ in
$\Perv(W_0)$ and\/ $\pi_X\vert_{W_1}^*(\be)=j_1^*[-1](\ga)$ in
$\Perv(W_1)$. Then there exists a Zariski open neighbourhood\/ $X'$
of\/ $x$ in $X$ such that\/~$\al\vert_{X'}=\be\vert_{X'}:
\cP^\bu\vert_{X'}\ra\cQ^\bu\vert_{X'}$.
\label{sm2prop}
\end{prop}

Here we should think of $j_t^*[-1](\ga)$ for $t\in\C$ as a family of
perverse sheaf morphisms $\cP^\bu\ra\cQ^\bu$, defined near $x$ in
$X$ locally in the \'etale topology. But morphisms of perverse
sheaves are discrete (to see this, note that we can take $A=\Z$), so
as $j_t^*[-1](\ga)$ depends continuously on $t$, it should be
locally constant in $t$ near $x$, in a suitable sense. The
conclusion $\al\vert_{X'}=\be\vert_{X'}$ essentially says that
$j_0^*[-1](\ga)=j_1^*[-1](\ga)$ near $x$.

If $P\ra X$ is a principal $\Z/2\Z$-bundle on a $\C$-scheme $X$, and
$\cQ^\bu\in\Perv(X)$, we will define a perverse sheaf
$\cQ^\bu\ot_{\Z/2\Z}P$, which will be important
in~\S\ref{sm5}--\S\ref{sm6}.

\begin{dfn} Let $X$ be a $\C$-scheme. A {\it principal\/
$\Z/2\Z$-bundle\/} $P\ra X$ is a proper, surjective, \'etale
morphism of $\C$-schemes $\pi:P\ra X$ together with a free
involution $\si:P\ra P$, such that the orbits of $\Z/2\Z=\{1,\si\}$
are the fibres of $\pi$. We will use the ideas of {\it
isomorphism\/} of principal bundles $\io:P\ra P'$, {\it section\/}
$s:X\ra P$, {\it tensor product\/} $P\ot_{\Z/2\Z}P'$, and {\it
pullback\/} $f^*(P)\ra W$ under a $\C$-scheme morphism $f:W\ra X$,
all of which are defined in the obvious ways.

Let $P\ra X$ be a principal $\Z/2\Z$-bundle. Write $\cL_P\in
D^b_c(X)$ for the rank one $A$-local system on $X$ induced from $P$
by the nontrivial representation of $\Z/2\Z\cong\{\pm 1\}$ on $A$.
It is characterized by $\pi_*(A_P)\cong A_X\op\cL_P$. For each
$\cQ^\bu\in D^b_c(X)$, write $\cQ^\bu\ot_{\Z/2\Z}P\in D^b_c(X)$ for
$\cQ^\bu\otL\cL_P$, and call it $\cQ^\bu$ {\it twisted by\/} $P$. If
$\cQ^\bu$ is perverse then $\cQ^\bu\ot_{\Z/2\Z}P$ is perverse.

Perverse sheaves and complexes twisted by principal $\Z/2\Z$-bundles
have the obvious functorial behaviour. For example, if $P\ra X$,
$P'\ra X$ are principal $\Z/2\Z$-bundles and $\cQ^\bu\in D^b_c(X)$
there is a canonical isomorphism $(\cQ^\bu\ot_{\Z/2\Z}P)
\ot_{\Z/2\Z}P'\cong\cQ^\bu\ot_{\Z/2\Z}(P\ot_{\Z/2\Z}P')$, and if
$f:W\ra X$ is a $\C$-scheme morphism there is a canonical
isomorphism $f^*(\cQ^\bu\ot_{\Z/2\Z}P)\cong f^*(\cQ^\bu)\ot_{\Z/2\Z}
f^*(P)$.
\label{sm2def3}
\end{dfn}

\subsection{Nearby cycles and vanishing cycles on $\C$-schemes}
\label{sm23}

We explain nearby cycles and vanishing cycles, as in Dimca \cite[\S
4.2]{Dimc}. The definition is complex analytic,
$\widetilde{X_*^\an},\widetilde{\C^*}$ in \eq{sm2eq2} do not come
from $\C$-schemes.

\begin{dfn} Let $X$ be a $\C$-scheme, and $f:X\ra\C$ a
regular function. Define $X_0=f^{-1}(0)$, as a $\C$-subscheme of
$X$, and $X_*=X\sm X_0$. Consider the commutative diagram of complex
analytic spaces:
\e
\begin{gathered}
\xymatrix@R=15pt@C=30pt{ X_0^\an \ar[r]_i \ar[d]^(0.45)f & X^\an
\ar[d]^(0.45)f & X_*^\an \ar[l]^j \ar[d]^(0.45)f &
{\widetilde{X_*^\an}\phantom{.}} \ar[l]^p \ar@/_.7pc/[ll]_\pi
\ar[d]^(0.45){\ti f} \\ \{0\} \ar[r] & \C & \C^* \ar[l] &
{\widetilde{\C^*}.} \ar[l]_\rho}
\end{gathered}
\label{sm2eq2}
\e
Here $X^\an,X_0^\an,X_*^\an$ are the complex analytic spaces
associated to the $\C$-schemes $X_0,X,X_*$, and $i:X_0^\an\hookra
X^\an$, $j:X_*^\an\hookra X^\an$ are the inclusions,
$\rho:\widetilde{\C^*}\ra\C^*$ is the universal cover of
$\C^*=\C\sm\{0\}$, and $\widetilde{X_*^\an}=X_*^\an\t_{f,\C^*,\rho}
\widetilde{\C^*}$ the corresponding cover of $X_*^\an$, with
covering map $p:\widetilde{X_*^\an}\ra X_*^\an$, and $\pi=j\ci p$.

As in \S\ref{sm26}, the triangulated categories $D(X),D^b_c(X)$ and
six operations $f^*,f^!,Rf_*,Rf_!,{\cal RH}om,\otL$ also make sense
for complex analytic spaces. So we can define the {\it nearby cycle
functor\/} $\psi_f:D^b_c(X)\ra D^b_c(X_0)$ to be $\psi_f=i^*\ci
R\pi_*\ci \pi^*$. Since this definition goes via
$\widetilde{X_*^\an}$ which is not a $\C$-scheme, it is not obvious
that $\psi_f$ maps to (algebraically) constructible complexes
$D^b_c(X_0)$ rather than just to $D(X_0)$, but it does
\cite[p.~103]{Dimc}, \cite[p.~352]{KaSc1}.

There is a natural transformation $\Xi:i^*\Ra\psi_f $ between the
functors $i^*,\psi_f:D^b_c(X)\ra D^b_c(X_0)$. The {\it vanishing
cycle functor\/} $\phi_f:D^b_c(X)\ra D^b_c(X_0)$ is a functor such
that for every $\cC^\bu$ in $D^b_c(X)$ we have a distinguished
triangle
\begin{equation*}
\smash{\xymatrix@C=40pt{i^*(\cC^\bu) \ar[r]^{\Xi(\cC^\bu)} & \psi_f
(\cC^\bu) \ar[r] & \phi_f (\cC^\bu) \ar[r]^{[+1]} & i^*(\cC^\bu)}}
\end{equation*}
in $D^b_c(X_0)$. Following Dimca \cite[p.~108]{Dimc}, we write
$\psi_f^p,\phi_f^p$ for the shifted functors~$\psi_f[-1],\ab
\phi_f[-1]:D^b_c(X)\ra D^b_c(X_0)$.

The generator of $\Z=\pi_1(\C^*)$ on $\widetilde{\C^*}$ induces a
deck transformation $\de_{\smash{\C^*}}:\widetilde{\C^*}\ra
\widetilde{\C^*}$ which lifts to a deck transformation
$\de_{\smash{X^*}}:\widetilde{X^*}\ra\widetilde{X^*}$ with
$p\ci\de_{\smash{X^*}}=p$ and $\ti f\ci\de_{\smash{X^*}}=
\de_{\smash{\C^*}}\ci\ti f$. As in \cite[p.~103, p.~105]{Dimc}, we
can use $\de_{\smash{X^*}}$ to define natural transformations
$M_{X,f}:\psi^p_f\Ra\psi^p_f$ and $M_{X,f}:\phi^p_f\Ra\phi^p_f$,
called {\it monodromy}.
\label{sm2def4}
\end{dfn}

Alternative definitions of $\psi_f,\phi_f$ in terms of
specialization and microlocalization functors are given by Kashiwara
and Schapira \cite[Prop.~8.6.3]{KaSc1}. Here are some properties of
nearby and vanishing cycles. Parts (i),(ii) can be found in Dimca
\cite[Th.~5.2.21 \& Prop.~4.2.11]{Dimc}. Part (iv) is proved by
Massey \cite{Mass3}; compare also Proposition \ref{smAprop1}
in the Appendix. 

\begin{thm}{\bf(i)} If\/ $X$ is a $\C$-scheme and\/ $f:X\ra\C$ is
regular, then\/ $\psi_f^p,\phi_f^p:D^b_c(X)\ra D^b_c(X_0)$ both
map\/~$\Perv(X)\ra\Perv(X_0)$.
\smallskip

\noindent{\bf(ii)} Let\/ $\Phi:X\ra Y$ be a proper morphism of\/
$\C$-schemes, and\/ $g:Y\ra\C$ be regular. Write
$f=g\ci\Phi:X\ra\C,$ $X_0=f^{-1}(0)\subseteq X,$
$Y_0=g^{-1}(0)\subseteq Y,$ and\/ $\Phi_0=\Phi\vert_{X_0}:X_0\ra
Y_0$. Then we have natural isomorphisms
\e
R(\Phi_0)_*\ci\psi^p_f\cong \psi^p_g\ci
R\Phi_*\quad\text{and\/}\quad R(\Phi_0)_*\ci\phi^p_f\cong
\phi^p_g\ci R\Phi_*.
\label{sm2eq3}
\e
Note too that\/ $R\Phi_*\cong R\Phi_!$ and\/ $R(\Phi_0)_*\cong
R(\Phi_0)_!,$ as $\Phi,\Phi_0$ are proper.
\smallskip

\noindent{\bf(iii)} Let\/ $\Phi:X\ra Y$ be an \'etale morphism of\/
$\C$-schemes, and\/ $g:Y\ra\C$ be regular. Write
$f=g\ci\Phi:X\ra\C,$ $X_0=f^{-1}(0)\subseteq X,$
$Y_0=g^{-1}(0)\subseteq Y,$ and\/ $\Phi_0=\Phi\vert_{X_0}:X_0\ra
Y_0$. Then we have natural isomorphisms
\e
\Phi_0^*\ci\psi^p_f\cong\psi^p_g\ci\Phi^*\quad\text{and\/}\quad
\Phi_0^*\ci\phi^p_f\cong\phi^p_g\ci\Phi^*.
\label{sm2eq4}
\e
Note too that\/ $\Phi^*\cong\Phi^!$ and\/ $\Phi_0^*\cong\Phi_0^!,$
as $\Phi,\Phi_0$ are \'etale.

More generally, if\/ $\Phi:X\ra Y$ is smooth of relative (complex)
dimension $d$ and $g,f,X_0,Y_0,\Phi_0$ are as above then we have
natural isomorphisms
\e
\Phi_0^*[d]\ci\psi^p_f\cong\psi^p_g\ci\Phi^*[d]\quad\text{and\/}\quad
\Phi_0^*[d]\ci\phi^p_f\cong\phi^p_g\ci\Phi^*[d].
\label{sm2eq5}
\e
Note too that\/ $\Phi^*[d]\cong\Phi^![-d]$
and\/~$\Phi_0^*[d]\cong\Phi_0^![-d]$.
\smallskip

\noindent{\bf(iv)} If\/ $X$ is a $\C$-scheme and\/ $f:X\ra\C$ is
regular, then there are natural isomorphisms $\psi_f^p\ci\bD_X\cong
\bD_{X_0}\ci\psi_f^p$ and\/~$\phi_f^p\ci\bD_X\cong
\bD_{X_0}\ci\phi_f^p$.
\label{sm2thm4}
\end{thm}

\subsection{Perverse sheaves of vanishing cycles on $\C$-schemes}
\label{sm24}

We can now define the main subject of this paper, the perverse sheaf
of vanishing cycles $\PV^\bu_{U,f}$ for a regular
function~$f:U\ra\C$.

\begin{dfn} Let $U$ be a smooth $\C$-scheme, and $f:U\ra\C$ a
regular function. Write $X=\Crit(f)$, as a closed $\C$-subscheme of
$U$. Then as a map of topological spaces, $f\vert_X:X\ra\C$ is
locally constant, with finite image $f(X)$, so we have a
decomposition $X=\coprod_{c\in f(X)}X_c$, for $X_c\subseteq X$ the
open and closed $\C$-subscheme with $f(x)=c$ for each
$\C$-point~$x\in X_c$.

(Note that if $X$ is non-reduced, then $f\vert_X:X\ra\C$ need not be
locally constant as a morphism of $\C$-schemes, but
$f\vert_{X^\red}:X^\red\ra\C$ is locally constant, where $X^\red$ is
the reduced $\C$-subscheme of $X$. Since $X,X^\red$ have the same
topological space, $f\vert_X:X\ra\C$ is locally constant on
topological spaces.)

For each $c\in\C$, write $U_c=f^{-1}(c)\subseteq U$. Then as in
\S\ref{sm23}, we have a vanishing cycle functor
$\phi_{f-c}^p:\Perv(U)\ra\Perv(U_c)$. So we may form
$\phi_{f-c}^p(A_U[\dim U])$ in $\Perv(U_c)$, since $A_U[\dim
U]\in\Perv(U)$ by Theorem \ref{sm2thm2}(f). One can show
$\phi_{f-c}^p(A_U[\dim U])$ is supported on the closed subset
$X_c=\Crit(f)\cap U_c$ in $U_c$, where $X_c=\es$ unless $c\in f(X)$.
That is, $\phi_{f-c}^p(A_U[\dim U])$ lies in~$\Perv(U_c)_{X_c}$.

But Theorem \ref{sm2thm2}(c) says $\Perv(U_c)_{X_c}$ and
$\Perv(X_c)$ are equivalent categories, so we may regard
$\phi_{f-c}^p(A_U[\dim U])$ as a perverse sheaf on $X_c$. That is,
we can consider $\phi_{f-c}^p(A_U[\dim U])\vert_{X_c}=i_{X_c,U_c}^*
\bigl(\phi_{f-c}^p(A_U[\dim U])\bigr)$ in $\Perv(X_c)$, where
$i_{X_c,U_c}:X_c\ra U_c$ is the inclusion morphism.

As $X=\coprod_{c\in f(X)}X_c$ with each $X_c$ open and closed in
$X$, we have $\Perv(X)=\bigop_{c\in f(X)}\Perv(X_c)$. Define the
{\it perverse sheaf of vanishing cycles\/ $\PV_{U,f}^\bu$ of\/}
$U,f$ in $\Perv(X)$ to be $\PV_{U,f}^\bu=\bigop_{c\in
f(X)}\phi_{f-c}^p(A_U[\dim U])\vert_{X_c}$. That is, $\PV_{U,f}^\bu$
is the unique perverse sheaf on $X=\Crit(f)$ with
$\PV_{U,f}^\bu\vert_{X_c}=\phi_{f-c}^p(A_U[\dim U])\vert_{X_c}$ for
all~$c\in f(X)$.

Under Verdier duality, we have $A_U[\dim U]\cong\bD_U(A_U[\dim U])$
by Theorem \ref{sm2thm2}(f), so $\phi_{f-c}^p(A_U[\dim U])
\cong\bD_{U_c}\bigl(\phi_{f-c}^p(A_U[\dim U])\bigr)$ by Theorem
\ref{sm2thm4}(iv). Applying $i_{X_c,U_c}^*$ and using $\bD_{X_c}\ci
i_{X_c,U_c}^*\cong i_{X_c,U_c}^!\ci\bD_{U_c}$ by Theorem
\ref{sm2thm1}(iv) and $i_{X_c,U_c}^!\cong i_{X_c,U_c}^*$ on
$\Perv(U_c)_{X_c}$ by Theorem \ref{sm2thm2}(c) also gives
\begin{equation*}
\phi_{f-c}^p(A_U[\dim U])\vert_{X_c}\cong \bD_{X_c}\bigl(
\phi_{f-c}^p(A_U[\dim U])\vert_{X_c}\bigr).
\end{equation*}
Summing over all $c\in f(X)$ yields a canonical isomorphism
\e
\si_{U,f}:\PV_{U,f}^\bu\,{\buildrel\cong\over\longra}\,\bD_X(\PV_{U,f}^\bu).
\label{sm2eq6}
\e

For $c\in f(X)$, we have a monodromy operator
$M_{U,f-c}:\phi_{f-c}^p(A_U[\dim U])\ab\ra \phi_{f-c}^p(A_U[\dim
U])$, which restricts to $\phi_{f-c}^p(A_U[\dim U])\vert_{X_c}$.
Define the {\it twisted monodromy operator\/}
$\tau_{U,f}:\PV_{U,f}^\bu\ra\PV_{U,f}^\bu$ by
\e
\begin{split}
\tau_{U,f}\vert_{X_c}&=(-1)^{\dim U}M_{U,f-c}\vert_{X_c}:\\
&\phi_{f-c}^p(A_U[\dim U])\vert_{X_c}\longra \phi_{f-c}^p(A_U[\dim
U])\vert_{X_c},
\end{split}
\label{sm2eq7}
\e
for each $c\in f(X)$. Here `twisted' refers to the sign $(-1)^{\dim
U}$ in \eq{sm2eq7}. We include this sign change as it makes
monodromy act naturally under transformations which change dimension
--- without it, equation \eq{sm5eq15} below would only commute up to
a sign $(-1)^{\dim V-\dim U}$, not commute --- and it normalizes the
monodromy of any nondegenerate quadratic form to be the identity, as
in \eq{sm2eq13}. The sign $(-1)^{\dim U}$ also corresponds to the
twist `$(\ha\dim U)$' in the definition \eq{sm2eq24} of the mixed
Hodge module of vanishing cycles $\HV_{U,f}^{\bu}$ in~\S\ref{sm210}.

\label{sm2def5}
\end{dfn}

The (compactly-supported) hypercohomology $\bH^*(\PV_{U,f}^\bu),
\bH^*_\compact(\PV_{U,f}^\bu)$ from \eq{sm2eq1} is an important
invariant of $U,f$. If $A$ is a field then the isomorphism
$\si_{U,f}$ in \eq{sm2eq6} implies that $\bH^k(\PV_{U,f}^\bu)
\cong\bH^{-k}_\compact(\PV_{U,f}^\bu)^*$, a form of Poincar\'e
duality.

We defined $\smash{\PV_{U,f}^\bu}$ in perverse sheaves over a base
ring $A$. Writing $\PV_{U,f}^\bu(A)$ to denote the base ring, one
can show that $\PV_{U,f}^\bu(A)\cong \PV_{U,f}^\bu(\Z)\otL_\Z A$.
Thus, we may as well take $A=\Z$, or $A=\Q$ if we want $A$ to be a
field, since the case of general $A$ contains no more information.

There is a Thom--Sebastiani Theorem for perverse sheaves, due to
Massey \cite{Mass1} and Sch\"urmann \cite[Cor.~1.3.4]{Schu}. Applied
to $\PV_{U,f}^\bu$, it yields:

\begin{thm} Let\/ $U,V$ be smooth\/ $\C$-schemes and\/ $f:U\ra\C,$
$g:V\ra\C$ be regular, so that\/ $f\boxplus g:U\t V\ra\C$ is regular
with\/ $(f\boxplus g)(u,v):=f(u)+g(v)$. Set\/ $X=\Crit(f)$ and\/
$Y=\Crit(g)$ as $\C$-subschemes of\/ $U,V,$ so that\/
$\Crit(f\boxplus g)=X\t Y$. Then there is a natural isomorphism
\e
\smash{\TS_{U,f,V,g}:\PV_{U\t V,f\boxplus g}^\bu\longra
\PV_{U,f}^\bu\boxtL \PV_{V,g}^\bu}
\label{sm2eq8}
\e
in $\Perv(X\t Y),$ such that the following diagrams commute:
\ea
\begin{gathered}
{}\!\!\!\!\!\!\!\!\!\!\!\!\!\!\! \xymatrix@!0@C=58pt@R=50pt{
*+[r]{\PV_{U\t V,f\boxplus g}^\bu} \ar[rrrrr]_{\si_{U\t V,f\boxplus
g}} \ar[d]^(0.4){\TS_{U,f,V,g}} &&&&&
*+[l]{\bD_{X\t Y}(\PV_{U\t V,f\boxplus g}^\bu)} \\
*+[r]{\raisebox{25pt}{\quad\,\,\,$\begin{subarray}{l}\ts\PV_{U,f}^\bu\boxtL \\
\ts\PV_{V,g}^\bu\end{subarray}$}}
\ar[rr]^(0.57){\si_{U,f}\boxtL\si_{V,g}} &&
*+[r]{\raisebox{25pt}{${}\,\,\,\begin{subarray}{l}\ts
\bD_X(\PV_{U,f}^\bu)\boxtL\!\!\!\!\!{} \\
\ts \bD_Y(\PV_{V,g}^\bu)\end{subarray}$}} \ar[rrr]^(0.33)\cong &&&
*+[l]{\bD_{X\t Y}\bigl(\PV_{U,f}^\bu\boxtL \PV_{V,g}^\bu\bigr),\!\!{}}
\ar[u]^{\bD_{X\t Y}(\TS_{U,f,V,g})} }\!\!\!\!\!{}
\end{gathered}
\label{sm2eq9}
\\[-10pt]
\begin{gathered}
{}\!\!\!\!\xymatrix@!0@C=290pt@R=40pt{ *+[r]{\PV_{U\t V,f\boxplus
g}^\bu} \ar[r]_{\tau_{U\t V,f\boxplus g}} \ar[d]^{\TS_{U,f,V,g}} &
*+[l]{\PV_{U\t V,f\boxplus g}^\bu} \ar[d]_{\TS_{U,f,V,g}} \\
*+[r]{\PV_{U,f}^\bu\boxtL \PV_{V,g}^\bu} \ar[r]^{\tau_{U,f}\boxtL\tau_{V,g}}
& *+[l]{\PV_{U,f}^\bu\boxtL \PV_{V,g}^\bu.\!\!{}} }\!\!\!\!\!{}
\end{gathered}
\label{sm2eq10}
\ea
\label{sm2thm5}
\end{thm}

The next example will be important later.

\begin{ex} Define $f:\C^n\ra\C$ by
$f(z_1,\ldots,z_n)=z_1^2+\cdots+z_n^2$ for $n>1$. Then
$\Crit(f)=\{0\}$, so $\PV_{\C^n,z_1^2+\cdots+z_n^2}^\bu=
\phi_f^p(A_{\C^n}[n])\vert_{\{0\}}$ is a perverse sheaf on the point
$\{0\}$. Following Dimca \cite[Prop.~4.2.2, Ex.~4.2.3 \&
Ex.~4.2.6]{Dimc}, we find that there is a canonical isomorphism
\e
\PV_{\C^n,z_1^2+\cdots+z_n^2}^\bu\cong
H^{n-1}\bigl(MF_f(0);A\bigr)\ot_AA_{\{0\}},
\label{sm2eq11}
\e
where $MF_f(0)$ is the {\it Milnor fibre\/} of $f$ at 0, as in
\cite[p.~103]{Dimc}. Since $f(z)=z_1^2+\cdots+z_n^2$ is homogeneous,
we see that
\begin{equation*}
MF_f(0)\cong\bigl\{(z_1,\ldots,z_n)\in\C^n:
f(z_1,\ldots,z_n)=1\bigr\}\cong T^*{\cal S}^{n-1},
\end{equation*}
so that $H^{n-1}\bigl(MF_f(0);A\bigr)\cong H^{n-1}\bigl({\cal
S}^{n-1};A\bigr)\cong A$. Therefore we have
\e
\PV_{\C^n,z_1^2+\cdots+z_n^2}^\bu\cong A_{\{0\}}.
\label{sm2eq12}
\e

This isomorphism \eq{sm2eq12} is {\it natural up to sign} (unless
the base ring $A$ has characteristic 2, in which case \eq{sm2eq12}
is natural), as it depends on the choice of isomorphism
$H^{n-1}({\cal S}^{n-1},A) \cong A$, which corresponds to an
orientation for ${\cal S}^{n-1}$. This uncertainty of signs will be
important in~\S\ref{sm5}--\S\ref{sm6}.

We can also use Milnor fibres to compute the monodromy operator on
$\PV_{\C^n,z_1^2+\cdots+z_n^2}^\bu$. There is a monodromy map
$\mu_f:MF_f(0)\ra MF_f(0)$, natural up to isotopy, which is the
monodromy in the Milnor fibration of $f$ at 0. Under the
identification $MF_f(0)\cong T^*{\cal S}^{n-1}$ we may take $\mu_f$
to be the map $\d(-1): T^*{\cal S}^{n-1}\ra T^*{\cal S}^{n-1}$
induced by $-1:{\cal S}^{n-1}\ra {\cal S}^{n-1}$ mapping
$-1:(x_1,\ldots,x_n)\mapsto (-x_1,\ldots,-x_n)$. This multiplies
orientations on ${\cal S}^{n-1}$ by $(-1)^n$. Thus,
$\mu_{f*}:H^{n-1}({\cal S}^{n-1},A)\ra H^{n-1}({\cal S}^{n-1},A)$
multiplies by $(-1)^n$.

By \cite[Prop.~4.2.2]{Dimc}, equation \eq{sm2eq11} identifies the
action of the monodromy operator $M_{\C^n,f}\vert_{\{0\}}$ on
$\PV_{\C^n,z_1^2+\cdots+z_n^2}^\bu$ with the action of $\mu_{f*}$ on
$H^{n-1}({\cal S}^{n-1},A)$. So $M_{\C^n,f}\vert_{\{0\}}$ is
multiplication by $(-1)^n$. Combining this with the sign change
$(-1)^{\dim U}$ in \eq{sm2eq7} for $U=\C^n$ shows that the twisted
monodromy is
\e
\tau_{\C^n,z_1^2+\cdots+z_n^2}=\id:
\PV_{\C^n,z_1^2+\cdots+z_n^2}^\bu \longra
\PV_{\C^n,z_1^2+\cdots+z_n^2}^\bu.
\label{sm2eq13}
\e

Equations \eq{sm2eq12}--\eq{sm2eq13} also hold for $n=0,1$, though
\eq{sm2eq11} does not.

Note also that these results are compatible with the Thom--Sebastiani Theorem~\ref{sm2thm5}, 
and can be deduced from it and the case $n=1$.
\label{sm2ex1}
\end{ex}

We introduce some notation for pullbacks of $\PV_{V,g}^\bu$ by
\'etale morphisms.

\begin{dfn} Let $U,V$ be smooth $\C$-schemes, $\Phi:U\ra V$ an
\'etale morphism, and $g:V\ra\C$ a regular function. Write
$f=g\ci\Phi:U\ra\C$, and $X=\Crit(f)$, $Y=\Crit(g)$ as
$\C$-subschemes of $U,V$. Then $\Phi\vert_X:X\ra Y$ is \'etale.
Define an isomorphism
\e
\PV_\Phi:\PV_{U,f}^\bu\longra\Phi\vert_X^*\bigl(\PV_{V,g}^\bu
\bigr)\quad \text{in $\Perv(X)$}
\label{sm2eq14}
\e
by the commutative diagram for each $c\in f(X)\subseteq g(Y)$:
\e
\begin{gathered}
\xymatrix@C=145pt@R=17pt{
*+[r]{\PV_{U,f}^\bu\vert_{X_c}\!=\!\phi_{f-c}^p(A_U[\dim U])
\vert_{X_c}} \ar[r]_(0.53)\al \ar@<2ex>[d]^{\PV_\Phi\vert_{X_c}} &
*+[l]{\phi_{f-c}^p\ci\Phi^*(A_V[\dim V]))\vert_{X_c}}
\ar@<-2ex>[d]_\be \\
*+[r]{\Phi\vert_{X_c}^*\bigl(\PV_{V,g}^\bu\bigr)} \ar@{=}[r] &
*+[l]{\Phi_0^*\ci\phi_{g-c}^p\ci(A_V[\dim V]))\vert_{X_c}.\!\!{}}}\!\!\!\!{}
\end{gathered}
\label{sm2eq15}
\e
Here $\al$ is $\phi_{f-c}^p$ applied to the canonical isomorphism
$A_U\ra\Phi^*(A_V)$, noting that $\dim U=\dim V$ as $\Phi$ is
\'etale, and $\be$ is induced by \eq{sm2eq4}.

By naturality of the isomorphisms $\al,\be$ in \eq{sm2eq15} we find
the following commute, where $\si_{U,f},\tau_{U,f}$ are as
in~\eq{sm2eq6}--\eq{sm2eq7}:
\ea
\begin{gathered}
\xymatrix@!0@C=140pt@R=35pt{ *+[r]{\PV_{U,f}^\bu}
\ar[rr]_{\si_{U,f}} \ar[d]^{\PV_\Phi} &&
*+[l]{\bD_X(\PV_{U,f}^\bu)} \\
*+[r]{\Phi\vert_X^*\bigl(\PV_{V,g}^\bu\bigr)}
\ar[r]^{\Phi\vert_X^*(\si_{V,g})} &
{\Phi\vert_X^*\bigl(\bD_Y(\PV_{V,g}^\bu)\bigr)} \ar[r]^(0.37)\cong &
*+[l]{\bD_X\bigl(\Phi\vert_X^*(\PV_{V,g}^\bu)\bigr),\!\!{}}
\ar[u]^{\bD_X(\PV_\Phi)} }
\end{gathered}
\label{sm2eq16}\\
\begin{gathered}
\xymatrix@!0@C=280pt@R=35pt{ *+[r]{\PV_{U,f}^\bu}
\ar[r]_{\tau_{U,f}} \ar[d]^{\PV_\Phi} &
*+[l]{\PV_{U,f}^\bu} \ar[d]_{\PV_\Phi} \\
*+[r]{\Phi\vert_X^*(\PV_{V,g}^\bu)} \ar[r]^{\Phi\vert_X^*(\tau_{V,g})}
& *+[l]{\Phi\vert_X^*(\PV_{V,g}^\bu).\!\!{}} }
\end{gathered}
\label{sm2eq17}
\ea

If $U=V$, $f=g$ and $\Phi=\id_U$
then~$\PV_{\id_U}=\id_{\PV_{U,f}^\bu}$.

If $W$ is another smooth $\C$-scheme, $\Psi:V\ra W$ is \'etale, and
$h:W\ra\C$ is regular with $g=h\ci\Psi:V\ra\C$, then composing
\eq{sm2eq15} for $\Phi$ with $\Phi\vert_{X_c}^*$ of \eq{sm2eq15} for
$\Psi$ shows that
\e
\PV_{\Psi\ci\Phi}=\Phi\vert_X^*(\PV_\Psi)\ci\PV_\Phi:
\PV_{U,f}^\bu\longra(\Psi\ci\Phi)\vert_X^*\bigl(\PV_{W,h}^\bu\bigr).
\label{sm2eq18}
\e
That is, the isomorphisms $\PV_\Phi$ are functorial.
\label{sm2def6}
\end{dfn}

\begin{ex} In Definition \ref{sm2def6}, set $U=V=\C^n$ and
$f(z_1,\ldots,z_n)=g(z_1,\ldots,z_n)=z_1^2+\cdots+z_n^2$, so that
$Y=Z=\{0\}\subset\C^n$. Let $M\in{\rm O}(n,\C)$ be an orthogonal
matrix, so that $M:\C^n\ra\C^n$ is an isomorphism with $f=g\ci M$
and $M\vert_{\{0\}}=\id_{\{0\}}$. As $M\vert_Y=\id_Y$, Definition
\ref{sm2def6} defines an isomorphism
\e
\PV_M:\PV_{\C^n,z_1^2+\cdots+z_n^2}^\bu \longra
\PV_{\C^n,z_1^2+\cdots+z_n^2}^\bu.
\label{sm2eq19}
\e

Equation \eq{sm2eq11} describes $\PV_{\C^n,z_1^2+\cdots+z_n^2}^\bu$
in terms of $MF_f(0)\cong T^*{\cal S}^{n-1}$. Now
$M\vert_{MF_f(0)}:MF_f(0)\ra MF_f(0)$ multiplies orientations on
${\cal S}^{n-1}$ by $\det M$, so
$(M\vert_{MF_f(0)})_*:H^{n-1}\bigl(MF_f(0);A\bigr)\ra
H^{n-1}\bigl(MF_f(0); A\bigr)$ is multiplication by $\det M$. Thus
\eq{sm2eq11} implies that $\PV_M$ in \eq{sm2eq19} is multiplication
by~$\det M=\pm 1$.

\label{sm2ex2}
\end{ex}

\subsection{Summary of the properties we use in this paper}
\label{sm25}

Since parts of \S\ref{sm21}--\S\ref{sm24} do not work for the other
kinds of perverse sheaves, $\cD$-modules and mixed Hodge modules in
\S\ref{sm26}--\S\ref{sm210}, we list what we will need for
\S\ref{sm3}--\S\ref{sm6}, to make it easy to check they are also
valid in the settings of~\S\ref{sm26}--\S\ref{sm210}.
\begin{itemize}
\setlength{\itemsep}{0pt}
\setlength{\parsep}{0pt}
\item[(i)] There should be an $A$-linear abelian category
$\PP(X)$ of $\PP$-objects defined for each scheme or complex
analytic space $X$, over a fixed, well-behaved base ring $A$. We
do not require $A$ to be a field.
\item[(ii)] There should be a Verdier duality functor $\bD_X$
with $\bD_X\ci\bD_X\cong\id$, defined on a suitable subcategory
of $\PP$-objects on $X$ which includes the objects we are
interested in. We do not need $\bD_X$ to be defined on all
objects in~$\PP(X)$.
\item[(iii)] If $U$ is a smooth scheme or complex manifold,
then there should be a canonical object $A_U[\dim U]\in\PP(U)$,
with a canonical isomorphism
\begin{equation*}
\bD_U(A_U[\dim U])\cong A_U[\dim U].
\end{equation*}
\item[(iv)] Let $f:X\hookra Y$ be a closed embedding of schemes
or complex analytic spaces; this implies $f$ is proper. Then
$f_*,f_!:\PP(X)\ra \PP(Y)$ should exist, inducing an equivalence
of categories $\PP(X)\,{\buildrel\sim\over\longra}\, \PP_X(Y)$
as in Theorem \ref{sm2thm2}(c), where $\PP_X(Y)$ is the full
subcategory of objects in $\PP(Y)$ supported on $X$.
\item[(v)] Let $f:X\ra Y$ be an \'etale morphism. Then the
pullbacks $f^*,f^!:\PP(Y)\ra \PP(X)$ should exist. More
generally, if $f:X\ra Y$ is smooth of relative dimension $d$,
then there should be pullbacks $f^*[d],f^![-d]$ mapping
$\PP(Y)\ra\PP(X)$. If $X,Y$ are smooth, there should be a
canonical isomorphism $f^*[d](A_Y[\dim Y])\cong A_X[\dim X]$. We
do not need pullbacks to exist for general morphisms $f:X\ra Y$,
though see (xi) below.
\item[(vi)] An external tensor product $\smash{\boxtL}:\PP(X)\t
\PP(Y)\ra \PP(X\t Y)$ should exist for all $X,Y$.
\item[(vii)] If $X$ is a scheme or complex analytic space,
$P\ra X$ a principal $\Z/2\Z$-bundle, and $\cQ^\bu\in\PP(X)$,
the twisted perverse sheaf $\cQ^\bu\ot_{\Z/2\Z}P\in\PP(X)$
should make sense as in Definition \ref{sm2def3}, and have the
obvious functorial properties.
\item[(viii)] A vanishing cycle functor $\phi_f^p:\PP(U)\ra
\PP(U_0)$ and monodromy transformation $M_{U,f}:\phi^p_f
\Ra\phi^p_f$ in \S\ref{sm24} should exist for all smooth $U$ and
regular/holomorphic $f:U\ra\bA^1$.
\item[(ix)] The functors $\bD_X,f^*,f^!,f_*,f_!,\phi_f^p$
should satisfy the natural isomorphisms in Theorems
\ref{sm2thm1} and \ref{sm2thm4}, provided they exist. They
should have the obvious compatibilities with $\boxtL$, and
restriction to (Zariski) open sets.
\item[(x)] There should be suitable subcategories of
$\PP$-objects which form a stack in the \'etale or complex
analytic topologies, as in Theorem \ref{sm2thm3}. In the
algebraic case we only need Theorem \ref{sm2thm3}(ii) to hold
for Zariski open covers, not \'etale open covers.
\item[(xi)] Proposition \ref{sm2prop} must hold. This involves
pullbacks $j_t^*$ by a morphism $j_t:W_t\hookra W$ which is not
\'etale or smooth, as in (v) above. But on objects we only
consider $j_t^*\bigl(\pi_X^*(\cP^\bu)\bigr)=\pi_X\vert_{W_t}^*
(\cP^\bu)$ which exists in $\PP(W_t)$ by (v) as
$\pi_X\vert_{W_t}$ is \'etale, so $j_t^*$ is defined on the
objects we need.
\item[(xii)] There should be a Thom--Sebastiani Theorem for
$\PP$-objects, so that the analogue of Theorem~\ref{sm2thm5}
holds.
\end{itemize}

\begin{rem} The existence of a (bounded) derived category of
$\PP$-objects will not be assumed, or used, in this paper. On the
other hand, in all the cases we consider, there will be a
realization functor from the category of $\PP$-objects to an
appropriate category of constructible complexes, and the notation
used above reflects this. So in (iii),(v) above, $[1]$ does not
stand for a shift in any derived category; the notation means a
$\PP$-object or morphism whose realization is the appropriate
constructible object or morphism. See Remark~\ref{sm2rem4} below.
\label{sm2rem3}
\end{rem}

\subsection{Perverse sheaves on complex analytic spaces}
\label{sm26}

Next we discuss perverse sheaves on complex analytic spaces, as in
Dimca \cite{Dimc}. The theory follows \S\ref{sm21}--\S\ref{sm24},
replacing (smooth) $\C$-schemes by complex analytic spaces (complex
manifolds), and regular functions by holomorphic functions.

Let $X$ be a complex analytic space, always assumed locally of
finite type (that is, locally embeddable in $\C^n$). In the analogue
of Definition \ref{sm2def1}, we fix a well-behaved commutative ring
$A$, and consider sheaves of $A$-modules $\cS$ on $X$ in the complex
analytic topology. A sheaf $\cS$ is called ({\it analytically\/})
{\it constructible\/} if all the stalks $\cS_x$ for $x\in X$ are
finite type $A$-modules, and there is a locally finite
stratification $X=\coprod_{j\in J}X_j$ of $X$, where now
$X_j\subseteq X$ for $j\in J$ are complex analytic subspaces of $X$,
such that $\cS\vert_{X_j}$ is an $A$-local system for all $j\in J$.

Write $D(X)$ for the derived category of complexes $\cC^\bu$ of
sheaves of $A$-modules on $X$, exactly as in \S\ref{sm21}, and
$D^b_c(X)$ for the full subcategory of bounded complexes $\cC^\bu$
in $D(X)$ whose cohomology sheaves $\cH^m(\cC^\bu)$ are analytically
constructible for all $m\in\Z$. Then $D(X),D^b_c(X)$ are
triangulated categories.

When we wish to distinguish the complex algebraic and complex
analytic theories, we will write $D^b_c(X)^\alg,\Perv(X)^\alg$ for
the algebraic versions in \S\ref{sm21}--\S\ref{sm22} with $X$ a
$\C$-scheme, and $D^b_c(X)^\an,\Perv(X)^\an$ for the analytic
versions.

Here are the main differences between the material of
\S\ref{sm21}--\S\ref{sm24} for perverse sheaves on $\C$-schemes and
on complex analytic spaces:
\begin{itemize}
\setlength{\itemsep}{0pt}
\setlength{\parsep}{0pt}
\item[(a)] If $f:X\ra Y$ is an arbitrary morphism of
$\C$-schemes, then as in \S\ref{sm21} the pushforwards
$Rf_*,Rf_!:D(X)\ra D(Y)$ also map $D^b_c(X)^\alg\ra
D^b_c(Y)^\alg$.

However, if $f:X\ra Y$ is a morphism of complex analytic spaces,
then $Rf_*,Rf_!:D(X)\ra D(Y)$ {\it need not\/} map
$D^b_c(X)^\an\ra D^b_c(Y)^\an$ without extra assumptions on $f$,
for example, if $f:X\ra Y$ is proper.
\item[(b)] The analogue of Theorem \ref{sm2thm3} says that perverse
sheaves on a complex analytic space $X$ form a stack in the
complex analytic topology. This is proved in the subanalytic context in
\cite[Th.~10.2.9]{KaSc1}; the analytic case follows upon noting that a sheaf is complex 
analytically constructible if and only if is locally at all points, as proved
in \cite[Prop.~4.1.13]{Dimc}. See also~\cite[Prop.~8.1.26]{HTT}.
\end{itemize}
The analogues of (i)--(xii) in \S\ref{sm25} work for complex
analytic perverse sheaves, and so our main results hold in this
context.

If $X$ is a $\C$-scheme, and $X^\an$ the corresponding complex
analytic space, then $D(X)$ in \S\ref{sm21} for $X$ a $\C$-scheme
coincides with $D(X^\an)$ for $X^\an$ a complex analytic space, and
$D^b_c(X)^\alg\subset D^b_c(X^\an)^\an$,
$\Perv(X)^\alg\subset\Perv(X^\an)^\an$ are full subcategories, and
the six functors $f^*,f^!,Rf_*,Rf_!, {\cal RH}om,\otL$ for
$\C$-scheme morphisms $f:X\ra Y$ agree in the algebraic and analytic
cases.

\subsection{$\cD$-modules on $\C$-schemes and complex analytic spaces}
\label{sm27}

{\it $\cD$-modules\/} on a smooth $\C$-scheme or smooth complex analytic space $X$
are sheaves of modules over a certain sheaf of rings of differential
operators $\cD_X$ on $X$. Some books on them are Borel et al.\
\cite{Bore}, Coutinho \cite{Cout}, and Hotta, Takeuchi and Tanisaki
\cite{HTT} in the $\C$-scheme case, and Bj\"ork \cite{Bjor} and
Kashiwara \cite{Kash2} in the complex analytic case. For a singular
complex $\C$-scheme or complex analytic space $X$,
the definition of a well-behaved category of $\cD$-modules
is given by Saito \cite{Sait4}, via locally embedding $X$ into a
smooth scheme or space.

The analogue of perverse sheaves on $X$ are called {\it regular
holonomic\/ $\cD$-modules}, which form an abelian category
$\mathop{\rm Mod}_{\rm rh}(\cD_X)$, the heart in the derived category
$D^b_{\rm rh}(\mathop{\rm Mod}(\cD_X))$ of bounded complexes of $\cD_X$-modules with 
regular holonomic cohomology modules. The whole package of
\S\ref{sm21}--\S\ref{sm24} works for $\cD$-modules. Our next theorem
is known as the {\it Riemann--Hilbert correspondence\/} \cite[\S
V.5]{Bjor}, \cite[Th.~7.2.1]{HTT}, see Borel \cite[\S
14.4]{Bore} for $\C$-schemes, Kashiwara \cite{Kash1} for complex
manifolds, and Saito \cite[\S 6]{Sait4} for complex analytic spaces, and
also Maisonobe and Mekhbout~\cite{MaMe}. 

\begin{thm} Let\/ $X$ be a $\C$-scheme or complex analytic space.
Then there is a \begin{bfseries}de Rham functor\end{bfseries} ${\rm
DR}:D^b_{\rm rh}(\mathop{\rm Mod}(\cD_X))\,{\buildrel\sim\over\longra}\, D^b_c(X,\C),$ which is an
equivalence of categories, restricts to an equivalence $\mathop{\rm
Mod}_{\rm rh}(\cD_X)\,{\buildrel\sim \over\longra}\,\Perv(X,\C),$
and commutes with\/ $f^*,f^!,Rf_*,\ab Rf_!,\ab {\cal
RH}om,\ab\otL$, and also with $\psi_f^p,\phi_f^p$ for $X$ smooth. 
Here $D^b_c(X,\C),\Perv(X,\C)$ are constructible complexes and perverse sheaves over the base
ring\/~$A=\C$.
\label{sm2thm6}
\end{thm}

Because of the Riemann--Hilbert correspondence, all our results on
perverse sheaves of vanishing cycles on $\C$-schemes and complex
analytic spaces in \S\ref{sm3}--\S\ref{sm6} over a well-behaved base
ring $A$, translate immediately when $A=\C$ to the corresponding
results for $\cD$-modules of vanishing cycles, with no extra work.

\subsection{Mixed Hodge modules: basics}
\label{sm28}

We write this section in the minimal generality needed for our
applications. The statements made work equally well in the category
of (algebraic) $\C$-schemes and the category of complex analytic
spaces. By space, we will mean an object in either of these
categories. The theory of mixed Hodge modules works with reduced
spaces; should a space $X$ be non-reduced, the following
constructions are taken by definition on its reduction.

For a space $X$, let $\HM(X)$ denote Saito's category \cite{Sait1}
of polarizable pure Hodge modules, (locally) 
a direct sum of subcategories
$\HM(X)^w$ of pure Hodge modules of fixed weight $w$. On a smooth
$X$, a pure Hodge module $M^\bu$ consists of a triple of data: a
filtered holonomic $\cD$-module $(M, F)$, a $\Q$-perverse sheaf, and a comparison map 
identifying the former with the complexification of the latter
under the Riemann--Hilbert correspondence; see \cite[\S 5.1.1,
p.~952]{Sait1} and \cite[\S 4]{Sait3}. This triple has to satisfy
many other properties; in particular, the underlying holonomic $\cD$-module is automatically regular,
and algebraic Hodge modules are asked to be extendable
to an algebraic compactification. Thus there is a forgetful functor 
$\HM(X)\ra \mathop{\rm Mod}_{\rm rh}(\cD_X)$ 
from Hodge modules to regular holonomic (algebraic) $\cD$-modules.
Hodge modules on singular spaces are defined,
similarly to $\cD$-modules, via embeddings into smooth varieties;
see Saito \cite{Sait3} and also Maxim, Saito and
Sch\"urmann~\cite[\S 1.8]{MSS}.

There is a duality functor $\bD^H_X: \HM(X)\ra \HM(X)$. Pure Hodge
modules also admit a Tate twist functor $M^\bu\mapsto M^\bu(1)$, see
\cite[\S 5.1.3, p.~952]{Sait1}. This functor shifts the filtration and
rotates the rational structure on the underlying perverse sheaf: the $\cD$-module filtration $(M, F)$ is
shifted to $(M, F[n])$ with $(F[n])_i = F_{i-n}$; 
the underlying perverse sheaf is tensored by $\Z(n) = (2 \pi i)^n \Z \subset \C$, as in 
\cite[(2.0.2), p.~876]{Sait1}.

A polarization of weight $w$ on a
pure Hodge module $M^\bu\in\HM(X)^w$ is a morphism of pure Hodge
modules
\begin{equation*}
\si: M^\bu\longra \bD^H_X(M^\bu)(-w),
\end{equation*}
satisfying the extra conditions using vanishing cycles described on 
\cite[(5.1.6.2) on p.~956 and (5.2.10.2) on p.~968]{Sait1}, as well as the
condition that on points it should correspond to the classical
notion of a polarization of a pure Hodge structure (including
positive definiteness).

Next, let $\MHM(X)$ denote the category of graded polarizable mixed Hodge
modules \cite{Sait1,Sait3}. A graded polarizable mixed Hodge module carries a functorial
weight filtration $W$, with graded pieces being polarizable pure Hodge
modules, see \cite[\S 5.2.10, p.~967-8]{Sait1}. 
The forgetful functor $\rat: \MHM(X) \ra \Perv(X)$ to the
appropriate category of perverse $\Q$-sheaves on $X$ is faithful and
exact; faithfulness in particular means that a morphism in $\MHM(X)$
is uniquely determined by the underlying morphism of perverse
sheaves. The Tate twist functor extends to $\MHM(X)$; under 
this functor, the weight filtration
$W$ of the mixed Hodge module is changed to $W[−2n]$ with $W[−2n]_i =
W_{i+2n}$ as on \cite[p.~855]{Sait1}. The duality
functor $\bD^H_X$ also extends to $\MHM(X)$ and is compatible with
Verdier duality on the perverse realization. There is also a forgetful functor 
$\MHM(X)\ra \mathop{\rm Mod}_{\rm rh}(\cD_X)$ to regular holonomic $\cD$-modules, even 
for singular spaces.

\begin{thm} The categories of graded polarizable mixed Hodge modules
have the following properties:
\begin{itemize}
\setlength{\itemsep}{0pt}
\setlength{\parsep}{0pt}
\item[{\bf(i)}] By\/ {\rm\cite[Th.~3.9, p.~288]{Sait3},} the
category of mixed Hodge modules for $X$ a point is canonically
equivalent to Deligne's category of graded polarizable mixed Hodge
structures.
\item[{\bf(ii)}] For a smooth space $U,$ we have a canonical
object of weight\/ $\dim U$
\begin{equation*}
\Q^H_U[\dim U]\in\HM(U)\subset\MHM(U),
\end{equation*}
which by\/ {\rm\cite[Prop.~5.2.16, p.~971]{Sait1}} possesses a
canonical polarization $\si: \Q^H_U[\dim U]\ra
\bD^H_U\Q^H_U[\dim U](-\dim U)$.
\item[{\bf(iii)}] For an open inclusion $f: Y\hookra X$ of
spaces, there is a pullback functor $f^*=f^!:\MHM(X)\ra
\MHM(Y)$. More generally, by {\rm\cite[Prop.~2.19,
p.~258]{Sait3},} for an arbitrary morphism $f:Y\ra X,$ there
exist cohomological pullback functors $L^jf^*,L^j f^!:
\MHM(X)\ra\MHM(Y)$ compatible with (perverse) cohomological
pullback on the perverse sheaf level.
\item[{\bf(iv)}] For a closed embedding $i: X\hookra Y,$ there
is a pushforward functor
\begin{equation*}
\smash{i_*=i_!:\MHM(X)\longra \MHM(Y),}
\end{equation*}
whose essential image is the full subcategory $\MHM_X(Y)$ of
objects in\/ $\MHM(Y)$ supported on $X$. Its inverse is
$i^*\!=\!i^!: \MHM_X(Y)\!\ra\! \MHM(X)$. More generally, by
{\rm\cite[Th.~5.3.1, p.~977]{Sait1}} and\/ {\rm\cite[Th.~2.14,
p.~252]{Sait3},} for a \begin{bfseries}projective\end{bfseries}
map $f:X\ra Y$ there are cohomological pushforward functors
\begin{equation*}
\smash{R^jf_*: \MHM(X)\longra \MHM(Y).}
\end{equation*}
\item[{\bf(v)}] There is an external tensor product functor
\begin{equation*}
\smash{\boxtL: \MHM(X)\t \MHM(Y)\longra \MHM(X\t Y),}
\end{equation*}
which is compatible with duality in the sense that for
$M^\bu\in\MHM(X)$ and $N^\bu\in\MHM(Y),$ there is a natural
isomorphism
\begin{equation*}
\smash{\bD^H_X M^\bu \boxtL \bD^H_Y N^\bu \cong
\bD^H_{X\t Y}(M^\bu \boxtL N^\bu).}
\end{equation*}
\end{itemize}
\label{sm2thm7}
\end{thm}

\begin{rem} We will not need to use any derived category
$D^?\MHM(X)$ of mixed Hodge modules in this paper, which is just as
well since on singular analytic $X$, the appropriate boundedness
conditions do not appear to be well understood, and the general
pullback and pushforward functors of Theorem \ref{sm2thm7}(iii),(iv)
are not known to exist as derived functors outside of the algebraic context of
\cite[\S 4]{Sait3}. Hence, in part (ii)
above, $[1]$ does not stand for a shift in the derived category;
$\Q^H_U[\dim U]$ just denotes a mixed Hodge module whose realization
is the perverse sheaf $\Q_U[\dim U]$ on $U$. Compare Remark
\ref{sm2rem3} above.

\label{sm2rem4}
\end{rem}

Using the functors above, we can now define the twist of a mixed
Hodge module by a principal $\Z/2\Z$-bundle. In the setup of
Definition \ref{sm2def3}, given a $\Z/2\Z$-bundle $\pi: P\ra X$, and
an object $M^\bu\in\MHM(X)$ on a space $X$, we have a natural map
$M^\bu\ra \pi_*\pi^*M^\bu$, which is an injection by faithfulness of
the realization functor and the fact that it is an injection on the
perverse sheaf level. The quotient object will be denoted, by abuse
of notation, by $M^\bu\ot_{\Z/2\Z}P$ in~$\MHM(X)$.

\subsection{Monodromic mixed Hodge modules}
\label{sm29}

To discuss nearby and vanishing cycle functors in a way
consistent with monodromy, we need an extension of the category
of mixed Hodge modules. For a space $X$, following Saito \cite[\S
4.2]{Sait5} denote by $\MHM(X;T_s,N)$ the category of mixed Hodge
modules $M^\bu$ on $X$ with commuting actions of a finite order
operator $T_s:M^\bu\ra M^\bu$ and a locally nilpotent operator
$N:M^\bu\ra M^\bu(-1)$. There is an embedding of categories
$\MHM(X)\ra \MHM(X;T_s,N)$ defined by setting $T=\id$ and~$N=0$.
As proved by \cite[(4.6.2)]{Sait5}, the category $\MHM(X;T_s,N)$ is equivalent 
to the category $\MHM(X \t \C)_{\rm mon,!}$ of monodromic mixed Hodge modules 
on $X\t\C^*$ extended by zero to $X \t \C$; compare also~\cite[\S 4.2]{KoSo2}.

Every object $M^\bu\in \MHM(X;T_s,N)$ decomposes into a direct sum
$M^\bu=M^\bu_1\op M^\bu_{\neq 1}$ of the $T_s$-invariant part and
its $T_s$-equivariant complement. The Tate twist, and appropriate
cohomological pullback and pushforward functors continue to exist.
There is a duality functor
\begin{equation*}
\bD_X^T: \MHM(X;T_s,N)\longra \MHM(X;T_s,N)
\end{equation*}
defined by
\begin{equation*}
\bD_X^T(M^\bu) = \bD_X^H(M^\bu_1) \op \bD_T^H(M^\bu_{\neq 1})(1),
\end{equation*}
equipped with the finite-order operator $\bD_X(T_s)^{-1}$ and the nilpotent operator
$-\bD_X(N)$. This duality functor still satisfies $\bD_X^T\circ \bD_X^T=\id$. 

Saito \cite[\S 5.1]{Sait5} also defines an external tensor product
\begin{equation*}
\boxtT : \MHM(X_1;T_s,N) \t \MHM(X_2;T_s,N)\longra \MHM(X_1\t X_2;T_s,N).
\end{equation*}
defined on the monodromic category as follows. The addition map on fibres
\begin{equation*}
\pi: (X_1\t\C) \t (X_2\t\C) \ra (X_1\t X_2)\t \C
\end{equation*}
induces the additive convolution 
\begin{equation*}
\pi_*(-\boxtimes -): \MHM(X_1 \t \C)_{\rm mon,!}\t \MHM(X_2 \t \C)_{\rm mon,!} \ra \MHM(X_1 \t x_2 \t \C)_{\rm mon,!}.
\end{equation*}

One can translate this external tensor product $\boxtT$ to the $\MHM(X;T_s,N)$ defined by 
concrete data $(M^\bu, T_s, N)$. On the underlying $\cD$-modules and perverse sheaves, it is just
the usual product $\boxtimes$. The operators are defined by $T_s = T_s \boxtimes T_s$ and
$N = N \boxtimes {\rm id} + {\rm id}\boxtimes N$. However, the Hodge and weight filtrations on the underlying 
$\cD$-modules and perverse sheaves are shifted using the finite order endomorphisms $T_s$; for details, see \cite[(5.1.1)--(5.1.2)]{Sait5}. Note that as a consequence of these definitions, 
the forgetful functors $\MHM(-;T_s,N)\ra\MHM(-)$ do not map $\boxtT$ to $\boxtL$.
Twisted duality and the twisted tensor product commute in the sense
that given $M^\bu\in\MHM(X;T_s,N)$ and $N^\bu\in \MHM(Y;T_s,N)$, we
have a natural isomorphism in $\MHM(X\t Y;T_s,N)$
\e
\bD_X^T(M^\bu)\boxtT\bD_Y^T(N^\bu)\cong\bD_{X\t Y}^T(M^\bu\boxtT
N^\bu).
\label{sm2eq20}
\e

For an object $M^\bu\in\MHM(X;T_s,N)$ whose weight filtration is a (suitable shifted) monodromy 
filtration of the nilpotent morphism $N$, there is a stronger notion of
polarization which will be useful for us. A {\it strong
polarization of weight\/} $w$ of such an object $M^\bu$ is a morphism $\si: M^\bu \ra
\bD^T_X(M^\bu)(-w)$ in $\MHM(X)$, compatible with
$T_s$ and $N$, such that $\si$ defines polarizations on the
$N$-primitive parts of $M^\bu$, compatible with Hodge filtrations;
for precise conditions, see \cite[p.~855]{Sait1}.
A polarization on a pure Hodge module is a strong polarization (with
$N=0$); a strongly polarized mixed Hodge module is graded polarizable.
The partial twist in the definition of $\bD_X^T$ implies that $M^\bu$ is of weight $w$ if and only if
$M^\bu_1$, respectively $M^\bu_{\neq 1}$ are of weights $w,w-1$ in the sense of \cite[p.~855]{Sait1}. 

Given strongly polarized mixed Hodge modules $M_i^\bu\in\MHM(X_i;T_s,N)$ of weight $w_i$ for $i=1,2$, polarized by $\si_i: M_i^\bu \ra \bD^T_{X_i}(M_i^\bu)(-w_i)$, 
there is an induced morphism $\si$ in a commutative diagram
\begin{equation*}
\xymatrix@C=140pt@R=11pt{ *+[r]{M_1^\bu\boxtT M_2^\bu} \ar[dr]_(0.3)\si \ar[r] & 
*+[l]{\bD_{X_1}^T(M_1^\bu)(-w_1)\boxtT \bD_{X_2}^T(M_2^\bu)(-w_2)} \ar[d] \\
& *+[l]{\bD_{X_1\t X_2}^T(M_1^\bu\boxtT M_2^\bu)(-w_1-w_2),} }
\end{equation*}
where the top map is $\si_1\boxtT \si_2$ and the right is the
isomorphism \eq{sm2eq20}. In general, it is {\em not clear} 
whether this morphism is a strong polarization of the 
tensor product $M_1^\bu\boxtT M_2^\bu$; this result is not available in the literature. 
However, in this paper we only use this construction in cases
where one of the monodromic mixed Hodge modules is essentially trivial, 
living on $X_1={\rm pt}$ with $N=0$, in which case it is easy to check that
the resulting $\si$ is a strong polarization. 

Note also that if $M^\bu$ is strongly
polarized by $\si: M^\bu \ra \bD^T_X(M^\bu)(-w)$, then its Tate
twist is also strongly polarized by the composition
\e
\smash{\xymatrix@C=27pt{ M^\bu(1) \ar[r]^(0.35){\si(1)} & \bD^T_X(M^\bu)(-w+1)
\ar[r]^(0.45)\sim & \bD^T_X(M^\bu(1))(-w+2). }}
\label{sm2eq21}
\e

The notion of strong polarization leads to gluing, in the following
way.

\begin{thm} Let\/ $X=\bigcup_i U_i$ be an open cover of a space
$X,$ in any of the Zariski, \'etale or complex analytic topologies.
Then:
\begin{itemize}
\setlength{\itemsep}{0pt}
\setlength{\parsep}{0pt}
\item[{\bf(i)}] Suppose we are given mixed Hodge modules $M^\bu,
N^\bu\in\MHM(X),$ with morphisms $f_i: M^\bu\vert_{U_{i}}\ra
N^\bu\vert_{U_{i}}$ in $\MHM(U_i)$ which agree on overlaps
$U_{ij}$. Then there is a unique
$f\in\Hom_{\MHM(X)}(M^\bu,N^\bu)$ with\/~$f\vert_{U_{i}}=f_i$.
\item[{\bf(ii)}] Suppose we are given mixed Hodge modules
$M_i^\bu\in\MHM(U_i; T_s, N),$ each equipped with a strong
polarization $\si_i$. Suppose also that we are given
isomorphisms $\al_{ij}:M_i^\bu\vert_{U_{ij}}\ra
M_j^\bu\vert_{U_{ij}}$ on intersections, commuting with the
restrictions of the maps $T_{si},N_i$ and\/ $\si_i,$ with\/
$\al_{jk}\vert_{U_{ijk}}\ci
\al_{ij}\vert_{U_{ijk}}=\al_{ik}\vert_{U_{ijk}}$ on triple
intersections. Then there is a strongly polarized mixed Hodge
module $M^\bu\in\MHM(X; T_s, N),$ restricting to $M_i^\bu$ on $U_i$.
\end{itemize}
\label{sm2thm8}
\end{thm}

\begin{proof} To prove (i), it is enough to note that the $f_i$
glue on the perverse sheaf and $\cD$-module levels, respecting
filtrations.

To prove (ii), we begin with the case of pure Hodge modules of fixed
weight $w$. The data of a pure Hodge module consists of a pair of a
filtered holonomic $\cD$-module and a $\Q$-perverse sheaf, with an
identification of the former with the complexification of the latter
under the Riemann--Hilbert correspondence. Since both filtered
holonomic $\cD$-modules and $\Q$-perverse sheaves form stacks (both
in the algebraic and the analytic case), this data glues over $X$.
As for the (strong) polarization, $T_{si}=\id$ and $N_i=0$ glue to
$T_s=\id$ and $N=0$, whereas the map $\si$ on the perverse sheaf
level glues from the maps $\si_i$ once again from the stack property
(now for morphisms) of perverse sheaves.

The conditions \cite[\S 5.1.6, p.~955]{Sait1} which make such a pair
a pure Hodge module come from local conditions as well as conditions
on vanishing cycles; the latter glue by induction on the dimension.
So strongly polarized pure Hodge modules form a stack. The case of
mixed Hodge modules is similar: we need to glue filtrations and
polarizations, as well as the maps $T_{si}, N_i$ and $\si_i$, first
on the level of perverse sheaves, and then checking the axioms,
which are local or follow by induction.
\end{proof}

\begin{rem} Given a projective $\C$-scheme $X$, and a polarizable
mixed Hodge module $M^\bu\in\MHM(X)$ on it, the second part of
Theorem \ref{sm2thm7}(iv) applied to $f:X\ra{\rm pt}$ shows that the
hypercohomology ${\mathbb H}^*(X, M^\bu)$ carries a mixed Hodge
structure. In particular, it carries a weight filtration and
therefore has a weight polynomial, which will be useful in
refinements of Donaldson--Thomas theory, see the discussion in
Remark \ref{sm6rem2} below. So we need to glue polarizable objects
from local data. On the other hand, graded polarizable mixed Hodge modules
may fail to form a stack in the analytic category unless the
polarizations glue. This is the reason for using the stronger form
of polarization, which allows for gluing as shown above.
\label{sm2rem5}
\end{rem}

\subsection{Mixed Hodge modules of vanishing cycles}
\label{sm210}

By Saito's work \cite{Sait1,Sait2,Sait3}, for a regular function
$f:U\ra\C$ on a smooth space $U$, the perverse nearby and vanishing
cycle functors $\psi_f^p,\phi_f^p$ defined on perverse sheaves in
\S\ref{sm23} lift to functors
\begin{equation*}
\psi_f^{H},\phi_f^{H}:\MHM(U)\longra\MHM(U_0;T_s,N),
\end{equation*}
where $U_0=f^{-1}(0)$. The actions of the finite order and nilpotent
operators $T_s,N$ are given by the semisimple part of the monodromy
operator, and the logarithm of its unipotent part. The analogue 
\begin{equation*}
\phi_f^{H} \circ \bD_U^T \cong  \bD_{U_0}^T \circ  \phi_f^{H}
\end{equation*}
of Theorem \ref{sm2thm4}(iv) is proved in \cite{Sait2}; 
to make this isomorphism work is the reason for the the twist in the definition 
of $\bD_U^T$. Note also that \cite[Th.~1.6]{Sait2} fits with the convention that $T_s$ 
and $N$ are defined on dual objects as $\bD_U(T_s)^{-1}$ and $−\bD_U(N)$, respectively.

By \cite[\S 5.2]{Sait1}, if $M^\bu\in\HM(U)$ is a pure Hodge module, then a
polarization of $M^\bu$ induces a strong polarization on the (mixed)
Hodge module of vanishing cycles $\phi_f^{H}(M^\bu)$, of the same
weight. In particular, if $M^\bu=\Q^H_U[\dim U]$ is the canonical
object with its canonical polarization from Theorem
\ref{sm2thm7}(ii), then $\phi_f^{H}\bigl(\Q^H_U[\dim U]\bigr)\in
\MHM(\Crit(f); T_s, N)$ is a strongly polarized mixed Hodge module
on the critical locus of $f$, with polarization
\e
\si: \phi_f^{H}\bigl(\Q^H_U[\dim U]\bigr)\longra
\bD_{\Crit(f)}^T\ci\phi_f^{H} \bigl(\Q^H_U[\dim U]\bigr)(-\dim U).
\label{sm2eq22}
\e

\begin{ex} Define $f:\C\ra\C$ by $f(z)=z^2$. Then $\Crit(f)=\{0\}$,
and we obtain an object $\phi_f^{H}\bigl(\Q_\C^H[1]\bigr)$ in
$\MHM({\rm pt}; T_s, N)$, a one-dimensional polarized mixed Hodge
structure with monodromy acting by $T_s=-{\rm id}$ and $N=0$. For $g:\C^2\ra\C$ given by
$g(z_1, z_2)=z_1^2 + z_2^2$, it is well known that
\begin{equation*}
\phi_{z_1^2+z_2^2}^{H}\bigl(\Q_{\C^2}^H[2]\bigr)\cong \Q(-1),
\end{equation*}
with trivial monodromy action. Applying the Thom--Sebastiani formula
for mixed Hodge modules \cite[Th.~5.4]{Sait5}, we see that
\begin{equation*}
\phi_{z^2}^{H}\bigl(\Q_\C^H[1]\bigr)\boxtT \phi_{z^2}^{H}
\bigl(\Q_\C^H[1]\bigr)\cong \Q(-1)
\end{equation*}
in the category $\MHM({\rm pt}; T_s, N)$. The objects $\Q(1)$ and
$\Q(-1)$ thus admit square roots under $\boxtT$ in this category,
which we will denote by $\Q(\ha)$ and $\Q(-\ha)$, where
\e
\phi_{z^2}^{H}\bigl(\Q_\C^H[1]\bigr)=\Q(-\ha).
\label{sm2eq23}
\e
More explicitly, we have
\begin{equation*}
\Q(-\ha)= (\Q(0), -{\rm id}, 0)
\end{equation*}
and
\begin{equation*}
\Q(\ha)= (\Q(1), -{\rm id}, 0).
\end{equation*}

Define an object $\Q(\frac{n}{2})\in\MHM({\rm pt};T_s,N)$ for each
$n\in\Z$ by $\Q(\frac{n}{2})=\Q(\ha)^{\boxtT{}^n}$ for $n\ge 0$, and
$\Q(\frac{n}{2})=\Q(-\ha)^{\boxtT{}^{-n}}$ for $n<0$. For any space $X$
with structure morphism $\pi:X\ra{\rm pt}$, and any $M^\bu\in
D^b\MHM(X;T_s,N)$, we define the $\frac{n}{2}$ {\it twist of\/}
$M^\bu$ to be $M^\bu(\frac{n}{2})=M^\bu\boxtT
\bigl(\Q(\frac{n}{2})\bigr)$. If $M^\bu$ is strongly polarized, then
this tensor product is also strongly polarized by the tensor
polarization by our comments above.
\label{sm2ex3}
\end{ex}

Let $U$ be a smooth space, $f:U\ra\C$ a regular function, and
$X=\Crit(f)$ its critical locus, as a subspace of $U$. The perverse
sheaf of vanishing cycles $\PV^\bu_{U,f}\in\Perv(X)$ from
\S\ref{sm24} has a lift to a mixed Hodge module $\HV_{U,f}^{\bu}$ in
$\MHM(X;T_s,N)$, defined for each $c\in f(X)$ by
\e
\HV_{U,f}^{\bu}\vert_{X_c}=\phi_{f-c}^{H}\bigl(\Q^H_U[\dim
U]\bigr)\big\vert{}_{X_c}\bigl(\ha\dim U\bigr)\in\MHM(X_c;T_s,N).
\label{sm2eq24}
\e
This mixed Hodge module inherits a strong polarization of weight $0$
(compare \eq{sm2eq21} and \eq{sm2eq22})
\e
\si^H_{U,f}: \HV_{U,f}^{\bu}\longra \bD_X^T\bigl(\HV_{U,f}^{\bu}\bigr).
\label{sm2eq25}
\e
The twist $(\ha\dim U)$ in \eq{sm2eq24}, using the notation of
Example \ref{sm2ex3}, is included for the same reason as the
$(-1)^{\dim U}$ in the definition \eq{sm2eq7} of $\tau_{U,f}$. It
makes $\HV_{U,f}^{\bu}$ act naturally under transformations which
change dimension --- without it, the mixed Hodge module version of
\eq{sm5eq15} below would have to include a twist $(\ha n)$ for
$n=\dim V-\dim U$. Then $\HV_{U,f}^{\bu}$,
$T_s:\HV_{U,f}^{\bu}\ra\HV_{U,f}^{\bu}$,
$N:\HV_{U,f}^{\bu}\ra\HV_{U,f}^{\bu}(-1)$ and $\si^H_{U,f}:
\HV_{U,f}^{\bu}\ra \bD_X^T\bigl(\HV_{U,f}^{\bu}\bigr)$ are related
to $\PV_{U,f}^{\bu}$, $\tau_{U,f}:\PV_{U,f}^{\bu}\ra
\PV_{U,f}^{\bu}$ and $\si_{U,f}:\PV_{U,f}^\bu
\,{\buildrel\cong\over\longra}\,\bD_X(\PV_{U,f}^\bu)$ in
\S\ref{sm24} by
\begin{equation*}
\PV_{U,f}^{\bu}=\rat\bigl(\HV_{U,f}^{\bu}\bigr), \;\>
\tau_{U,f}=\rat(T_s)\ci\exp\bigl(2 \pi i \rat(N)\bigr), \;\>
\si_{U,f}=\rat(\si^H_{U,f});
\end{equation*}
for the last statement, see Proposition \ref{smAprop1} in the Appendix. 

The following Thom--Sebastiani type result is the analogue of
Theorem~\ref{sm2thm5}.

\begin{thm} Let\/ $U,V$ be smooth spaces and\/ $f:U\ra\C,$
$g:V\ra\C$ be regular functions, so that\/ $f\boxplus g:U\t V\ra\C$
is given by\/ $(f\boxplus g)(u,v):=f(u)+g(v)$. Set\/ $X=\Crit(f)$
and\/ $Y=\Crit(g)$ as subspaces of\/ $U,V,$ so that\/
$\Crit(f\boxplus g)=X\t Y$. Then there is a natural isomorphism
\e
\smash{\TS^H_{U,f,V,g}:\HV_{U\t V,f\boxplus
g}^\bu\,{\buildrel\cong\over\longra}\, \HV_{U,f}^\bu\boxtT
\HV_{V,g}^\bu} \;\>\text{in\/ $\MHM(X\t Y;T_s,N),$}
\label{sm2eq26}
\e
so that the following diagram commutes:
\ea
\begin{gathered}
{}\!\!\!\!\!\!\!\!\!\!\!\!\!\!\! \xymatrix@!0@C=58pt@R=50pt{
*+[r]{\HV_{U\t V,f\boxplus g}^\bu} \ar[rrrrr]_{\si_{U\t V,f\boxplus
g}} \ar[d]^(0.4){\TS^H_{U,f,V,g}} &&&&&
*+[l]{\bD^T_{X\t Y}(\HV_{U\t V,f\boxplus g}^\bu)} \\
*+[r]{\raisebox{25pt}{\quad\,\,\,$\begin{subarray}{l}\ts\HV_{U,f}^\bu\boxtT \\
\ts\HV_{V,g}^\bu\end{subarray}$}}
\ar[rr]^(0.57){\si_{U,f}\boxtT\si_{V,g}} &&
*+[r]{\raisebox{25pt}{${}\,\,\,\begin{subarray}{l}\ts
\bD^T_X(\HV_{U,f}^\bu)\boxtT\!\!\!\!\!{} \\
\ts \bD^T_Y(\HV_{V,g}^\bu)\end{subarray}$}} \ar[rrr]^(0.33)\cong &&&
*+[l]{\bD^T_{X\t Y}\bigl(\HV_{U,f}^\bu\boxtT \HV_{V,g}^\bu\bigr).\!\!{}}
\ar[u]^{\bD^T_{X\t Y}(\TS^H_{U,f,V,g})} }\!\!\!\!\!{}
\end{gathered}
\label{sm2eq27}
\ea

\label{sm2thm9}
\end{thm}

\begin{proof} The existence of the isomorphism \eq{sm2eq26}
follows from the Thom--Seb\-ast\-iani Theorem for mixed Hodge
modules due to Saito \cite[Th.~5.4]{Sait5}, applied to
$\HV_{U,f}^\bu$. The diagram \eq{sm2eq27} exists by \eq{sm2eq20};
its commutativity can be checked on the level of the underlying
perverse sheaves which is \eq{sm2eq9}, in light 
of Propositions \ref{smAprop1}--\ref{smAprop2}
in the Appendix. Note that \eq{sm2eq26} also
includes the analogue of \eq{sm2eq10} in Theorem \ref{sm2thm5},
according to which we have a matching of the monodromy actions
$\tau_{U\t V,f\boxplus g}\cong\tau_{U,f}\boxtL\tau_{V,g}$,
as \eq{sm2eq26} holds in $\MHM(X\t Y;T_s,N)$ rather than just
$\MHM(X\t Y)$.
\end{proof}

In this paper we will only ever apply Theorem \ref{sm2thm9} when
$V=\C^n$, $g=z_1^2+\cdots+z_n^2$ and $Y=\{0\}$. Combining
\eq{sm2eq23} and \eq{sm2eq24} shows that
\begin{equation*}
\HV_{\C,z^2}^\bu=\bigl(\Q(-\ha)\bigr)(\ha)\cong\Q(0)\cong\Q^H_{\{0\}}.
\end{equation*}
Thus, by Theorem \ref{sm2thm9},
$\Q^H_{\{0\}}\boxtT\Q^H_{\{0\}}\cong\Q^H_{\{0\}}$, and induction on
$n$, we see that
\begin{equation*}
\HV_{\C^n,z_1^2+\cdots+z_n^2}^\bu\cong\Q^H_{\{0\}}.
\end{equation*}
As for \eq{sm2eq12}, this isomorphism is natural up to sign,
depending on a choice of orientation for the complex Euclidean space
$(\C^n,\d z_1^2+\cdots+\d z_n^2)$.

\section{Action of symmetries on vanishing cycles}
\label{sm3}

Here is our first main result.

\begin{thm} Let\/ $U,V$ be smooth\/ $\C$-schemes, $\Phi,\Psi:U\ra
V$ \'etale morphisms, and\/ $f:U\ra\C,$ $g:V\ra\C$ regular functions
with\/ $g\ci\Phi=f=g\ci\Psi$. Write $X=\Crit(f)$ and\/ $Y=\Crit(g)$
as $\C$-subschemes of\/ $U,V,$ so that\/ $\Phi\vert_X,\Psi\vert_X:
X\ra Y$ are \'etale morphisms. Suppose $\Phi\vert_X=\Psi\vert_X$.
Then:
\begin{itemize}
\setlength{\itemsep}{0pt}
\setlength{\parsep}{0pt}
\item[{\bf(a)}] As $\Phi,\Psi$ are \'etale, $\d\Phi:TU\ra \Phi^*(TV),$
$\d\Psi:TU\ra\Psi^*(TV)$ are isomorphisms of vector bundles.
Restricting to the reduced\/ $\C$-subscheme $X^\red$ of\/ $X,$
and using $\Phi\vert_{X^\red}=\Psi\vert_{X^\red}$ as
$\Phi\vert_X=\Psi\vert_X,$ gives isomorphisms
\begin{gather*}
\d\Phi\vert_{X^\red},\d\Psi\vert_{X^\red}:TU\vert_{X^\red}\longra
\Phi\vert_{X^\red}^*(TV),\\
\text{and thus\/}\quad \d\Psi\vert_{X^\red}^{-1}\ci
\d\Phi\vert_{X^\red}:TU\vert_{X^\red}\longra TU\vert_{X^\red}.
\end{gather*}
Hence $\det\bigl(\d\Psi\vert_{X^\red}^{-1}\ci \d\Phi
\vert_{X^\red}\bigr):X^\red\ra\C\sm\{0\}$ is a regular function.
Then $\det\bigl(\d\Psi\vert_{X^\red}^{-1}\ci \d\Phi
\vert_{X^\red}\bigr)$ is a locally constant map $X^\red\ra\{\pm
1\}\subset\C\sm\{0\}$.

\item[{\bf(b)}] Definition\/ {\rm\ref{sm2def6}} defines isomorphisms
$\PV_\Phi,\PV_\Psi:\PV_{U,f}^\bu\ra\Phi\vert_X^*\bigl(\PV_{V,g}^\bu
\bigr)$ in $\Perv(X)$. These are related by
\e
\PV_\Phi=\det\bigl(\d\Psi\vert_{X^\red}^{-1}\ci \d\Phi
\vert_{X^\red}\bigr)\cdot \PV_\Psi,
\label{sm3eq1}
\e
regarding\/ $\det\bigl(\d\Psi\vert_{X^\red}^{-1}\ci \d\Phi
\vert_{X^\red}\bigr):X\ra\{\pm 1\}$ as a locally constant map of
topological spaces, where\/ $X,X^\red$ have the same topological
space.
\end{itemize}

The analogues of these results also hold for $\cD$-modules and mixed
Hodge modules on $\C$-schemes, and (with\/ $\Phi,\Psi$ local
biholomorphisms and\/ $f,g$ analytic functions) for perverse
sheaves, $\cD$-modules and mixed Hodge modules on complex analytic
spaces, as in\/~{\rm\S\ref{sm26}--\S\ref{sm210}}.
\label{sm3thm}
\end{thm}

By taking $U=V$, $f=g$, $\Phi$ an isomorphism and $\Psi=\id_U$, we
deduce a result on the action of symmetries on perverse sheaves of
vanishing cycles:

\begin{cor} Let\/ $U$ be a smooth\/ $\C$-scheme, $\Phi:U\ra U$ an
isomorphism, and\/ $f:U\ra\C$ be regular with\/ $f\ci\Phi=f$. Write
$X=\Crit(f)$ as a $\C$-subscheme of\/ $U$ and\/ $X^\red$ for its
reduced\/ $\C$-subscheme, and suppose $\Phi\vert_X=\id_X$. Then
$\det\bigl(\d\Phi\vert_{X^\red}:TU\vert_{X^\red}\ra
TU\vert_{X^\red}\bigr)$ is a locally constant map $X^\red\ra\{\pm
1\},$ and\/ $\PV_\Phi:\PV_{U,f}^\bu\,{\buildrel\cong\over\longra}\,
\PV_{U,f}^\bu$ in $\Perv(X)$ from Definition\/ {\rm\ref{sm2def6}} is
multiplication by $\det\bigl(\d\Phi \vert_{X^\red}\bigr)=\pm 1$. The
analogues hold in the settings
of\/~{\rm\S\ref{sm26}--\S\ref{sm210}}.
\label{sm3cor1}
\end{cor}

\begin{ex} Let $U=V=\C^n$ and $f(z_1,\ldots,z_n)=
g(z_1,\ldots,z_n)=z_1^2+\cdots+z_n^2$, so that
$X=Y=\{0\}\subset\C^n$. Let $\Phi,\Psi\in{\rm O}(n,\C)$ be
orthogonal matrices, so that $\det\Phi,\det\Psi\in\{\pm 1\}$ and
$\Phi,\Psi:\C^n\ra\C^n$ are isomorphisms with $f=g\ci\Phi=g\ci\Phi$
and $\Phi\vert_{\{0\}}=\Psi\vert_{\{0\}}=\id_{\{0\}}$. In Theorem
\ref{sm3thm}(a) we have
\begin{equation*}
\d\Psi\vert_{X^\red}^{-1}\ci \d\Phi\vert_{X^\red}=\Psi^{-1}\ci\Phi:
\C^n\longra\C^n,
\end{equation*}
so that $\det\bigl(\d\Psi\vert_{X^\red}^{-1}\ci \d\Phi
\vert_{X^\red}\bigr)=\det\Psi^{-1}\det\Phi=\pm 1$.

For Theorem \ref{sm3thm}(b), Example \ref{sm2ex2} shows that
$\PV_\Phi,\PV_\Psi:A_{\{0\}}\ra A_{\{0\}}$ are multiplication by
$\det\Phi,\det\Psi$, so $\PV_\Phi=(\det\Psi^{-1}\det\Phi)
\cdot\PV_\Psi$, as in~\eq{sm3eq1}.
\label{sm3ex1}
\end{ex}

The proof of Theorem \ref{sm3thm}(b) uses the following proposition.
To interpret it, pretend for simplicity that the \'etale morphisms
$\pi_U\vert_{W_t}:W_t\ra U$ in (b) are invertible. Then
$\Th_t:=\pi_V\vert_{W_t}\ci \pi_U\vert_{W_t}^{-1}$ for $t\in\C$ are
a 1-parameter family of morphisms $U\ra V$, which satisfy
$f=g\ci\Th_t$ and $\Th_t\vert_X=\Phi\vert_X=\Psi\vert_X$ for
$t\in\C$, with $\Th_0=\Phi$ and $\Th_1=\Psi$. Thus, modulo taking
\'etale covers of $U$, the family $\{\Th_t:t\in\C\}$ interpolates
between $\Phi$ and $\Psi$.

\begin{prop} Let\/ $U,V$ be smooth\/ $\C$-schemes, $\Phi,\Psi:U\ra
V$ \'etale morphisms, and\/ $f:U\ra\C,$ $g:V\ra\C$ regular functions
with\/ $g\ci\Phi=f=g\ci\Psi$. Write $X=\Crit(f)$ and\/ $Y=\Crit(g)$
as $\C$-subschemes of\/ $U,V,$ so that\/ $\Phi\vert_X,\Psi\vert_X:
X\ra Y$ are \'etale. Suppose $\Phi\vert_X=\Psi\vert_X,$ and\/ $x\in
X$ such that\/ $\d\Psi\vert_x^{-1}\ci\d\Phi\vert_x:T_xU\ra T_xU$
satisfies $\bigl(\d\Psi\vert_x^{-1}\ci \d\Phi\vert_x -
\id_{T_xU}\bigr){}^2=0$. Then there exist a smooth\/ $\C$-scheme $W$
and morphisms $\pi_\C:W\ra\C,$ $\pi_U:W\ra U,$ $\pi_V:W\ra V$ and\/
$\io:\C\ra W$ such that:
\begin{itemize}
\setlength{\itemsep}{0pt}
\setlength{\parsep}{0pt}
\item[{\bf(a)}] $\pi_\C\ci\io(t)=t,$ $\pi_U\ci\io(t)=x$ and\/
$\pi_V\ci \io(t)=\Phi(x)$ for all\/~$t\in\C;$
\item[{\bf(b)}] $\pi_\C\!\t\pi_U:W\!\ra\!\C\!\t U$ and\/
$\pi_\C\!\t\pi_V:W\!\ra\!\C\!\t V$ are \'etale. Thus,
$W_t:=\!\pi_\C^{-1}(t)$\break is a smooth\/ $\C$-scheme for
each\/ $t\in\C,$ and\/ $\pi_U\vert_{W_t}:W_t\!\ra\! U,$
$\pi_V\vert_{W_t}:W_t\!\ra\! V$ are \'etale, and\/ $\io(t)\in
W_t$ with\/ $\pi_U:\io(t)\mapsto x,$ $\pi_V:\io(t)\mapsto
\Phi(x);$
\item[{\bf(c)}] $h:=f\ci\pi_U=g\ci\pi_V:W\ra\C$. Thus,
$(\pi_\C\t\pi_U)\vert_Z:Z\ra\C\t X$ and\/
$(\pi_\C\t\pi_V)\vert_Z:Z\ra\C\t Y$ are \'etale, where
$Z:=\Crit(h);$
\item[{\bf(d)}] $\Phi\vert_X\ci\pi_U\vert_Z=\Psi\vert_X
\ci\pi_U\vert_Z=\pi_V\vert_Z:Z\ra Y\subseteq V;$ and\/
\item[{\bf(e)}] $\Phi\ci\pi_U\vert_{W_0}=\pi_V\vert_{W_0}$ and\/
$\Psi\ci\pi_U\vert_{W_1}=\pi_V\vert_{W_1},$ for $W_0,W_1$ as
in\/~{\bf(b)}.
\end{itemize}
\label{sm3prop1}
\end{prop}

We will prove Proposition \ref{sm3prop1} in \S\ref{sm31}, and
Theorem \ref{sm3thm} in~\S\ref{sm32}--\S\ref{sm34}.

\subsection{Proof of Proposition \ref{sm3prop1}}
\label{sm31}

Let $U,V,\Phi,\Psi,f,g,X,Y,x$ be as in Proposition \ref{sm3prop1}.
Choose a Zariski open neighbourhood $V'$ of $\Phi(x)=\Psi(x)$ in $V$
and \'etale coordinates $(z_1,\ldots,z_n):V'\ra\C^n$ on $V'$, with
$z_1=\cdots=z_n=0$ at $\Phi(x)$. Let $m$ be the rank of the
symmetric matrix $\bigl(\frac{\pd^2g}{\pd z_i\pd
z_j}\vert_{\Phi(x)}\bigr){}_{i,j=1}^n$, so that
$m\in\{0,\ldots,n\}$. By
applying an element of $\GL(n,\C)$ to the coordinates
$(z_1,\ldots,z_n)$ we can suppose that
\e
\frac{\pd^2g}{\pd z_i\pd z_j}\Big\vert_{\Phi(x)}=\begin{cases} 1, &
i=j\in\{1,\ldots,m\}, \\
0, & \text{otherwise.} \end{cases}
\label{sm3eq2}
\e
Then $\ha\frac{\pd g}{\pd z_i}$ agrees with $z_i$ to first order at
$\Phi(x)$ for $i=1,\ldots,m$, so replacing $z_i$ by $\ha\frac{\pd
g}{\pd z_i}$ for $i=1,\ldots,m$ and making $V'$ smaller, we can
suppose \eq{sm3eq2} holds and $z_1,\ldots,z_m$ lie in the ideal
$\bigl(\frac{\pd g}{\pd z_i},\; i=1,\ldots,n\bigr)$ in $\O_{V'}$.
Thus we may write
\e
z_i=\ts\sum\limits_{j=1}^nA_{ij}\cdot\frac{\pd g}{\pd z_j}, \qquad
i=1,\ldots,m,
\label{sm3eq3}
\e
where $A_{ij}:V'\ra\bA^1$ are regular functions for $i=1,\ldots,m$
and $j=1,\ldots,n$. Taking $\frac{\pd}{\pd z_j}$ of \eq{sm3eq3} for
$j=1,\ldots,m$ and using \eq{sm3eq2} gives
\e
A_{ij}\vert_{\Phi(x)}=\begin{cases} 1, &
i=j,\;\> i,j\in\{1,\ldots,m\}, \\
0, & i\ne j,\;\> i,j\in\{1,\ldots,m\}.\end{cases}
\label{sm3eq4}
\e

Set $U'=\Phi^{-1}(V')\cap\Psi^{-1}(V')$, so that $U'$ is a Zariski
open neighbourhood of $x$ in $U$. Define \'etale coordinates
$(x_1,\ldots,x_n):U'\ra\C^n$ and $(y_1,\ldots,y_n):U'\ra\C^n$ by
$x_i=z_i\ci\Phi$ and $y_i=z_i\ci\Psi$, so that $x_1=\cdots=x_n=
y_1=\cdots=y_n=0$ at $x$. Since $f=g\ci\Phi=g\ci\Psi$ we have
$\frac{\pd f}{\pd x_j}=\frac{\pd g}{\pd z_j}\ci\Phi$ and $\frac{\pd
f}{\pd y_j}=\frac{\pd g}{\pd z_j}\ci\Psi$. Thus \eq{sm3eq2} and
\eq{sm3eq3} imply that
\ea
\frac{\pd^2f}{\pd x_i\pd x_j}\Big\vert_x&= \frac{\pd^2f}{\pd y_i\pd
y_j}\Big\vert_x=\begin{cases} 1, & i=j\in\{1,\ldots,m\}, \\
0, & \text{otherwise,} \end{cases}
\label{sm3eq5}\\
x_i=\ts\sum\limits_{j=1}^n\bigl(A_{ij}\ci\Phi\bigr)\cdot\frac{\pd
f}{\pd x_j}, \quad y_i&=\ts\sum\limits_{j=1}^n\bigl(A_{ij}\ci
\Psi\bigr)\cdot\frac{\pd f}{\pd y_j}, \quad i=1,\ldots,m.
\label{sm3eq6}
\ea

Now $\d\Phi\vert_x:T_xU\ra T_{\Phi(x)}V$ maps $\frac{\pd}{\pd
x_j}\mapsto\frac{\pd}{\pd z_j}$, as $x_j=z_j\ci\Phi$, and
$\d\Psi\vert_x:T_xU\ra T_{\Phi(x)}V$ maps $\frac{\pd}{\pd
y_j}\mapsto\frac{\pd}{\pd z_j}$. Hence $\d\Psi\vert_x^{-1}\ci
\d\Phi\vert_x: T_xU\ra T_xU$ maps $\frac{\pd}{\pd
x_j}\mapsto\frac{\pd}{\pd y_j}=\sum_{i=1}^n\frac{\pd x_i}{\pd
y_j}\cdot \frac{\pd}{\pd x_i}$. Define $B_{ij}\in\C$ for
$i,j=1,\ldots,n$ by
\e
\de_{ij}+B_{ij}=\ts\frac{\pd x_i}{\pd y_j}\big\vert_x.
\label{sm3eq7}
\e
Then $\bigl(\de_{ij}+B_{ij}\bigr){}_{i,j=1}^n$ is the matrix of
$\d\Psi\vert_x^{-1}\ci\d\Phi\vert_x$ w.r.t.\ the basis
$\frac{\pd}{\pd x_1},\ldots,\frac{\pd}{\pd x_n}$, and
$\bigl(B_{ij}\bigr){}_{i,j=1}^n$ is the matrix of
$\d\Psi\vert_x^{-1}\ci \d\Phi\vert_x - \id_{T_xU}$, so by assumption
$\bigl(B_{ij}\bigr){}^2=0$. Therefore the inverse matrix of
$\bigl(\de_{ij}+B_{ij}\bigr)$ is $\bigl(\de_{ij}-B_{ij}\bigr)$, so
\eq{sm3eq7} gives
\e
\de_{ij}-B_{ij}=\ts\frac{\pd y_i}{\pd x_j}\big\vert_x.
\label{sm3eq8}
\e
More generally, $\bigl(\de_{ij}+tB_{ij}\bigr)$ is invertible for
$t\in\C$, with inverse~$\bigl(\de_{ij}-tB_{ij}\bigr)$.

Now $\Phi\vert_X=\Psi\vert_X$ implies that $\d\Psi\vert_x^{-1}\ci
\d\Phi\vert_x$ is the identity on $T_xX\subseteq T_xU$, and
$T_xX=\Ker\bigl(\Hess_xf\bigr)=\langle \frac{\pd}{\pd
x_{m+1}},\ldots,\frac{\pd}{\pd x_n}\rangle$ by \eq{sm3eq5}, so
\e
B_{ij}=0\quad\text{for all $i=1,\ldots,n$ and $j=m+1,\ldots,n$.}
\label{sm3eq9}
\e
We have $\frac{\pd^2f}{\pd y_i\pd y_l}\big\vert_x=\sum_{j,k}
\frac{\pd x_j}{\pd y_i}\frac{\pd^2f}{\pd x_j\pd x_k}\frac{\pd
x_k}{\pd y_l}\big\vert_x$, $\frac{\pd^2f}{\pd x_i\pd
x_l}\big\vert_x=\sum_{j,k} \frac{\pd y_j}{\pd x_i}\frac{\pd^2f}{\pd
y_j\pd y_k}\frac{\pd y_k}{\pd x_l}\big\vert_x$, which by \eq{sm3eq5}
and \eq{sm3eq7}--\eq{sm3eq9} give equations equivalent to
\e
B_{ij}+B_{ji}=\ts\sum\limits_{k=1}^nB_{ki}B_{kj}=0\quad \text{for
all $i,j=1,\ldots,m$.}
\label{sm3eq10}
\e

Define regular $t',x_1',\ldots,x_n', y_1',\ldots,y_n',
z_1',\ldots,z_n':\C\t U'\t V'\ra\C$ by
\begin{align*}
t'(t,u,v)&=t, & x_i'(t,u,v)&=x_i(u)=z_i\ci\Phi(u),\\
y_i'(t,u,v)&=y_i(u)=z_i\ci\Psi(u), & z_i'(t,u,v)&=z_i(v).
\end{align*}
Then $(t',y_1',\ldots,y_n',z_1',\ldots,z_n')$ are \'etale
coordinates on~$\C\t U'\t V'$.

Let $S$ be an affine Zariski open neighbourhood of $\C\t(x,\Phi(x))$
in $\C\t U'\t V'$, satisfying a series of smallness conditions we
will give during the proof. Regard $(t',y_1',\ldots,y_n',
z_1',\ldots,z_n')$ as \'etale coordinates on $S$, and write
$\pi_\C:S\ra\C$, $\pi_U:S\ra U$, $\pi_V:S\ra V$ for the projections.
We will work with (sheaves of) ideals in $\O_S$, using notation
$\bigl(x_i'-z_i',\;i=1,\ldots,n\bigr)$ to denote the ideal generated
by the functions $x_1'-z_1',\ldots,x_n'-z_n'$, and
$f\ci\pi_U-g\ci\pi_V\in\bigl(x_i'-z_i',\;i=1,\ldots,n\bigr)$ to mean
that $f\ci\pi_U-g\ci\pi_V\in H^0(\O_S)$ is a section of the ideal
$\bigl(x_i'-z_i',\;i=1,\ldots,n\bigr)$. Write $I_X\subset\O_U$,
$I_Y\subset\O_V$ for the ideals of functions on $U,V$ vanishing on
$X,Y$, and $\pi_U^{-1}(I_X),\pi_V^{-1}(I_Y)\subset\O_S$ for the
preimage ideals.

Since $x_i=z_i\ci\Phi$, the functions $x_i'-z_i'$ for $i=1,\ldots,n$
vanish on the smooth, closed $\C$-subscheme
$\bigl(\C\t(\id\t\Phi)(U)\bigr)\cap S$ in $S$, and locally these
functions cut out this $\C$-subscheme. So making $S$ smaller we can
suppose $\bigl(\C\t(\id\t\Phi)(U)\bigr)\cap S$ is the $\C$-subscheme
$x_1'-z_1'=\cdots=x_n'-z_n'=0$ in $S$. As $f=g\ci\Phi$, the function
$f\ci\pi_U-g\ci\pi_V$ is zero on $\bigl(\C\t(\id\t\Phi)(U)\bigr)\cap
S$. Hence
\e
f\ci\pi_U-g\ci\pi_V\in \bigl(x_i'-z_i',
\;i=1,\ldots,n\bigr)\subset\O_S.
\label{sm3eq11}
\e

Lifting \eq{sm3eq11} from $\bigl(x_i'-z_i', \;i=1,\ldots,n\bigr)$ to
$\bigl(x_i'-z_i', \;i=1,\ldots,n\bigr)^2$, making $S$ smaller if
necessary, we may choose regular $C_i:S\ra\C$ for $i=1,\ldots,n$
with
\e
f\ci\pi_U-g\ci\pi_V-\ts\sum\limits_{i=1}^nC_i\cdot(x_i'-z_i')\in
\bigl(x_i'-z_i', \;i=1,\ldots,n\bigr)^2.
\label{sm3eq12}
\e
Apply $\frac{\pd}{\pd z_i'}$ to \eq{sm3eq12}, using the \'etale
coordinates $(t',y_1',\ldots,y_n',z_1',\ldots,z_n')$ on $S$. Since
$\frac{\pd}{\pd z_i'}\bigl(g\ci\pi_V\bigr)=\frac{\pd g}{\pd
z_i}\ci\pi_V$ and $\frac{\pd}{\pd z_i'}(f\ci\pi_U)=0=\frac{\pd
x_j'}{\pd z_i'}$, this gives
\begin{equation*}
C_i-\ts\frac{\pd g}{\pd z_i}\ci\pi_V\in \bigl(x_i'-z_i',
\;i=1,\ldots,n\bigr).
\end{equation*}

Combining this with \eq{sm3eq12} yields
\begin{equation*}
f\ci\pi_U-g\ci\pi_V-\ts\sum\limits_{i=1}^n\bigl(\frac{\pd g}{\pd
z_i}\ci\pi_V\bigr)\cdot(x_i'-z_i')\in \bigl(x_i'-z_i',
\;i=1,\ldots,n\bigr)^2.
\end{equation*}
So making $S$ smaller we can choose regular $D_{ij}:S\ra\C$ for
$i,j=1,\ldots,n$ with $D_{ij}=D_{ji}$ and
\e
f\ci\pi_U-g\ci\pi_V=\ts\sum\limits_{i=1}^n\bigl(\frac{\pd g}{\pd
z_i}\ci\pi_V\bigr)\cdot(x_i'-z_i')+
\sum\limits_{i,j=1}^nD_{ij}\cdot(x_i'-z_i')(x_j'-z_j').
\label{sm3eq13}
\e
Similarly, starting from $y_i=z_i\ci\Psi$ and $f=g\ci\Psi$ we may
choose regular $E_{ij}:S\ra\C$ for $i,j=1,\ldots,n$ with
$E_{ij}=E_{ji}$ and
\e
f\ci\pi_U-g\ci\pi_V=\ts\sum\limits_{i=1}^n\bigl(\frac{\pd g}{\pd
z_i}\ci\pi_V\bigr)\cdot(y_i'-z_i')+
\sum\limits_{i,j=1}^nE_{ij}\cdot(y_i'-z_i')(y_j'-z_j').
\label{sm3eq14}
\e

Applying $\frac{\pd^2}{\pd z_i'\pd z_j'}$ to \eq{sm3eq13} and
\eq{sm3eq14}, restricting to $(t,x,\Phi(x))$ for $t\in\C$, noting
that $x_i'=y_i'=z_i'=0$ at $(t,x,\Phi(x))$, and using \eq{sm3eq2},
we deduce that
\e
\begin{split}
D_{ij}(t,x,\Phi(x))=E_{ij}(t,x,\Phi(x))&=\ts\ha\frac{\pd^2g}{\pd
z_i\pd z_j}\Big\vert_{\Phi(x)}\\
& =\begin{cases} \ha, &
i=j\in\{1,\ldots,m\}, \\
0, & \text{otherwise.} \end{cases}
\end{split}
\label{sm3eq15}
\e

Summing $1-t'$ times \eq{sm3eq13} with $t'$ times \eq{sm3eq14} and
rearranging yields
\ea
&f\ci\pi_U-g\ci\pi_V=\ts\sum\limits_{i=1}^n\begin{aligned}[t]
\ts\Bigl[ \frac{\pd g}{\pd z_i}\ci\pi_V
\!+\!2t'(1\!-\!t')\sum\limits_{j=1}^n
(D_{ij}\!-\!E_{ij})(x_j'\!-\!y_j')\Bigr]&\\
{}\cdot\bigl((1-t')x_i'+t'y_i'-z_i'\bigr)&\end{aligned}
\nonumber\\
&\quad+
\ts\sum\limits_{i,j=1}^n\bigl[(1-t')D_{ij}\!+\!t'E_{ij}\bigr]\cdot
\bigl((1\!-\!t')x_i'\!+\!t'y_i'\!-\!z_i'\bigr)
\bigl((1\!-\!t')x_j'\!+\!t'y_j'\!-\!z_j'\bigr)
\nonumber\\
&\quad+
\ts\sum\limits_{i,j=1}^nt'(1-t')\bigl[t'D_{ij}+(1\!-\!t')E_{ij}
\bigr]\cdot(x_i'-y_i')(x_j'-y_j').
\label{sm3eq16}
\ea

Since $x_i'-y_i'=(x_i-y_i)\ci\pi_U$, and $(x_i-y_i)\vert_X=
z_i\ci\Phi\vert_X-z_i\ci\Psi\vert_X=0$ as $\Phi\vert_X=\Psi\vert_X$,
we see that $x_i'-y_i'\in\pi_U^{-1}(I_X)$. Thus making $S$ smaller
if necessary, we may choose regular $F_{ij}:S\ra\C$ such that
\e
x_i'-y_i'=\ts\sum_{j=1}^nF_{ij}\cdot \bigl(\frac{\pd f}{\pd
y_j}\ci\pi_U\bigr)\quad\text{for $i=1,\ldots,n$.}
\label{sm3eq17}
\e
Furthermore, by \eq{sm3eq6} when $i=1,\ldots,m$ we may take
\begin{equation*}
F_{ij}=\Bigl(\ts\sum\limits_{k=1}^n(A_{ik}\ci\Phi)\cdot\frac{\pd
y_j}{\pd x_k}-A_{ij}\ci \Psi\Bigr)\ci\pi_U.
\end{equation*}
Restricting to $(t,x,\Phi(x))$ and using \eq{sm3eq4}, \eq{sm3eq8},
\eq{sm3eq9} and $\Phi(x)=\Psi(x)$ gives
\e
F_{ij}(t,x,\Phi(x))=-B_{ji} \quad\text{for $i=1,\ldots,m$ and
$j=1,\ldots,n$.}
\label{sm3eq18}
\e

Applying $\frac{\pd}{\pd x_i'}$ to equation \eq{sm3eq13} shows that
\begin{equation*}
\ts\frac{\pd g}{\pd z_i}\ci\pi_V-\frac{\pd f}{\pd x_i}\ci\pi_U\in \bigl(x_j'-z_j',\;j=1,\ldots,n\bigr).
\end{equation*}
Thus we may write
\e
\ts\frac{\pd g}{\pd z_i}\ci\pi_V=\sum_{j=1}^n\bigl(\frac{\pd y_j}{\pd x_i}\ci\pi_U\bigr)\cdot\bigl(\frac{\pd f}{\pd y_j}\ci\pi_U\bigr)+\sum_{j=1}^nG_{ij}\cdot(x_j'-z_j'),
\label{sm3eq19}
\e
where $G_{ij}:S\ra\C$ are regular. Applying $\frac{\pd}{\pd z_j'}$ to \eq{sm3eq19}, restricting to $(t,x,\Phi(x))$ and using \eq{sm3eq2} yields 
\e
G_{ij}(t,x,\Phi(x))=\begin{cases} -1, &
i=j\in\{1,\ldots,m\}, \\
0, & \text{otherwise.} \end{cases}
\label{sm3eq20}
\e
From \eq{sm3eq17} and \eq{sm3eq19} we see that
\e
\ts\frac{\pd g}{\pd z_i}\ci\pi_V=\sum\limits_{j=1}^nH_{ij}\cdot
\bigl(\frac{\pd f}{\pd y_j}\ci\pi_U\bigr)+\sum_{j=1}^nG_{ij}\cdot
\bigl((1-t')x_j'+t'y_j'-z_j'\bigr),
\label{sm3eq21}
\e
where $H_{ij}=\frac{\pd y_j}{\pd x_i}\ci\pi_U+t'\sum_{k=1}^nG_{ik}F_{kj}$, so that from equations \eq{sm3eq8}, \eq{sm3eq9}, \eq{sm3eq18} and \eq{sm3eq20} we deduce that
\e
H_{ij}(t,x,\Phi(x))=\de_{ij}-(1-t)B_{ji}.
\label{sm3eq22}
\e
 
Combining \eq{sm3eq16}, \eq{sm3eq17} and \eq{sm3eq21} gives
\ea
&f\ci\pi_U-g\ci\pi_V-\ts\sum\limits_{i=1}^n
I_i\cdot\bigl((1-t')x_i'+t'y_i'-z_i'\bigr)
\nonumber\\
&\;\>-\ts\sum\limits_{i,j,k,l=1}^nt'(1\!-\!t')
\bigl[t'D_{ij}\!+\!(1\!-\!t')E_{ij} \bigr] F_{ik}F_{jl}\cdot
\bigl(\frac{\pd f}{\pd y_k}\ci\pi_U\bigr)\bigl(\frac{\pd f}{\pd
y_l}\ci\pi_U\bigr)
\nonumber\\
&\;\>\in\bigl((1-t')x_i'+t'y_i'-z_i',\;i=1,\ldots,n\bigr)^2,
\label{sm3eq23}\\
&\text{where}\quad I_i= \ts
\sum\limits_{j=1}^n
\Bigl[H_{ij}\!+\!2t'(1\!-\!t')\sum\limits_{k=1}^n
(D_{ik}\!-\!E_{ik})F_{kj}\Bigr]\cdot \bigl(\frac{\pd f}{\pd
y_j}\ci\pi_U\bigr).
\label{sm3eq24}
\ea

Consider the matrix of functions $[H_{ij}+\cdots]_{i,j=1}^n$
appearing in \eq{sm3eq24}. Equations \eq{sm3eq15} and \eq{sm3eq22}
imply that at $(t,x,\Phi(x))$ this reduces to
$(\de_{ij}-(1-t)B_{ji})$, which is invertible from above. Thus,
making $S$ smaller, we can suppose that $[H_{ij}+\cdots]_{i,j=1}^n$
in \eq{sm3eq24} is an invertible matrix on $S$. Write
$[J_{ij}]_{i,j=1}^n$ for the inverse matrix. Then we have
\ea
&\ts\sum\limits_{i,j,k,l=1}^nt'(1\!-\!t')
\bigl[t'D_{ij}\!+\!(1\!-\!t')E_{ij} \bigr] F_{ik}F_{jl}\cdot
\ts\bigl(\frac{\pd f}{\pd y_k}\ci\pi_U\bigr)\bigl(\frac{\pd f}{\pd
y_l}\ci\pi_U\bigr)
\nonumber\\
&\qquad=t'(1-t')\ts\sum\limits_{i,j=1}^nK_{ij}\cdot I_iI_j,\quad
\text{where}
\label{sm3eq25}\\
&K_{ij}=\ts\sum\limits_{k,l,p,q=1}^nt'(1\!-\!t')
\bigl[t'D_{kl}\!+\!(1\!-\!t')E_{kl} \bigr] F_{kp}F_{lq}J_{pi}J_{qj}.
\nonumber
\ea
Using \eq{sm3eq9}, \eq{sm3eq10}, \eq{sm3eq15} and \eq{sm3eq18} we
find that
\e
K_{ij}(t,x,\Phi(x))=0\quad\text{for all $t\in\C$.}
\label{sm3eq26}
\e

Combining \eq{sm3eq23} and \eq{sm3eq25}, making $S$ smaller if
necessary we may write
\ea
&f\ci\pi_U\!-\!g\ci\pi_V\!=\!\ts\sum\limits_{i=1}^n
I_i\cdot\bigl((1\!-\!t')x_i'\!+\!t'y_i'\!-\!z_i'\bigr)
+t'(1\!-\!t')\ts\sum\limits_{i,j=1}^nK_{ij}\cdot I_iI_j
\nonumber\\
&\quad+\ts\sum\limits_{i,j=1}^nL_{ij}\cdot\bigl((1-t')x_i'+
t'y_i'-z_i'\bigr)\bigl((1-t')x_j'+t'y_j'-z_j'\bigr),
\label{sm3eq27}
\ea
for regular $L_{ij}:S\ra\C$ for $i,j=1,\ldots,n$.

Write $(r_{ij})_{i,j=1}^n$ for the coordinates on $\C^{n^2}$. Let
$T$ be a Zariski open neighbourhood of
$\C\t\bigl(x,\Phi(x),(0)_{i,j=1}^n\bigr)$ in $S\t\C^{n^2}$ to be
chosen shortly, and let $W$ be the closed $\C$-subscheme of $T$
defined by
\ea
W=\bigl\{&\bigl(t,u,v,(r_{ij})_{i,j=1}^n\bigr)\in T\subseteq
S\t\C^{n^2}\subseteq\C\t U\t V\t\C^{n^2}:
\nonumber\\
&\bigl((1\!-\!t')x_i'\!+\!t'y_i'\!-\!z_i'\bigr)
(t,u,v)=\ts\sum\limits_{j=1}^n r_{ij}\cdot I_j(t,u,v), \;\>
i\!=\!1,\ldots,n,
\label{sm3eq28}\\
&r_{ij}+t(1-t)K_{ij}(t,u,v)+\ts\sum\limits_{k,l=1}^nL_{kl}(t,u,v)
\cdot r_{ki}r_{lj} =0,\;\> i,j=1,\ldots,n\bigr\}.
\nonumber
\ea
Define $\C$-scheme morphisms $\pi_\C:W\ra\C$, $\pi_U:W\ra U$,
$\pi_V:W\ra V$ to map $\bigl(t,u,v,(r_{ij})_{i,j=1}^n\bigr)$ to
$t,u,v$, respectively.

At $(t,x,\Phi(x))\in S$ for $t\in\C$ we have $x_i'=y_i'=z_i'=0$, and
$I_i=0$ by \eq{sm3eq24} as $\frac{\pd f}{\pd y_j}\big\vert_x=0$, and
$K_{ij}=0$ by \eq{sm3eq26}. Hence
$\bigl(t,x,\Phi(x),(0)_{i,j=1}^n\bigr)$ satisfies the equations of
\eq{sm3eq28}, and lies in $W$. Define $\io:\C\ra W$ by
\e
\io(t)=\bigl(t,x,\Phi(x),(0)_{i,j=1}^n\bigr).
\label{sm3eq29}
\e

Now $T\subseteq \C\t U\t V\t\C^{n^2}$ is smooth of dimension
$1+n+n+n^2$, and in \eq{sm3eq28} we impose $n+n^2$ equations, so the
expected dimension of $W$ is $(1+2n+n^2)-(n+n^2)=n+1$. The
linearizations of the $n+n^2$ equations in \eq{sm3eq28} at
$\bigl(t,u,v, (r_{ij})_{i,j=1}^n\bigr)=\bigl(t,x,
\Phi(x),(0)_{i,j=1}^n\bigr)=\io(t)$ are
\e
\begin{split}
\d y_i\vert_x(\de u)-\d z_i\vert_{\Phi(x)}(\de v)=0,\quad
i=1,\ldots,n,\\
\de r_{ij}+\d K_{ij}\vert_{(t,x,\Phi(x))}(\de t\op\de u\op\de v)=0,
\quad i,j=1,\ldots,n,
\end{split}
\label{sm3eq30}
\e
for $\de t\in T_t\C$, $\de u\in T_xU$, $\de v\in T_{\Phi(x)}V$, and
$(\de r_{ij})_{i,j=1}^n\in T_{(0)_{i,j=1}^n}\C^{n^2}$, where we have
used $\d x_i'=\d y_i'$ and $I_j=K_{ij}=0$ at $(t,x,\Phi(x))$. As $\d
y_1\vert_x,\ldots,\d y_n\vert_x$ are a basis for $T_x^*U$, equations
\eq{sm3eq30} are transverse, so $W$ is smooth of dimension $n+1$
near $\io(t)$. Hence, taking $T$ small enough, we can suppose $W$ is
smooth.

It remains to prove Proposition \ref{sm3prop1}(a)--(e). Part (a) is
immediate from \eq{sm3eq29}. For (b), the vector space of solutions
$\bigl(\de t,\de u,\de v,(\de r_{ij})_{i,j=1}^n\bigr)$ to
\eq{sm3eq30} is $T_{\io(t)}W$, where
$\d(\pi_\C\t\pi_U)\vert_{\io(t)}:T_{\io(t)}W\ra T_{(t,x)}(\C\!\t
\!U)$ and $\d(\pi_\C\t\pi_V)\vert_{\io(t)}:T_{\io(t)}W\ra
T_{(t,\Phi(x))} (\C\!\t\! V)$ map $\bigl(\de t,\de u,\de v,(\de
r_{ij})_{i,j=1}^n\bigr)$ to $(\de t,\de u)$ and $(\de t,\de v)$. By
\eq{sm3eq30}, these are isomorphisms, so $\pi_\C\t\pi_U$ and
$\pi_\C\t\pi_V$ are \'etale near $\io(\C)$. Making $T,W$ smaller, we
can suppose $\pi_\C\t\pi_U$ and $\pi_\C\t\pi_V$ are \'etale.

For (c), we have
\begin{align*}
\bigl(f\ci&\pi_W\!-\!g\ci\pi_U\bigr)\bigl(t,u,v,(r_{ij})_{i,j=1}^n\bigr)
\!=\!f(u)\!-\!g(v)\!=\!(f\ci\pi_U\!-\!g\ci\pi_V)(t,u,v)\\
&=\ts\sum\limits_{i=1}^n I_i\cdot\bigl((1-t')x_i'+t'y_i'
-z_i'\bigr)+t'(1-t')\ts\sum\limits_{i,j=1}^nK_{ij}I_iI_j\\
&+\ts\sum\limits_{i,j=1}^nL_{ij}\cdot\bigl((1-t')x_i'+t'y_i'-z_i'\bigr)
\bigl((1-t')x_j'+t'y_j'-z_j'\bigr)
\\
&=\ts\sum\limits_{i=1}^n I_i\cdot
\Bigl(\ts\sum\limits_{j=1}^n r_{ij}\cdot I_j\Bigr)
+t'(1-t')\ts\sum\limits_{i,j=1}^nK_{ij}I_iI_j
\\
&+\ts\sum\limits_{i,j=1}^nL_{ij}\cdot
\Bigl(\ts\sum\limits_{k=1}^n
r_{ik}\cdot I_k\Bigr)\Bigl(\ts\sum\limits_{l=1}^n r_{jl}\cdot I_l\Bigr)
\\
&=\ts\sum\limits_{i,j=1}^n I_iI_j\cdot
\Bigl[r_{ij}+t'(1-t')K_{ij}+\ts\sum\limits_{k,l=1}^nL_{kl}
\cdot r_{ki}r_{lj}\Bigr]=0,
\end{align*}
using \eq{sm3eq27} in the third step, the first equation of
\eq{sm3eq28} in the fourth, rearranging and exchanging labels $i,k$
and $j,l$ in the fifth, and the second equation of \eq{sm3eq28} in
the sixth. Hence $f\ci\pi_U-g\ci\pi_V=0:W\ra\C$, proving (c).

For (d), from \eq{sm3eq28} we can show that $\bigl(\C\t(\id\t\Phi)
(X)\t\C^{n^2}\bigr)\cap W$ is open and closed in $Z=\Crit(h)$, and
contains $\io(\C)$. So making $T,W$ smaller we can take
$Z=\bigl(\C\t(\id\t\Phi)(X)\t\C^{n^2}\bigr)\cap W$, and then (d)
follows as $\Phi\vert_X=\Psi\vert_X$.

For (e), observe that when $t=0$ in \eq{sm3eq28}, the second
equation reduces to $r_{ij}=0$ near $\io(\C)$ as
$t(1-t)K_{ij}(t,u,v)=0$, so making $T,W$ smaller gives
\begin{align*}
W_0&=\bigl\{\bigl(0,u,v,(0)_{i,j=1}^n\bigr)\in T:
\bigl(x_i'-\!z_i'\bigr) (0,u,v)=0,\; i=1,\ldots,n\bigr\}\\
&=\bigl\{\bigl(0,u,v,(0)_{i,j=1}^n\bigr)\in T: v=\Phi(u)\bigr\}.
\end{align*}
Hence $\Phi\ci\pi_W\vert_{W_0}=\pi_U\vert_{W_0}$. Similarly, when
$t=1$ we have
\begin{align*}
W_1&=\bigl\{\bigl(1,u,v,(0)_{i,j=1}^n\bigr)\in T:
\bigl(y_i'-\!z_i'\bigr) (0,u,v)=0,\; i=1,\ldots,n\bigr\}\\
&=\bigl\{\bigl(0,u,v,(0)_{i,j=1}^n\bigr)\in T: v=\Psi(u)\bigr\},
\end{align*}
so that $\Psi\ci\pi_W\vert_{W_1}=\pi_U\vert_{W_1}$. This proves (e),
and Proposition~\ref{sm3prop1}.

\subsection[Theorem \ref{sm3thm}(a): $\det(\d\Psi\vert_{X^\red}^{-1}
\ci\d\Phi\vert_{X^\red})=\pm 1$]{Part (a):
$\det(\d\Psi\vert_{X^\red}^{-1}\ci \d\Phi\vert_{X^\red})=\pm 1$}
\label{sm32}

We work in the situation of Theorem \ref{sm3thm}. For each $x\in
X\subseteq U$, consider the diagram of linear maps of vector spaces:
\e
\begin{gathered}
\xymatrix@C=20pt@R=25pt{ 0\ar[r] & T_xX
\ar[d]_{\d(\Phi\vert_X)\vert_x} \ar[r] & T_xU \ar[d]_{\d\Phi\vert_x
} \ar[rr]_{\Hess_xf} && T_x^*U \ar[r]
\ar[d]^{(\d\Phi\vert_x^{-1})^*} & T_x^*X
\ar[d]^{(\d(\Phi\vert_X)\vert_x^{-1})^*} \ar[r] & 0 \\
0\ar[r] & T_{\Phi(x)}Y \ar[r] & T_{\Phi(x)}V
\ar[rr]^{\Hess_{\Phi(x)}g} && T_{\Phi(x)}^*V \ar[r] & T_{\Phi(x)}^*Y
\ar[r] & 0, }
\end{gathered}
\label{sm3eq31}
\e
where $T_xX$ is the Zariski tangent space of $X$, and
$\Hess_xf=(\pd^2f)\vert_x$ the Hessian of $f$ at $x$. The rows of
\eq{sm3eq31} are exact, and the columns isomorphisms. The outer
squares of \eq{sm3eq31} clearly commute. We can show the central
square commutes by taking second derivatives of $f=g\ci\Phi$ to get
$\pd^2 f\vert_x=\pd^2 g\vert_{\Phi(x)}\ci(\d\Phi\vert_x
\ot\d\Phi\vert_x)$, and composing with $\id\ot\d\Phi\vert_x^{-1}$.
Thus \eq{sm3eq31} is commutative.

There is also an analogue of \eq{sm3eq31} for $\Psi$. Since
$\Psi(x)=\Phi(x)$, we may compose the columns of \eq{sm3eq31} for
$\Phi$ with the inverses of the columns of \eq{sm3eq31} for $\Psi$
to get a commutative diagram
\e
\begin{gathered}
\xymatrix@C=27pt@R=40pt{ 0\ar[r] & T_xX \ar[d]_{\begin{subarray}{l}
\d(\Psi\vert_X)\vert^{-1}_x\ci \\  \d(\Phi\vert_X)\vert_x\\
=\id_{T_xX}\end{subarray}} \ar[r] & T_xU \ar[d]_{\begin{subarray}{l}
\d\Psi\vert^{-1}_x\ci \\
\d\Phi\vert_x\end{subarray} } \ar[rr]_{\Hess_xf} && T_x^*U \ar[r]
\ar[d]^{\begin{subarray}{l} \d\Psi\vert_x^*\ci \\
(\d\Phi\vert_x^{-1})^*\end{subarray}} & T_x^*X
\ar[d]^{\begin{subarray}{l} (\d(\Psi\vert_X)\vert_x)^*\ci \\
(\d(\Phi\vert_X)\vert_x^{-1})^* \\ =\id_{T_x^*X}\end{subarray}}
\ar[r] & 0 \\
0\ar[r] & T_xX \ar[r] & T_xU \ar[rr]^{\Hess_xf} && T_x^*U \ar[r] &
T_x^*X \ar[r] & 0, }
\end{gathered}
\label{sm3eq32}
\e
where the outer morphisms are identities
as~$\Phi\vert_X=\Psi\vert_X$.

Choose a complementary vector subspace $N_x$ to $T_xX$ in $T_xU$,
which we think of as the normal to $X$ in $U$ at $x$, so that
$T_xU=T_xX\op N_x$. Write $\Hess_x'f$ for the restriction of
$\Hess_xf$ to a symmetric bilinear form on $N_x$. Since
$T_xX=\Ker(\Hess_xf)$, we see that $\Hess_{\smash{x}}'f$ is a
nondegenerate symmetric bilinear form on $N_x$. We may write
equation \eq{sm3eq32} as
\ea
\nonumber\\[-17pt]
\begin{gathered}
\xymatrix@C=11.3pt@R=68pt{ 0 \ar[r] & T_xX \ar[d]^\id
\ar[rr]_{\begin{pmatrix}\st \id \\ \st 0 \end{pmatrix}} &&
\raisebox{-6pt}{${\begin{subarray}{l}  \ts T_xX\op
\\[2pt] \ts \;\; N_x \end{subarray}}$}
\ar[d]^{
\begin{subarray}{l}
\d\Psi\vert^{-1}_x \\
\ci\d\Phi\vert_x=\end{subarray}
\begin{pmatrix}\st \id & \st A \\
\st 0 & \st B \end{pmatrix}}
\ar[rrrr]_{\begin{pmatrix}\st 0 & \st 0 \\
\st 0 & \st \Hess_{\smash{x}}'f \end{pmatrix}} &&&&
\raisebox{-7pt}{${\begin{subarray}{l} \ts T_x^*X\op
\\[1.5pt] \ts \;\; N_x^* \end{subarray}}$}
\ar[d]^(0.45){\begin{subarray}{l}\st \d\Psi\vert_x^*\ci
(\d\Phi\vert_x^{-1})^*=\\
\ts\begin{pmatrix}\st \id & \st -AB^{-1} \\
\st 0 & \st B^{-1} \end{pmatrix}\end{subarray}}
\ar[rrr]_{\begin{pmatrix}\st \id & \st 0
\end{pmatrix}} &&& T_x^*X \ar[d]^\id \ar[r] & 0\phantom{.}
\\
0 \ar[r] & T_xX \ar[rr]^{\begin{pmatrix}\st \id \\
\st 0 \end{pmatrix}} &&
\raisebox{6pt}{${\begin{subarray}{l} \ts T_xX\op
\\[2pt] \ts \;\; N_x \end{subarray}}$}
\ar[rrrr]^{\begin{pmatrix}\st 0 & \st 0 \\
\st 0 & \st \Hess_{\smash{x}}'f \end{pmatrix}} &&&&
\raisebox{6pt}{${\begin{subarray}{l}  \ts T_x^*X\op
\\[1.5pt] \ts \;\; N_x^* \end{subarray}}$}
\ar[rrr]^{\begin{pmatrix}\st \id & \st 0
\end{pmatrix}} &&& T_x^*X \ar[r] & 0, }
\end{gathered}
\label{sm3eq33}\\[-17pt]
\nonumber
\ea
for some linear $A:N_x\ra T_xX$ and $B:N_x\ra N_x$. Then
\eq{sm3eq33} commuting implies that $B$ preserves the nondegenerate
symmetric bilinear form $\Hess_{\smash{x}}'f$ on $N_x$, and $\det
B=\pm 1$. So $\det\bigl(\d\Psi\vert^{-1}_x\ci\d\Phi\vert_x\bigr)
\!=\!\det\bigl(\begin{smallmatrix} \id & A \\ 0 & B
\end{smallmatrix}\bigr)\!=\!\det B\!=\!\pm 1$ for~$x\in X$.

Thus, as a map of topological spaces,
$\det\bigl(\d\Psi\vert_{X^\red}^{-1}\ci\d\Phi
\vert_{X^\red}\bigr):X^\red\ra\C\sm\{0\}$ actually maps
$X^\red\ra\{\pm 1\}$. Since it is continuous, it is locally
constant. Now if $f,g:Y\ra Z$ are morphisms of $\C$-schemes with $Y$
reduced, then $f=g$ if and only if $f(y)=g(y)$ for each point $y\in
Y$. Applying this to compare
$\det\bigl(\d\Psi\vert_{X^\red}^{-1}\ci\d\Phi\vert_{X^\red}\bigr):
X^\red\ra\C\sm\{0\}$ locally with the constant maps $1$ or $-1$ on
$X^\red$ shows that $\det\bigl(\d\Psi\vert_{X^\red}^{-1}\ci
\d\Phi\vert_{X^\red}\bigr)$ is a locally constant map
$X^\red\ra\{\pm 1\}\subset\C\sm\{0\}$ as a $\C$-scheme morphism.
This proves Theorem~\ref{sm3thm}(a).

\subsection[Theorem \ref{sm3thm}(b): $\PV_\Phi=\det\bigl(\d\Psi
\vert_{X^\red}^{-1}\ci \d\Phi\vert_{X^\red}\bigr)\cdot
\PV_\Psi$]{Part (b): $\PV_\Phi=\det\bigl(\d\Psi\vert_{X^\red}^{-1}
\ci\d\Phi \vert_{X^\red}\bigr)\cdot \PV_\Psi$}
\label{sm33}

For Theorem \ref{sm3thm}(b), we begin with the following
proposition.

\begin{prop} Let\/ $U,V,\Phi,\Psi,f,g,X,Y$ be as in Theorem\/
{\rm\ref{sm3thm},} and suppose $x\in X$ with\/ $\bigl(\d\Psi
\vert_x^{-1}\ci \d\Phi\vert_x - \id_{T_xU}\bigr){}^2=0$. Then there
exists a Zariski open neighbourhood\/ $X'$ of\/ $x$ in $X$ such
that\/~$\PV_\Phi\vert_{X'}=\PV_\Psi\vert_{X'}$.
\label{sm3prop2}
\end{prop}

\begin{proof} Apply Proposition \ref{sm3prop1} to get
$W,\pi_\C,\pi_U,\pi_V,\io,h,Z$. Then apply Proposition \ref{sm2prop}
with $Z,X,x,\pi_\C\vert_Z,\pi_U\vert_Z,\io,\PV_{U,f}^\bu,
\Phi\vert_X^*\bigl(\PV_{V,g}^\bu\bigr),\PV_\Phi,\PV_\Psi$ in place
of $W,X,x,\pi_\C,\pi_X,\io,\cP^\bu, \cQ^\bu,\al,\be$, respectively,
and with $\ga$ defined by the commuting diagram of isomorphisms:
\e
\begin{gathered}
\xymatrix@C=130pt@R=13pt{ *+[r]{\PV_{W,h}^\bu}
\ar[d]^{\PV_{\pi_\C\t\pi_U}} \ar[r]_(0.45){\PV_{\pi_\C\t\pi_V}} &
*+[l]{(\pi_\C\!\t\!\pi_V)\vert_Z^*
\bigl(\PV_{\C\t V,0\boxplus g}^\bu\bigr)}
\ar@<-2ex>[d]_{(\pi_\C\t\pi_V)\vert_Z^*(\TS_{\C,0,V,g})} \\
*+[r]{(\pi_\C\t\pi_U)\vert_Z^*
\bigl(\PV_{\C\t U,0\boxplus f}^\bu\bigr)} \ar[d]^{
(\pi_\C\t\pi_V)\vert_Z^*(\TS_{\C,0,U,f})} &
*+[l]{(\pi_\C\t\pi_V)\vert_Z^*
\bigl(\PV_{\C,0}^\bu\boxtL\PV_{V,g}^\bu\bigr)} \ar@<-2ex>[d]_{\de'} \\
*+[r]{(\pi_\C\t\pi_U)\vert_Z^*
\bigl(\PV_{\C,0}^\bu\boxtL\PV_{U,f}^\bu\bigr)} \ar[d]^\de &
*+[l]{\bigl[\pi_\C\vert_Z^*\bigl(A_\C[1]\bigr)\bigr]\otL\bigl[\pi_V\vert_Z^*
\bigl(\PV_{V,g}^\bu\bigr)\bigr]} \ar@<-2ex>[d]_{\ep'} \\
*+[r]{\bigl[\pi_\C\vert_Z^*\bigl(A_\C[1]\bigr)\bigr]\otL\bigl[\pi_U\vert_Z^*
\bigl(\PV_{U,f}^\bu\bigr)\bigr]} \ar[d]^\ep &
*+[l]{\pi_V\vert_Z^*[1]\bigl(\PV_{V,g}^\bu\bigr)}
\ar@<-2ex>@{=}[d] \\
*+[r]{\pi_U\vert_Z^*[1]\bigl(\PV_{U,f}^\bu\bigr)}
\ar[r]^(0.45)\ga &
*+[l]{\pi_U\vert_Z^*[1]\bigl(\Phi\vert_X^*(\PV_{V,g}^\bu)\bigr).\!\!{}} }
\end{gathered}
\label{sm3eq34}
\e
Here $\TS_{\C,0,U,f},\TS_{\C,0,V,g}$ are as in \eq{sm2eq8},
$\de,\de'$ come from $\PV_{\C,0}\cong A_\C[1]$, and $\ep,\ep'$ from
$\pi_\C\vert_Z^*(A_\C)\cong A_Z$ and $A_Z\smash{\otL}\cP^\bu\cong
\cP^\bu$ for $\cP^\bu\in\Perv(Z)$.

Then the hypothesis $\pi_X\vert_{W_0}^*(\al)=j_0^*[-1](\ga)$  in
Proposition \ref{sm2prop} follows from comparing $j_0^*[-1]$ applied
to \eq{sm3eq34} with the commuting diagram
\begin{equation*}
\xymatrix@C=170pt@R=17pt{ *+[r]{j_0^*[-1]\bigl(\PV_{W,h}^\bu\bigr)}
\ar[d]^\cong \ar[r]_(0.35){j_0^*[-1](\PV_{\pi_\C\t\pi_V})} &
*+[l]{j_0^*[-1]\!\ci\!(\pi_\C\!\t\!\pi_V)\vert_Z^*
\bigl(\PV_{\C\t V,0\boxplus g}^\bu\bigr)}
\ar@<-2ex>[d]_\cong \\
*+[r]{\PV_{W_0,h\vert_{W_0}}^\bu}
\ar[d]^{\PV_{\pi_U\vert_{W_0}}}
\ar[r]^(0.35){\PV_{\pi_V\vert_{W_0}}} &
*+[l]{\pi_V\vert_{Z\cap W_0}^*\bigl(\PV_{V,g}^\bu\bigr)}
\ar@<-2ex>@{=}[d] \\
*+[r]{\pi_V\vert_{Z\cap W_0}^*\bigl(\PV_{U,f}^\bu\bigr)}
\ar[r]^(0.4){\pi_U\vert_{Z\cap W_0}^*(\PV_\Phi)} &
*+[l]{\pi_U\vert_{Z\cap
W_0}^*\ci\Phi\vert_X^*\bigl(\PV_{V,g}^\bu\bigr),\!\!{}} }
\end{equation*}
where $j_0:Z\cap W_0\hookra Z$ is the inclusion, and the bottom
square commutes by Proposition \ref{sm3prop1}(e) and \eq{sm2eq18}.
Similarly $\pi_X\vert_{W_1}^*(\be)=j_1^*[-1](\ga)$. Hence
Proposition \ref{sm2prop} gives Zariski open $x\in X'\subseteq X$
with~$\PV_\Phi\vert_{X'}=\PV_\Psi\vert_{X'}$.
\end{proof}

Now to prove Theorem \ref{sm3thm}(b), let $x\in X$ be arbitrary. As
in \S\ref{sm32}, we can choose a splitting $T_xU=T_xX\op N_x$ such
that
\e
\d\Psi\vert^{-1}_x \ci\d\Phi\vert_x=\begin{pmatrix} \id &  A \\
0 & B \end{pmatrix}: \begin{matrix} T_xX\op \\ N_x\end{matrix}
\longra\begin{matrix} T_xX\op \\ N_x\end{matrix}
\label{sm3eq35}
\e
for linear $A:N_x\ra T_xX$ and $B:N_x\ra N_x$, where $B$ preserves
the nondegenerate symmetric bilinear form $\Hess_{\smash{x}}'f$ on
$N_x$.

Choose a Zariski open neighbourhood $U'$ of $x$ in $U$ and a
splitting $TU'=E\op F$ for algebraic vector subbundles $E,F\subseteq
TU$ with $E\vert_x=T_xX$ and $F\vert_x=N_x$. Then $\d f\vert_{U'}=
\al\op\be$ for unique $\al\in H^0(E)$ and $\be\in H^0(F)$, and
$X\cap U'$ is defined by $\al=\be=0$. Since $\Hess_xf=\pd(\d
f)\vert_x$ is nondegenerate on $N_x$, we see that
$\nabla\be\vert_x:T_xU\ra F\vert_x$ induces an isomorphism $N_x\ra
F\vert_x$, so $\nabla\be\vert_x$ is surjective. Therefore
$S:=\be^{-1}(0)$ is a smooth $\C$-subscheme of $U'$ near $x$, and
making $U'$ smaller, we can suppose $S$ is smooth. Set
$e=f\vert_S:S\ra\C$. Then the isomorphism $T^*S\cong E\vert_S$
identifies $\d e\in H^0(T^*S)$ with $\al\vert_S\in H^0(E\vert_S)$.
Hence $\Crit(e:S\ra\C)=\Crit(f\vert_{U'}:U'\ra\C)=X\cap U'$, as
$\C$-subschemes of~$U$.

By \cite[Prop.~2.23]{Joyc} quoted in Theorem \ref{sm5thm1}(i) below,
there exist a smooth $\C$-scheme $R$, morphisms $\ga:R\!\ra\! U'$,
$\de:R\!\ra\! S$, $\ep:R\!\ra\!\C^n$ where $n=\dim U'-\dim S$, and
$r\in R$, such that $\ga(r)=x$, $\ga\vert_Q=\de\vert_Q$,
$f\ci\ga=e\ci\de+(z_1^2+\cdots+z_n^2)\ci\ep:R\ra\C$, and the
following commutes with horizontal morphisms \'etale:
\e
\begin{gathered}
\xymatrix@C=80pt@R=13pt{ *+[r]{S} \ar[d]^\subset & Q:=\ga^{-1}(S)
\ar[l]^{\ga\vert_Q} \ar[r]_{\ga\vert_Q} \ar[d]^\subset
& *+[l]{S} \ar[d]_{\id_S \t 0} \\
*+[r]{U'} & R  \ar[r]^{\de\t\ep} \ar[l]_\ga &
*+[l]{S\t\C^n.\!\!{}} }
\end{gathered}
\label{sm3eq36}
\e
Taking derivatives at $r\in Q\subseteq R$ in \eq{sm3eq36} gives a
commutative diagram
\begin{equation*}
\xymatrix@C=80pt@R=13pt{ *+[r]{T_xX=T_xS} \ar[d]^\subset & T_rQ
\ar[l]^(0.3){\d(\ga\vert_Q)\vert_r}_(0.3)\cong
\ar[r]_(0.3){\d(\ga\vert_Q)\vert_r}^(0.3)\cong
\ar[d]^\subset & *+[l]{T_xX=T_xS} \ar[d]_{\id \t 0} \\
*+[r]{T_xX\op N_x=T_xU} & T_rR \ar[l]_(0.3){\d\ga\vert_r}^(0.3)\cong
\ar[r]^(0.3){\d(\de\t\ep)\vert_r}_(0.3)\cong &
*+[l]{T_xX\op T_0\C^n.\!{}} }
\end{equation*}

Therefore $\d(\de\t\ep)\vert_r\ci\d\ga\vert_r^{-1}:T_xX\op N_x\ra
T_xX\op T_0\C^n$ is the identity on $T_xX$, and induces an
isomorphism $N_x\ra T_0\C^n$, which as
$f\ci\ga=e\ci\de+(z_1^2+\cdots+z_n^2)\ci\ep$ identifies
$\Hess_{\smash{x}}'f$ on $N_x$ with $\Hess_0( z_1^2+\cdots+z_n^2)=\d
z_1\ot\d z_1+\cdots+\d z_n\ot\d z_n$ on $T_0\C^n$. Thus, the linear
isomorphism $B:N_x\ra N_x$ above preserving $\Hess_{\smash{x}}'f$ is
identified with a linear isomorphism $M:\C^n\ra\C^n$ preserving $\d
z_1\ot\d z_1+\cdots+\d z_n\ot\d z_n$, that is, $M\in{\rm O}(n,\C)$
satisfies
\e
\begin{pmatrix} \st \id & \st 0 \\ \st 0 & \st B \end{pmatrix}\ci
\d\ga\vert_r\ci\d(\de\t\ep)\vert_r^{-1}=\d\ga\vert_r\ci
\d(\de\t\ep)\vert_r^{-1}\ci \begin{pmatrix} \st \id & \st 0
\\ \st  0 & \st M\end{pmatrix}.
\label{sm3eq37}
\e

Define $P$ to be the $\C$-scheme fibre product $P=R\t_{\de\t(M\ci
\ep),S\t\C^n,\de\t\ep}R$, with projections $\pi_1,\pi_2:P\ra R$.
Then $P$ is smooth and $\pi_1,\pi_2$ are \'etale, as $R,S\t\C^n$ are
smooth and $\de\t (M\ci\ep),\de\t\ep:R\ra S\t\C^n$ are \'etale. As
$r\in R$ with $(\de\t( M\ci\ep))(r)=(x,0)=(\de\t\ep)(r)$, we have a
point $p\in P$ with $\pi_1(p)=\pi_2(p)=r$. Define
$d=f\ci\ga\ci\pi_1:P\ra\C$ and $Z=\Crit(d)$. Then
\e
\begin{split}
d&=f\ci\ga\ci\pi_1=(e\boxplus z_1^2+\cdots+z_n^2)\ci(\de\t\ep)
\ci\pi_1\\
&=(e\boxplus z_1^2+\cdots+z_n^2)\ci(\de\t(M\ci\ep))
\ci\pi_1\\
&=(e\boxplus z_1^2+\cdots+z_n^2)\ci(\de\t\ep)
\ci\pi_2=f\ci\ga\ci\pi_2.
\end{split}
\label{sm3eq38}
\e

Consider the \'etale morphisms $\Phi\ci\ga\ci\pi_1,\Psi\ci\ga
\ci\pi_2:P\ra V$. Both map $p\mapsto \Phi(x)$, and satisfy
$g\ci(\Phi\ci\ga\ci\pi_1)=d=g\ci(\Psi\ci\ga\ci\pi_2)$ by
\eq{sm3eq38} and $g\ci\Phi=f=g\ci\Psi$. Taking derivatives at $p$ to
get linear maps $T_pP\ra T_{\Phi(x)}V$, we find that
\e
\begin{split}
\d(\Psi&\ci\ga\ci\pi_2)\vert_p=\d\Psi\vert_x \ci\d\ga\vert_r\ci
\d(\de\t\ep)\vert_r^{-1}\ci\d((\de\t\ep)\ci\pi_2)\vert_p\\
&=\d\Psi\vert_x\ci\d\ga\vert_r\ci \d(\de\t\ep)\vert_r^{-1}\ci
\d((\de\t (M\ci\ep))\ci\pi_1)\vert_p\\
&=\d\Psi\vert_x\ci\d\ga\vert_r\ci
\d(\de\t\ep)\vert_r^{-1}\ci\begin{pmatrix} \st \id & \st 0
\\ \st  0 & \st M\end{pmatrix}\ci
\d(\de\t\ep)\vert_r\ci
\d\pi_1\vert_p\\
&=\d\Psi\vert_x\ci \begin{pmatrix} \st \id & \st 0 \\ \st 0 & \st B
\end{pmatrix}\ci\d\ga\vert_r\ci\d(\de\t\ep)\vert_r^{-1}\ci
\d(\de\t\ep)\vert_r\ci\d\pi_1\vert_p\\
&= \d\Psi\vert_x\ci\begin{pmatrix} \st \id & \st -AB^{-1} \\ \st 0 &
\st \id
\end{pmatrix}\begin{pmatrix} \st \id & \st A \\
\st 0 & \st B
\end{pmatrix} \ci\d\ga\vert_r\ci
\d\pi_1\vert_p\\
&= \d\Psi\vert_x\ci\begin{pmatrix} \st \id & \st -AB^{-1} \\ \st 0 &
\st \id\end{pmatrix} \ci\d\Psi\vert^{-1}_x \ci
\d(\Phi\ci\ga\ci\pi_1)\vert_p,
\end{split}
\label{sm3eq39}
\e
using $(\de\t (M\ci\ep))\ci\pi_1=(\de\t\ep)\ci\pi_2$ in the second
step, \eq{sm3eq37} in the fourth, and \eq{sm3eq35} in the sixth.
Since $\bigl[\bigl(\begin{smallmatrix} \id &  -AB^{-1} \\ 0 &
\id\end{smallmatrix}\bigr)-\id\bigr]{}^2=0$, equation \eq{sm3eq39}
implies that $\bigl(\d(\Psi\ci\ga\ci\pi_2)\vert_p^{-1}\ci
\d(\Phi\ci\ga\ci\pi_1)\vert_p - \id_{T_pP}\bigr){}^2=0$, and thus
Proposition \ref{sm3prop2} gives a Zariski open neighbourhood $P'$
of $p$ in $P$ such that
\e
\PV_{\Phi\ci\ga\ci\pi_1}\vert_{P'}=\PV_{\Psi\ci\ga\ci\pi_2}
\vert_{P'}:\PV_{P,d}^\bu\vert_{P'}\longra
(\Phi\ci\ga\ci\pi_1)\vert_Z^*\bigl(\PV_{V,g}^\bu\bigr)\vert_{P'}.
\label{sm3eq40}
\e

Since $(\de\t(M\ci\ep))\ci\pi_1=(\de\t\ep)\ci\pi_2:P\ra S\t\C^n$ are
\'etale with
\begin{equation*}
(e\boxplus z_1^2+\cdots+z_n^2)\ci(\de\t(M\ci\ep))\ci\pi_1=d=
(e\boxplus z_1^2+\cdots+z_n^2)\ci(\de\t\ep)\ci\pi_2,
\end{equation*}
we see using \eq{sm2eq8} and \eq{sm2eq18} that
\ea
\pi_1\vert_Z^*&\bigl[\PV_\de\boxtL
\bigl(M\vert_{\{0\}}^*(\PV_\ep)\ci\PV_M\bigr)\bigr]\ci\PV_{\pi_1}
=\pi_1\vert_Z^*\bigl[\PV_\de\boxtL\PV_{M\ci\ep}\bigr] \ci\PV_{\pi_1}
\nonumber\\
&\cong\pi_1\vert_Z^*(\PV_{\de\t(M\ci\ep)})\ci\PV_{\pi_1}
=\PV_{(\de\t(M\ci\ep))\ci\pi_1}=\PV_{(\de\t\ep)\ci\pi_2}
\nonumber\\
&=\pi_2\vert_Z^*(\PV_{\de\t\ep})\ci\PV_{\pi_2}\cong
\pi_2\vert_Z^*\bigl[\PV_\de\boxtL \PV_{\ep}\bigr]\ci\PV_{\pi_2},
\label{sm3eq41}
\ea
where `$\cong$' are equalities after identifying both sides of
\eq{sm2eq8}. Since $\pi_1\vert_Z=\pi_2\vert_Z$, and
$M\vert_{\{0\}}=\id_{\{0\}}$, and Example \ref{sm2ex2} shows that
$\PV_M$ in \eq{sm2eq19} is multiplication by $\det M$, equation
\eq{sm3eq41} implies that
\begin{equation*}
\det M\cdot\pi_1\vert_Z^*\bigl[\PV_\de\boxtL\PV_{\ep}\bigr]\ci
\PV_{\pi_1}=\pi_1\vert_Z^*\bigl[\PV_\de\boxtL \PV_{\ep}\bigr]
\ci\PV_{\pi_2}.
\end{equation*}
As $\pi_1\vert_Z^*\bigl[\PV_\de\boxtL \PV_{\ep}\bigr]$ is an
isomorphism, this gives
\e
\det M\cdot\PV_{\pi_1}=\PV_{\pi_2}:\PV_{P,d}^\bu\longra
\pi_1\vert_Z^*\bigl(\PV_{R,e}^\bu\bigr).
\label{sm3eq42}
\e

Writing $Z'=Z\cap P'$, we now have
\e
\begin{split}
(\ga\ci\pi_1)&\vert_{Z'}^*(\PV_\Phi)\ci\PV_{\ga\ci\pi_1}\vert_{P'}=
\PV_{\Phi\ci\ga\ci\pi_1}\vert_{P'}=\PV_{\Psi\ci\ga\ci\pi_2}
\vert_{P'}\\
&=\pi_2\vert_{Z'}^*(\PV_{\Psi\ci\ga})\ci\PV_{\pi_2}\vert_{P'}\\
&=\det M\cdot
\pi_1\vert_{Z'}^*(\PV_{\Psi\ci\ga})\ci\PV_{\pi_1}\vert_{P'} =\det
M\cdot\PV_{\Psi\ci\ga\ci\pi_1}\vert_{P'} \\
&=\det M\cdot
(\ga\ci\pi_1)\vert_{Z'}^*(\PV_\Psi)\ci\PV_{\ga\ci\pi_1}\vert_{P'},
\end{split}
\label{sm3eq43}
\e
using \eq{sm3eq40} in the second step, \eq{sm3eq42} and
$\pi_1\vert_{Z'}=\pi_2\vert_{Z'}$ in the fourth, and \eq{sm2eq18} in
the rest. As $\PV_{\ga\ci\pi_1}\vert_{P'}$ is an isomorphism,
\eq{sm3eq43} implies that
\begin{equation*}
(\ga\ci\pi_1)\vert_{Z'}^*(\PV_\Phi)=\det M\cdot
(\ga\ci\pi_1)\vert_{Z'}^*(\PV_\Psi),
\end{equation*}
and by Theorem \ref{sm2thm3}(i) this implies that
\e
\PV_\Phi\vert_{X'}=\det M\cdot\PV_\Phi\vert_{X'},
\label{sm3eq44}
\e
where $X'=(\ga\ci\pi_1)(Z')$ is a Zariski open neighbourhood of $x$
in $X$, since $(\ga\ci\pi_1)\vert_{Z'}:Z'\ra X$ is \'etale with
$\ga\ci\pi_1(p)=x$. Now \eq{sm3eq35} and \eq{sm3eq37} give
$\det\bigl(\d\Psi\vert^{-1}_x \ci\d\Phi\vert_x\bigr)=\det\bigl(
\begin{smallmatrix} \id &  A \\ 0 & B \end{smallmatrix}\bigr)=\det
B=\det M$. So \eq{sm3eq44} proves that \eq{sm3eq1} holds near $x$ in
$X$. As this is true for all $x\in X$, Theorem \ref{sm3thm}(b)
follows.

\subsection{$\cD$-modules and mixed Hodge modules}
\label{sm34}

The proof of Proposition \ref{sm3prop1} 
applies verbatim also in the analytic context. Theorem
\ref{sm3thm}(a),(b) then follow from Proposition \ref{sm3prop1} and
the argument given above, using \S\ref{sm25}, including the Sheaf
Property (x) for morphisms. Hence all these results carry over to
our other contexts~\S\ref{sm26}--\S\ref{sm210}.

\section{Dependence of $\PV_{U,f}^\bu$ on $f$}
\label{sm4}

We will use the following notation:

\begin{dfn} Let $U$ be a smooth $\C$-scheme, $f:U\ra\C$ a regular
function, and $X=\Crit(f)$ as a closed $\C$-subscheme of $U$. Write
$I_X\subseteq \O_U$ for the sheaf of ideals of regular functions
$U\ra\C$ vanishing on $X$, so that $I_X=I_{\d f}$. For each
$k=1,2,\ldots,$ write $X^{(k)}$ for the $k^{\rm th}$ {\it order
thickening of\/ $X$ in\/} $U$, that is, $X^{(k)}$ is the closed
$\C$-subscheme of $U$ defined by the sheaf of
ideals $I_X^k$ in $\O_U$. Also write $X^\red$ for the reduced
$\C$-subscheme of $U$.

Then we have a chain of inclusions of closed $\C$-subschemes
\e
X^\red\subseteq X=X^{(1)}\subseteq X^{(2)}\subseteq X^{(3)}\subseteq
\cdots\subseteq U.
\label{sm4eq1}
\e
Write $f^{(k)}:=f\vert_{\smash{X^{(k)}}}:X^{(k)}\ra\C$, and
$f^\red:=f\vert_{\smash{X^\red}}:X^\red\ra\C$, so that
$f^{(k)},f^\red$ are regular functions on the $\C$-schemes
$X^{(k)},X^\red$. Note that
$f^\red:X^\red\ra\C$ is locally constant, since~$X=\Crit(f)$.

We also use the same notation for complex analytic spaces.
\label{sm4def1}
\end{dfn}

In \S\ref{sm24} we defined the perverse sheaf of vanishing cycles
$\PV_{U,f}^\bu$ in $\Perv(X)$. So we can ask: how much of the
sequence \eq{sm4eq1} does $\PV_{U,f}^\bu$ depend on? That is, is
$\PV_{U,f}^\bu$ (canonically?) determined by $(X^\red,f^\red)$, or
by $(X^{(k)},f^{(k)})$ for some $k\geq 1$, as well as by $(U,f)$? 
Our next theorem shows that
$\PV_{U,f}^\bu$ is determined up to canonical isomorphism by
$(X^{(3)},f^{(3)})$, and hence a fortiori also by
$(X^{(k)},f^{(k)})$ for $k>3$:

\begin{thm} Let\/ $U,V$ be smooth\/ $\C$-schemes, $f:U\ra\C,$
$g:V\ra\C$ be regular functions, and\/ $X=\Crit(f),$ $Y=\Crit(g)$ as
closed\/ $\C$-subschemes of\/ $U,V,$ so that\/ {\rm\S\ref{sm24}}
defines perverse sheaves $\PV_{U,f}^\bu,\PV_{V,g}^\bu$ on $X,Y$.
Define $X^{(3)},f^{(3)}$ and\/ $Y^{(3)},g^{(3)}$ as in Definition\/
{\rm\ref{sm4def1},} and suppose $\Phi:X^{(3)}\ra Y^{(3)}$ is an
isomorphism with\/ $g^{(3)}\ci\Phi=f^{(3)},$ so that\/
$\Phi\vert_X:X\ra Y\subseteq Y^{(3)}$ is an isomorphism.

Then there is a canonical isomorphism in $\Perv(X)$
\e
\Om_\Phi:\PV_{U,f}^\bu\longra \Phi\vert_X^*(\PV_{V,g}^\bu),
\label{sm4eq2}
\e
which is characterized by the property that if\/ $T$ is a smooth\/
$\C$-scheme and\/ $\pi_U:T\ra U,$ $\pi_V:T\ra V$ are \'etale
morphisms with\/ $e:=f\ci\pi_U=g\ci\pi_V:T\ra\C,$ so that\/
$\pi_U\vert_Q:Q\ra X,$ $\pi_V\vert_Q:Q\ra Y$ are \'etale for
$Q:=\Crit(e),$ and\/ $\Phi\ci\pi_U\vert_{Q^{(2)}}=\pi_V
\vert_{Q^{(2)}}:Q^{(2)}\ra Y^{(2)},$ then
\e
\pi_U\vert_Q^*(\Om_\Phi)\ci\PV_{\pi_U}=\PV_{\pi_V}:\PV_{T,e}^\bu\longra
\pi_V\vert_Q^*(\PV_{U,f}^\bu).
\label{sm4eq3}
\e

Also the following commute, where $\si_{U,f},\si_{V,g},\tau_{U,f},
\tau_{V,g}$ are as in\/~{\rm\eq{sm2eq6}--\eq{sm2eq7}:}
\ea
\begin{gathered}
\xymatrix@!0@C=140pt@R=35pt{ *+[r]{\PV_{U,f}^\bu}
\ar[rr]_{\si_{U,f}} \ar[d]^{\Om_\Phi} &&
*+[l]{\bD_X(\PV_{U,f}^\bu)} \\
*+[r]{\Phi\vert_X^*(\PV_{V,g}^\bu)} \ar[r]^{\Phi\vert_X^*(\si_{V,g})}
& {\Phi\vert_X^*\bigl(\bD_Y(\PV_{V,g}^\bu)\bigr)} \ar[r]^(0.37)\cong
& *+[l]{\bD_X\bigl(\Phi\vert_X^*(\PV_{V,g}^\bu)\bigr),\!\!{}}
\ar[u]^{\bD_X(\Om_\Phi)} }
\end{gathered}
\label{sm4eq4}\\
\begin{gathered}
\xymatrix@!0@C=280pt@R=35pt{ *+[r]{\PV_{U,f}^\bu}
\ar[r]_{\tau_{U,f}} \ar[d]^{\Om_\Phi} &
*+[l]{\PV_{U,f}^\bu} \ar[d]_{\Om_\Phi} \\
*+[r]{\Phi\vert_X^*(\PV_{V,g}^\bu)} \ar[r]^{\Phi\vert_X^*(\tau_{V,g})}
& *+[l]{\Phi\vert_X^*(\PV_{V,g}^\bu).\!\!{}} }
\end{gathered}
\label{sm4eq5}
\ea

If there exists an \'etale morphism $\Xi:U\ra V$ with\/
$g\ci\Xi=f:U\ra\C$ and\/ $\Xi\vert_{X^{(3)}}=\Phi:X^{(3)}\ra
Y^{(3)}$ then $\Om_\Phi=\PV_\Xi,$ for $\PV_\Xi$ as in~\eq{sm2eq14}.

If\/ $W$ is another smooth\/ $\C$-scheme, $h:W\ra\C$ is regular,
$Z=\Crit(h),$ and\/ $\Psi:Y^{(3)}\ra Z^{(3)}$ is an isomorphism
with\/ $h^{(3)}\ci\Psi=g^{(3)}$ then
\e
\Om_{\Psi\ci\Phi}=\Phi\vert_X^*(\Om_\Psi)\ci\Om_\Phi:\PV_{U,f}^\bu
\longra(\Psi\ci\Phi)\vert_X^*(\PV_{W,h}^\bu).
\label{sm4eq6}
\e
If\/ $U=V,$ $f=g,$ $X=Y$ and\/ $\Phi=\id_{X^{(3)}}$
then\/~$\Om_{\id_{X^{(3)}}}=\id_{\PV_{U,f}^\bu}$.

The analogues of all the above also hold with appropriate
modifications for $\cD$-modules on
$\C$-schemes, for perverse sheaves and $\cD$-modules on complex
analytic spaces, and for mixed Hodge modules on $\C$-schemes and
complex analytic spaces,
as in\/~{\rm\S\ref{sm26}--\S\ref{sm210}}.
\label{sm4thm}
\end{thm}

We will prove Theorem \ref{sm4thm} in \S\ref{sm42}--\S\ref{sm43}.
The proof for $\C$-schemes depends on the case $k=2$ of the
following proposition, proved in~\S\ref{sm41}:

\begin{prop} Let\/ $U,V$ be smooth\/ $\C$-schemes, $f:U\ra\C,$
$g:V\ra\C$ be regular functions, and\/ $X=\Crit(f)\subseteq U,$
$Y=\Crit(g)\subseteq V$. Using the notation of Definition\/
{\rm\ref{sm4def1},} suppose $\Phi:X^{(k+1)}\ra Y^{(k+1)}$ is an
isomorphism with\/ $g^{(k+1)}\ci\Phi=f^{(k+1)}$ for some $k\ge 2$.
Then for each\/ $x\in X$ we can choose a smooth\/ $\C$-scheme $T$
and \'etale morphisms $\pi_U:T\ra U,$ $\pi_V:T\ra V$ such that\/
\begin{itemize}
\setlength{\itemsep}{0pt}
\setlength{\parsep}{0pt}
\item[{\bf(a)}] $e:=f\ci\pi_U=g\ci\pi_V:T\ra\C;$
\item[{\bf(b)}] setting $Q=\Crit(e),$ then
$\pi_U\vert_{Q^{(k)}}:Q^{(k)}\ra X^{(k)}\subseteq U$ is an
isomorphism with a Zariski open neighbourhood\/ $\ti X^{(k)}$
of\/ $x$ in $X^{(k)};$ and
\item[{\bf(c)}] $\Phi\ci\pi_U\vert_{Q^{(k)}}=\pi_V
\vert_{Q^{(k)}}:Q^{(k)}\ra Y^{(k)}$.
\end{itemize}
\label{sm4prop1}
\end{prop}

The proof of Proposition \ref{sm4prop1} is similar to that of
Proposition \ref{sm3prop2} in \S\ref{sm31}. One can also prove an
analogue of Proposition \ref{sm4prop1} when $k=1$, but in part (b)
$\pi_U\vert_{Q^{(1)}}:Q^{(1)}\ra X^{(1)}$ must be \'etale rather
than a Zariski open inclusion.

In Proposition \ref{sm4prop1}, we start with
$\Phi:X^{(k+1)}\,{\buildrel\cong\over\longra}\, Y^{(k+1)}$, but we
construct $T,\pi_U,\pi_V$ with $\Phi\ci\pi_U\vert_{Q^{(k)}}=\pi_V
\vert_{Q^{(k)}}$. One might expect to find $T,\pi_U,\pi_V$ with
$\Phi\ci\pi_U\vert_{Q^{(k+1)}}=\pi_V \vert_{Q^{(k+1)}}$, but the
next example shows this is not possible.

\begin{ex} Let $U,V$ be open neighbourhoods of 0 in $\C$, and
$f:U\ra\C$, $g:V\ra\C$ be regular functions given as power series by
$f(x)=x^{m+1}$ and $g(y)=y^{m+1}+Ay^{(k+1)m}+\cdots$, for $k,m\ge 2$
and $0\ne A\in\C$, where 0 is the only critical point of $g$. Then
$X:=\Crit(f)=\Spec\bigl(\C[x]/(x^m)\bigr)$ and
$Y:=\Crit(g)=\Spec\bigl(\C[y]/(y^m)\bigr)$, so
$X^{(k+1)}=\Spec\bigl(\C[x]/(x^{(k+1)m})\bigr)$,
$f^{(k+1)}=x^{m+1}+(x^{(k+1)m})$, and
$Y^{(k+1)}=\Spec\bigl(\C[y]/(y^{(k+1)m})\bigr)$,
$g^{(k+1)}=y^{m+1}+(y^{(k+1)m})$. Thus $\Phi:X^{(k+1)}\ra Y^{(k+1)}$
acing on functions by $y+(y^{(k+1)m})\mapsto x+(x^{(k+1)m})$ is an
isomorphism with~$f^{(k+1)}=g^{(k+1)}\ci\Phi$.

Suppose $T,\pi_U,\pi_V$ are as in Proposition \ref{sm4prop1}, and
use $w=x\ci\pi_U$ as a coordinate on $T$. Then $e(w)=w^{m+1}$, and
$Q=\Crit(e)=\Spec\bigl(\C[w]/(w^m)\bigr)$. We have $\pi_U(w)=w$, so
$\Phi\ci\pi_U(w)=w$, but $w^{m+1}=\pi_V(w)^{m+1}+
A\pi_V(w)^{(k+1)m}+\cdots,$ so that $\pi_V(w)=w-\frac{1}{m+1}
Aw^{km}+\cdots$. Thus, $\Phi\ci\pi_U:T\ra V$ and $\pi_V:T\ra V$
differ by $-\frac{1}{m+1}Aw^{km}+\cdots$, which is zero on $Q^{(k)}$
but not on $Q^{(k+1)}$. Hence in this example there do not exist
$T,\pi_U,\pi_V$ with~$\Phi\ci\pi_U\vert_{Q^{(k+1)}}=\pi_V
\vert_{Q^{(k+1)}}$.
\label{sm4ex1}
\end{ex}

\begin{rem} We can also ask: can we improve $(X^{(3)},f^{(3)})$ in
Theorem \ref{sm4thm} to $(X^{(2)},f^{(2)})$ or $(X^{(1)},f^{(1)})$
or $(X^\red,f^\red)$? Here are some thoughts on this.
\smallskip

\noindent{\bf(a)} The analogue of Proposition \ref{sm4prop1} for
$k=1$ mentioned above implies that \'etale or complex analytically
locally on $X$, $(U,f)$ and hence $\smash{\PV_{U,f}^\bu}$ are
determined up to {\it non-canonical\/} isomorphism by
$(X^{(2)},f^{(2)})$. Using the ideas of \S\ref{sm5}--\S\ref{sm6},
one can show that these non-canonical isomorphisms of
$\smash{\PV_{U,f}^\bu}$ are unique {\it up to sign}.
\smallskip

\noindent{\bf(b)} Consider the following example: let
$U=(\C\sm\{0\})\t\C=V$, and define $f:U\ra\C$ and $g:V\ra\C$ by
$f(x,y)=y^2$ and $g(x,y)=xy^2$. Then
$X:=\Crit(f)=\{y=0\}=\Crit(g)=:Y$, and $f^{(2)}=g^{(2)}=0$, so that
$(X^{(2)},f^{(2)})=(Y^{(2)},g^{(2)})$. However, as in Example
\ref{sm5ex} below, $\PV_{U,f}^\bu\not\cong\PV_{V,g}^\bu$. Thus,
globally, $\smash{\PV_{U,f}^\bu}$ is not determined up to
isomorphism by~$(X^{(2)},f^{(2)})$.
\smallskip

\noindent{\bf(c)} Suppose $U$ is a complex manifold and $f:U\ra\C$
is holomorphic, with $\Crit(f)$ a single (not necessarily reduced)
point $x$. The {\it Mather--Yau Theorem\/} \cite{MaYa} shows that
the germ of $(U,f)$ at $x$ is determined up to non-canonical
isomorphism by the complex analytic subspace $f^{(1)}=0$ in
$X^{(1)}$, and hence by the pair $(X^{(1)},f^{(1)})$. Therefore, for
{\it isolated\/} singularities, $\smash{\PV_{U,f}^\bu}$ is
determined up to non-canonical isomorphism by~$(X^{(1)},f^{(1)})$.
\smallskip

\noindent{\bf(d)} Define $f:U\ra\C$ by $U=\C$ and $f(z)=cz^n$ for
$0\ne c\in\C$ and $n>2$. This has an isolated singularity at $0$,
and $(X^{(1)},f^{(1)})$ is independent of $c$. By moving $c$ in a
circle round zero, we see that in this example
$\smash{\PV_{U,f}^\bu}$ is determined up to a $\Z/n\Z$ group of
automorphisms. So the non-canonical isomorphisms of
$\smash{\PV_{U,f}^\bu}$ are not unique up to sign, in contrast
to~{\bf(a)}.
\smallskip

\noindent{\bf(e)} Parts {\bf(a)}\rm--{\bf(d)} leave open the
question of whether $\smash{\PV_{U,f}^\bu}$ is determined locally up
to non-canonical isomorphism by $(X^{(1)},f^{(1)})$ for non-isolated
singularities. We do not have a counterexample to this.

However, Gaffney and Hauser \cite[\S 4]{GaHa} give examples of
complex manifolds $U$ and holomorphic $f:U\ra\C$ with $X=\Crit(f)$
non-isolated, such that the germ of $(U,f)$ at $x\in X$ is not
determined up to non-canonical isomorphism by the germ of
$(X^{(1)},f^{(1)})$ at $x$, in contrast to the Mather--Yau Theorem,
and continuous families of distinct germs $[U,f,x]$ can have the
same germ $[X^{(1)},f^{(1)},x]$. It seems likely that in examples of
this kind, the mixed Hodge module $\HV_{U,f}^\bu$ (which contains
continuous Hodge-theoretic information) is not locally determined up
to non-canonical isomorphism by~$(X^{(1)},f^{(1)})$.
\smallskip

\noindent{\bf(f)} For the example in {\bf(d)}, $\PV_{U,f}^\bu$
depends on $n=3,4,\ldots,$ but $(X^\red,f^\red)=(\{0\},0)$ is
independent of $n$. So $\smash{\PV_{U,f}^\bu}$ is not determined
even locally up to non-canonical isomorphism by~$(X^\red,f^\red)$.
\label{sm4rem}
\end{rem}

\subsection{Proof of Proposition \ref{sm4prop1}}
\label{sm41}

The $\C$-subscheme $X^{(k+1)}$ in $U$ is the zeroes of the ideal
$I_X^{k+1}\subset\O_U$, which vanishes to order $k+1\ge 2$ at $x\in
X\subseteq X^{(k+1)}\subseteq U$. Hence $T_xX^{(k+1)}=T_xU$. As
$\Phi:X^{(k+1)}\ra Y^{(k+1)}$ is an isomorphism, it follows that
\e
T_xU=T_xX^{(k+1)}\cong T_{\Phi(x)}Y^{(k+1)}=T_{\Phi(x)}V.
\label{sm4eq7}
\e
Therefore $n:=\dim U=\dim V$.

Choose a Zariski open neighbourhood $V'$ of $\Phi(x)$ in $V$ and
\'etale coordinates $(y_1,\ldots,y_n):V'\ra\C^n$ on $V'$. Write
$g'=g\vert_{V'}$ and $Y'=\Crit(g')=Y\cap V'$, so that
$Y^{\prime(k+1)}=Y^{(k+1)}\cap V'$. Then $y_a\ci\Phi$ are regular
functions on the open neighbourhood $\Phi^{-1}(V')\subseteq
X^{(k+1)}$ of $x$ in $X^{(k+1)}$, so they extend Zariski locally
from $X^{(k+1)}$ to $U$. Thus we can choose a Zariski open
neighbourhood $U'$ of $x$ in $U$ with $\Phi(X^{(k+1)}\cap U')
\subseteq Y^{(k+1)}\cap V'$, and regular functions $x_i:U'\ra\C$
with $x_i\vert_{X^{(k+1)}\cap U'}=y_i\ci\Phi\vert_{X^{(k+1)}\cap
U'}$ for~$i=1,\ldots,n$.

Write $f'=f\vert_{U'}$ and $X'=\Crit(f')=X\cap U'$, so that
$X^{\prime(k+1)}=X^{(k+1)}\cap X'$. Since $(y_1,\ldots,y_n)$ are
\'etale coordinates, $\d y_1\vert_{\Phi(x)}, \ldots,\d
y_n\vert_{\Phi(x)}$ are a basis for $T_{\Phi(x)}^*V$, so $\d
x_1\vert_x,\ldots,\d x_n\vert_x$ are a basis for $T_x^*X$ by
\eq{sm4eq7}. Hence by making $U'$ smaller, we can suppose
$(x_1,\ldots,x_n)$ are \'etale coordinates on~$U'$.

Consider the $\C$-scheme $U'\t V'$, with projections $\pi_{U'}:U'\t
V'\ra U'$, $\pi_{V'}:U'\t V'\ra V'$, and write $x_i'=x_i\ci\pi_{U'}:
U'\t V'\ra\C$, $y_i'=y_i\ci\pi_{V'}: U'\t V'\ra\C$, so that
$(x_1',\ldots,x_n',y_1',\ldots,y_n')$ are \'etale coordinates on
$U'\t V'$. We have a morphism $\id\t\Phi\vert_{X^{\prime(k+1)}}:
X^{\prime(k+1)}\ra U'\t V'$ which embeds $X^{\prime(k+1)}$ as a
closed $\C$-subscheme of $U'\t V'$. The image
$\bigl(\id\t\Phi\vert_{X^{\prime(k+1)}}\bigr)(X^{\prime(k+1)})$ is
locally the zeroes of the sheaf of ideals
\begin{equation*}
\bigl(x_i'-y_i',\;i=1,\ldots,n\bigr)+
\pi_{U'}^{-1}\bigl(I_X^{k+1}\bigr)\subset\O_{U'\t V'},
\end{equation*}
where $\bigl(x_i'-y_i',i=1,\ldots,n\bigr)$ denotes the ideal
generated by $x_i'-y_i':U'\t V'\ra\C$ for $i=1,\ldots,n$, and
$\pi_{U'}^{-1}\bigl(I_X^{k+1}\bigr)\subset \O_{U'\t V'}$ the
preimage ideal of $I_X^{k+1}\vert_{U'}\subset\O_{U'}$.

Now $\bigl(f\ci\pi_{U'}- g\ci\pi_{V'}\bigr)\vert_{(\id \t\Phi)
(X^{\prime(k+1)})}=0$ as $f^{(k+1)}=g^{(k+1)}\ci\Phi$. Hence
\e
f\ci\pi_{U'}-g\ci\pi_{V'}\in \bigl(x_i'-y_i',\;i=1,\ldots,n\bigr)+
\pi_{U'}^{-1}\bigl(I_X^{k+1}\bigr).
\label{sm4eq8}
\e
Lifting \eq{sm4eq8} from $\bigl(x_i'-y_i',\;i=1,\ldots,n\bigr)$ to
$\bigl(x_i'-y_i',\;i=1,\ldots,n\bigr)^2$, after making $U',V'$
smaller if necessary, we can choose regular functions $A_i:U'\t
V'\ra\C$ for $i=1,\ldots,n$ such that
\e
f\ci\pi_{U'}\!-\!g\ci\pi_{V'}\!-\!\ts\sum\limits_{i=1}^nA_i
\cdot(x_i'\!-\!y_i')\in
\bigl(x_i'\!-\!y_i',\;i=1,\ldots,n\bigr)^2\!+\!
\pi_{U'}^{-1}\bigl(I_X^{k+1}\bigr).
\label{sm4eq9}
\e
Apply $\frac{\pd}{\pd x_i'}$ to \eq{sm4eq9}, using the \'etale
coordinates $(x_1',\ldots,x_n',y_1',\ldots,y_n')$ on $U'\t V'$.
Since $\frac{\pd}{\pd x_i'}\bigl(f\ci\pi_{U'}\bigr)=\frac{\pd f}{\pd
x_i}\ci\pi_{U'}$ and $\frac{\pd}{\pd
x_i'}\bigl(g\ci\pi_{V'}\bigr)=0$, this gives
\e
\ts\frac{\pd f}{\pd x_i}\ci\pi_{U'}-A_i\in
\bigl(x_i'-y_i',\;i=1,\ldots,n\bigr)+
\pi_{U'}^{-1}\bigl(I_X^k\bigr).
\label{sm4eq10}
\e

Changing $A_i$ by an element of
$\bigl(x_i'-y_i',\;i=1,\ldots,n\bigr)$ can be absorbed in the ideal
$\bigl(x_i'-y_i',\;i=1,\ldots,n\bigr)^2$ in \eq{sm4eq9}, so we can
suppose $\frac{\pd f}{\pd x_i}\ci\pi_{U'}-A_i\in
\pi_{U'}^{-1}\bigl(I_X^k\bigr)$. As $I_X=\bigl(\frac{\pd f}{\pd
x_j},\;j=1,\ldots,n\bigr)$, after making $U',V'$ smaller we may
write
\e
A_i=\ts\frac{\pd f}{\pd x_i}\ci\pi_{U'}+\sum\limits_{j=1}^n
B_{ij}\cdot\frac{\pd f}{\pd x_j}\ci\pi_{U'},
\label{sm4eq11}
\e
with $B_{ij}\in\pi_{U'}^{-1}\bigl(I_X^{k-1}\bigr)$ for
$i,j=1,\ldots,n$. Consider the matrix of functions
$\bigl(\de_{ij}+B_{ij}\bigr){}_{i,j=1}^n$ on $U'\t V'$. At the point
$(x,\Phi(x))$ in $U'\t V'$ this matrix is the identity, since
$B_{ij}(x,\Phi(x))=0$ as $B_{ij}\in \pi_{U'}^{-1}(I_X^{k-1})$ with
$k\ge 2$, so $\bigl(\de_{ij}+B_{ij}\bigr){}_{i,j=1}^n$ is invertible
near $(x,\Phi(x))$, and making $U',V'$ smaller we can suppose
$\bigl(\de_{ij}+B_{ij}\bigr){}_{i,j=1}^n$ is invertible on $U'\t
V'$. But in matrix notation we have
\begin{equation*}
\bigl(A_i\bigr){}_{i=1}^n=\bigl(\de_{ij}+B_{ij}\bigr){}_{i,j=1}^n\bigl(
\ts\frac{\pd f}{\pd x_j}\ci\pi_{U'}\bigr){}_{j=1}^n.
\end{equation*}
Hence in ideals in $\O_{U'\t V'}$ we have
\e
\bigl(A_i,\;i=1,\ldots,n\bigr)=\bigl( \ts\frac{\pd f}{\pd
x_j}\ci\pi_{U'},\;j=1,\ldots,n\bigr)=\pi_{U'}^{-1}\bigl(I_X\bigr)
\subset\O_{U'\t V'}.
\label{sm4eq12}
\e

Now by \eq{sm4eq9}, after making $U',V'$ smaller if necessary, we
may write
\e
\begin{split}
f\ci\pi_{U'}&-g\ci\pi_{V'}=\ts\sum\limits_{i=1}^nA_i\cdot
\bigl(x_i'-y_i'\bigr)\\
&\quad+\ts\sum\limits_{i,j=1}^nC_{ij}\cdot\bigl(x_i'-y_i'\bigr)
\bigl(x_j'-y_j'\bigr)+\sum\limits_{i,j=1}^nD_{ij}\cdot A_i A_j,
\end{split}
\label{sm4eq13}
\e
for regular functions $C_{ij},D_{ij}:U'\t V'\ra\C$ with $D_{ij}\in
H^0\bigl(\pi_{U'}^{-1}(I_X^{k-1})\bigr)$ for $i,j=1,\ldots,n$, where
in the last term we have used \eq{sm4eq12} to write two factors of
$\pi_{U'}^{-1}(I_X)$ in terms of $A_1,\ldots,A_n$.

Write $(z_{ij})_{i,j=1}^n$ for the coordinates on $\C^{n^2}$. Let
$W$ be a Zariski open neighbourhood of
$\bigl(x,\Phi(x),(0)_{i,j=1}^n\bigr)$ in $U'\t V'\t\C^{n^2}$ to be
chosen shortly, and let $T$ be the $\C$-subscheme of $W$ defined by
\e
\begin{split}
T=\bigl\{&\bigl(u,v,(z_{ij})_{i,j=1}^n\bigr)\in W\subseteq U'\t
V'\t\C^{n^2}: \\
&x_i(u)-y_i(v)=\ts\sum\limits_{j=1}^nz_{ij}\cdot A_j(u,v),\quad
i=1,\ldots,n,
\\
&z_{ij}+\ts\sum\limits_{l,m=1}^nC_{lm}(u,v)\cdot
z_{li}z_{mj}+D_{ij}(u,v) =0,\quad i,j=1,\ldots,n\bigr\}.
\end{split}
\label{sm4eq14}
\e
Define $\C$-scheme morphisms $\pi_U:T\ra U$ by
$\pi_U:\bigl(u,v,(z_{ij})_{i,j=1}^n\bigr)\mapsto u$ and $\pi_V:T\ra
V$ by~$\pi_V:\bigl(u,v,(z_{ij})_{i,j=1}^n\bigr)\mapsto v$.

Now $W\subseteq U'\t V'\t\C^{n^2}$ is smooth of dimension $n+n+n^2$,
and in \eq{sm4eq14} we impose $n+n^2$ equations, so the expected
dimension of $T$ is $(2n+n^2)-(n+n^2)=n$. The linearizations of the
$n+n^2$ equations in \eq{sm4eq14} at $\bigl(u,v,
(z_{ij})_{i,j=1}^n\bigr)=\bigl(x,\Phi(x),(0)_{i,j=1}^n\bigr)$ are
\e
\begin{split}
\d x_i\vert_x(\de u)-\d y_i\vert_{\Phi(x)}(\de v)=0,\quad
i=1,\ldots,n,\\
\de z_{ij}+\d D_{ij}\vert_{(x,\Phi(x))}(\de u\op\de v)=0, \quad
i,j=1,\ldots,n,
\end{split}
\label{sm4eq15}
\e
for $\de u\in T_xU'$, $\de v\in T_{\Phi(x)}V'$, and $(\de
z_{ij})_{i,j=1}^n\in T_{(0)_{i,j=1}^n}\C^{n^2}$. As $\d
x_1\vert_x,\ldots,\d x_n\vert_x$ are a basis for $T_x^*U'$, the
equations \eq{sm4eq15} are transverse, so that $T$ is smooth of
dimension $n$ near $\bigl(x,\Phi(x),(0)_{i,j=1}^n\bigr)$.

The vector space of solutions $\bigl(\de u,\de v,(\de
z_{ij})_{i,j=1}^n\bigr)$ to \eq{sm4eq15} is $T_{(x,\Phi(x),(0))}T$,
where $\d\pi_U\vert_{(x,\Phi(x),(0))}: T_{(x,\Phi(x),(0))}T\ra T_xU$
maps $\bigl(\de u,\de v,(\de z_{ij})_{i,j=1}^n\bigr)\mapsto\de u$,
and $\d\pi_V\vert_{(x,\Phi(x),(0))}: T_{(x,\Phi(x),(0))}T\ra
T_{\Phi(x)}V$ maps $\bigl(\de u,\de v,(\de
z_{ij})_{i,j=1}^n\bigr)\mapsto\de v$. Clearly,
$\d\pi_U\vert_{(x,\Phi(x),(0))},\d\pi_V\vert_{(x,\Phi(x),(0))}$ are
isomorphisms, so as $T$ is smooth near $\bigl(x,\Phi(x),
(0)_{i,j=1}^n\bigr)$ and $U,V$ are smooth, we see that $\pi_U,\pi_V$
are \'etale near $\bigl(x,\Phi(x),(0)\bigr)$. Thus, by choosing the
open neighbourhood $\bigl(x,\Phi(x),(0)\bigr)\in W\subseteq U'\t
V'\t\C^{n^2}$ sufficiently small, we can suppose that $T$ is smooth
of dimension $n$ and $\pi_U:T\ra U$, $\pi_V:T\ra V$ are \'etale.

It remains to prove Proposition \ref{sm4prop1}(a)--(c). For (a), we
have
\begin{align*}
\bigl(f\ci&\pi_U\!-\!g\ci\pi_V\bigr)\bigl(u,v,(z_{ij})_{i,j=1}^n\bigr)
\!=\!f(u)\!-\!g(v)\!=\!(f\ci\pi_{U'}\!-\!g\ci\pi_{V'})(u,v)\\
&=\ts\sum\limits_{i=1}^nA_i(u,v)\cdot\bigl(x_i(u)\!-\!y_i(v)\bigr)
\!+\!\ts\sum\limits_{i,j=1}^n\begin{aligned}[t]
C_{ij}(u,v)\cdot\bigl(x_i(u)\!-\!y_i(v)\bigr)&\\
\bigl(x_j(u)\!-\!y_j(v)\bigr)&\end{aligned} \\[-7pt]
&\quad+\ts\sum\limits_{i,j=1}^nD_{ij}(u,v)\cdot A_i(u,v)A_j(u,v) \\
&= \ts\sum\limits_{i=1}^nA_i(u,v)\cdot\Bigl(\,
\sum\limits_{j=1}^nz_{ij}\cdot A_j(u,v)\Bigr)
\!+\!\ts\sum\limits_{i,j=1}^n\begin{aligned}[t]
C_{ij}(u,v)\cdot\Bigl(\,
\ts\sum\limits_{l=1}^nz_{il}\cdot A_l(u,v)\Bigr)&\\
\ts\Bigl(\,\sum\limits_{m=1}^nz_{jm}\cdot
A_m(u,v)\Bigr)&\end{aligned} \\[-15pt]
&\quad+\ts\sum\limits_{i,j=1}^nD_{ij}(u,v)\cdot A_i(u,v)A_j(u,v) \\
&=\ts\sum\limits_{i,j=1}^nA_i(u,v)A_j(u,v)\Bigl[
z_{ij}\!+\!\ts\sum\limits_{l,m=1}^nC_{lm}(u,v)\cdot
z_{li}z_{mj}\!+\!D_{ij}(u,v) \Bigr]=0,
\end{align*}
using \eq{sm4eq13} in the third step, the first equation of
\eq{sm4eq14} in the fourth, rearranging and exchanging labels $i,l$
and $j,m$ in the fifth, and the second equation of \eq{sm4eq14} in
the sixth. Hence $f\ci\pi_U-g\ci\pi_V=0:T\ra\C$, proving (a).

For (b), using the morphism $\id\t\Phi\vert_X\t(0):X\ra U\t
V\t\C^{n^2}\supseteq W$, define $\ti
X=\bigl(\id\t\Phi\vert_X\t(0)\bigr)^{-1}(W)$, so that $\ti X$ is a
Zariski open neighbourhood of $x$ in $X$. Then
$\bigl(\id\t\Phi\vert_X\t(0)\bigr)(\ti X)$ is a closed
$\C$-subscheme of $W$. We claim that:
\begin{itemize}
\setlength{\itemsep}{0pt}
\setlength{\parsep}{0pt}
\item[(i)] $\bigl(\id\t\Phi\vert_X\t(0)\bigr)(\ti X)$ is a closed
$\C$-subscheme of $T\subseteq W$; and
\item[(ii)] $\bigl(\id\t\Phi\vert_X\t(0)\bigr)(\ti X)$ is open
and closed in $Q:=\Crit(e)\subseteq T$,
\end{itemize}
where $e:=f\ci\pi_U=g\ci\pi_V:T\ra\C$. To prove (i), we have to show
that the equations of \eq{sm4eq14} hold on
$\bigl(\id\t\Phi\vert_X\t(0)\bigr)(\ti X)$, which is true as
$x_i\vert_{\ti X}=y_i\ci\Phi\vert_{\ti X}$, and $z_{ij}\ci(0)=0$,
and $D_{ij}\ci(\id\t\Phi\vert_{\ti X})=0$ as $D_{ij}\in
H^0\bigl(\pi_{U'}^{-1}(I_X^{k-1})\bigr)$ for~$k\ge 2$.

For (ii), as $\pi_U:T\ra U$ is \'etale with $e=f\ci\pi_U$, we see
that $\pi_U\vert_Q:Q\ra X$ is \'etale. But
$\pi_U\vert_Q\ci\bigl(\id\t\Phi\vert_X\t(0)\bigr)\vert_{\ti
X}=\id_{\ti X}$. Hence $\bigl(\id\t\Phi\vert_X\t(0)\bigr)(\ti X)$ is
open in $Q$, and is also closed in $Q$ as it is closed in $T$. Thus,
by making $W,T$ smaller to delete other components of $Q$, we can
suppose that $Q=\bigl(\id\t\Phi\vert_X\t(0)\bigr)(\ti X)$. Then
$\pi_U\vert_Q:Q\ra\ti X$ is an isomorphism with the Zariski open
neighbourhood $\ti X$ of $x$ in $X$. Since $\pi_U:T\ra U$ is \'etale
with $e=f\ci\pi_U$, this extends to the $k^{\rm th}$ order
thickenings, so $\pi_U\vert_{Q^{(k)}}:Q^{(k)}\ra\ti X^{(k)}$ is an
isomorphism, proving~(b).

For (c), first note that $Q=\bigl(\id\t\Phi\vert_X\t(0)\bigr)(\ti
X)$, so $\Phi\ci\pi_U\vert_Q=\pi_V\vert_Q$ is immediate. We have to
extend this to the thickening $Q^{(k)}$. Write $I_Q\subset\O_T$ for
the ideal of functions vanishing on $Q$. Then $I_Q=\pi_U^{-1}(I_X)$
as $\pi_U$ identifies $Q$ with $\ti X\subseteq X$. We have
\begin{equation*}
A_i\ci\pi_{U'\t V'}\in I_Q\quad\text{and}\quad  D_{ij}\ci\pi_{U'\t
V'}\in I_Q^{k-1},
\end{equation*}
as $A_i\in H^0\bigl(\pi_{U'}^{-1}(I_X)\bigr)$, $D_{ij}\in
H^0\bigl(\pi_{U'}^{-1}(I_X^{k-1})\bigr)$. The second equation of
\eq{sm4eq14} then shows that
\begin{equation*}
z_{ij}\ci\pi_{\C^{n^2}}\in I_Q^{k-1},
\end{equation*}
since $Q=\bigl(\id\t\Phi\vert_X\t(0)\bigr)(\ti X)$ implies that
$z_{ij}\ci\pi_{\C^{n^2}}=0$ on $Q$, so we can neglect the terms
$\sum_{l,m=1}^nC_{lm}(u,v)\cdot z_{li}z_{mj}$. Hence the first
equation of \eq{sm4eq14} gives
\begin{equation*}
x_i\ci\pi_U-y_i\ci\pi_V\in I_Q^k.
\end{equation*}
As $I_Q^k$ vanishes on $Q^{(k)}$, and
$x_i\vert_{X'}=y_i\ci\Phi\vert_{X'}$, this gives
\begin{equation*}
y_i\ci\bigl(\Phi\ci\pi_U\vert_{Q^{(k)}}\bigr)=x_i\ci
\pi_U\vert_{Q^{(k)}}=y_i\ci\bigl(\pi_V \vert_{Q^{(k)}}\bigr).
\end{equation*}
Thus $\Phi\ci\pi_U\vert_{Q^{(k)}}=\pi_V \vert_{Q^{(k)}}$ follows, as
$(y_1,\ldots,y_n)$ are \'etale coordinates on $V$ near $\pi_V(Q)$
and $\Phi\ci\pi_U\vert_Q=\pi_V\vert_Q$. This proves (c), and
Proposition~\ref{sm4prop1}.

\subsection[Proof of Theorem \ref{sm4thm} for perverse sheaves on
$\C$-schemes]{Proof of Theorem \ref{sm4thm} for $\C$-schemes}
\label{sm42}

Let $U,V,f,g,X,Y$ and $\Phi:X^{(3)}\ra Y^{(3)}$ be as in Theorem
\ref{sm4thm}. Pick $x\in X$, and apply Proposition \ref{sm4prop1}
with $k=2$. This gives a smooth $\C$-scheme $T$ and \'etale
morphisms $\pi_U:T\ra U$, $\pi_V:T\ra V$ with
$e:=f\ci\pi_U=g\ci\pi_V:T\ra\C$ and $Q:=\Crit(e)$, such that
$\pi_U\vert_{Q^{(2)}}:Q^{(2)}\ra X^{(2)}$ is an \'etale open
neighbourhood of $x$ in $X^{(2)}$, and
$\Phi\ci\pi_U\vert_{Q^{(2)}}=\pi_V \vert_{Q^{(2)}}:Q^{(2)}\ra
Y^{(2)}$. Actually Proposition \ref{sm4prop1} proves more, that
$\pi_U\vert_{Q^{(2)}}:Q^{(2)}\ra X^{(2)}$ is an isomorphism with a
Zariski open set $x\in\ti X^{(2)}\subseteq X^{(2)}$, but we will not
use this.

Thus, we can choose $\bigl\{(T^a,\pi_U^a,\pi_V^a,e^a,Q^a):a\in
A\bigr\}$, where $A$ is an indexing set, such that
$T^a,\pi_U^a,\pi_V^a,e^a,Q^a$ satisfy the conditions above for each
$a\in A$, and $\bigl\{\pi_U^a\vert_{Q^a}:Q^a\ra X\bigr\}{}_{a\in A}$
is an \'etale open cover of $X$. Then for each $a\in A$, by
Definition \ref{sm2def6} we have isomorphisms
\begin{align*}
\PV_{\pi_U^a}:\PV_{T^a,e^a}^\bu\longra \pi_U^a\vert_{Q^a}^*\bigl(
\PV_{U,f}^\bu\bigr),\quad
\PV_{\pi_V^a}:\PV_{T^a,e^a}^\bu\longra \pi_V^a\vert_{Q^a}^*\bigl(
\PV_{V,f}^\bu\bigr).
\end{align*}
Noting that $\pi_V^a\vert_{Q^a}=\Phi\vert_X\ci\pi_U^a\vert_{Q^a}$,
we may define an isomorphism
\e
\Om^a=\PV_{\pi_V^a}\ci\PV_{\pi_U^a}^{-1}: \pi_U^a\vert_{Q^a}^*\bigl(
\PV_{U,f}^\bu\bigr)\longra \pi_U^a\vert_{Q^a}^*\bigl(
\Phi\vert_X^*\bigl(\PV_{V,f}^\bu\bigr)\bigr).
\label{sm4eq16}
\e

For $a,b\in A$, define $T^{ab}=T^a\t_{\pi_U^a,U,\pi_U^b}T^b$ to be
the $\C$-scheme fibre product, so that $T^{ab}$ is a smooth
$\C$-scheme and the projections $\Pi_{T^a}:T^{ab}\ra T^a$,
$\Pi_{T^b}:T^{ab}\ra T^b$ are \'etale. Define
$e^{ab}=e^a\ci\Pi_{T^a}:T^a\ra\C$. Then
\e
\begin{split}
e^{ab}&=e^a\ci\Pi_{T^a}=g\ci\pi_V^a\ci\Pi_{T^a}=f\ci\pi_U^a
\ci\Pi_{T^a}\\
&=f\ci\pi_U^b \ci\Pi_{T^b}=g\ci\pi_V^b\ci\Pi_{T^b}= e^b\ci\Pi_{T^b}.
\end{split}
\label{sm4eq17}
\e
Write $Q^{ab}=\Crit(e^{ab})$. Then
$\Pi_{T^a}\vert_{Q^{ab}}:Q^{ab}\ra Q^a$ and
$\Pi_{T^b}\vert_{Q^{ab}}:Q^{ab}\ra Q^b$ are \'etale. Now
$\pi_U^a\ci\Pi_{T^a}=\pi_U^b\ci\Pi_{T^b}$ and
$\Phi\ci\pi_U^a\vert_{Q^{a\,(2)}}=\pi_V^a\vert_{Q^{a\,(2)}}$ imply
that
\e
\begin{split}
\bigl(\pi_V^a&\ci\Pi_{T^a}\bigr)\vert_{Q^{ab\,(2)}}=
\pi_V^a\vert_{Q^{a\,(2)}}\ci\Pi_{T^a}\vert_{Q^{ab\,(2)}}\\
&=\Phi\vert_{X^{(2)}}\ci\pi_U^a\vert_{Q^{a\,(2)}}\ci
\Pi_{T^a}\vert_{Q^{ab\,(2)}}
=\Phi\vert_{X^{(2)}}\ci(\pi_U^a\ci\Pi_{T^a})\vert_{Q^{ab\,(2)}}\\
&=\Phi\vert_{X^{(2)}}\ci(\pi_U^b\ci\Pi_{T^b})\vert_{Q^{ab\,(2)}}
=\Phi\vert_{X^{(2)}}\ci\pi_U^a\vert_{Q^{b\,(2)}}\ci
\Pi_{T^b}\vert_{Q^{ab\,(2)}}\\
&=\pi_V^a\vert_{Q^{b\,(2)}}\ci\Pi_{T^b}\vert_{Q^{ab\,(2)}}
=\bigl(\pi_V^b\ci\Pi_{T^b}\bigr)\vert_{Q^{ab\,(2)}}.
\end{split}
\label{sm4eq18}
\e
Hence $(\pi_V^a\ci\Pi_{T^a})\vert_{Q^{ab}}=
(\pi_V^b\ci\Pi_{T^b})\vert_{Q^{ab}}$. Moreover, as
$TQ^{ab\,(2)}\vert_{Q^{ab}}=T(T^{ab})\vert_{Q^{ab}}$, we see that
$\d(\pi_V^a\ci\Pi_{T^a})\vert_{Q^{ab}}=\d(\pi_V^b\ci\Pi_{T^b})
\vert_{Q^{ab}}$, so that
\begin{equation*}
\d(\pi_V^b\ci\Pi_{T^b}) \vert_{Q^{ab}}^{-1}\ci
\d(\pi_V^a\ci\Pi_{T^a})\vert_{Q^{ab}}=\id:T(T^{ab})\vert_{Q^{ab}}\longra
T(T^{ab})\vert_{Q^{ab}}.
\end{equation*}
So $\det\bigl(\d(\pi_V^b\ci\Pi_{T^b}) \vert_{Q^{ab}}^{-1}\ci
\d(\pi_V^a\ci\Pi_{T^a})\vert_{Q^{ab}}\bigr)=1$. Thus, applying
Theorem \ref{sm3thm} with $T^{ab},V,Q^{ab},\pi_V^a\ci\Pi_{T^a},
\pi_V^b\ci\Pi_{T^b},e^{ab},f$ in place of $V,W,X,\Phi,\Psi,f,g$
gives
\e
\PV_{\pi_V^a\ci\Pi_{T^a}}=\PV_{\pi_V^b\ci\Pi_{T^b}}:
\PV_{T^{ab},e^{ab}}^\bu\longra
(\pi_V^a\ci\Pi_{T^a})\vert_{Q^{ab}}^*\bigl(\PV_{V,g}^\bu\bigr).
\label{sm4eq19}
\e

Now
\e
\begin{split}
\Pi_{T^a}\vert_{Q^{ab}}^*&(\Om^a)=\Pi_{T^a}\vert_{Q^{ab}}^*
(\PV_{\pi_V^a})\ci\Pi_{T^a}\vert_{Q^{ab}}^*(\PV_{\pi_U^a})^{-1}\\
&=\bigl[\Pi_{T^a}\vert_{Q^{ab}}^*(\PV_{\pi_V^a})
\ci\PV_{\Pi_{T^a}}\bigr]\ci\bigl[\Pi_{T^a}\vert_{Q^{ab}}^*
(\PV_{\pi_U^a})\ci\PV_{\Pi_{T^a}}\bigr]^{-1}\\
&=\PV_{\pi_V^a\ci\Pi_{T^a}}\ci\PV_{\pi_U^a\ci\Pi_{T^a}}^{-1}
=\PV_{\pi_V^b\ci\Pi_{T^b}}\ci\PV_{\pi_U^b\ci\Pi_{T^b}}^{-1}\\
&=\bigl[\Pi_{T^b}\vert_{Q^{ab}}^*(\PV_{\pi_V^a})\ci
\PV_{\Pi_{T^b}}\bigr]\ci\bigl[\Pi_{T^b}\vert_{Q^{ab}}^*
(\PV_{\pi_U^b})\ci\PV_{\Pi_{T^b}}\bigr]^{-1}\\
&=\Pi_{T^b}\vert_{Q^{ab}}^*(\PV_{\pi_V^b})\ci
\Pi_{T^b}\vert_{Q^{ab}}^*(\PV_{\pi_U^b})^{-1}=
\Pi_{T^b}\vert_{Q^{ab}}^*(\Om^b),
\end{split}
\label{sm4eq20}
\e
using \eq{sm4eq16} in the first and seventh steps, \eq{sm2eq18} in
the third and fifth, and \eq{sm4eq19} and
$\pi_U^a\ci\Pi_{T^a}=\pi_U^b\ci\Pi_{T^b}$ in the fourth. Therefore
Theorem \ref{sm2thm3}(i) applied to the \'etale open cover
$\bigl\{\pi_U^a\vert_{Q^a}:Q^a\ra X\bigr\}{}_{a\in A}$ of $X$ shows
that there is a unique isomorphism $\Om_\Phi$ in \eq{sm4eq2} with
$\pi_U^a\vert_{Q^a}^*(\Om_\Phi)=\Om^a$ for all $a\in A$.

Suppose $\bigl\{(T^a,\ldots,Q^a):a\in A\bigr\}$ and
$\bigl\{(T^{\prime a},\ldots,Q^{\prime a}):a\in A'\bigr\}$ are
alternative choices above, yielding morphisms $\Om_\Phi$ and
$\Om_\Phi'$ in \eq{sm4eq2}. By running the same construction using
the family $\bigl\{(T^a,\ldots,Q^a):a\in
A\bigr\}\amalg\bigl\{(T^{\prime a},\ldots,Q^{\prime a}):a\in
A'\bigr\}$, we get a third morphism $\Om_\Phi''$ in \eq{sm4eq2},
such that $\pi_U^a\vert_{Q^a}^*
(\Om_\Phi)=\Om^a=\pi_U^a\vert_{Q^a}^*(\Om_\Phi'')$ for $a\in A$,
giving $\Om_\Phi=\Om_\Phi''$, and $\pi_U^{\prime a}\vert_{Q^{\prime
a}}^* (\Om_\Phi')=\Om^{\prime a}=\pi_U^{\prime a}\vert_{Q^{\prime
a}}^*(\Om_\Phi'')$ for $a\in A'$, which forces
$\Om_\Phi'=\Om_\Phi''$. Thus $\Om_\Phi=\Om_\Phi'$, so $\Om_\Phi$ is
independent of the choice of $\bigl\{(T^a,\ldots,Q^a):a\in A\bigr\}$
above.

Let $T,\pi_U,\pi_V,e,Q$ be as in Theorem \ref{sm4thm}. Applying the
argument above using the family $\bigl\{(T^a,\ldots,Q^a):a\in
A\bigr\}\amalg\bigl\{(T,\pi_U, \pi_V,e,Q)\bigr\}$ shows that
$\Om_\Phi$ satisfies $\pi_U\vert_Q^*(\Om_\Phi)=
\PV_{\pi_V}\ci\PV_{\pi_U}^{-1}$, by \eq{sm4eq16}. Thus \eq{sm4eq3}
holds.

To show that \eq{sm4eq4}--\eq{sm4eq5} commute, we can combine
equations \eq{sm2eq16}--\eq{sm2eq17}, \eq{sm4eq16} and
$\pi_U^a\vert_{Q^a}^*(\Om_\Phi)=\Om^a$ to show that
$\pi_U^a\vert_{Q^a}^*$ applied to \eq{sm4eq4}--\eq{sm4eq5} commute
in $\Perv(Q^a)$ for each $a\in A$, so \eq{sm4eq4}--\eq{sm4eq5}
commute by Theorem~\ref{sm2thm3}(i).

Suppose there exists an \'etale morphism $\Xi:U\ra V$ with
$f=g\ci\Xi:U\ra\C$ and $\Xi\vert_{X^{(3)}}=\Phi:X^{(3)}\ra Y^{(3)}$.
Then as we have to prove, we have
\begin{equation*}
\smash{\PV_\Xi=\id_X^*(\Om_\Phi)\ci\PV_{\id_U}=\Om_\Phi\ci
\id_{\PV_{U,f}^\bu}=\Om_\Phi,}
\end{equation*}
where in the first step we use \eq{sm4eq3} with $T=U$,
$\pi_U=\id_U$, $\pi_V=\Xi$, $e=f$, and $Q=X$, and in the second we
use $\PV_{\id_U}=\id_{\PV_{U,f}^\bu}$ from Definition~\ref{sm2def6}.

Suppose $W$ is another smooth $\C$-scheme, $h:W\ra\C$ is regular,
$Z=\Crit(h),$ and $\Psi:Y^{(3)}\ra Z^{(3)}$ is an isomorphism with
$h^{(3)}\ci\Psi=g^{(3)}$. Let $x\in X$, and set $y=\Phi(x)\in Y$.
Proposition \ref{sm4prop1} for $x,\Phi$ gives a smooth $T$ and
\'etale $\pi_U:T\ra U$, $\pi_V:T\ra V$ with $e:=f\ci\pi_U=g\ci\pi_V$
and $Q:=\Crit(e)$, such that $\pi_U\vert_{Q^{(2)}}:Q^{(2)}\ra
X^{(2)}$ is an \'etale open neighbourhood of $x$, and
$\Phi\ci\pi_U\vert_{Q^{(2)}}=\pi_V\vert_{Q^{(2)}}$. Proposition
\ref{sm4prop1} for $y,\Psi$ gives smooth $\ti T$ and \'etale
$\ti\pi_V:\ti T\ra V$, $\ti\pi_W:\ti T\ra W$ with $\ti
e:=g\ci\ti\pi_V=h\ci\ti\pi_W$ and $\ti Q:=\Crit(\ti e)$, such that
$\ti\pi_V\vert_{\smash{\ti Q{}^{(2)}}}:\ti Q{}^{(2)}\ra Y^{(2)}$ is
an \'etale open neighbourhood of $y$, and
$\Psi\ci\ti\pi_V\vert_{\smash{\ti Q{}^{(2)}}}=\ti\pi_W
\vert_{\smash{\ti Q{}^{(2)}}}$.

Define $\hat T=T\t_{\pi_V,V,\ti\pi_V}\ti T$ with projections
$\Pi_T:\hat T\ra T$, $\Pi_{\smash{\ti T}}:\hat T\ra\ti T$. Then
$\hat T$ is smooth and $\Pi_T,\Pi_{\smash{\ti T}}$ are \'etale, as
$T,\ti T,V$ are smooth and $\pi_V,\ti\pi_V$ \'etale. Define
$\hat\pi_U=\pi_U\ci\Pi_T:\hat T\ra U$ and
$\hat\pi_W=\ti\pi_W\ci\Pi_{\smash{\ti T}}:\hat T\ra W$. Then
$\hat\pi_U,\hat\pi_W$ are \'etale. Set $\hat e=f\ci\hat\pi_U:\hat
T\ra\C$, and write $\hat Q=\Crit(\hat e)$. Then
\begin{equation*}
\hat e\!=\!f\!\ci\!\hat\pi_U\!=\!f\!\ci\!\pi_U\!\ci\!\Pi_T\!=\!
g\!\ci\!\pi_V\!\ci\!\Pi_T\!=\!g\!\ci\!\ti\pi_V\!\ci\!\Pi_{\smash{\ti T}}
\!=\!h\!\ci\!\ti\pi_W\!\ci\!\Pi_{\smash{\ti T}}\!=\!h\!\ci\!\hat\pi_W.
\end{equation*}
Also $\hat\pi_U\vert_{\smash{\hat Q{}^{(2)}}}:\hat Q{}^{(2)}\ra
X^{(2)}$ is an \'etale open neighbourhood of $x$, and
\begin{align*}
(\Psi\ci\Phi)\ci\hat\pi_U\vert_{\smash{\hat Q{}^{(2)}}}&=
\Psi\ci\Phi\ci\pi_U\vert_{Q^{(2)}}\ci\Pi_T\vert_{\smash{\hat Q{}^{(2)}}}
=\Psi\ci\pi_V\vert_{Q^{(2)}}\ci\Pi_T\vert_{\smash{\hat Q{}^{(2)}}}\\
&=\Psi\ci\ti\pi_V\vert_{\smash{\ti Q{}^{(2)}}}\ci\Pi_{\smash{\ti T}}
\vert_{\smash{\hat Q{}^{(2)}}}=\ti\pi_W\vert_{\smash{\ti Q{}^{(2)}}}
\ci\Pi_{\smash{\ti T}}\vert_{\smash{\hat Q{}^{(2)}}}
=\hat\pi_W\vert_{\smash{\hat Q{}^{(2)}}}.
\end{align*}

Thus we may apply \eq{sm4eq3} for $\Om_\Phi$ with
$T,\pi_U,\pi_V,\ldots,Q$, and for $\Om_\Psi$ with $\ti
T,\ti\pi_V,\ti\pi_W,\ldots,\ti Q$, and for $\Om_{\Psi\ci\Phi}$ with
$\hat T,\hat\pi_U,\hat\pi_W,\ldots,\hat Q$. This yields
\e
\begin{gathered}
\pi_U\vert_Q^*(\Om_\Phi)=\PV_{\pi_V}\ci\PV_{\pi_U}^{-1},\quad
\ti\pi_V\vert_{\ti Q}^*(\Om_\Psi)= \PV_{\ti\pi_W}\ci
\PV_{\ti\pi_V}^{-1},\\
\hat\pi_U\vert_{\hat Q}^*(\Om_{\Psi\ci\Phi})=
\PV_{\hat\pi_W}\ci\PV_{\hat\pi_U}^{-1}.
\end{gathered}
\label{sm4eq21}
\e
Now
\begin{align*}
\hat\pi_U&\vert_{\hat Q}^*\bigl(
\Om_{\Psi\ci\Phi}\bigr)=\PV_{\hat\pi_W}\ci
\PV_{\hat\pi_U}^{-1} =\PV_{\ti\pi_W\ci\Pi_{\smash{\ti T}}}
\ci\PV_{\pi_U\ci\Pi_T}^{-1}\\
&=\bigl[\Pi_{\smash{\ti T}}\vert_{\hat Q}^*
(\PV_{\ti\pi_W})\ci\PV_{\Pi_{\smash{\ti T}}}\bigr]
\ci\bigl[\Pi_T\vert_{\hat
Q}^*(\PV_{\pi_U})\ci\PV_{\Pi_T}\bigr]^{-1} \\
&=\Pi_{\smash{\ti T}}\vert_{\hat Q}^* \bigl(\PV_{\ti\pi_W}
\!\ci\!\PV_{\ti\pi_V}^{-1}\bigr)\!\ci\!\Pi_{\smash{\ti T}}\vert_{\hat
Q}^*(\PV_{\ti\pi_V})\!\ci\!\PV_{\Pi_{\smash{\ti T}}}\ci\PV_{\Pi_T}^{-1}
\!\ci\!\Pi_T\vert_{\hat Q}^*(\PV_{\pi_U}^{-1})\\
&=\Pi_{\smash{\ti T}}\vert_{\hat Q}^* \bigl(
\ti\pi_V\vert_{\ti Q}^*(\Om_\Psi)\bigr)\ci
\PV_{\ti\pi_V\ci\Pi_{\smash{\ti T}}} \ci\PV_{\Pi_T}^{-1}
\ci\Pi_T\vert_{\hat Q}^*(\PV_{\pi_U}^{-1})\\
&=\bigl[\ti\pi_V\vert_{\ti Q}\ci\Pi_{\smash{\ti T}}\vert_{\hat
Q}\bigr]^*(\Om_\Psi)\ci \PV_{\pi_V\ci\Pi_T}
\ci\PV_{\Pi_T}^{-1}
\ci\Pi_T\vert_{\hat Q}^*(\PV_{\pi_U}^{-1})\\
&=\bigl[\pi_V\vert_Q\!\ci\!\Pi_T\vert_{\hat
Q}\bigr]^*(\Om_\Psi)\!\ci\! \Pi_T\vert_{\hat
Q}^*(\PV_{\pi_V})\!\ci\!\PV_{\Pi_T}
\!\ci\!\PV_{\Pi_T}^{-1}
\!\ci\!\Pi_T\vert_{\hat Q}^*(\PV_{\pi_U}^{-1})\\
&=\bigl[\Phi\vert_X\ci\pi_U\vert_Q\ci\Pi_T\vert_{\hat
Q}\bigr]^*(\Om_\Psi) \ci \Pi_T\vert_{\hat
Q}^*(\PV_{\pi_V}\ci\PV_{\pi_U}^{-1})\\
&=(\pi_U\ci\Pi_T)\vert_{\hat
Q}^*(\Phi\vert_X^*(\Om_\Psi)) \ci \Pi_T\vert_{\hat
Q}^*(\pi_U\vert_Q^*(\Om_\Phi))
=\hat\pi_U\vert_{\hat
Q}^*\bigl(\Phi\vert_X^*(\Om_\Psi)\ci\Om_\Phi\bigr),
\end{align*}
using \eq{sm4eq21} in the first, fifth and ninth steps, \eq{sm2eq18}
in the third, fifth and seventh, $\pi_V\ci\Pi_T=
\ti\pi_V\ci\Pi_{\smash{\ti T}}$ in the sixth and seventh, and
$\Phi\ci\pi_U\vert_Q=\pi_V\vert_Q$ in the eighth. Thus, for each
$x\in X$, we have constructed an \'etale open neighbourhood
$\hat\pi_U\vert_{\hat Q}:\hat Q\ra X$ such that
$\hat\pi_U\vert_{\vphantom{Q}\smash{\hat Q}}^*$ applied to
\eq{sm4eq6} holds. Equation \eq{sm4eq6} follows by Theorem
\ref{sm2thm3}(i). Finally, if $U=V$, $f=g$, $X=Y$ and
$\Phi=\id_{\smash{X^{(3)}}}$ then
$\Om_{\id_{\smash{X^{(3)}}}}=\id_{\PV_{U,f}^\bu}$ follows by taking
$\Xi=\id_U$ in the fourth paragraph of the theorem. This proves
Theorem \ref{sm4thm} for perverse sheaves on $\C$-schemes.

\subsection{$\cD$-modules and mixed Hodge modules}
\label{sm43}

Once again, the proof of Proposition \ref{sm4prop1} is completely
algebraic, so applies in the other contexts of
\S\ref{sm26}--\S\ref{sm210}. Theorem \ref{sm4thm} then follows for
our other contexts from that and the general framework
of~\S\ref{sm25}.

\section[Stabilizing perverse sheaves of vanishing
cycles]{Stabilizing vanishing cycles}
\label{sm5}

To set up notation for our main result, which is Theorem
\ref{sm5thm2} below, we need the following theorem, which is proved
in Joyce \cite[Prop.s 2.22, 2.23 \& 2.25]{Joyc}.

\begin{thm}[Joyce \cite{Joyc}] Let\/ $U,V$ be smooth $\C$-schemes,
$f:U\ra\C,$ $g:V\ra\C$ be regular, and\/ $X=\Crit(f),$ $Y=\Crit(g)$
as $\C$-subschemes of\/ $U,V$. Let\/ $\Phi:U\hookra V$ be a closed
embedding of\/ $\C$-schemes with\/ $f=g\ci\Phi:U\ra\C,$ and suppose
$\Phi\vert_X:X\ra V\supseteq Y$ is an isomorphism $\Phi\vert_X:X\ra
Y$. Then:
\smallskip

\noindent{\bf(i)} For each\/ $x\in X\subseteq U$ there exist
smooth\/ $\C$-schemes $U',V',$ a point\/ $x'\in U'$ and morphisms
$\io:U'\ra U,$ $\jmath:V'\ra V,$ $\Phi':U'\ra V',$ $\al:V'\ra U$ and
$\be:V'\ra\C^n,$ where $n=\dim V-\dim U,$ such that\/ $\io(x')=x,$
and\/ $\io,\jmath$ and\/ $\al\t\be:V\ra U\t\C^n$ are \'etale, and
the following diagram commutes
\e
\begin{gathered}
\xymatrix@C=85pt@R=15pt{ *+[r]{U} \ar[d]^\Phi & U' \ar[l]^(0.4)\io
\ar[r]_\io \ar[d]^{\Phi'}
& *+[l]{U} \ar[d]_{\id_U \t 0} \\
*+[r]{V} & V' \ar[l]_(0.4)\jmath \ar[r]^{\al\t\be} & *+[l]{U\t\C^n,\!{}} }
\end{gathered}
\label{sm5eq1}
\e
and\/ $g\ci\jmath=f\ci\al+(z_1^2+\cdots+z_n^2)\ci\be:V'\ra\C$. Thus,
setting $f':=f\ci\io:U'\ra\C,$ $g':=g\ci\jmath:V'\ra\C,$
$X':=\Crit(f')\subseteq U',$ and\/ $Y':=\Crit(g')\subseteq V',$ then
$f'=g'\ci\Phi':U'\ra\C,$ and\/ $\Phi'\vert_{X'}:X'\ra Y',$
$\io\vert_{X'}:X'\ra X,$ $\jmath\vert_{Y'}:Y'\ra Y,$
$\al\vert_{Y'}:Y'\ra X$ are \'etale. We also require
that\/~$\Phi\ci\al\vert_{Y'}=\jmath\vert_{Y'}:Y'\ra Y$.
\smallskip

\noindent{\bf(ii)} Write\/ $N_{\sst UV}$ for the normal bundle of\/
$\Phi(U)$ in $V,$ regarded as an algebraic vector bundle on $U$ in
the exact sequence of vector bundles on $U\!:$
\e
\xymatrix@C=20pt{ 0 \ar[r] & TU \ar[rr]^(0.4){\d\Phi} && \Phi^*(TU)
\ar[rr]^(0.55){\Pi_{\sst UV}} && N_{\sst UV} \ar[r] & 0.}
\label{sm5eq2}
\e
Then there exists a unique $q_{\sst UV}\in H^0(S^2N_{\sst
UV}^*\vert_X)$ which is a nondegenerate quadratic form on $N_{\sst
UV}\vert_X,$ such that whenever $U',V',\io,\jmath,\Phi',\be,n,X'$
are as in {\bf(i)\rm,} writing $\langle\d z_1,\ldots,\d
z_n\rangle_{U'}$ for the trivial vector bundle on $U'$ with basis
$\d z_1,\ldots,\d z_n,$ there is a natural isomorphism
$\hat\be:\langle\d z_1,\ldots,\d z_n\rangle_{U'}\ra\io^*(N_{\sst
UV}^*)$ making the following diagram commute:
\begin{gather}
\begin{gathered}
\xymatrix@C=130pt@R=15pt{
*+[r]{\io^*(N_{\sst UV}^*)} \ar[r]_(0.3){\io^*(\Pi_{\sst UV}^*)}
 & *+[l]{\io^*\ci\Phi^*(T^*V)=\Phi^{\prime
*}\ci
\jmath^*(T^*V)} \ar[d]_{\Phi^{\prime *}(\d\jmath^*)} \\
*+[r]{\langle\d z_1,\ldots,\d z_n\rangle_{U'}=\Phi^{\prime *}
\ci\be^*(T_0^*\C^n)} \ar[r]^(0.7){\Phi^{\prime *}(\d\be^*)}
\ar@{.>}[u]_{\hat\be} & *+[l]{\Phi^{\prime *}(T^*V'),}}
\end{gathered}
\label{sm5eq3}\\
\text{and\/}\qquad\io\vert_{X'}^*(q_{\sst UV})=(S^2\hat\be)
\vert_{X'}(\d z_1\ot\d z_1+\cdots+\d z_n\ot\d z_n).
\label{sm5eq4}
\end{gather}

\noindent{\bf(iii)} Now suppose $W$ is another smooth\/ $\C$-scheme,
$h:W\ra\C$ is regular, $Z=\Crit(h)$ as a $\C$-subscheme of\/ $W,$
and\/ $\Psi:V\hookra W$ is a closed embedding of\/ $\C$-schemes
with\/ $g=h\ci\Psi:V\ra\C$ and\/ $\Psi\vert_Y:Y\ra Z$ an
isomorphism. Define $N_{\sst VW},q_{\sst VW}$ and\/ $N_{\sst
UW},q_{\sst UW}$ using $\Psi:V\hookra W$ and\/ $\Psi\ci\Phi:U\hookra
W$ as in {\bf(ii)} above. Then there are unique morphisms $\ga_{\sst
UVW},\de_{\sst UVW}$ which make the following diagram of vector
bundles on $U$ commute, with straight lines exact:
\e
\begin{gathered}
\xymatrix@!0@C=19pt@R=9pt{ &&&&&&&&&&&&&&& 0 \ar[ddl] \\
&&&&&&&&&&&&&&&& 0 \ar[dll] \\
&&&&&& 0 \ar[dddrr] &&&&&&&& TU \ar[dddllllll]_{\d\Phi}
\ar[ddddddllll]^(0.4){\d(\Psi\ci\Phi)} \\ \\ \\
&&&&&&&& \Phi^*(TV) \ar[dddllllll]_{\Pi_{\sst UV}}
\ar[dddrr]^(0.6){\Phi^*(\d\Psi)} \\ \\
0 \ar[drr] \\
&& N_{\sst UV} \ar[dll]  \ar@{.>}[dddrrrrrr]_{\ga_{\sst UVW}}
&&&&&&&& (\Psi\ci\Phi)^*(TW) \ar[ddddddrrrr]^{\Phi^*(\Pi_{\sst VW})}
\ar[dddll]^(0.4){\Pi_{\sst UW}} \\
0 \\ \\
&&&&&&&& N_{\sst UW}  \ar[dddll] \ar@{.>}[dddrrrrrr]_{\de_{\sst
UVW}} \\ \\ \\
&&&&&& 0  &&&&&&&& \Phi^*(N_{\sst VW}) \ar[drr] \ar[ddr] \\
&&&&&&&&&&&&&&&& 0 \\
&&&&&&&&&&&&&&& 0.}\!\!\!\!\!\!\!\!\!\!\!\!\!\!\!\!\!\!\!\!\!{}
\end{gathered}
\label{sm5eq5}
\e

Restricting to $X$ gives an exact sequence of vector bundles:
\e
\xymatrix@C=14pt{ 0 \ar[r] & N_{\sst UV}\vert_X
\ar[rrr]^(0.48){\ga_{\sst UVW}\vert_X} &&& N_{\sst UW}\vert_X
\ar[rrr]^(0.45){\de_{\sst UVW}\vert_X} &&& \Phi\vert_X^*(N_{\sst
VW}) \ar[r] & 0.}
\label{sm5eq6}
\e
Then there is a natural isomorphism of vector bundles on $X$
\e
N_{\sst UW}\vert_X\cong N_{\sst UV}\vert_X\op \Phi\vert_X^*(N_{\sst
VW}),
\label{sm5eq7}
\e
compatible with the exact sequence {\rm\eq{sm5eq6},} which
identifies
\e
\begin{split}
q_{\sst UW}&\cong q_{\sst UV}\op \Phi\vert_X^*(q_{\sst VW})\op 0
\qquad\text{under
the splitting}\\
S^2N_{\sst UW}\vert_X^*&\cong S^2N_{\sst
UV}\vert_X^*\op\Phi\vert_X^*\bigl(S^2N_{\sst VW}^*\vert_Y\bigr) \op
\bigl(N_{\sst UV}^*\vert_X\ot \Phi\vert_X^*(N_{\sst VW}^*)\bigr).
\end{split}
\label{sm5eq8}
\e

\smallskip

\noindent{\bf(iv)} Analogues of\/ {\bf(i)--(iii)} hold for complex
analytic spaces, replacing the smo\-oth\/ $\C$-schemes $U,V,W$ by
complex manifolds, the regular functions $f,g,h$ by holomorphic
functions, the $\C$-schemes $X,Y,Z$ by complex analytic spaces, the
\'etale open sets $\io:U'\ra U,$ $\jmath:V'\ra V$ by complex
analytic open sets $U'\subseteq U,$ $V'\subseteq V,$ and with\/
$\al\t\be:V'\ra U\t\C^n$ a biholomorphism with a complex analytic
open neighbourhood of\/ $(x,0)$ in $U\t\C^n$.
\label{sm5thm1}
\end{thm}

Following \cite[Def.s 2.26 \& 2.34]{Joyc}, we define:

\begin{dfn} Let $U,V$ be smooth $\C$-schemes, $f:U\ra\C$,
$g:V\ra\C$ be regular, and $X=\Crit(f)$, $Y=\Crit(g)$ as
$\C$-subschemes of $U,V$. Suppose $\Phi:U\hookra V$ is a closed
embedding of $\C$-schemes with $f=g\ci\Phi:U\ra\C$ and
$\Phi\vert_X:X\ra Y$ an isomorphism. Then Theorem \ref{sm5thm1}(ii)
defines the normal bundle $N_{\sst UV}$ of $U$ in $V$, a vector
bundle on $U$ of rank $n=\dim V-\dim U$, and a nondegenerate
quadratic form $q_{\sst UV}\in H^0(S^2N_{\sst UV}^*\vert_X)$. Taking
top exterior powers in the dual of \eq{sm5eq2} gives an isomorphism
of line bundles on $U$
\begin{equation*}
\rho_{\sst UV}:K_U\ot\La^nN_{\sst UV}^*
\,{\buildrel\cong\over\longra}\,\Phi^*(K_V),
\end{equation*}
where $K_U,K_V$ are the canonical bundles of $U,V$.

Write $X^\red$ for the reduced $\C$-subscheme of $X$. As $q_{\sst
UV}$ is a nondegenerate quadratic form on $N_{\sst VW}\vert_X,$ its
determinant $\det(q_{\sst VW})$ is a nonzero section of
$(\La^nN_{\sst VW}^*)\vert_X^{\ot^2}$. Define an isomorphism of line
bundles on~$X^\red$:
\e
J_\Phi=\rho_{\sst UV}^{\ot^2}\ci
\bigl(\id_{K_U^2\vert_{X^\red}}\ot\det(q_{\sst
UV})\vert_{X^\red}\bigr):K_U^{\ot^2}\big\vert_{X^\red}
\,{\buildrel\cong\over\longra}\,\Phi\vert_{X^\red}^*
\bigl(K_V^{\ot^2}\bigr).
\label{sm5eq9}
\e

Since principal $\Z/2\Z$-bundles $\pi:P\ra X$ in the sense of
Definition \ref{sm2def3} are an (\'etale or complex analytic)
topological notion, and $X^\red$ and $X$ have the same topological
space (even in the \'etale or complex analytic topology), principal
$\Z/2\Z$-bundles on $X^\red$ and on $X$ are equivalent. Define
$\pi_\Phi:P_\Phi\ra X$ to be the principal $\Z/2\Z$-bundle which
parametrizes square roots of $J_\Phi$ on $X^\red$. That is, (\'etale
or complex analytic) local sections $s_\al:X\ra P_\Phi$ of $P_\Phi$
correspond to local isomorphisms $\al: K_U\vert_{X^\red}
\ra\Phi\vert_{X^\red}^*(K_V)$ on $X^\red$ with~$\al\ot\al=J_\Phi$.

Now suppose $W$ is another smooth $\C$-scheme, $h:W\ra\C$ is
regular, $Z=\Crit(h)$ as a $\C$-subscheme of $W$, and $\Psi:V\hookra
W$ is a closed embedding of $\C$-schemes with $g=h\ci\Psi:V\ra\C$
and $\Psi\vert_Y:Y\ra Z$ an isomorphism. Then Theorem
\ref{sm5thm1}(iii) applies, and from \eq{sm5eq7}--\eq{sm5eq8} we can
deduce that
\e
\begin{split}
J_{\Psi\ci\Phi}=\Phi\vert_{X^\red}^*(J_\Psi)\ci J_\Phi:
K_U^{\ot^2}\big\vert_{X^\red}\,{\buildrel\cong\over\longra}\,
&(\Psi\ci\Phi)\vert_{X^\red}^*\bigl(K_W^{\ot^2}\bigr)\\
=&\,\Phi\vert_{X^\red}^*\bigl[\Psi\vert_{Y^\red}^*
\bigl(K_W^{\ot^2}\bigr)\bigr].
\end{split}
\label{sm5eq10}
\e
For the principal $\Z/2\Z$-bundles $\pi_\Phi:P_\Phi\ra X$,
$\pi_\Psi:P_\Psi\ra Y$, $\pi_{\Psi\ci\Phi}:P_{\Psi\ci\Phi}\ra X$,
equation \eq{sm5eq10} implies that there is a canonical isomorphism
\e
\Xi_{\Psi,\Phi}:P_{\Psi\ci\Phi}\,{\buildrel\cong\over\longra}\,
\Phi\vert_X^*(P_\Psi)\ot_{\Z/2\Z}P_\Phi.
\label{sm5eq11}
\e
It is also easy to see that these $\Xi_{\Psi,\Phi}$ have an
associativity property under triple compositions, that is, given
another smooth $\C$-scheme $T$, regular $e:T\ra\C$ with
$Q:=\Crit(e)$, and $\Up:T\hookra U$ a closed embedding with
$e=f\ci\Up:T\ra\C$ and $\Up\vert_Q:Q\ra X$ an isomorphism, then
\e
\begin{split}
\bigl(\id_{(\Phi\ci\Up)\vert_Q^*(P_\Psi)}&\ot\kern
.1em\Xi_{\Phi,\Up}\bigr)\ci \Xi_{\Psi,\Phi\ci\Up}=
\bigl(\Up\vert_Q^*(\Xi_{\Psi,\Phi})\ot\id_{P_\Up}\bigr)\ci
\Xi_{\Psi\ci\Phi,\Up}:\\
&P_{\Psi\ci\Phi\ci\Up}\longra
(\Phi\ci\Up)\vert_Q^*(P_\Psi)\ot_{\Z/2\Z}
\Up\vert_Q^*(P_\Phi)\ot_{\Z/2\Z}P_\Up.
\end{split}
\label{sm5eq12}
\e

Analogues of all the above also work for complex manifolds and complex
analytic spaces, as in Theorem~\ref{sm5thm1}(v).
\label{sm5def}
\end{dfn}

The reason for restricting to $X^\red$ above is the following
\cite[Prop.~2.27]{Joyc}, whose proof uses the fact that $X^\red$ is
reduced in an essential way.

\begin{lem} In Definition\/ {\rm\ref{sm5def},} the isomorphism
$J_\Phi$ in \eq{sm5eq9} and the principal\/ $\Z/2\Z$-bundle\/
$\pi_\Phi:P_\Phi\ra X$ depend only on $U,\ab V,\ab X,\ab Y,\ab f,g$
and\/ $\Phi\vert_X:X\ra Y$. That is, they do not depend on
$\Phi:U\ra V$ apart from\/~$\Phi\vert_X:X\ra Y$.
\label{sm5lem}
\end{lem}

Using the notation of Definition \ref{sm5def}, we can state our main
result:

\begin{thm}{\bf(a)} Let\/ $U,V$ be smooth $\C$-schemes, $f:U\ra\C,$
$g:V\ra\C$ be regular, and\/ $X=\Crit(f),$ $Y=\Crit(g)$ as
$\C$-subschemes of\/ $U,V$. Let\/ $\Phi:U\hookra V$ be a closed
embedding of\/ $\C$-schemes with\/ $f=g\ci\Phi:U\ra\C,$ and suppose
$\Phi\vert_X:X\ra V\supseteq Y$ is an isomorphism $\Phi\vert_X:X\ra
Y$. Then there is a natural isomorphism of perverse sheaves on
$X\!:$
\e
\Th_\Phi:\PV_{U,f}^\bu\longra\Phi\vert_X^*\bigl(\PV_{V,g}^\bu\bigr)
\ot_{\Z/2\Z}P_\Phi,
\label{sm5eq13}
\e
where\/ $\PV_{U,f}^\bu,\PV_{V,g}^\bu$ are the perverse sheaves of
vanishing cycles from\/ {\rm\S\ref{sm24},} and\/ $P_\Phi$ the
principal\/ $\Z/2\Z$-bundle from Definition\/ {\rm\ref{sm5def},} and
if\/ $\cQ^\bu$ is a perverse sheaf on\/ $X$ then
$\cQ^\bu\ot_{\Z/2\Z}P_\Phi$ is as in Definition\/
{\rm\ref{sm2def3}}. Also the following diagrams commute, where
$\si_{U,f},\si_{V,g},\tau_{U,f},\tau_{V,g}$ are as
in\/~{\rm\eq{sm2eq6}--\eq{sm2eq7}:}
\ea
\begin{gathered}
\xymatrix@!0@C=145pt@R=35pt{
*+[r]{\PV_{U,f}^\bu} \ar[r]_(0.25){\Th_\Phi}
\ar@<2ex>[d]^{\si_{U,f}} &
*+[l]{\Phi\vert_X^*\bigl(\PV_{V,g}^\bu\bigr)\ot_{\Z/2\Z}P_\Phi}
\ar[r]_(0.18){\raisebox{-11pt}{$\st\Phi\vert_X^*(\si_{V,g})\ot\id$}}
& *+[l]{\Phi\vert_X^*\bigl(\bD_Y(\PV_{V,g}^\bu)\bigr)\!\ot_{\Z/2\Z}
\!P_\Phi} \ar@<-2ex>[d]_\cong
\\
*+[r]{\bD_X(\PV_{U,f}^\bu)} &&
*+[l]{\bD_X\bigl(\Phi\vert_X^*(\PV_{V,g}^\bu)\!\ot_{\Z/2\Z}\!P_\Phi\bigr),}
\ar[ll]_(0.7){\bD_X(\Th_\Phi)} }\!\!\!\!\!{}
\end{gathered}
\label{sm5eq14}\\
\begin{gathered}
\xymatrix@!0@C=290pt@R=35pt{
*+[r]{\PV_{U,f}^\bu} \ar[r]_(0.45){\Th_\Phi}
\ar@<2ex>[d]^{\tau_{U,f}} &
*+[l]{\Phi\vert_X^*\bigl(\PV_{V,g}^\bu\bigr)\ot_{\Z/2\Z}P_\Phi}
\ar@<-2ex>[d]_{\Phi\vert_X^*(\tau_{V,g})\ot\id} \\
*+[r]{\PV_{U,f}^\bu} \ar[r]^(0.45){\Th_\Phi} &
*+[l]{\Phi\vert_X^*\bigl(\PV_{V,g}^\bu\bigr)\ot_{\Z/2\Z}P_\Phi.}
}\!\!\!\!\!{}
\end{gathered}
\label{sm5eq15}
\ea

If\/ $U=V,$ $f=g,$ $\Phi=\id_U$ then $\pi_\Phi:P_\Phi\ra X$ is
trivial, and\/ $\Th_\Phi$ corresponds to $\id_{\PV_{U,f}^\bu}$ under
the natural isomorphism $\id_X^*(\PV_{U,f}^\bu)
\ot_{\Z/2\Z}P_\Phi\cong\PV_{U,f}^\bu$.
\smallskip

\noindent{\bf(b)} The isomorphism $\Th_\Phi$ in \eq{sm5eq13} depends
only on $U,\ab V,\ab X,\ab Y,\ab f,g$ and\/ $\Phi\vert_X:X\ra Y$.
That is, if\/ $\ti\Phi:U\ra V$ is an alternative choice for $\Phi$
with\/ $\Phi\vert_X=\ti\Phi\vert_X:X\ra Y,$ then
$\Th_\Phi=\Th_{\smash{\ti \Phi}},$ noting that\/
$P_\Phi=P_{\smash{\ti\Phi}}$ by Lemma\/~{\rm\ref{sm5lem}}.
\smallskip

\noindent{\bf(c)} Now suppose $W$ is another smooth\/ $\C$-scheme,
$h:W\ra\C$ is regular, $Z=\Crit(h),$ and\/ $\Psi:V\hookra W$ is a
closed embedding with\/ $g=h\ci\Psi:V\ra\C$ and\/ $\Psi\vert_Y:Y\ra
Z$ an isomorphism. Then Definition\/ {\rm\ref{sm5def}} defines
principal\/ $\Z/2\Z$-bundles $\pi_\Phi:P_\Phi\ra X,$
$\pi_\Psi:P_\Psi\ra Y,$ $\pi_{\Psi\ci\Phi}:P_{\Psi\ci\Phi}\ra X$ and
an isomorphism $\Xi_{\Psi,\Phi}$ in {\rm\eq{sm5eq11},} and part\/
{\bf(a)} defines isomorphisms of perverse sheaves
$\Th_\Phi,\Th_{\Psi\ci\Phi}$ on $X$ and\/ $\Th_\Psi$ on $Y$. Then
the following commutes in $\Perv(X)\!:$
\e
\begin{gathered}
\xymatrix@C=145pt@R=17pt{*+[r]{\PV_{U,f}^\bu}
\ar[r]_(0.45){\Th_{\Psi\ci\Phi}} \ar@<2ex>[d]^{\Th_\Phi} &
*+[l]{(\Psi\ci\Phi)\vert_X^*\bigl(\PV_{W,h}^\bu\bigr)\ot_{\Z/2\Z}
P_{\Psi\ci\Phi}} \ar@<-2ex>[d]_{\id\ot \Xi_{\Psi,\Phi}} \\
*+[r]{\Phi\vert_X^*\bigl(\PV_{V,g}^\bu\bigr)\!\ot_{\Z/2\Z}\!P_\Phi}
\ar[r]^(0.33){\raisebox{6pt}{$\st\Phi\vert_X^*(\Th_\Psi)\ot\id$}} &
*+[l]{\Phi\vert_X^*\!\ci\!\Psi\vert_Y^*\bigl(\PV_{W,h}^\bu\bigr)
\!\ot_{\Z/2\Z}\!\Phi\vert_X^*(P_\Psi)\!\ot_{\Z/2\Z}\!P_\Phi.}
}\!\!\!\!\!{}
\end{gathered}
\label{sm5eq16}
\e

\noindent{\bf(d)} The analogues of\/ {\bf(a)--\bf(c)} also hold for
$\cD$-modules on $\C$-schemes, for perverse sheaves and
$\cD$-modules on complex analytic spaces, and for mixed Hodge
modules on $\C$-schemes and complex analytic spaces, as
in\/~{\rm\S\ref{sm26}--\S\ref{sm210}}.
\label{sm5thm2}
\end{thm}

\begin{ex} Let $U=\C\sm\{0\}$ and $V=(\C\sm\{0\})\t\C$ as smooth
$\C$-schemes, define regular $f:U\ra\C$ and $g:V\ra\C$ by $f(x)=0$
and $g(x,y)=x^ky^2$ for fixed $k\in\Z$, and define $\Phi:U\ra V$ by
$\Phi:x\mapsto (x,0)$, so that $f=g\ci\Phi:U\ra\C$. Then
$X:=\Crit(f)=U$, and $Y:=\Crit(g)=\bigl\{(x,y)\in
V:kx^{k-1}y^2=2x^ky=0\bigr\}=\bigl\{(x,y)\in V:y=0\bigr\}$, as $x\ne
0$. Thus $\Phi\vert_X:X\ra Y$ is an isomorphism.

In Theorem \ref{sm5thm1}(ii), $N_{\sst UV}^*$ is the trivial line
bundle on $U$ with basis $\d y$, and $q_{\sst UV}=x^k\d y\ot\d y$.
In Definition \ref{sm5def}, $K_U\vert_X$ and $\Phi\vert_X^*(K_V)$
are the trivial line bundles on $X=X^\red=U$ with bases $\d x$ and
$\d x\w\d y$, and $J_\Phi$ in \eq{sm5eq9} maps
\begin{equation*}
J_\Phi:\d x\ot\d x\longmapsto x^k\,(\d x\w\d y)\ot(\d x\w\d y).
\end{equation*}
The principal $\Z/2\Z$-bundle $\pi_\Phi:P_\Phi\ra X$ in Definition
\ref{sm5def} parametrizes $\al:K_U\vert_X\ra\Phi\vert_X^*(K_V)$ with
$\al\ot\al=J_\Phi$. Writing $\al:\d x\mapsto p\,\d x\w\d y$ for $p$
a local function on $X=\C\sm\{0\}$, $\al\ot\al=J_\Phi$ reduces to
$p^2=x^k$. Thus, $P_\Phi$ parametrizes (\'etale local) square roots
$p$ of~$x^k:\C\sm\{0\}\ra\C\sm\{0\}$.

If $k$ is even then $x^k$ has a global square root $p=x^{k/2}$, so
the principal $\Z/2\Z$-bundle $P_\Phi$ has a global section, and is
trivial. If $k$ is odd then $x^k$ has no global square root on
$X=\C\sm\{0\}$, so $P_\Phi$ has no global section, and is
nontrivial.

Thus, Theorem \ref{sm5thm2} implies that if $k$ is even then
$\PV_{V,g}^\bu\cong A_Y[1]$ is the constant perverse sheaf on $Y$,
but if $k$ is odd then $\PV_{V,g}^\bu$ is the twist of $A_Y[1]$ by
the unique nontrivial principal $\Z/2\Z$-bundle
on~$Y\cong\C\sm\{0\}$.
\label{sm5ex}
\end{ex}

\subsection{Theorem \ref{sm5thm2}(a): the isomorphism $\Th_\Phi$}
\label{sm51}

Let $U,V,f,g,X,Y,\Phi$ be as in Theorem \ref{sm5thm2}(a), and use
the notation $N_{\sst UV},q_{\sst UV}$ from Theorem
\ref{sm5thm1}(ii) and $J_\Phi,P_\Phi$ from Definition \ref{sm5def}.
We will show that there exists a unique perverse sheaf morphism
$\Th_\Phi$ in \eq{sm5eq13} which is characterized by the property
that whenever $U',V',\io,\jmath,\Phi',\al,\be,X',Y',f',g'$ are as in
Theorem \ref{sm5thm1}(i) then the following diagram of isomorphisms
in $\Perv(X')$ commutes:
\e
\begin{gathered}
\xymatrix@!0@C=270pt@R=50pt{
*+[r]{\io\vert_{X'}^*\bigl(\PV_{U,f}^\bu\bigr)}
\ar[r]_(0.25){\io\vert_{X'}^*(\ga)}
\ar[d]^(0.45){\io\vert_{X'}^*(\Th_\Phi)} &
*+[l]{\io\vert_{X'}^*\ci(\id_X\t 0)^*\bigl(\PV_{U,f}^\bu\boxtL
\PV_{\C^n,z_1^2+\cdots+z_n^2}^\bu\bigr)}
\ar[d]_(0.45){\io\vert_{X'}^*\ci(\id_X\t
0)^*(\TS_{U,f,\C^n,z_1^2+\cdots+z_n^2}^{-1})} \\
*+[r]{\begin{subarray}{l}
\ts\io\vert_{X'}^*\ci\Phi\vert_X^*\bigl(\PV_{V,g}^\bu\bigr)\\
\ts\ot_{\Z/2\Z}\io\vert_{X'}^*(P_\Phi)\end{subarray}}
\ar[d]^(0.45)\de & *+[l]{\begin{subarray}{l}\ts
\io\vert_{X'}^*\!\ci\!(\id_X\!\t
0)^*\bigl(\PV_{U\t\C^n,f\boxplus z_1^2+\cdots+z_n^2}^\bu\bigr)\!=\\
\ts\Phi'\vert_{X'}^*\!\ci\!(\al\!\t\!\be)\vert_{Y'}^*
\bigl(\PV_{U\t\C^n,f\boxplus z_1^2+\cdots+z_n^2}^\bu
\bigr)\end{subarray}\;\>{}}
\ar[d]_(0.55){\Phi'\vert_{X'}^*(\PV_{\al\t\be}^{-1})}
\\
*+[r]{\begin{subarray}{l}
\ts\io\vert_{X'}^*\ci\Phi\vert_X^*\bigl(\PV_{V,g}^\bu\bigr)=\\
\ts\Phi'\vert_{X'}^*\ci\jmath\vert_{Y'}^*\bigl(\PV_{V,g}^\bu\bigr)
\end{subarray}} &
*+[l]{\Phi'\vert_{X'}^*\bigl(\PV_{V',g'}^\bu\bigr),}
\ar[l]_{\Phi'\vert_{X'}^*(\PV_\jmath)}  }
\end{gathered}
\label{sm5eq17}
\e
where $\TS_{U,f,\C^n,z_1^2+\cdots+z_n^2}$ is as in \eq{sm2eq8}, and
$\ga,\de$ are defined as follows:
\begin{itemize}
\setlength{\itemsep}{0pt}
\setlength{\parsep}{0pt}
\item[(A)] $\ga:\PV_{U,f}^\bu\ra(\id_X\t 0)^*\bigl(\PV_{U,f}^\bu
\boxtL\PV_{\C^n,z_1^2+\cdots+z_n^2}^\bu\bigr)$ in $\Perv(X)$
comes from the isomorphism $\PV_{\C^n,z_1^2+\cdots+z_n^2}^\bu
\cong A_{\{0\}}$ in~\eq{sm2eq12}.
\item[(B)] The principal $\Z/2\Z$-bundle $P_\Phi\ra X$
comes from $(N_{\sst UV}\vert_{X^\red},\ab q_{\sst
UV}\vert_{X^\red})$, as the bundle of square roots of
$\det(q_{\sst UV}\vert_{X^\red})$. Thus, the pullback
$\io\vert_{X'}^*(P_\Phi)\ra X'$ comes from
$\bigl(\io\vert_{X^{\prime\red}}^*(N_{\sst
UV}),\io\vert_{X^{\prime\red}}^*(q_{\sst UV})\bigr)$. Now
Theorem \ref{sm5thm1}(ii) defines
$\hat\be\vert_{X^{\prime\red}}:\langle\d z_1,\ldots,\d
z_n\rangle_{X^{\prime\red}}\,{\buildrel\cong\over\longra}\,
\io\vert_{X^{\prime\red}}^*(N_{\sst UV}^*) $ identifying
$\sum_{j=1}^n\d z_j^2$ with $\io\vert_{X^{\prime\red}}^*(q_{\sst
UV})$. Thus, $\hat\be\vert_{X^{\prime\red}}$ induces a
trivialization of~$\io\vert_{X'}^*(P_\Phi)\ra X'$.

Then $\de:\io\vert_{X'}^*\ci\Phi\vert_X^*\bigl(\PV_{V,g}^\bu
\bigr)\ot_{\Z/2\Z}\io\vert_{X'}^*(P_\Phi)\ra\io\vert_{X'}^*\ci
\Phi\vert_X^*\bigl(\PV_{V,g}^\bu\bigr)$ in $\Perv(X')$ comes
from this trivialization of the principal $\Z/2\Z$-bundle
$\io\vert_{X'}^*(P_\Phi)\ra X'$.
\end{itemize}

Since Theorem \ref{sm5thm1}(i) holds for each $x\in X$, we may
choose a family $\bigl\{(U_a',\ab V_a',\ab\io_a,\ab\jmath_a,\ab
\Phi'_a, \al_a,\be_a,f'_a,g'_a,X'_a,Y'_a):a\in A\bigr\}$, such that
$U_a',V_a',\ldots,Y'_a$ satisfy Theorem \ref{sm5thm1}(i) for each
$a\in A$, and $\bigl\{\io_a'\vert_{\smash{X_a'}}:X_a'\ra
X\bigr\}{}_{a\in A}$ is an \'etale open cover of $X$. For each $a\in
A$, define an isomorphism
\begin{equation*}
\Th_a:\io_a\vert_{X'_a}^*\bigl(\PV_{U,f}^\bu\bigr)\longra
\io_a\vert_{X_a'}^*\ci\Phi\vert_X^*\bigl(\PV_{V,g}^\bu\bigr)
\end{equation*}
to make the following diagram of isomorphisms commute:
\e
\begin{gathered}
\xymatrix@!0@C=270pt@R=50pt{
*+[r]{\io_a\vert_{X'_a}^*\bigl(\PV_{U,f}^\bu\bigr)}
\ar[r]_(0.25){\io_a\vert_{X_a'}^*(\ga)} \ar[d]^(0.45){\Th_a} &
*+[l]{\io_a\vert_{X'_a}^*\ci(\id_X\t 0)^*\bigl(\PV_{U,f}^\bu\boxtL
\PV_{\C^n,z_1^2+\cdots+z_n^2}^\bu\bigr)}
\ar[d]_(0.45){\io_a\vert_{X_a'}^*\ci(\id_X\t
0)^*(\TS_{U,f,\C^n,z_1^2+\cdots+z_n^2}^{-1})} \\
*+[r]{\begin{subarray}{l}
\ts\io_a\vert_{X_a'}^*\ci\Phi\vert_X^*\bigl(\PV_{V,g}^\bu\bigr)\\
\ts\ot_{\Z/2\Z}\io_a\vert_{X_a'}^*(P_\Phi)\end{subarray}}
\ar[d]^(0.45){\de_a} & *+[l]{\begin{subarray}{l}\ts
\io_a\vert_{X_a'}^*\!\ci\!(\id_X\!\t
0)^*\bigl(\PV_{U\t\C^n,f\boxplus z_1^2+\cdots+z_n^2}^\bu\bigr)\!=\\
\ts\Phi_a'\vert_{X_a'}^*\!\ci\!(\al_a\!\t\!\be_a)\vert_{Y_a'}^*
\bigl(\PV_{U\t\C^n,f\boxplus z_1^2+\cdots+z_n^2}^\bu
\bigr)\end{subarray}\;\>{}}
\ar[d]_(0.55){\Phi_a'\vert_{X_a'}^*(\PV_{\al_a\t\be_a}^{-1})}
\\
*+[r]{\begin{subarray}{l}
\ts\io_a\vert_{X_a'}^*\ci\Phi\vert_X^*\bigl(\PV_{V,g}^\bu\bigr)=\\
\ts\Phi_a'\vert_{X_a'}^*\ci\jmath_a\vert_{Y_a'}^*\bigl(\PV_{V,g}^\bu\bigr)
\end{subarray}} &
*+[l]{\Phi_a'\vert_{X_a'}^*\bigl(\PV_{V_a',g_a'}^\bu\bigr),}
\ar[l]_{\Phi_a'\vert_{X_a'}^*(\PV_{\jmath_a})}  }
\end{gathered}
\label{sm5eq18}
\e
where $\ga$ is as in (A), and $\de_a$ defined as in (B) above.

For $a,b\in A$, define $U'_{ab}=U'_a\t_{\io_a,U,\io_b}U'_b$ and
$V'_{ab}=V'_a\t_{\jmath_a,V,\jmath_b}V'_b$, with projections
$\Pi_{\smash{U'_a}}:U'_{ab}\ra U'_a$, $\Pi_{\smash{U'_b}}:
U'_{ab}\ra U'_b$, $\Pi_{\smash{V'_a}}:V'_{ab}\ra V'_a$,
$\Pi_{\smash{V'_b}}:V'_{ab}\ra V'_b$. Then $U'_{ab},V'_{ab}$ are
smooth and $\Pi_{\smash{U'_a}},\Pi_{\smash{U'_b}},
\Pi_{\smash{V'_a}},\Pi_{\smash{V'_b}}$ \'etale. The universal
property of $V'_a\t_{\jmath_a,V,\jmath_b}V'_b$ gives a unique
morphism $\Phi'_{ab}:U'_{ab}\ra V'_{ab}$ with
\e
\Pi_{\smash{V_a'}}\ci\Phi_{ab}'=\Phi_a'\ci\Pi_{\smash{U_a'}}
\quad\text{and}\quad
\Pi_{\smash{V_b'}}\ci\Phi_{ab}'=\Phi_b'\ci\Pi_{\smash{U_b'}}.
\label{sm5eq19}
\e

Set $f'_{ab}=f_a'\ci\Pi_{U_a'}:U'_{ab}\ra\C$,
$g'_{ab}=g_a'\ci\Pi_{V_a'}:V'_{ab}\ra\C$ and
$X'_{ab}=\Crit(f'_{ab})\subseteq U'_{ab}$,
$Y'_{ab}=\Crit(g'_{ab})\subseteq V'_{ab}$. As for \eq{sm4eq17} we
have
\begin{align*}
f'_{ab}&=f_a'\ci\Pi_{\smash{U_a'}}=f\ci \io_a\ci\Pi_{\smash{U_a'}}=
f\ci\io_b\ci\Pi_{\smash{U_b'}}=f_b'\ci\Pi_{\smash{U_b'}},\\
g'_{ab}&=g_a'\ci\Pi_{\smash{V_a'}}=g\ci\jmath_a\ci\Pi_{\smash{V_a'}}=
g\ci \jmath_b\ci\Pi_{\smash{V_b'}}=g_b'\ci\Pi_{\smash{V_b'}}\\
&=(f\boxplus z_1^2+\cdots+z_n^2)\ci(\al_a\t\be_a)\ci
\Pi_{\smash{V_a'}}\\
&=(f\boxplus z_1^2+\cdots+z_n^2)\ci
(\al_b\t\be_b)\ci\Pi_{\smash{V_b'}}.
\end{align*}

Apply Theorem \ref{sm3thm} with $V_{ab}',U\t\C^n,(\al_a\t\be_a)
\ci\Pi_{\smash{V_a'}},(\al_b\t\be_b)\ci\Pi_{\smash{V_b'}},g'_{ab}$,
and $f\boxplus z_1^2+\cdots+z_n^2$ in place of $V,W,\Phi,\Psi,f,g$.
The analogue of $\Phi\vert_X=\Psi\vert_X$ is
\e
\begin{aligned}
(\al_a&\t\be_a) \ci\Pi_{\smash{V_a'}}\vert_{\smash{Y_{ab}'}}=
((\Phi\vert_X^{-1}\ci\Phi\vert_X\ci\al_a)\vert_{Y_a'}\t 0)
\ci\Pi_{\smash{V_a'}}\vert_{\smash{Y_{ab}'}}\\
&=(\Phi\vert_X^{-1}\ci\jmath_a\ci\Pi_{\smash{V_a'}}
\vert_{\smash{Y_{ab}'}})\t 0=(\Phi\vert_X^{-1}\ci\jmath_b
\ci\Pi_{\smash{V_b'}}\vert_{\smash{Y_{ab}'}})\t 0 \\
&=((\Phi\vert_X^{-1}\ci\Phi\vert_X\ci\al_b)\vert_{Y_b'}\t 0)
\ci\Pi_{\smash{V_b'}}\vert_{\smash{Y_{ab}'}}=
(\al_b\t\be_b) \ci\Pi_{\smash{V_b'}}\vert_{\smash{Y_{ab}'}},
\end{aligned}
\label{sm5eq20}
\e
using $\Phi\vert_X:X\ra Y$ an isomorphism and $\be_a\vert_{Y_a'}=0$
in the first step, $\jmath_a\vert_{Y_a'}=
\Phi\vert_X\ci\al_a\vert_{Y_a'}$ in the second,
$\jmath_a\ci\Pi_{\smash{V_a'}}= \jmath_b\ci\Pi_{\smash{V_b'}}$ in
the third, $\jmath_b\vert_{Y_b'}=\Phi\vert_X\ci\al_b\vert_{Y_b'}$ in
the fourth, and $\be_b\vert_{Y_b'}=0$ in the fifth. Thus Theorem
\ref{sm3thm} gives
\e
\begin{split}
\PV_{(\al_a\t\be_a)\ci\Pi_{\smash{V_a'}}}\!=\!
\det\bigl[\d\bigl((\al_b\!\t\!\be_b)\!\ci\!\Pi_{\smash{V_b'}}\bigr)
\vert_{Y_{ab}^{\prime\red}}^{-1}\!\ci\!\d\bigl((\al_a\!\t\!
\be_a)\!\ci\!\Pi_{\smash{V_a'}}\bigr)\vert_{Y_{ab}^{\prime\red}}
\bigr]\cdot \\
\PV_{(\al_b\t\be_b)\ci\Pi_{\smash{V_b'}}}:\PV_{V_{ab}',g_{ab}'}\!
\longra\!(\al_a\!\t\!\be_a)\!\ci\!\Pi_{\smash{V_a'}}
\vert_{\smash{Y_{ab}'}}^*\bigl(\PV_{U\t\C^n,f\boxplus
z_1^2+\cdots+z_n^2}^\bu\bigr)
\end{split}
\label{sm5eq21}
\e
in $\Perv(Y_{ab}')$, where $\det[\cdots]$ maps
$Y_{ab}^{\prime\red}\ra\{\pm 1\}$.

Consider the morphisms
\ea
\Pi_{U_a'}\vert_{X_{ab}'}^*(\de_a)&,\Pi_{U_b'}\vert_{X_{ab}'}^*(\de_b):
(\Phi\!\ci\!\io_a\!\ci\!\Pi_{U_a'})\vert_{X_{ab}'}^*\bigl(\PV_{V,g}^\bu
\bigr)\!\ot_{\Z/2\Z}\!(\io_a\!\ci\!\Pi_{U_a'})\vert_{X_{ab}'}^*(P_\Phi)
\nonumber\\
&\longra(\Phi\ci\io_a\ci\Pi_{U_a'})\vert_{X_{ab}'}^*
\bigl(\PV_{V,g}^\bu\bigr).
\label{sm5eq22}
\ea
As in (B) above, these are defined using two different
trivializations of the principal $\Z/2\Z$-bundle
$(\io_a\ci\Pi_{U_a'})\vert_{X_{ab}'}^*(P_\Phi)\ra X_{ab}'$, defined
using
\begin{equation*}
\Pi_{U_a'}\vert_{X_{ab}^{\prime\red}}^*(\hat\be_a),
\Pi_{U_b'}\vert_{X_{ab}^{\prime\red}}^*(\hat\be_b):\langle\d z_1,
\ldots,\d z_n\rangle{}_{X^{\prime\red}_{ab}}\longra (\io_a\ci\Pi_{U_a'})
\vert_{X_{ab}^{\prime\red}}^*\bigl(N_{\sst UV}^*\bigr),
\end{equation*}
which are isomorphisms of vector bundles on $X^{\prime\red}_{ab}$
identifying the nondegenerate quadratic forms $\sum_{j=1}^n\d z_j^2$
on $\langle\d z_1,\ldots,\d z_n\rangle_{X^{\prime\red}_{ab}}$ and
$(\io_a\ci\Pi_{U_a'}) \vert_{X_{ab}^{\prime\red}}^*(q_{\sst UV})$ on
$(\io_a\ci\Pi_{U_a'})\vert_{X_{ab}^{\prime\red}}^*\bigl(N_{\sst
UV}^*\bigr)$, for $\hat\be_a,\hat\be_b$ as in \eq{sm5eq3}. Thus we
see that
\e
\Pi_{U_a'}\vert_{X_{ab}'}^*(\de_a)=\det\bigl[\Pi_{U_a'}
\vert_{X_{ab}^{\prime\red}}^*(\hat\be_a)\ci
\Pi_{U_b'}\vert_{X_{ab}^{\prime\red}}^*(\hat\be_b)^{-1}\bigr]
\cdot\Pi_{U_b'}\vert_{X_{ab}'}^*(\de_b),
\label{sm5eq23}
\e
where $\det[\cdots]$ maps $X_{ab}^{\prime\red}\ra\{\pm 1\}$ since
both isomorphisms in \eq{sm5eq22} identify the same nondegenerate
quadratic forms.

We have an exact sequence of vector bundles on
$X_{ab}^{\prime\red}$:
\begin{equation*}
\xymatrix@C=20pt{0 \ar[r] & TU_{ab}'\vert_{X_{ab}^{\prime\red}} \ar[r] &
\Phi_{ab}'\vert_{X_{ab}^{\prime\red}}^*(TV'_{ab})
\ar[r] & (\io_a\!\ci\!\Pi_{U_a'})\vert_{X_{ab}^{\prime\red}}^*
(N_{\sst UV}) \ar[r] & 0.}
\end{equation*}
Choosing a local splitting of this sequence, we may identify
\begin{align*}
\Phi_{ab}'\vert_{X_{ab}^{\prime\red}}^*\bigl[\d\bigl(
(\al_b&\t\be_b)\ci\Pi_{\smash{V_b'}}\bigr)
\vert_{Y_{ab}^{\prime\red}}^{-1}\ci\d\bigl((\al_a\t
\be_a)\ci\Pi_{\smash{V_a'}}\bigr)\vert_{Y_{ab}^{\prime\red}}\bigr]\\
&\cong\begin{pmatrix} \id_{TU_{ab}'\vert_{X_{ab}^{\prime\red}}} & * \\
0 & \bigl(\Pi_{U_a'}\vert_{X_{ab}^{\prime\red}}^*(\hat\be_a)\ci
\Pi_{U_b'}\vert_{X_{ab}^{\prime\red}}^*(\hat\be_b)^{-1}\bigr){}^*
\end{pmatrix}.
\end{align*}
Therefore
\e
\begin{split}
\Phi_{ab}'&\vert_{X_{ab}^{\prime\red}}^*\bigl(
\det\bigl[\d\bigl((\al_b\t\be_b)\ci\Pi_{\smash{V_b'}}\bigr)
\vert_{Y_{ab}^{\prime\red}}^{-1}\ci\d\bigl((\al_a\t
\be_a)\ci\Pi_{\smash{V_a'}}\bigr)\vert_{Y_{ab}^{\prime\red}}\bigr]\bigr)\\
&=\det\bigl[\Pi_{U_a'}\vert_{X_{ab}^{\prime\red}}^*(\hat\be_a)\ci\Pi_{U_b'}
\vert_{X_{ab}^{\prime\red}}^*(\hat\be_b)^{-1}\bigr]:
X_{ab}^{\prime\red}\longra\{\pm 1\}.
\end{split}
\label{sm5eq24}
\e

Now
\e
\text{\begin{footnotesize}$\displaystyle
\begin{aligned}
&\Pi_{U_a'}\vert_{X_{ab}'}^*(\Th_a)=\Pi_{U_a'}\vert_{X_{ab}'}^*
(\de_a^{-1})\ci(\Phi_a'\ci\Pi_{U_a'})\vert_{X_{ab}'}^*
(\PV_{\jmath_a})\ci(\Phi_a'\ci\Pi_{U_a'})\vert_{X_{ab}'}^*
(\PV_{\al_a\t\be_a}^{-1})\\
&\;\>\ci((\id_X\t
0)\ci\io_a\ci\Pi_{U_a'})\vert_{X_{ab}'}^*
(\TS_{U,f,\C^n,\Si_jz_j^2}^{-1})\ci
(\io_a\ci\Pi_{U_a'})\vert_{X_{ab}'}^*(\ga)\\
&=\Pi_{U_a'}\vert_{X_{ab}'}^*(\de_a^{-1})\ci
(\Pi_{\smash{V_a'}}\ci\Phi_{ab}')\vert_{X_{ab}'}^*(\PV_{\jmath_a})
\ci\Phi_{ab}'\vert_{X_{ab}'}^*(\PV_{\Pi_{\smash{V_a'}}})\ci
\Phi_{ab}'\vert_{X_{ab}'}^*(\PV_{\Pi_{\smash{V_a'}}}^{-1})\\
&\;\>\ci\!(\Pi_{\smash{V_a'}}\!\ci\!\Phi_{ab}')\vert_{X_{ab}'}^*
(\PV_{\al_a\t\be_a}^{-1})
\!\ci\! ((\id_X\!\t 0)\!\ci\!\io_a\!\ci\!\Pi_{U_a'})\vert_{X_{ab}'}^*
(\TS_{U,f,\C^n,\Si_jz_j^2}^{-1})\!\ci\!
(\io_a\!\ci\!\Pi_{U_a'})\vert_{X_{ab}'}^*(\ga)\\
&=\Pi_{U_a'}\vert_{X_{ab}'}^*(\de_a^{-1})
\ci\Phi_{ab}'\vert_{X_{ab}'}^*(\PV_{\jmath_a\ci\Pi_{\smash{V_a'}}})
\ci\Phi_{ab}'\vert_{X_{ab}'}^*(\PV_{(\al_a\t\be_a)\ci
\Pi_{\smash{V_a'}}}^{-1})\\
&\;\>\ci ((\id_X\t 0)\ci\io_a\ci\Pi_{U_a'})\vert_{X_{ab}'}^*
(\TS_{U,f,\C^n,\Si_jz_j^2}^{-1})\ci
(\io_a\ci\Pi_{U_a'})\vert_{X_{ab}'}^*(\ga)\\
&=\det\bigl[\Pi_{U_a'} \vert_{X_{ab}^{\prime\red}}^*(\hat\be_a)\ci
\Pi_{U_b'}\vert_{X_{ab}^{\prime\red}}^*(\hat\be_b)^{-1}\bigr]{}^{-1}
\cdot{}\\
&\;\>\Phi_{ab}'\vert_{X_{ab}^{\prime\red}}^*\bigl(
\det\bigl[\d\bigl((\al_b\t\be_b)\ci\Pi_{\smash{V_b'}}\bigr)
\vert_{Y_{ab}^{\prime\red}}^{-1}\ci\d\bigl((\al_a\t\be_a)
\ci\Pi_{\smash{V_a'}}\bigr)\vert_{Y_{ab}^{\prime\red}}\bigr]\bigr)\cdot{}\\
&\;\>\,\Pi_{U_b'}\vert_{X_{ab}'}^*(\de_b^{-1})
\ci\Phi_{ab}'\vert_{X_{ab}'}^*(\PV_{\jmath_b\ci\Pi_{\smash{V_b'}}})
\ci\Phi_{ab}'\vert_{X_{ab}'}^*(\PV_{(\al_b\t\be_b)\ci
\Pi_{\smash{V_b'}}}^{-1})\\
&\;\>\ci ((\id_X\t 0)\ci\io_b\ci\Pi_{U_b'})\vert_{X_{ab}'}^*
(\TS_{U,f,\C^n,\Si_jz_j^2}^{-1})\ci
(\io_b\ci\Pi_{U_b'})\vert_{X_{ab}'}^*(\ga)\\
&=\Pi_{U_b'}\vert_{X_{ab}'}^*(\de_b^{-1})
\ci(\Pi_{\smash{V_b'}}\ci\Phi_{ab}')\vert_{X_{ab}'}^*(\PV_{\jmath_b})\ci
\Phi_{ab}'\vert_{X_{ab}'}^*(\PV_{\Pi_{\smash{V_b'}}})\ci
\Phi_{ab}'\vert_{X_{ab}'}^*(\PV_{\Pi_{\smash{V_b'}}}^{-1})\\
&\;\>\ci\!(\Pi_{\smash{V_b'}}\!\ci\!\Phi_{ab}')\vert_{X_{ab}'}^*
(\PV_{\al_b\t\be_b}^{-1})\!\ci\!((\id_X\!\t
0)\!\ci\!\io_b\!\ci\!\Pi_{U_b'})\vert_{X_{ab}'}^*
(\TS_{U,f,\C^n,\Si_jz_j^2}^{-1})\!\ci\!
(\io_b\!\ci\!\Pi_{U_b'})\vert_{X_{ab}'}^*(\ga)\\
&=\Pi_{U_b'}\vert_{X_{ab}'}^*(\de_b^{-1})
\ci(\Phi_b'\ci\Pi_{U_b'})\vert_{X_{ab}'}^*(\PV_{\jmath_b})\ci
(\Phi_b'\ci\Pi_{U_b'})\vert_{X_{ab}'}^*(\PV_{\al_b\t\be_b}^{-1})\\
&\;\>\ci ((\id_X\t 0)\ci\io_b\ci\Pi_{U_b'})\vert_{X_{ab}'}^*
(\TS_{U,f,\C^n,\Si_jz_j^2}^{-1})\ci
(\io_b\ci\Pi_{U_b'})\vert_{X_{ab}'}^*(\ga)
=\Pi_{U_b'}\vert_{X_{ab}'}^*(\Th_b),
\end{aligned}$\end{footnotesize}}
\!\!\!\!\!\!\!\!\!\!\!\!\!\!\!\!\!\!\!\!\!\!{}
\label{sm5eq25}
\e
using \eq{sm5eq18} in the first and seventh steps, \eq{sm5eq19} in
the second and sixth, \eq{sm2eq18} in the third, \eq{sm5eq21},
\eq{sm5eq23}, $\io_a\ci\Pi_{U_a'}=\io_b\ci\Pi_{U_b'}$ and
$\jmath_a\ci\Pi_{\smash{V_a'}}=\jmath_b\ci\Pi_{\smash{V_b'}}$ in the
fourth, and \eq{sm2eq18} and \eq{sm5eq24} in the fifth. Therefore
Theorem \ref{sm2thm3}(i) applied to the \'etale open cover
$\bigl\{\io_a\vert_{X_a'}:X_a'\ra X\bigr\}{}_{a\in A}$ of $X$ shows
that there is a unique isomorphism $\Th_\Phi$ in \eq{sm5eq13} with
$\io_a\vert_{X_a'}^*(\Th_\Phi)=\Th_a$ for all~$a\in A$.

Suppose $\bigl\{(U_a',\ldots,Y_a'):a\in A\bigr\}$ and $\bigl\{(\ti
U_a',\ldots,\ti Y_a'):a\in\ti A\bigr\}$ are alternative choices
above, yielding morphisms $\Th_\Phi$ and $\ti\Th_\Phi$ in
\eq{sm5eq13}. By running the same construction using the family
$\bigl\{(U_a',\ldots,Y_a'):a\in A\bigr\}\amalg\bigl\{(\ti
U_a',\ldots,\ti Y_a'):a\in\ti A\bigr\}$, we can show that
$\Th_\Phi=\ti\Th_\Phi$, so $\Th_\Phi$ is independent of the choice
of $\bigl\{(U_a',\ldots,Y_a'):a\in A\bigr\}$ above. Let
$U',V',\io,\jmath,\Phi',\al,\be,X',Y',f',g'$ be as in Theorem
\ref{sm5thm1}(i). Constructing $\Th_\Phi$ using
$\bigl\{(U_a',\ldots,Y_a'):a\in A\bigr\}\amalg\bigl\{(U',\ldots,
Y')\bigr\}$, we see from \eq{sm5eq18} that \eq{sm5eq17} commutes.
This completes the construction of~$\Th_\Phi$.

To see that \eq{sm5eq14}--\eq{sm5eq15} commute, in the situation of
\eq{sm5eq17} we show that Verdier duality and monodromy operators
commute with each morphism in \eq{sm5eq17}. Going clockwise from the
top left corner, $\io\vert_{X'}^*(\ga)$ is compatible with Verdier
duality and monodromy because of the commutative diagrams
\begin{equation*}
\xymatrix@C=80pt@R=11pt{ *+[r]{A_{\{0\}}} \ar[d]^\Ga \ar[r]_\cong &
*+[l]{\bD_{\{0\}}\bigl(A_{\{0\}}\bigr)} \\
*+[r]{\PV_{\C^n,\Si_jz_j^2}^\bu}
\ar[r]^(0.35){\si_{\C^n,\Si_jz_j^2}} &
*+[l]{\bD_{\{0\}}\bigl(\PV_{\C^n,\Si_jz_j^2}^\bu\bigr),} \ar[u]^{\bD(\Ga)} }
\quad
\xymatrix@C=80pt@R=11pt{ *+[r]{A_{\{0\}}} \ar[d]^\Ga \ar[r]_\id &
*+[l]{A_{\{0\}}} \ar[d]_\Ga \\
*+[r]{\PV_{\C^n,\Si_jz_j^2}^\bu}
\ar[r]^(0.5){\tau_{\C^n,\Si_jz_j^2}} &
 *+[l]{\PV_{\C^n,\Si_jz_j^2}^\bu,} }
\end{equation*}
where $\Ga:A_{\{0\}}\ra\PV_{\C^n,\Si_jz_j^2}^\bu$ is the isomorphism
used to define $\ga$ in (A) above. Equations
\eq{sm2eq9}--\eq{sm2eq10} imply that $\io\vert_{X'}^*\ci(\id_X\t
0)^*(\TS_{U,f,\C^n,\Si_jz_j^2}^{-1})$ is compatible with Verdier
duality and monodromy, and \eq{sm2eq16}--\eq{sm2eq17} imply that
$\Phi'\vert_{X'}^*(\PV_\jmath), \Phi'\vert_{X'}^*
(\PV_{\al\t\be}^{-1})$ are. Also $\de$ is compatible with Verdier
duality and monodromy, since these do not affect the trivialization
of $\io\vert_{X'}^*(P_\Phi)\ra X'$ used to define $\de$ in (B)
above.

Thus by \eq{sm5eq17} we see that $\io\vert_{X'}^*(\Th_\Phi)$ is
compatible with Verdier duality and monodromy, that is,
$\io\vert_{X'}^*$ applied to \eq{sm5eq14}--\eq{sm5eq15} commute.
Since we can form an \'etale open cover of $X$ by such
$\io\vert_{X'}:X'\ra X$, Theorem \ref{sm2thm3}(i) implies that
\eq{sm5eq14}--\eq{sm5eq15} commute.

Finally, if $U=V$, $f=g$ and $\Phi=\id_U$ then $J_\Phi=\id:K_U^2
\vert_{X^\red}\ra K_U^2\vert_{X^\red}$ in \eq{sm5eq9}, which has a
natural square root $\al=\id:K_U\vert_{X^\red}\ra
K_U\vert_{X^\red}$, so $\pi_\Phi:P_\Phi\ra X$ is trivial in
Definition \ref{sm5def}. In \eq{sm5eq17} we may put $U'=V'=U$,
$\io=\jmath=\al=\id_U$, $n=0$, $\be=0$, $X'=Y'=X$, $f'=g'=f$, and
then each morphism in \eq{sm5eq17} is essentially the identity on
$\PV_{U,f}^\bu$, so $\Th_\Phi=\id_X^*(\Th_\Phi)=
\id_{\PV_{U,f}^\bu}$. This proves Theorem~\ref{sm5thm2}(a).

\subsection{Theorem \ref{sm5thm2}(b): $\Th_\Phi$ depends only
on $\Phi\vert_X:X\ra Y$}
\label{sm52}

Suppose $\Phi,\ti\Phi:U\ra V$ are alternative choices in Theorem
\ref{sm5thm2}(a) with $\Phi\vert_X=\ti\Phi\vert_X:X\ra Y$, so that
$P_\Phi=P_{\smash{\ti\Phi}}$ by Lemma \ref{sm5lem}. Fix $x\in X$,
let $a\ne b$ be labels, and let $U'_a,V'_a,\io_a,\jmath_a,
\Phi'_a,\al_a,\be_a,X_a',Y_a',f_a',g_a'$ be as in Theorem
\ref{sm5thm1}(i) for $x,\Phi$ and $U'_b,V'_b,\ldots,g_b'$ as in
Theorem \ref{sm5thm1}(i) for $x,\ti\Phi$. As in \S\ref{sm51}, define
$\Th_a,\Th_b$ and $U'_{ab},V'_{ab},
\Pi_{\smash{U'_a}},\Pi_{\smash{U'_b}},\Pi_{\smash{V'_a}},
\Pi_{\smash{V'_b}},\Phi'_{ab},f'_{ab},g'_{ab}, X'_{ab},Y'_{ab}$, and
follow the proof in \S\ref{sm51} from \eq{sm5eq19} as far
as~\eq{sm5eq25}.

This proof does not actually need $U'_a,\ldots,g_a,\Th_a$ and
$U'_b,\ldots,g_b,\Th_b$ to be defined using the same $\Phi:U\ra V$,
it only uses in \eq{sm5eq20}--\eq{sm5eq22} that $\Phi\vert_X:X\ra Y$
is the same for $U'_a,\ldots,\Th_a$ and $U'_b,\ldots,\Th_b$. Thus we
can apply it with $U'_a,\ldots,\Th_a$ defined using $\Phi$, and
$U'_b,\ldots,\Th_b$ defined using $\ti\Phi$. Hence
\begin{align*}
(\io_a\ci\Pi_{U_a'})\vert_{X_{ab}'}^*(\Th_\Phi)&=
\Pi_{U_a'}\vert_{X_{ab}'}^*(\Th_a)=\Pi_{U_b'}\vert_{X_{ab}'}^*(\Th_b)\\
&=(\io_b\ci\Pi_{U_b'})\vert_{X_{ab}'}^*(\Th_{\smash{\ti\Phi}})
=(\io_a\ci\Pi_{U_a'})\vert_{X_{ab}'}^*(\Th_{\smash{\ti\Phi}}),
\end{align*}
using $\io_a\vert_{X_a'}^*(\Th_\Phi)=\Th_a$ in the first step,
\eq{sm5eq25} in the second, $\io_b\vert_{X_b'}^*
(\Th_{\smash{\ti\Phi}})=\Th_b$ in the third, and
$\io_a\ci\Pi_{U_a'}=\io_b\ci\Pi_{U_b'}$ in the fourth. As such
$\io_a\ci\Pi_{U_a'}\vert_{X_{ab}'}:X_{ab}'\ra X$ form an \'etale
open cover of $X$, this implies that $\Th_\Phi=
\Th_{\smash{\ti\Phi}}$ by Theorem~\ref{sm2thm3}(i).

\subsection{Theorem \ref{sm5thm2}(c): composition of the $\Th_\Phi$}
\label{sm53}

Let $U,V,W,f,g,h,X,Y,Z,\Phi,\Psi$ be as in Theorem \ref{sm5thm2}(c).
Let $x\in X$, and set $y=\Phi(x)\in Y$. Apply Theorem
\ref{sm5thm1}(i) to $U,V,f,g,X,Y,\Phi,x$ to get $\C$-schemes
$U',V'$, a point $x'\in U'$, morphisms $\io:U'\ra U,$ $\jmath:V'\ra
V,$ $\Phi':U'\ra V',$ $\al:V'\ra U$ and $\be:V'\ra\C^m$ where
$m=\dim V-\dim U$, and $f':=f\ci\io:U'\ra\C$, $g':=g\ci\jmath
:V'\ra\C$, $X':=\Crit(f')\subseteq U'$, $Y':=\Crit(g')\subseteq V'$,
satisfying conditions including $\io,\jmath,\al\t\be$ \'etale,
\eq{sm5eq1} commutes, and~$\io(x')=x$.

Similarly, apply Theorem \ref{sm5thm1}(i) to $V,W,g,h,Y,Z,\Psi,y$ to
get $\C$-schemes $\ti V,\ti W$, a point $\ti y\in\ti V$, morphisms
$\ti\io:\ti V\ra V,$ $\ti\jmath:\ti W\ra W,$ $\ti\Psi:\ti V\ra\ti
W$, $\ti\al:\ti W\ra V$ and $\ti\be:\ti W\ra\C^n$ where $n=\dim
W-\dim V$, and $\ti g:=g\ci\ti\io:\ti V\ra\C$, $\ti
h:=h\ci\ti\jmath:\ti W\ra\C$, $\ti Y:=\Crit(\ti g)\subseteq\ti V$,
$\ti Z:=\Crit(\ti h)\subseteq\ti W$, satisfying conditions.

Define $\hat U=U'\t_{\Phi\ci\io,V,\ti\io}\ti V$ and $\hat
W=V'\t_{\jmath,V,\ti\al}\ti W$, with projections $\Pi_{U'}:\hat U\ra
U'$, $\Pi_{\smash{\ti V}}:\hat U\ra\ti V$, $\Pi_{V'}:\hat W\ra V'$,
$\Pi_{\smash{\ti W}}:\hat W\ra\ti W$. As $x'\in U'$ and $\ti y\in\ti
V$ with $\Phi\ci\io(x')=y=\ti\io(\ti y)$, there exists $\hat
x\in\hat U$ with $\Pi_{U'}(\hat x)=x'$ and $\Pi_{\ti V}(\hat x)=\ti
y$. Set $\hat f:=f'\ci\Pi_{U'}: \hat U\ra\C$ and $\hat h:=\ti h\ci
\Pi_{\smash{\ti W}}:\hat W\ra\C$, and $\hat X:=\Crit(\hat
f)\subseteq\hat U$, $\hat Z:=\Crit(\hat h)\subseteq\hat W$. The
morphisms $\Phi'\ci\Pi_{U'}:\hat U\ra V'$,
$\ti\Psi\ci\Pi_{\smash{\ti V}}:\hat U\ra\ti W$ satisfy
\begin{equation*}
\jmath\ci(\Phi'\ci\Pi_{U'})=\Phi\ci\io\ci\Pi_{U'}=\ti\io
\ci\Pi_{\smash{\ti V}}=\ti\al\ci(\ti\Psi\ci\Pi_{\smash{\ti V}}).
\end{equation*}

Hence there exists a unique morphism
${}\,\,\,\,\widehat{\!\!\!\!\Psi\ci\Phi\!\!\!\!}\,\,\,\,:\hat
U\ra\hat W$ such that
$\Pi_{V'}\ci\,\,\,\,\widehat{\!\!\!\!\Psi\ci\Phi\!\!\!\!}\,\,\,\,=
\Phi'\ci\Pi_{U'}$ and $\Pi_{\smash{\ti
W}}\ci\,\,\,\,\widehat{\!\!\!\!\Psi\ci\Phi\!\!\!\!}\,\,\,\,=
\ti\Psi\ci\Pi_{\smash{\ti V}}$. Then the following diagram
\begin{equation*}
\xymatrix@C=43pt@R=15pt{ *+[r]{U} \ar[d]^{\Psi\ci\Phi} && \hat U
\ar[ll]^(0.4){\hat\io=\io\ci\Pi_{U'}}
\ar[rrr]_{\hat\io=\io\ci\Pi_{U'}}
\ar[d]_{{}\,\,\,\,\widehat{\!\!\!\!\Psi\ci\Phi\!\!\!\!}\,\,\,\,}
&&& *+[l]{U} \ar[d]_{\id_U \t 0\t 0} \\
*+[r]{W} && \hat W \ar[ll]_(0.4){\hat\jmath=\ti\jmath\ci\Pi_{\smash{\ti W}}}
\ar[rrr]^(0.35){\begin{subarray}{l}\hat\al\t\hat\be=\\
(\al\ci\Pi_{V'})\t(\be\ci\Pi_{V'})\t(\ti\be\ci \Pi_{\smash{\ti
W}})\end{subarray}} &&& *+[l]{U\!\t\!(\C^m\!\t\!\C^n)} }
\end{equation*}
is the analogue of \eq{sm5eq1} for $U,W,f,h,X,Z,\Psi\ci\Phi,x$, and
the conclusions of Theorem \ref{sm5thm1}(i) hold. Thus \eq{sm5eq17}
holds for $\Th_\Phi$ using $U',V',X',Y'\io,\jmath,\Phi',\al,\be,m$,
and for $\Th_\Psi$ using $\ti V,\ti W,\ti Y,\ti
Z,\ti\io,\ti\jmath,\ti\Psi,\ti\al,\ti\be,n$, and for
$\Th_{\Psi\ci\Phi}$ using $\hat U,\hat W,\hat X,\hat
Z,\hat\io,\ab\hat\jmath,\ab\,\,\,\,\widehat{\!\!\!\!\Psi\ci\Phi\!\!\!\!}
\,\,\,\,,\ab\hat\al,\ab\hat\be,m+n$.

We have a commutative diagram in~$\Perv(\hat X)$:
\e
\text{\begin{small}$\displaystyle \xymatrix@!0@C=51pt@R=55pt{
*+[r]{{}\!\!\!\!\!\hat\io\vert_{\hat X}^*\!\ci\!(\Psi\!\ci\!\Phi)\vert_X^*
\bigl(\PV_{W,h}^\bu\bigr)}
&&&& {\vphantom{\biggl[}}  &&
*+[l]{\hat\io\vert_{\hat X}^*\!\ci\!(\id_X\!\t
0)^*\bigl(\PV_{U\t\C^{m+n},f\boxplus \Si_jy_j^2+\Si_jz_j^2}^\bu\bigr)}
\ar[llllll]^(0.63){\,\,\,\,\widehat{\!\!\!\!\Psi\ci\Phi\!\!\!\!}
\,\,\,\,\vert_{\hat X}^*(\PV_{\hat\jmath}\ci\PV_{\hat\al\t\hat\be}^{-1})}
\\
*+[r]{\hat\io\vert_{\hat X}^*\!\ci\!\Phi\vert_X^*
\!\ci\!(\id_Y\!\t
0)^*\bigl(\PV_{V\t\C^n,g\boxplus \Si_jz_j^2}^\bu\bigr)}
\ar[u]_{(\Psi'\ci\Pi_{\smash{\ti V}})\vert_{\hat X}^*
(\PV_{\ti\jmath}\ci\PV_{\ti\al\t\ti\be}^{-1})} &&&&&&
*+[l]{\begin{subarray}{l}\ts
\hat\io\vert_{\hat X}^*\ci(\id_X\t 0)^*\bigl(\PV_{U,f}^\bu\\
\ts\boxtL\PV_{\C^{m+n},\Si_jy_j^2+\Si_jz_j^2}^\bu\bigr)
\end{subarray}}
\ar[u]^(0.6){\substack{\hat\io\vert_{\hat X}^*\ci(\id_X\t
0)^*\\ (\TS_{U,f,\C^{m+n},\Si_jy_j^2+\Si_jz_j^2}^{-1})}}
\\
&& {{}\hskip -20pt\begin{subarray}{l}
\ts \hat\io\vert_{\hat X}^*\ci(\id_X\!\t 0\!\t 0)^*\\
\ts\bigl(\PV_{U\t\C^m,f\boxplus\Si_jy_j^2}^\bu \\[-4pt]
\ts\;\>\boxtL\PV_{\C^n,\Si_jz_j^2}^\bu\bigr)\end{subarray}}
\ar[dll]_(0.55){
\substack{(\Phi'\ci\Pi_{U'})\vert_{\hat X}^*\\
{}(\PV_\jmath\ci\PV_{\al\t\be}^{-1}) \\
\boxtL\id}} \ar@<-1ex>[uurr]_(0.45){{}\!\!\!\!\!\!\!\!\!\!\!\substack{
\hat\io\vert_{\hat X}^*\ci(\id_X\t
0)^*\\ (\TS_{U\t\C^m,f\boxplus\Si_jy_j^2,\C^n,\Si_jz_j^2}^{-1})}} &&
*+[r]{\raisebox{-45pt}{$\begin{subarray}{l}
\ts \hat\io\vert_{\hat X}^*\ci(\id_X\!\t 0\!\t 0)^*\\
\ts\bigl(\PV_{U,f}^\bu
\!\boxtL\!\PV_{\C^m,\Si_jy_j^2}^\bu\!\!\!\!\!\!{} \\[-4pt]
\ts\qquad\boxtL \PV_{\C^n,\Si_jz_j^2}^\bu\bigr)\end{subarray}$}}
\ar@<1ex>[ll]^(0.5){\begin{subarray}{l}
\hat\io\vert_{\hat X}^*\ci(\id_X\t
0)^*\\ (\TS_{U,f,\C^m,\Si_jy_j^2}^{-1})\\ {}\qquad\boxtL\id
\end{subarray}} \ar[urr]_{{}\,\,\,\,\,\,\,\,\,\begin{subarray}{l}
\hat\io\vert_{\hat X}^*\ci(\id_X\t
0)^*\\ (\id\boxtL \\
\TS_{\begin{subarray}{l}\C^m,\Si_jy_j^2,\\
\C^n,\Si_jz_j^2\end{subarray}}^{-1})\end{subarray}}
\\
*+[r]{\begin{subarray}{l}\ts
\hat\io\vert_{\hat X}^*\!\ci\!\Phi\vert_X^* \!\ci\!(\id_Y\!\t 0)^* \\
\ts\bigl(\PV_{V,g}^\bu\boxtL \PV_{\C^n,\Si_jz_j^2}^\bu\bigr)
\end{subarray}}
\ar[uu]_(0.8){\hat\io\vert_{\hat
X}^*\ci\Phi\vert_X^*\ci(\id_Y\t 0)^*
(\TS_{V,g,\C^n,\Si_jz_j^2}^{-1})} &&&&&&
*+[l]{\hat\io\vert_{\hat X}^*\bigl(\PV_{U,f}^\bu\bigr)}
\ar[uu]^(0.2){\hat\io\vert_{\hat X}^*(\hat\ga)} \ar[ull]
\ar[d]_(0.45){\hat\io\vert_{\hat X}^*(\ga)}
\\
*+[r]{\hat\io\vert_{\hat X}^*\!\ci\!\Phi\vert_X^*
\bigl(\PV_{V,g}^\bu\bigr)} \ar[u]_(0.4){\hat\io\vert_{\hat
X}^*\ci\Phi\vert_X^*(\ti\ga)} &&&
{\begin{subarray}{l}\ts{}\quad\hat\io\vert_{\hat X}^*\!\ci\!(\id_X\!\t
0)^*\\
\ts\bigl(\PV_{U\t\C^m,f\boxplus\Si_jy_j^2}^\bu\bigr)\end{subarray}}
\ar[lll]_(0.42){\substack{(\Phi'\ci\Pi_{U'})\vert_{\hat X}^*\\
(\PV_\jmath\ci\PV_{\al\t\be}^{-1})}} \ar[uul] &&&
*+[l]{\begin{subarray}{l}\ts{}\quad\qquad
\hat\io\vert_{\hat X}^*\ci(\id_X\!\t 0)^*\\
\ts\bigl(\PV_{U,f}^\bu\boxtL\PV_{\C^m,\Si_jy_j^2}^\bu\bigr),\end{subarray}}
\ar[lll]_(0.58){\substack{\hat\io\vert_{\hat X}^*\ci(\id_X\t
0)^*\\ (\TS_{U,f,\C^m,\Si_jy_j^2}^{-1})}}}$\end{small}}
\label{sm5eq26}
\e
where the top right quadrilateral commutes because of associativity
in the Thom--Sebastiani Theorem for $\PV_{V,f}^\bu$, Theorem
\ref{sm2thm5}. Also we have
\e
\begin{gathered}
\xymatrix@!0@C=275pt@R=50pt{
*+[r]{(\Psi\!\ci\!\Phi\!\ci\!\hat\io)\vert_{\hat X}^*
\bigl(\PV_{W,h}^\bu\bigr)}  &
*+[l]{(\Psi\!\ci\!\Phi\!\ci\!\hat\io)\vert_{\hat X}^*
\bigl(\PV_{W,h}^\bu\bigr) \ot_{\Z/2\Z}\hat\io\vert_{\hat
X}^*(P_{\Psi\ci\Phi})} \ar[l]^(0.62){\hat\de}
\ar[d]_{\id\ot\hat\io\vert_{\hat
X}^*(\Xi_{\Psi,\Phi})} \\
*+[r]{\begin{subarray}{l}\ts
(\Psi\!\ci\!\Phi\!\ci\!\hat\io)\vert_{\hat X}^*\bigl(\PV_{W,h}^\bu\bigr)\\
\ts\ot_{\Z/2\Z}(\ti\io\!\ci\!\Pi_{\smash{\ti V}})\vert_{\hat
X}^*(P_\Psi)\end{subarray}} \ar[u]_{ \Pi_{\smash{\ti V}}\vert_{\hat
X}^*(\ti\de)} &
*+[l]{\begin{subarray}{l}\ts(\Psi\!\ci\!\Phi\!\ci\!\hat\io)\vert_{\hat X}^*
\bigl(\PV_{W,h}^\bu\bigr)\\
\ts\ot_{\Z/2\Z}(\ti\io\!\ci\!\Pi_{\smash{\ti V}})\vert_{\hat
X}^*(P_\Psi) \ot_{\Z/2\Z}\hat\io\vert_{\hat
X}^*(P_\Phi)\end{subarray} } \ar[d]_{(\Phi\ci\hat\io)\vert_{\hat
X}^*(\Th_\Psi^{-1})\ot\id} \\
*+[r]{(\Phi\!\ci\!\hat\io)\vert_{\hat X}^*\bigl(\PV_{V,g}^\bu\bigr)}
\ar[u]_{(\Phi\ci\hat\io)\vert_{\hat X}^*(\Th_\Psi)} &
*+[l]{(\Phi\!\ci\!\hat\io)\vert_{\hat
X}^*\bigl(\PV_{V,g}^\bu\bigr) \ts\ot_{\Z/2\Z}\hat\io\vert_{\hat
X}^*(P_\Phi),} \ar[l]_(0.55){ \Pi_{U'}\vert_{\hat X}^*(\de)} }
\end{gathered}
\label{sm5eq27}
\e
which commutes because the trivializations of
$\io\vert_{X'}(P_\Phi),\ti\io\vert_{\ti X}(P_\Psi),
\hat\io\vert_{\hat X}(P_{\Psi\ci\Phi})$ used to define
$\de,\ti\de,\hat\de$ are compatible with $\Xi_{\Psi,\Phi}$.

Combining \eq{sm5eq26} and \eq{sm5eq27} with $\Pi_{U'}\vert_{\hat
X}^*$ applied to \eq{sm5eq17} for $\Th_\Phi$, and $\Pi_{\smash{\ti
V}}\vert_{\hat X}^*$ applied to \eq{sm5eq17} for $\Th_\Psi$, and
\eq{sm5eq17} for $\Th_{\Psi\ci\Phi}$, we can show that the following
diagram commutes in $\Perv(\hat X)$:
\begin{equation*}
\xymatrix@C=163pt@R=21pt{
*+[r]{\hat\io\vert_{\hat X}^*\bigl(\PV_{U,f}^\bu\bigr)}
\ar[r]_(0.35){\hat\io\vert_{\hat X}^*(\Th_{\Psi\ci\Phi})}
\ar@<2ex>[d]^{\hat\io\vert_{\hat X}^*(\Th_\Phi)} &
*+[l]{(\Psi\!\ci\!\Phi\!\ci\!\hat\io)\vert_{\hat X}^*
\bigl(\PV_{W,h}^\bu\bigr) \ot_{\Z/2\Z}\hat\io\vert_{\hat
X}^*(P_{\Psi\ci\Phi})}
\ar@<-2ex>[d]_{\id\ot \hat\io\vert_{\hat X}^*(\Xi_{\Psi,\Phi})} \\
*+[r]{(\Phi\!\ci\!\hat\io)\vert_{\hat X}^*\bigl(\PV_{V,g}^\bu\bigr)
\ts\ot_{\Z/2\Z}\hat\io\vert_{\hat X}^*(P_\Phi)}
\ar[r]^(0.51){\raisebox{6pt}{$\st(\Phi\ci\hat\io)\vert_{\hat
X}^*(\Th_\Psi)\ot\id$}} &
*+[l]{\begin{subarray}{l}\ts{}\qquad (\Psi\!\ci\!\Phi\!\ci\!\hat\io)
\vert_{\hat X}^*\bigl(\PV_{W,h}^\bu\bigr)\\
\ts\ot_{\Z/2\Z}(\ti\io\!\ci\!\Pi_{\smash{\ti V}})\vert_{\hat
X}^*(P_\Psi) \ot_{\Z/2\Z}\hat\io\vert_{\hat X}^*(P_\Phi),
\end{subarray}} }
\end{equation*}
which is $\hat\io\vert_{\hat X}^*$ applied to \eq{sm5eq16}. Since
such $\hat\io\vert_{\hat X}:\hat X\ra X$ form an \'etale cover of
$X$, equation \eq{sm5eq16} commutes by Theorem \ref{sm2thm3}(i).
This proves Theorem~\ref{sm5thm2}(c).

\subsection{$\cD$-modules and mixed Hodge modules}
\label{sm54}

By Theorem \ref{sm5thm1}(iv),(v), the earlier parts of that result
hold for our other contexts in \S\ref{sm26}--\S\ref{sm210}. Once
again, the proofs of Theorem \ref{sm5thm2}(a)--(c) then carry over
to the other contexts using the general framework of \S\ref{sm25},
now also making use of property~(vii).

\section{Perverse sheaves on oriented d-critical loci}
\label{sm6}

\subsection{Background material on d-critical loci}
\label{sm61}

Here are some of the main definitions and results on d-critical
loci, from Joyce \cite[Th.s 2.1, 2.20, 2.28 \& Def.s 2.5, 2.18,
2.31]{Joyc}. For the algebraic case we work with $\C$-schemes.

\begin{thm} Let\/ $X$ be a $\C$-scheme. Then there exists a sheaf\/
$\cS_X$ of\/ $\C$-vector spaces on $X,$ unique up to canonical
isomorphism, which is uniquely characterized by the following two
properties:
\begin{itemize}
\setlength{\itemsep}{0pt}
\setlength{\parsep}{0pt}
\item[{\bf(i)}] Suppose $R\subseteq X$ is Zariski open, $U$ is a
smooth\/ $\C$-scheme, and\/ $i:R\hookra U$ is a closed
embedding. Then we have an exact sequence of sheaves of\/
$\C$-vector spaces on $R\!:$
\begin{equation*}
\smash{\xymatrix@C=30pt{ 0 \ar[r] & I_{R,U} \ar[r] &
i^{-1}(\O_U) \ar[r]^{i^\sharp} & \O_X\vert_R \ar[r] & 0, }}
\end{equation*}
where $\O_X,\O_U$ are the sheaves of regular functions on $X,U,$
and\/ $i^\sharp$ is the morphism of sheaves of\/ $\C$-algebras
on $R$ induced by $i$.

There is an exact sequence of sheaves of\/ $\C$-vector spaces on
$R\!:$
\begin{equation*}
\xymatrix@C=20pt{ 0 \ar[r] & \cS_X\vert_R
\ar[rr]^(0.4){\io_{R,U}} &&
\displaystyle\frac{i^{-1}(\O_U)}{I_{R,U}^2} \ar[rr]^(0.4)\d &&
\displaystyle\frac{i^{-1}(T^*U)}{I_{R,U}\cdot i^{-1}(T^*U)}\,, }
\end{equation*}
where $\d$ maps $f+I_{R,U}^2\mapsto \d f+I_{R,U}\cdot
i^{-1}(T^*U)$.
\item[{\bf(ii)}] Let\/ $R\subseteq S\subseteq X$ be Zariski open,
$U,V$ be smooth\/ $\C$-schemes, $i:R\hookra U,$ $j:S\hookra V$
closed embeddings, and\/ $\Phi:U\ra V$ a morphism with\/
$\Phi\ci i=j\vert_R:R\ra V$. Then the following diagram of
sheaves on $R$ commutes:
\e
\begin{gathered}
\xymatrix@C=12pt{ 0 \ar[r] & \cS_X\vert_R \ar[d]^\id
\ar[rrr]^(0.4){\io_{S,V}\vert_R} &&&
\displaystyle\frac{j^{-1}(\O_V)}{I_{S,V}^2}\Big\vert_R
\ar@<-2ex>[d]^{i^{-1}(\Phi^\sharp)} \ar[rr]^(0.4)\d &&
\displaystyle\frac{j^{-1}(T^*V)}{I_{S,V}\cdot
j^{-1}(T^*V)}\Big\vert_R \ar@<-2ex>[d]^{i^{-1}(\d\Phi)} \\
 0 \ar[r] & \cS_X\vert_R \ar[rrr]^(0.4){\io_{R,U}} &&&
\displaystyle\frac{i^{-1}(\O_U)}{I_{R,U}^2} \ar[rr]^(0.4)\d &&
\displaystyle\frac{i^{-1}(T^*U)}{I_{R,U}\cdot i^{-1}(T^*U)}\,.
}\!\!\!\!\!\!\!{}
\end{gathered}
\label{sm6eq1}
\e
Here $\Phi:U\ra V$ induces $\Phi^\sharp: \Phi^{-1}(\O_V)\ra\O_U$
on $U,$ so we have
\e
i^{-1}(\Phi^\sharp):j^{-1}(\O_V)\vert_R=i^{-1}\ci
\Phi^{-1}(\O_V)\longra i^{-1}(\O_U),
\label{sm6eq2}
\e
a morphism of sheaves of\/ $\C$-algebras on $R$. As $\Phi\ci
i=j\vert_R,$ equation \eq{sm6eq2} maps $I_{S,V}\vert_R\ra
I_{R,U},$ and so maps $I_{S,V}^2\vert_R\ra I_{R,U}^2$. Thus
\eq{sm6eq2} induces the morphism in the second column of\/
\eq{sm6eq1}. Similarly, $\d\Phi:\Phi^{-1}(T^*V)\ra T^*U$ induces
the third column of\/~\eq{sm6eq1}.
\end{itemize}

There is a natural decomposition\/ $\cS_X=\cSz_X\op\C_X,$ where\/
$\C_X$ is the constant sheaf on $X$ with fibre $\C,$ and\/
$\cSz_X\subset\cS_X$ is the kernel of the composition
\begin{equation*}
\xymatrix@C=40pt{ \cS_X \ar[r]^{\be_X} & \O_X
\ar[r]^(0.47){i_X^\sharp} & \O_{X^\red},  }
\end{equation*}
with\/ $X^\red$ the reduced\/ $\C$-subscheme of\/ $X,$ and\/
$i_X:X^\red\hookra X$ the inclusion.

The analogue of all the above also holds for complex analytic
spaces.
\label{sm6thm1}
\end{thm}

\begin{dfn} An {\it algebraic d-critical locus\/} over $\C$ is a
pair $(X,s)$, where $X$ is a $\C$-scheme, and $s\in H^0(\cSz_X)$ for
$\cSz_X$ as in Theorem \ref{sm6thm1}, satisfying the condition that
for each $x\in X$, there exists a Zariski open neighbourhood $R$ of
$x$ in $X$, a smooth $\C$-scheme $U$, a regular function
$f:U\ra\bA^1=\C$, and a closed embedding $i:R\hookra U$, such that
$i(R)=\Crit(f)$ as $\C$-subschemes of $U$,
and~$\io_{R,U}(s\vert_R)=i^{-1}(f)+I_{R,U}^2$.

Similarly, a {\it complex analytic d-critical locus\/} is a pair
$(X,s)$, where $X$ is a complex analytic space, and $s\in
H^0(\cSz_X)$ for $\cS_X$ as in Theorem \ref{sm6thm1}, such that each
$x\in X$ has an open neighbourhood $R\subset X$ with a closed
embedding $i:R\hookra U$ into a complex manifold $U$ and a
holomorphic function $f:U\ra\C$, such that $i(R)=\Crit(f)$,
and~$\io_{R,U}(s\vert_R)=i^{-1}(f)+I_{R,U}^2$.

In both cases we call the quadruple $(R,U,f,i)$ a {\it critical
chart\/} on~$(X,s)$.

Let $(X,s)$ be a d-critical locus (either algebraic or complex
analytic), and $(R,U,f,i)$ be a critical chart on $(X,s)$. Let
$U'\subseteq U$ be (Zariski) open, and set $R'=i^{-1}(U')\subseteq
R$, $i'=i\vert_{R'}:R'\hookra U'$, and $f'=f\vert_{U'}$. Then
$(R',U',f',i')$ is also a critical chart on $(X,s)$, and we call it
a {\it subchart\/} of $(R,U,f,i)$. As a shorthand we
write~$(R',U',f',i')\subseteq (R,U,f,i)$.

Let $(R,U,f,i),(S,V,g,j)$ be critical charts on $(X,s)$, with
$R\subseteq S\subseteq X$. An {\it embedding\/} of $(R,U,f,i)$ in
$(S,V,g,j)$ is a locally closed embedding $\Phi:U\hookra V$ such
that $\Phi\ci i=j\vert_R$ and $f=g\ci\Phi$. As a shorthand we write
$\Phi: (R,U,f,i)\hookra(S,V,g,j)$. If $\Phi:(R,U,f,i)\hookra
(S,V,g,j)$ and $\Psi:(S,V,g,j)\hookra(T,W,h,k)$ are embeddings, then
$\Psi\ci\Phi:(R,U,f,i)\hookra(T,W,h,k)$ is also an embedding.
\label{sm6def1}
\end{dfn}

\begin{thm} Let\/ $(X,s)$ be a d-critical locus (either algebraic
or complex analytic), and\/ $(R,U,f,i),(S,V,g,j)$ be critical charts
on $(X,s)$. Then for each\/ $x\in R\cap S\subseteq X$ there exist
subcharts $(R',U',f',i')\subseteq(R,U,f,i),$ $(S',V',g',j')\subseteq
(S,V,g,j)$ with\/ $x\in R'\cap S'\subseteq X,$ a critical chart\/
$(T,W,h,k)$ on $(X,s),$ and embeddings $\Phi:(R',U',f',i')\hookra
(T,W,h,k),$ $\Psi:(S',V',g',j')\hookra(T,W,h,k)$.
\label{sm6thm2}
\end{thm}

\begin{thm} Let\/ $(X,s)$ be a d-critical locus (either algebraic
or complex analytic), and\/ $X^\red\subseteq X$ the associated
reduced\/ $\C$-scheme or reduced complex analytic space. Then there
exists an (algebraic or holomorphic) line bundle $K_{X,s}$ on
$X^\red$ which we call the \begin{bfseries}canonical
bundle\end{bfseries} of\/ $(X,s),$ which is natural up to canonical
isomorphism, and is characterized by the following properties:
\begin{itemize}
\setlength{\itemsep}{0pt}
\setlength{\parsep}{0pt}
\item[{\bf(i)}] If\/ $(R,U,f,i)$ is a critical chart on
$(X,s),$ there is a natural isomorphism
\e
\io_{R,U,f,i}:K_{X,s}\vert_{R^\red}\longra
i^*\bigl(K_U^{\ot^2}\bigr)\vert_{R^\red},
\label{sm6eq3}
\e
where $K_U=\La^{\dim U}T^*U$ is the canonical bundle of\/ $U$ in
the usual sense.
\item[{\bf(ii)}] Let\/ $\Phi:(R,U,f,i)\hookra(S,V,g,j)$ be an
embedding of critical charts on $(X,s)$. Then \eq{sm5eq9}
defines an isomorphism of line bundles on $\Crit(f)^\red:$
\begin{equation*}
J_\Phi:K_U^{\ot^2}\vert_{\Crit(f)^\red}\,{\buildrel\cong\over\longra}\,
\Phi\vert_{\Crit(f)^\red}^*\bigl(K_V^{\ot^2}\bigr).
\end{equation*}
Since $i:R\ra\Crit(f)$ is an isomorphism with\/ $\Phi\ci
i=j\vert_R,$ this gives
\begin{equation*}
i\vert_{R^\red}^*(J_\Phi):i\vert_{R^\red}^*\bigl(K_U^{\ot^2}
\bigr)\,{\buildrel\cong\over\longra}\,
j\vert_{R^\red}^*\bigl(K_V^{\ot^2}\bigr),
\end{equation*}
and we must have
\e
\io_{S,V,g,j}\vert_{R^\red}=i\vert_{R^\red}^*(J_\Phi)\ci
\io_{R,U,f,i}:K_{X,s}\vert_{R^\red}\longra j^*
\bigl(K_V^{\ot^2}\bigr)\big\vert_{R^\red}.
\label{sm6eq4}
\e
\end{itemize}
\label{sm6thm3}
\end{thm}

\begin{dfn} Let $(X,s)$ be a d-critical locus (either algebraic or
complex analytic), and $K_{X,s}$ its canonical bundle from Theorem
\ref{sm6thm3}. An {\it orientation\/} on $(X,s)$ is a choice of
square root line bundle $K_{X,s}^{1/2}$ for $K_{X,s}$ on $X^\red$.
That is, an orientation is an (algebraic or holomorphic) line bundle
$L$ on $X^\red$, together with an isomorphism $L^{\ot^2}=L\ot L\cong
K_{X,s}$. A d-critical locus with an orientation will be called an
{\it oriented d-critical locus}.
\label{sm6def2}
\end{dfn}

In \cite[Th.~6.6]{BBJ} we show that algebraic d-critical loci are
classical truncations of objects in derived algebraic geometry known
as $-1$-{\it shifted symplectic derived schemes}, introduced by
Pantev, To\"en, Vaqui\'e and Vezzosi \cite{PTVV}.

\begin{thm}[Bussi, Brav and Joyce \cite{BBJ}] Suppose\/ $(\bs
X,\om)$ is a $-1$-shifted symplectic derived scheme in the sense of
Pantev et al.\ {\rm\cite{PTVV}} over $\C$, and let\/ $X=t_0(\bs X)$ be the
associated classical\/ $\C$-scheme of\/ ${\bs X}$. Then $X$ extends
naturally to an algebraic d-critical locus\/ $(X,s)$. The canonical
bundle $K_{X,s}$ from Theorem\/ {\rm\ref{sm6thm3}} is naturally
isomorphic to the determinant line bundle $\det(\bL_{\bs
X})\vert_{X^\red}$ of the cotangent complex\/ $\bL_{\bs X}$
of\/~$\bs X$.
\label{sm6thm4}
\end{thm}

Now Pantev et al.\ \cite{PTVV} show that derived moduli schemes of
coherent sheaves, or complexes of coherent sheaves, on a Calabi--Yau
3-fold $Y$ have $-1$-shifted symplectic structures. Using this, in
\cite[Cor.~6.7]{BBJ} we deduce:

\begin{cor} Suppose $Y$ is a Calabi--Yau\/ $3$-fold over $\C$, and\/
$\cM$ is a classical moduli\/ $\C$-scheme of simple coherent sheaves
in $\coh(Y),$ or simple complexes of coherent sheaves in
$D^b\coh(Y),$ with (symmetric) obstruction theory\/
$\phi:\cE^\bu\ra\bL_\cM$ as in Behrend\/ {\rm\cite{Behr},} Thomas\/
{\rm\cite{Thom},} or Huybrechts and Thomas\/ {\rm\cite{HuTh}}. Then
$\cM$ extends naturally to an algebraic d-critical locus $(\cM,s)$.
The canonical bundle $K_{\cM,s}$ from Theorem\/ {\rm\ref{sm6thm3}}
is naturally isomorphic to $\det(\cE^\bu)\vert_{\cM^\red}$.
\label{sm6cor1}
\end{cor}

Here we call $F\in\coh(Y)$ {\it simple\/} if $\Hom(F,F)=\C$, and $F^\bu\in
D^b\coh(Y)$ {\it simple\/} if $\Hom(F^\bu,F^\bu)=\C$ and
$\mathop{\rm Ext}^{<0}(F^\bu,F^\bu)=0$. Thus, d-critical loci will have applications in Donaldson--Thomas
theory for Calabi--Yau 3-folds \cite{JoSo,KoSo1,KoSo2,Thom}.
Orientations on $(\cM,s)$ are closely related to {\it orientation
data\/} in the work of Kontsevich and Soibelman~\cite{KoSo1,KoSo2}.

Pantev et al.\ \cite{PTVV} also show that derived intersections
$L\cap M$ of algebraic Lagrangians $L,M$ in an algebraic symplectic
manifold $(S,\om)$ have $-1$-shifted symplectic structures, so that
Theorem \ref{sm6thm4} gives them the structure of algebraic
d-critical loci. Bussi \cite[\S 3]{Buss} will prove a complex
analytic version of this:

\begin{thm}[Bussi \cite{Buss}] Suppose $(S,\om)$ is a complex
symplectic manifold, and\/ $L,M$ are complex Lagrangian submanifolds
in $S$. Then the intersection $X=L\cap M,$ as a complex analytic
subspace of\/ $S,$ extends naturally to a complex analytic
d-critical locus\/ $(X,s)$. The canonical bundle $K_{X,s}$ from
Theorem\/ {\rm\ref{sm6thm3}} is naturally isomorphic
to\/~$K_L\vert_{X^\red}\ot K_M\vert_{X^\red}$.
\label{sm6thm5}
\end{thm}

\subsection{The main result, and applications}
\label{sm62}

Here is our main result, which will be proved
in~\S\ref{sm63}--\S\ref{sm64}.

\begin{thm} Let\/ $(X,s)$ be an oriented algebraic d-critical locus
over $\C,$ with orientation $K_{X,s}^{1/2}$. Then for any
well-behaved base ring $A,$ such as $\Z,\Q$ or $\C,$ there exists a
perverse sheaf\/ $P_{X,s}^\bu$ in $\Perv(X)$ over $A,$ which is
natural up to canonical isomorphism, and Verdier duality and
monodromy isomorphisms
\e
\Si_{X,s}:P_{X,s}^\bu\longra \bD_X\bigl(P_{X,s}^\bu\bigr),\qquad
\Tau_{X,s}:P_{X,s}^\bu\longra P_{X,s}^\bu,
\label{sm6eq5}
\e
which are characterized by the following properties:
\begin{itemize}
\setlength{\itemsep}{0pt}
\setlength{\parsep}{0pt}
\item[{\bf(i)}] If\/ $(R,U,f,i)$ is a critical chart on
$(X,s),$ there is a natural isomorphism
\e
\om_{R,U,f,i}:P_{X,s}^\bu\vert_R\longra
i^*\bigl(\PV_{U,f}^\bu\bigr)\ot_{\Z/2\Z} Q_{R,U,f,i},
\label{sm6eq6}
\e
where $\pi_{R,U,f,i}:Q_{R,U,f,i}\ra R$ is the principal\/
$\Z/2\Z$-bundle parametrizing local isomorphisms
$\al:K_{X,s}^{1/2}\ra i^*(K_U)\vert_{R^\red}$ with\/ $\al\ot\al=
\io_{R,U,f,i},$ for $\io_{R,U,f,i}$ as in\/ \eq{sm6eq3}.
Furthermore the following commute in $\Perv(R)\!:$
\ea
&\begin{gathered} \xymatrix@!0@C=260pt@R=40pt{
*+[r]{P_{X,s}^\bu\vert_R} \ar[d]^{\Si_{X,s}\vert_R}
\ar[r]_(0.35){\om_{R,U,f,i}} &
*+[l]{i^*\bigl(\PV_{U,f}^\bu\bigr)\ot_{\Z/2\Z} Q_{R,U,f,i}}
\ar[d]_(0.4){i^*(\si_{U,f})\ot\id_{Q_{R,U,f,i}}} \\
*+[r]{\bD_R\bigl(P_{X,s}^\bu\vert_R\bigr)} &
*+[l]{\begin{aligned}[b]
\ts i^*\bigl(\bD_{\Crit(f)}(\PV_{U,f}^\bu)\bigr)\ot_{\Z/2\Z}
Q_{R,U,f,i}&\\
\ts \cong\bD_R\bigl(i^*(\PV_{U,f}^\bu)\ot_{\Z/2\Z}
Q_{R,U,f,i}\bigr)&,\end{aligned}}
\ar[l]_(0.65){\bD_R(\om_{R,U,f,i})}}
\end{gathered}
\label{sm6eq7}\\
&\begin{gathered} \xymatrix@!0@C=260pt@R=35pt{
*+[r]{P_{X,s}^\bu\vert_R} \ar[d]^{\Tau_{X,s}\vert_R}
\ar[r]_(0.35){\om_{R,U,f,i}} &
*+[l]{i^*\bigl(\PV_{U,f}^\bu\bigr)\ot_{\Z/2\Z} Q_{R,U,f,i}}
\ar[d]_{i^*(\tau_{U,f})\ot\id_{Q_{R,U,f,i}}} \\
*+[r]{P_{X,s}^\bu\vert_R} \ar[r]^(0.35){\om_{R,U,f,i}} &
*+[l]{i^*\bigl(\PV_{U,f}^\bu\bigr)\ot_{\Z/2\Z} Q_{R,U,f,i}.} }
\end{gathered}
\label{sm6eq8}
\ea

\item[{\bf(ii)}] Let\/ $\Phi:(R,U,f,i)\hookra(S,V,g,j)$ be an
embedding of critical charts on $(X,s)$. Then there is a natural
isomorphism of principal\/ $\Z/2\Z$-bundles
\e
\La_\Phi:Q_{S,V,g,j}\vert_R\,{\buildrel\cong\over\longra}\,
i^*(P_\Phi)\ot_{\Z/2\Z} Q_{R,U,f,i}
\label{sm6eq9}
\e
on $R,$ for $P_\Phi$ as in Definition\/ {\rm\ref{sm5def},}
defined as follows: local isomorphisms
\begin{gather*}
\al:K_{X,s}^{1/2}\vert_{R^\red}\longra
i^*(K_U)\vert_{R^\red},\quad
\be:K_{X,s}^{1/2}\vert_{R^\red}\longra j^*(K_V)\vert_{R^\red},\\
\text{and\/}\qquad\ga:i^*(K_U)\vert_{R^\red}\longra
j^*(K_V)\vert_{R^\red}
\end{gather*}
with $\al\ot\al=\io_{R,U,f,i},$
$\be\ot\be=\io_{S,V,g,j}\vert_{R^\red},$
$\ga\ot\ga=i\vert_{R^\red}^*(J_\Phi)$ correspond to local
sections $s_\al:R\ra Q_{R,U,f,i},$ $s_\be:R\ra
Q_{S,V,g,j}\vert_R,$ $s_\ga:R\ra i^*(P_\Phi)$. Equation
\eq{sm6eq4} shows that\/ $\be=\ga\ci\al$ is a possible solution
for $\be,$ and we define $\La_\Phi$ in \eq{sm6eq9} such that\/
$\La_\Phi(s_\be)=s_\ga\ot_{\Z/2\Z}s_\al$ if and only
if\/~$\be=\ga\ci\al$.

Then the following diagram commutes in $\Perv(R),$ for
$\Th_\Phi$ as in\/~{\rm\eq{sm5eq13}:}
\e
\begin{gathered}
\xymatrix@C=140pt@R=20pt{ *+[r]{P_{X,s}^\bu\vert_R}
\ar[r]_(0.38){\om_{R,U,f,i}} \ar[d]^(0.4){\om_{S,V,g,j}\vert_R}
& *+[l]{i^*\bigl(\PV_{U,f}^\bu\bigr)\ot_{\Z/2\Z} Q_{R,U,f,i}}
\ar[d]_(0.4){i^*(\Th_\Phi)\ot\id_{Q_{R,U,f,i}}} \\
*+[r]{\begin{subarray}{l}\ts j^*\bigl(\PV_{V,g}^\bu\bigr)\vert_R\\
\ts \ot_{\Z/2\Z}Q_{S,V,g,j}\vert_R\end{subarray}}
\ar[r]^(0.38){\id_{j^*(\PV_{V,g}^\bu)}\ot\La_\Phi} &
*+[l]{\begin{subarray}{l}\ts i^*\bigl(\Phi^*(\PV_{V,g}^\bu)
\ot_{\Z/2\Z}P_\Phi\bigr)\\
\ts\ot_{\Z/2\Z} Q_{R,U,f,i}.\end{subarray}} }
\end{gathered}
\label{sm6eq10}
\e
\end{itemize}

The analogues of all the above also hold for $\cD$-modules on
oriented algebraic d-critical loci over\/ $\C,$ for perverse sheaves
and\/ $\cD$-modules on oriented complex analytic d-critical loci,
and for mixed Hodge modules
on oriented algebraic d-critical loci over\/ $\C$ and oriented
complex analytic d-critical loci, as
in\/~{\rm\S\ref{sm26}--\S\ref{sm210}}.

\label{sm6thm6}
\end{thm}

\begin{rem} This sheaf-theoretic result is compatible with the motivic result of
Bussi, Joyce and Meinhardt in \cite{BJM}. Given $(X,s)$ an oriented algebraic d-critical locus
over $\C$, \cite{BJM} proves the existence of a natural 
motivic element $MF_{X,s}\in \overline{\mathcal M}_X^{\hat\mu}$ in a version of the 
relative Grothendieck ring of varieties over~$X$, equivariant with respect to suitable
actions of the group ${\hat\mu}$ of all roots of unity (for detailed definitions,
see \cite{BJM}). Since the mixed Hodge module realization factorizes over the additional relation
one has to impose in \cite{BJM} on the Grothendieck group, 
the ring ${\mathcal M}_X^{\hat\mu}$ has a map to $K_0(\MHM(X; T_s))$, the $K$-group 
of algebraic mixed Hodge modules on~$X$ with a finite order automorphism
(note that the Grothendieck group only sees the semisimple part
$T_s$ of the monodromy and not the nilpotent part $N$).
By a \v Cech-type argument using the corresponding comparison
result of \cite[Prop.~3.17]{GLM}, the image of $MF_{X,s}$ in 
$K_0(\MHM(X; T_s))$ agrees with the image of the mixed Hodge module realization of
$P_{X,s}^\bu$, since both sides are Zariski locally modelled by the 
same vanishing cycles. Thus, for example, they give the same
weight polynomial for global cohomology with compact support.
\label{sm6rem1}
\end{rem}

From Theorem \ref{sm6thm4}, Corollary \ref{sm6cor1} and Theorem
\ref{sm6thm5} we deduce:

\begin{cor} Let\/ $(\bs X,\om)$ be a $-1$-shifted symplectic
derived scheme over\/ $\C$ in the sense of Pantev et al.\
{\rm\cite{PTVV},} and\/ $X=t_0(\bs X)$ the associated classical\/
$\C$-scheme. Suppose we are given a square root\/
$\smash{\det(\bL_{\bs X})\vert_X^{1/2}}$ for $\det(\bL_{\bs
X})\vert_X$. Then we may define $P_{\bs X,\om}^\bu\in\Perv(X),$
uniquely up to canonical isomorphism, and isomorphisms $\Si_{\bs
X,\om}:P_{\bs X,\om}^\bu\ra \bD_X(P_{\bs X,\om}^\bu),$ $\Tau_{\bs
X,\om}:P_{\bs X,\om}^\bu\ra P_{\bs X,\om}^\bu$.

The same applies for $\cD$-modules and mixed Hodge modules on $X$.
\label{sm6cor2}
\end{cor}

\begin{cor} Let\/ $Y$ be a Calabi--Yau\/ $3$-fold over\/ $\C,$
and\/ $\cM$ a classical moduli\/ $\C$-scheme of simple coherent
sheaves in $\coh(Y),$ or simple complexes of coherent sheaves in
$D^b\coh(Y),$ with natural (symmetric) obstruction theory\/
$\phi:\cE^\bu\ra\bL_\cM$ as in Behrend\/ {\rm\cite{Behr},} Thomas\/
{\rm\cite{Thom},} or Huybrechts and Thomas\/ {\rm\cite{HuTh}}.
Suppose we are given a square root\/ $\det(\cE^\bu)^{1/2}$ for
$\det(\cE^\bu)$. Then we may define $P_\cM^\bu\in\Perv(\cM),$
uniquely up to canonical isomorphism, and isomorphisms
$\Si_\cM:P_\cM^\bu\ra \bD_\cM(P_\cM^\bu),$ $\Tau_\cM:P_\cM^\bu\ra
P_\cM^\bu$.

The same applies for $\cD$-modules and mixed Hodge modules on $\cM$.
\label{sm6cor3}
\end{cor}

\begin{cor} Let\/ $(S,\om)$ be a complex symplectic manifold and\/
$L,M$ complex Lagrangian submanifolds in $S,$ and write $X=L\cap M,$
as a complex analytic subspace of\/ $S$. Suppose we are given square
roots\/ $K_L^{1/2},K_M^{1/2}$ for $K_L,K_M$. Then we may define
$P_{L,M}^\bu\in\Perv(X),$ uniquely up to canonical isomorphism, and
isomorphisms $\Si_{L,M}:P_{L,M}^\bu\ra \bD_X(P_{L,M}^\bu),$
$\Tau_{L,M}:P_{L,M}^\bu\ra P_{L,M}^\bu$.

The same applies for $\cD$-modules and mixed Hodge modules on\/~$X$.
\label{sm6cor4}
\end{cor}

The next two remarks discuss applications of Corollaries
\ref{sm6cor3} and \ref{sm6cor4} to Donaldson--Thomas theory, and to
Lagrangian Floer cohomology.

\begin{rem} If $Y$ is a Calabi--Yau 3-fold over $\C$ and
$\tau$ a suitable stability condition on coherent sheaves on $Y$,
the {\it Donaldson--Thomas invariants\/} $DT^\al(\tau)$ are integers
which `count' the moduli schemes $\cM_{\rm st}^\al(\tau)$ of
$\tau$-stable coherent sheaves on $Y$ with Chern character $\al\in
H^{\rm even}(Y;\Q)$, provided there are no strictly
$\tau$-semistable sheaves in class $\al$ on $Y$. They were defined
by Thomas \cite{Thom}, who showed they are unchanged under
deformations of $Y$, following a suggestion of Donaldson and
Thomas~\cite{DoTh}.

Behrend \cite{Behr} showed that $DT^\al(\tau)$ may be written as a
weighted Euler characteristic $\chi(\cM_{\rm st}^\al(\tau),\nu)$,
where $\nu:\cM_{\rm st}^\al(\tau)\ra\Z$ is a certain constructible
function called the {\it Behrend function}. Joyce and Song
\cite{JoSo} extended the definition of $DT^\al(\tau)$ to classes
$\al$ including $\tau$-semistable sheaves (with
$DT^\al(\tau)\in\Q$), and proved a wall-crossing formula for
$DT^\al(\tau)$ under change of stability condition $\tau$.
Kontsevich and Soibelman \cite{KoSo1} gave a (partly conjectural)
motivic generalization of Donaldson--Thomas invariants, also with a
wall-crossing formula.

Corollary \ref{sm6cor3} is relevant to the {\it categorification\/}
of Donaldson--Thomas theory. As in \cite[\S 1.2]{Behr}, the perverse
sheaf $P_{\cM_{\rm st}^\al(\tau)}^\bu$ has pointwise Euler
characteristic $\chi\bigl(P_{\cM_{\rm st}^\al(\tau)}^\bu\bigr)=\nu$.
This implies that when $A$ is a field, say $A=\Q$, the
(compactly-supported) hypercohomologies $\bH^*\bigl(P_{\cM_{\rm
st}^\al(\tau)}^\bu\bigr), \bH^*_\compact\bigl(P_{\cM_{\rm
st}^\al(\tau)}^\bu\bigr)$ from \eq{sm2eq1} satisfy
\begin{align*}
\ts\sum\limits_{k\in\Z}(-1)^k\dim \bH^k\bigl(P_{\cM_{\rm
st}^\al(\tau)}^\bu\bigr)&= \ts\sum\limits_{k\in\Z}(-1)^k\dim
\bH^k_\compact\bigl(P_{\cM_{\rm st}^\al(\tau)}^\bu\bigr)\\
&=\chi\bigl(\cM_{\rm st}^\al(\tau),\nu\bigr)=DT^\al(\tau),
\end{align*}
where $\bH^k\bigl(P_{\cM_{\rm st}^\al(\tau)}^\bu\bigr) \cong
\bH^{-k}_\compact\bigl(P_{\cM_{\rm st}^\al(\tau)}^\bu\bigr){}^*$ by
Verdier duality. That is, we have produced a natural graded
$\Q$-vector space $\bH^*\bigl(P_{\cM_{\rm st}^\al(\tau)}^\bu\bigr)$,
thought of as some kind of generalized cohomology of $\cM_{\rm
st}^\al(\tau)$, whose graded dimension is $DT^\al(\tau)$. This gives
a new interpretation of the Donaldson--Thomas
invariant~$DT^\al(\tau)$.

In fact, as discussed at length in \cite[\S 3]{Szen}, the first
natural ``refinement'' or ``quantization'' direction of a
Donaldson--Thomas invariant $DT^\al(\tau)\in\Z$ is not the
Poincar\'e polynomial of this cohomology, but its weight polynomial
\begin{equation*}
w\bigl(\bH^*(P_{\cM_{\rm st}^\al(\tau)}^\bu), t\bigr)
\in\Z\bigl[t^{\pm\frac{1}{2}}\bigr],
\end{equation*}
defined using the mixed Hodge structure on the cohomology of the
mixed Hodge module version of $P_{\cM_{\rm st}^\al(\tau)}^\bu$
(which exists assuming that $\cM_{\rm st}^\al(\tau)$ is projective,
for example, see Remark~\ref{sm2rem5}).

The material above is related to work by other authors. The idea of
categorifying Donaldson--Thomas invariants using perverse sheaves or
$\cD$-modules is probably first due to Behrend \cite{Behr}, and for
Hilbert schemes $\mathop{\rm Hilb}^n(Y)$ of a Calabi--Yau 3-fold $Y$
is discussed by Dimca and Szendr\H oi \cite{DiSz} and Behrend, Bryan
and Szendr\H oi \cite[\S 3.4]{BBS}, using mixed Hodge modules.
Corollary \ref{sm6cor3} answers a question of Joyce and
Song~\cite[Question~5.7(a)]{JoSo}.

As in \cite{JoSo,KoSo1} representations of {\it quivers with
superpotentials\/} $(Q,W)$ give 3-Calabi--Yau triangulated
categories, and one can define Donaldson--Thomas type invariants
$DT^\al_{Q,W}(\tau)$ `counting' such representations, which are
simple algebraic `toy models' for Donaldson--Thomas invariants of
Calabi--Yau 3-folds. Kontsevich and Soibelman \cite{KoSo2} explain
how to categorify these quiver invariants $DT^\al_{Q,W}(\tau)$, and
define an associative multiplication on the categorification to make
a {\it Cohomological Hall Algebra}. This paper was motivated by the
aim of extending \cite{KoSo2} to define Cohomological Hall Algebras
for Calabi--Yau 3-folds.

The square root $\det(\cE^\bu)^{1/2}$ required in Corollary
\ref{sm6cor3} corresponds roughly to {\it orientation data\/} in the
work of Kontsevich and Soibelman \cite[\S 5]{KoSo1}, \cite{KoSo2}.

In a paper written independently of our programme
\cite{BBJ,BJM,Joyc}, Kiem and Li \cite{KiLi} have recently proved an
analogue of Corollary \ref{sm6cor3} by complex analytic methods,
beginning from Joyce and Song's result \cite[Th.~5.4]{JoSo}, proved
using gauge theory, that $\cM_{\rm st}^\al(\tau)$ is locally
isomorphic to $\Crit(f)$ as a complex analytic space, for $V$ a
complex manifold and $f:V\ra\C$ holomorphic.
\label{sm6rem2}
\end{rem}

\begin{rem} In the situation of Corollary \ref{sm6cor4}, with
$\dim_\C S=2n$, we claim that there ought morally to be some kind of
approximate comparison
\e
\bH^k(P_{L,M}^\bu)\approx HF^{k+n}(L,M),
\label{sm6eq11}
\e
where $HF^*(L,M)$ is the {\it Lagrangian Floer cohomology\/} of
Fukaya, Oh, Ohta and Ono \cite{FOOO}. We can compare and contrast
the two sides of \eq{sm6eq11} as follows:
\begin{itemize}
\setlength{\itemsep}{0pt}
\setlength{\parsep}{0pt}
\item[(a)] $\bH^*(P_{L,M}^\bu)$ is defined over any
well-behaved base ring $A$, e.g. $A=\Z$ or $\Q$, but $HF^*(L,M)$
is defined over a Novikov ring of power series $\La_{\rm nov}$.
\item[(b)] $\bH^*(P_{L,M}^\bu)$ has extra structure not visible
in $HF^*(L,M)$, from Verdier duality and monodromy operators
$\Si_{L,M},\Tau_{L,M}$, plus the mixed Hodge module version has
a mixed Hodge structure.
\item[(c)] $\bH^*(P_{L,M}^\bu)$ is defined for arbitrary complex
Lagrangians $L,M$, not necessarily compact or closed in $S$, but
$HF^*(L,M)$ is only defined for $L,M$ compact, or at least for
$L,M$ closed and well-behaved at infinity.
\item[(d)] To define $HF^*(L,M)$ one generally assumes $L,M$
intersect transversely, or at least cleanly. But
$\bH^*(P_{L,M}^\bu)$ is defined when $L\cap M$ is arbitrarily
singular, and the construction is only really interesting for
singular $L\cap M$.
\item[(e)] To define $HF^*(L,M)$ we need $L,M$ to be oriented and
spin, to orient moduli spaces of $J$-holomorphic curves. When
$L,M$ are complex Lagrangians they are automatically oriented,
and spin structures on $L,M$ correspond to choices of square
roots $K_L^{1/2},K_M^{1/2}$, as used in Corollary~\ref{sm6cor4}.
\end{itemize}

Some of the authors are working on defining a `Fukaya category' of
complex Lagrangians in a complex symplectic manifold, using
$\bH^*(P_{L,M}^\bu)$ as morphisms.

We now discuss related work. Nadler and Zaslow \cite{Nadl,NaZa} show
that if $X$ is a real analytic manifold (for instance, a complex
manifold), then the derived category $D^b_c(X)$ of constructible
sheaves on $X$ is equivalent to a certain derived Fukaya category
$D^b\cF(T^*X)$ of exact Lagrangians in~$T^*X$.

Let $L,M$ be complex Lagrangians in a complex symplectic manifold
$(S,\om)$. Regarding $\O_L,\O_M$ as coherent sheaves on $S$, Behrend
and Fantechi \cite[Th.s 4.3 \& 5.2]{BeFa} claim to construct
canonical $\C$-linear (not $\O_S$-linear) differentials
\begin{equation*}
\d:\cE xt^i_{\O_S}(\O_L,\O_M)\longra \cE xt^{i+1}_{\O_S}(\O_L,\O_M)
\end{equation*}
with $\d^2=0$, such that $\bigl(\cE xt^*_{\O_S}(\O_L,\O_M),\d\bigr)$
is a constructible complex. There is a mistake in the proof of
\cite[Th.~4.3]{BeFa}. To fix this one should instead work with $\cE
xt^*_{\O_S}(K_L^{1/2},K_M^{1/2})$ for square roots
$\smash{K_L^{1/2},K_M^{1/2}}$ as in Corollary \ref{sm6cor4}. Also
the proof of the constructibility of $\bigl(\cE xt^*_{\O_S}
(K_L^{1/2},K_M^{1/2}),\d\bigr)$ in \cite[Th.~5.2]{BeFa} depended on
a result of Kapranov, which later turned out to be false.

Our $P_{L,M}^\bu$ over $A=\C$ should be the natural perverse sheaf
on $L\cap M$ conjectured by Behrend and Fantechi
\cite[Conj.~5.16]{BeFa}, who also suggest there should be a spectral
sequence from $\bigl(\cE xt^*_{\O_S}(K_L^{1/2},K_M^{1/2}),
\d\bigr)[n]$ to $P_{L,M}^\bu$. (See Sabbah \cite[Th.~1.1]{Sabb} for
a related result.) In \cite[\S 5.3]{BeFa}, Behrend and Fantechi
discuss how to define a `Fukaya category' using their ideas.

Kashiwara and Schapira \cite{KaSc3} develop a theory of {\it
deformation quantization modules}, or {\it DQ-modules}, on a complex
symplectic manifold $(S,\om)$, which roughly may be regarded as
symplectic versions of $\cD$-modules. {\it Holonomic\/} DQ-modules
${\cal D}^\bu$ are supported on (possibly singular) complex
Lagrangians $L$ in $S$. If $L$ is a smooth, closed, complex
Lagrangian in $S$ and $K_L^{1/2}$ a square root of $K_L$, D'Agnolo
and Schapira \cite{DASc} show that there exists a simple holonomic
DQ-module ${\cal D}^\bu$ supported on~$L$.

If ${\cal D}^\bu,\cE^\bu$ are simple holonomic DQ-modules on $S$
supported on smooth Lagrangians $L,M$, then Kashiwara and Schapira
\cite{KaSc2} show that $R{\scr H}om({\cal D}^\bu,\cE^\bu)[n]$ is a
perverse sheaf on $S$ over the field $\C((\hbar))$, supported on
$X=L\cap M$. Pierre Schapira explained to the authors how to prove
that $R{\scr H}om({\cal D}^\bu,\cE^\bu)[n]\cong P_{L,M}^\bu$, when
$P_{L,M}^\bu$ is defined over the base ring~$A=\C((\hbar))$.

Now let $L,M,N$ be Lagrangians in $S$, with square roots
$K_L^{1/2},K_M^{1/2},K_N^{1/2}$. We have a product $HF^k(L,M)\t
HF^l(M,N)\ra HF^{k+l}(L,N)$ from composition of morphisms in
$D^b\cF(S)$. So \eq{sm6eq11} suggests there should be a
product
\e
\bH^k(P_{L,M}^\bu)\t \bH^l(P_{M,N}^\bu)\longra
\bH^{k+l+n}(P_{L,N}^\bu),
\label{sm6eq12}
\e
which would naturally be induced by a morphism in $D^b_c(S)$
\e
\mu_{L,M,N}:P_{L,M}^\bu\otL P_{M,N}^\bu \longra P_{L,N}^\bu[n].
\label{sm6eq13}
\e

Observe that the work of Behrend--Fantechi and Kashiwara--Schapira
cited above supports the existence of \eq{sm6eq12}--\eq{sm6eq13}:
there are natural products
\begin{align*}
\cE xt^k_{\O_S}( K_L^{1/2},K_M^{1/2}) \ot_{\O_S}
\cE xt^l_{\O_S}( K_M^{1/2},K_N^{1/2}) &\longra\cE
xt^{k+l}_{\O_S}( K_L^{1/2},K_N^{1/2}),\\
R{\scr H}om({\cal D}^\bu,\cE^\bu)\otL R{\scr
H}om(\cE^\bu,\cF^\bu)&\longra R{\scr H}om({\cal D}^\bu,\cF^\bu).
\end{align*}
But since \eq{sm6eq13} is a morphism of complexes, not of perverse
sheaves, Theorem \ref{sm2thm3}(i) does not apply, so we cannot
construct $\mu_{L,M,N}$ by na\"\i vely gluing data on an open cover,
as we have been doing in~\S\ref{sm3}--\S\ref{sm6}.
\label{sm6rem3}
\end{rem}

\subsection[Proof of Theorem \ref{sm6thm6} for perverse sheaves on
$\C$-schemes]{Proof of Theorem \ref{sm6thm6} for $\C$-schemes}
\label{sm63}

Let $(X,s)$ be an oriented algebraic d-critical locus over $\C,$
with orientation $K_{X,s}^{1/2}$. By Definition \ref{sm6def1} we may
choose a family $\bigl\{(R_a,U_a,f_a,i_a):a\in A\bigr\}$ of critical
charts $(R_a,U_a,f_a,i_a)$ on $(X,s)$ such that $\{R_a:a\in A\}$ is
a Zariski open cover of the $\C$-scheme $X$. Then for each $a\in A$
we have a perverse sheaf
\e
i_a^*\bigl(\PV_{U_a,f_a}^\bu\bigr)\ot_{\Z/2\Z}
Q_{R_a,U_a,f_a,i_a}\in\Perv(R_a),
\label{sm6eq14}
\e
for $Q_{R_a,U_a,f_a,i_a}$ as in Theorem \ref{sm6thm6}(i). The idea
of the proof is to use Theorem \ref{sm2thm3}(ii) to glue the
perverse sheaves \eq{sm6eq14} on the Zariski open cover $\{R_a:a\in
A\}$ to get a global perverse sheaf $P_{X,s}^\bu$ on $X$. Note that
Theorem \ref{sm2thm3}(ii) is written for \'etale open covers, but
this immediately implies the simpler Zariski version.

To do this, for all $a,b\in A$ we have to construct isomorphisms
\e
\begin{split}
\al_{ab}:\bigl[i_a^*\bigl(\PV_{U_a,f_a}^\bu\bigr)\ot_{\Z/2\Z}
Q_{R_a,U_a,f_a,i_a}\bigr]&\big\vert_{R_a\cap R_b}\longra\\
\bigl[i_b^*\bigl(\PV_{U_b,f_b}^\bu\bigr)\ot_{\Z/2\Z}
Q_{R_b,U_b,f_b,i_b}\bigr]&\big\vert_{R_a\cap R_b} \in\Perv(R_a\cap
R_b),
\end{split}
\label{sm6eq15}
\e
satisfying $\al_{aa}=\id$ for all $a\in A$ and
\e
\al_{bc}\vert_{R_a\cap R_b\cap R_c}\ci \al_{ab}\vert_{R_a\cap
R_b\cap R_c}=\al_{ac}\vert_{R_a\cap R_b\cap R_c}\quad\text{for all
$a,b,c\in A$.}
\label{sm6eq16}
\e

Fix $a,b\in A$. By applying Theorem \ref{sm6thm2} to the critical
charts $(R_a,U_a,f_a,i_a)$, $(R_b,U_b,f_b,i_b)$ at each $x\in
R_a\cap R_b$, we can choose an indexing set $D_{ab}$ and for each
$d\in D_{ab}$ subcharts $(R_a^{\prime d},U_a^{\prime d},f_a^{\prime
d},i_a^{\prime d})\subseteq (R_a,U_a,f_a,i_a)$, $(R_b^{\prime
d},U_b^{\prime d},f_b^{\prime d},i_b^{\prime d})\subseteq
(R_b,U_b,f_b,i_b)$, a critical chart $(S^d,V^d,g^d,j^d)$ on $(X,s)$,
and embeddings $\Phi^d:(R_a^{\prime d},U_a^{\prime d},f_a^{\prime
d},i_a^{\prime d})\hookra(S^d,V^d,g^d,j^d)$, $\Psi^d:(R_b^{\prime
d},U_b^{\prime d},f_b^{\prime d},i_b^{\prime
d})\hookra(S^d,V^d,g^d,j^d)$, such that $\bigl\{R_a^{\prime d}\cap
R_b^{\prime d}:d\in D_{ab}\bigr\}$ is a Zariski open cover of
$R_a\cap R_b$.

For each $d\in D_{ab}$, define an isomorphism
\begin{align*}
\al_{ab}^d:&\bigl[i_a^*\bigl(\PV_{U_a,f_a}^\bu\bigr)\ot_{\Z/2\Z}
Q_{R_a,U_a,f_a,i_a}\bigr]\big\vert_{R_a^{\prime d}\cap R_b^{\prime
d}}\longra\\
&\qquad\qquad\bigl[i_b^*\bigl(\PV_{U_b,f_b}^\bu\bigr)\ot_{\Z/2\Z}
Q_{R_b,U_b,f_b,i_b}\bigr]\big\vert_{R_a^{\prime d}\cap R_b^{\prime
d}}
\end{align*}
by the commutative diagram
\e
\begin{gathered}
\xymatrix@!0@C=275pt@R=50pt{
*+[r]{\begin{subarray}{l}\ts\bigl[i_a^*\bigl(\PV_{U_a,f_a}^\bu\bigr)
\ot_{\Z/2\Z} \\ \ts Q_{R_a,U_a,f_a,i_a}\bigr]\big\vert_{R_a^{\prime
d}\cap R_b^{\prime d}}\end{subarray}} \ar[dd]^{\al_{ab}^d}
\ar[r]_(0.41){\raisebox{-17pt}{$\st\begin{subarray}{l}
i_a\vert_{R_a^{\prime d}\cap R_b^{\prime
d}}^*(\Th_{\Phi^d})\\
\ot\id_{Q_{R_a,U_a,f_a,i_a}}\end{subarray}$}} &
*+[l]{\begin{subarray}{l}\ts
j^d\vert_{R_a^{\prime d}\cap R_b^{\prime d}}^*
\bigl(\PV_{V^d,g^d}^\bu\bigr)\!\ot_{\Z/2\Z}\!i_a\vert_{R_a^{\prime
d}\cap R_b^{\prime d}}^*\bigl(P_{\Phi^d}\bigr)\\
\ts {}\qquad\qquad\qquad\quad \ot_{\Z/2\Z}
Q_{R_a,U_a,f_a,i_a}\vert_{R_a^{\prime d}\cap R_b^{\prime
d}}\end{subarray}} \ar[d]_(0.55){\id\ot
\La_{\Phi^d}\vert_{R_a^{\prime d}\cap R_b^{\prime d}}^{-1}}
\\
& *+[l]{\bigl[(j^d)^*\bigl(\PV_{V^d,g^d}^\bu\bigr)
\ot_{\Z/2\Z}Q_{S^d,V^d,g^d,j^d}\bigr]\vert_{R_a^{\prime d}\cap
R_b^{\prime d}}} \ar[d]_(0.45){\id\ot \La_{\Psi^d}\vert_{R_a^{\prime
d}\cap R_b^{\prime d}}}
\\
*+[r]{\begin{subarray}{l}\ts\bigl[i_b^*\bigl(\PV_{U_b,f_b}^\bu\bigr)
\ot_{\Z/2\Z} \\ \ts Q_{R_b,U_b,f_b,i_b}\bigr]\big\vert_{R_a^{\prime
d}\cap R_b^{\prime d}}\end{subarray}} &
*+[l]{\begin{subarray}{l}\ts {}\qquad\qquad\qquad
j^d\vert_{R_a^{\prime d}\cap R_b^{\prime d}}^*
\bigl(\PV_{V^d,g^d}^\bu\bigr) \ot_{\Z/2\Z}\\
\ts i_b\vert_{R_a^{\prime d}\cap R_b^{\prime
d}}^*\bigl(P_{\Psi^d}\bigr)\!\ot_{\Z/2\Z}\!
Q_{R_b,U_b,f_b,i_b}\vert_{R_a^{\prime d}\cap R_b^{\prime
d}},\end{subarray}} \ar[l]_(0.6){\begin{subarray}{l}
i_b\vert_{R_a^{\prime d}\cap R_b^{\prime
d}}^*(\Th_{\Psi^d}^{-1})\\
\ot\id_{Q_{R_b,U_b,f_b,i_b}}\end{subarray}}  }\!\!\!{}
\end{gathered}
\label{sm6eq17}
\e
where $\Th_{\Phi^d},\Th_{\Psi^d}$ are as in Theorem \ref{sm5thm2},
and $\La_{\Phi^d},\La_{\Psi^d}$ as in \eq{sm6eq9}.

We claim that for all $d,e\in D_{ab}$ we have
\e
\al_{ab}^d\vert_{R_a^{\prime d}\cap R_b^{\prime d}\cap R_a^{\prime
e}\cap R_b^{\prime e}}=\al_{ab}^e\vert_{R_a^{\prime d}\cap
R_b^{\prime d}\cap R_a^{\prime e}\cap R_b^{\prime e}}.
\label{sm6eq18}
\e
To see this, let $x\in R_a^{\prime d}\cap R_b^{\prime d}\cap
R_a^{\prime e}\cap R_b^{\prime e}$, and apply Theorem \ref{sm6thm2}
to the critical charts $(S^d,V^d,g^d,j^d)$, $(S^e,V^e,g^e,j^e)$ and
point $x\in S^d\cap S^e$. This gives subcharts $(S^{\prime
d},V^{\prime d},g^{\prime d},j^{\prime d})\subseteq
(S^d,V^d,g^d,j^d)$, $(S^{\prime e},V^{\prime e},g^{\prime
e},j^{\prime e})\subseteq (S^e,V^e,g^e,j^e)$ with $x\in S^{\prime
d}\cap S^{\prime e}$, a critical chart $(T,W,h,k)$ on $(X,s)$, and
embeddings $\Om:(S^{\prime d},V^{\prime d},g^{\prime d},j^{\prime
d})\hookra(T,W,h,k)$, $\Up:(S^{\prime e},V^{\prime e},g^{\prime
e},j^{\prime e})\hookra(T,W,h,k)$.

Set $R^{de}=R_a^{\prime d}\cap R_b^{\prime d}\cap R_a^{\prime e}\cap
R_b^{\prime e}\cap S^{\prime d}\cap S^{\prime e}$, and consider the
diagram:
\e
\begin{gathered}
\text{\begin{small}$\displaystyle\xymatrix@!0@C=140pt@R=65pt{
*+[r]{\begin{subarray}{l}\ts\bigl[i_a^*\bigl(\PV_{U_a,f_a}^\bu\bigr)
\ot_{\Z/2\Z} \\ \ts{}\quad
Q_{R_a,U_a,f_a,i_a}\bigr]\big\vert_{R^{de}}\end{subarray}}
\ar[dr]^(0.6){\begin{subarray}{l} (\id\ci\La_{\Om\ci\Phi^d}^{-1})\ci  \\
(i_a^*(\Th_{\Om\ci\Phi^d})\ot\id)\vert_{R^{de}}\end{subarray}
}_(0.5){\begin{subarray}{l} =(\id\ci\La_{\Up\ci\Phi^e}^{-1})\ci  \\
(i_a^*(\Th_{\Up\ci\Phi^e})\ot\id)\vert_{R^{de}}\end{subarray}}
\ar@<1ex>[rr]_{(\id\ot \La_{\Phi^d}^{-1})\ci (i_a^*(\Th_{\Phi^d})
\ot\id)\vert_{R^{de}}}
\ar@<-1ex>[dd]^(0.6){\begin{subarray}{l} (\id\ot \La_{\Phi^e}^{-1})\ci\\
(i_a^*(\Th_{\Phi^e}) \ot\id)\vert_{R^{de}}\end{subarray}} &&
*+[l]{\begin{subarray}{l}\ts\bigl[(j^d)^*\bigl(\PV_{V^d,g^d}^\bu\bigr)
\ot_{\Z/2\Z}\\ \ts\quad Q_{S^d,V^d,g^d,j^d}\bigr]\vert_{R^{de}}
\end{subarray}}
\ar@<1ex>[dd]_(0.4){\begin{subarray}{l}
(i_b^*(\Th_{\Psi^d}^{-1})\ot\id)\ci \\
(\id\ot\La_{\Psi^d})\vert_{R^{de}}
\end{subarray} }
\ar[dl]_(0.6){\begin{subarray}{l} (\id\ot \La_\Om^{-1})\ci\\
((j^d)^*(\Th_\Om) \ot\id)\vert_{R^{de}}\end{subarray}} \\
&
{\begin{subarray}{l}\ts\bigl[k^*\bigl(\PV_{W,h}^\bu\bigr)\ot_{\Z/2\Z}\\
\ts\quad Q_{T,W,h,k}\bigr]\vert_{R^{de}}
\end{subarray}}
\ar[dr]^{\begin{subarray}{l}
(i_b^*(\Th_{\Om\ci\Psi^d}^{-1})\ot\id)\ci \\
(\id\ot\La_{\Om\ci\Psi^d})\vert_{R^{de}}
\end{subarray} }_(0.4){\begin{subarray}{l}
=(i_b^*(\Th_{\Up\ci\Psi^e}^{-1})\ot\id)\ci \\
(\id\ot\La_{\Up\ci\Psi^e})\vert_{R^{de}}
\end{subarray}} \\
*+[r]{\begin{subarray}{l}\ts\bigl[(j^e)^*\bigl(\PV_{V^e,g^e}^\bu\bigr)
\ot_{\Z/2\Z}\\ \ts\quad Q_{S^e,V^e,g^e,j^e}\bigr]\vert_{R^{de}}
\end{subarray}} \ar@<-1ex>[rr]^{
(i_b^*(\Th_{\Psi^e}^{-1})\ot\id)\ci
(\id\ot\La_{\Psi^e})\vert_{R^{de}}}
\ar[ur]_(0.6){\begin{subarray}{l} (\id\ot \La_\Up^{-1})\ci\\
((j^e)^*(\Th_\Up) \ot\id)\vert_{R^{de}}\end{subarray}} &&
*+[l]{\begin{subarray}{l}\ts\bigl[i_b^*\bigl(
\PV_{U_b,f_b}^\bu\bigr) \ot_{\Z/2\Z} \\ \ts\quad
Q_{R_b,U_b,f_b,i_b}\bigr]\big\vert_{R^{de}}.
\end{subarray}}}$\end{small}}\!\!\!\!{}
\end{gathered}
\label{sm6eq19}
\e

Here we have given two expressions for the top left diagonal
morphism in \eq{sm6eq19}. To see these are equal, set $R_a^{\prime
de}= R_a^{\prime d}\cap R_a^{\prime e} \cap S^{\prime d}\cap
S^{\prime e}$, $U_a^{\prime de}=(\Phi^d)^{-1}(V^{\prime d})\cap
(\Phi^e)^{-1}(V^{\prime e})$, $f_a^{\prime de}=f_a\vert_{U_a^{\prime
de}}$, and $i_a^{\prime de}=i_a\vert_{R_a^{\prime de}}$. Then
$(R_a^{\prime de},\ab U_a^{\prime de},\ab f_a^{\prime de},\ab
i_a^{\prime de})\subseteq (R_a,U_a,f_a,i_a)$ is a subchart and
\begin{equation*}
\Om\ci\Phi^d\vert_{U_a^{\prime de}},\Up\ci\Phi^e
\vert_{U_a^{\prime de}}:(R_a^{\prime de},U_a^{\prime de},
f_a^{\prime de},i_a^{\prime de})\longra(T,W,h,k)
\end{equation*}
are embeddings. As $\Om\ci\Phi^d\ci i_a^{\prime
de}=k\vert_{R_a^{\prime de}}=\Up\ci\Phi^e\ci i_a^{\prime de}$,
Theorem \ref{sm5thm2}(b) gives $\Th_{\Om\ci\Phi^d}
\vert_{i_a(R_a^{\prime de})}=\Th_{\Up\ci\Phi^e}
\vert_{i_a(R_a^{\prime de})}$, so that
$i_a\vert_{R^{de}}^*(\Th_{\Om\ci\Phi^d})=i_a\vert_{R^{de}}^*
(\Th_{\Up\ci\Phi^e})$ as $R^{de}\subseteq R_a^{\prime de}$. Also
$\La_{\Om\ci\Phi^d}=\La_{\Up\ci\Phi^e}$ as these are defined in
Theorem \ref{sm6thm6}(ii) using $J_{\Om\ci\Phi^d},J_{\Up\ci\Phi^e}$,
which are equal by Lemma \ref{sm5lem}. So the two expressions are
equal, and similarly for the bottom right diagonal morphism.

The upper triangle in \eq{sm6eq19} commutes because \eq{sm5eq16}
gives
\begin{equation*}
(\id\ot \Xi_{\Om,\Phi^d})\ci\Th_{\Om\ci\Phi^d}\vert_{i_a(R^{de})}
=(\Phi^d\vert_{i_a(R^{de})}^*
(\Th_\Om)\ot\id)\ci\Th_{\Phi^d}\vert_{i_a(R^{de})},
\end{equation*}
and the definitions of $\Xi_{\Om,\Phi^d}$ in \eq{sm5eq11} and
$\La_\Om,\La_{\Phi^d},\La_{\Om\ci\Phi^d}$ in \eq{sm6eq9} imply that
\begin{align*}
&\bigl(i_a\vert_{R^{de}}^*(\Xi_{\Om,\Phi^d})\ot\id\bigr)\ci
\La_{\Om\ci\Phi^d}\vert_{R^{de}}=(\id\ot\La_{\Phi^d})\ci
\La_\Om\vert_{R^{de}}:\\
&Q_{T,W,h,k}\vert_{R^{de}}\longra j^d\vert_{R^{de}}^*(P_\Om)\ot_{\Z/2\Z}
i_a\vert_{R^{de}}^*(P_{\Phi^d})
\ot_{\Z/2\Z}Q_{R_a,U_a,f_a,i_a}\vert_{R^{de}}.
\end{align*}
Similarly, the other three triangles in \eq{sm6eq19} commute, so
\eq{sm6eq19} commutes.

By \eq{sm6eq17}, the two routes round the outside of \eq{sm6eq19}
are $\al_{ab}^d\vert_{R^{de}}$ and $\al_{ab}^e\vert_{R^{de}}$, which
are equal as \eq{sm6eq19} commutes. As we can cover $R_a^{\prime
d}\cap R_b^{\prime d}\cap R_a^{\prime e}\cap R_b^{\prime e}$ by such
Zariski open $R^{de}$, equation \eq{sm6eq18} follows. Therefore by
the Zariski open cover version of Theorem \ref{sm2thm3}(i), there is
a unique isomorphism $\al_{ab}$ in \eq{sm6eq15} such that
$\al_{ab}\vert_{R_a^{\prime d}\cap R_b^{\prime d}}=\al_{ab}^d$ for
all~$d\in D_{ab}$.

If $D_{ab},R_a^{\prime d},\ldots,\Phi^d,\Psi^d$ are used to define
$\al_{ab}$ and $\ti D_{ab},\ti R_a^{\prime d},\ldots,\ti
\Phi^d,\ti\Psi^d$ are alternative choices yielding $\ti\al_{ab}$,
then by our usual argument using $D_{ab}\amalg\ti D_{ab}$ and both
sets of data we see that $\al_{ab}=\ti\al_{ab}$, so $\al_{ab}$ is
independent of choices.

Because the $\Th_{\Phi^d},\Th_{\Psi^d}$ used to define $\al_{ab}$
are compatible with Verdier duality and monodromy by
\eq{sm5eq14}--\eq{sm5eq15}, and the $\La_{\Phi^d},\La_{\Psi^d}$
affect only the principal $\Z/2\Z$-bundles rather than the perverse
sheaves, we can show $\al_{ab}$ is compatible with Verdier duality
and monodromy, in that the following commute:
\ea
\begin{gathered}
{}\!\!\!\!\!\!\!\!\!\!\!\!\!\!\!\!\!\!\!\!\!\!\!
\xymatrix@!0@C=270pt@R=70pt{
*+[r]{\begin{subarray}{l}\ts \bigl[i_a^*\bigl(\PV_{U_a,f_a}^\bu\bigr)
\ot_{\Z/2\Z}\\
\ts Q_{R_a,U_a,f_a,i_a}\bigr]\big\vert_{R_a\cap R_b}\end{subarray}}
\ar@<-2ex>[d]^(0.39){i_a^*(\si_{U_a,f_a})\ot\id_{Q_{R_a,U_a,f_a,i_a}}
\vert_{R_a\cap R_b}} \ar@<2ex>[r]_{\al_{ab}} &
*+[l]{\begin{subarray}{l}\ts\bigl[i_b^*\bigl(\PV_{U_b,f_b}^\bu\bigr)
\ot_{\Z/2\Z}\\ \ts Q_{R_b,U_b,f_b,i_b}\bigr]\big\vert_{R_a\cap
R_b}\end{subarray}}
\ar@<2ex>[d]_(0.39){i_b^*(\si_{U_b,f_b})\ot\id_{Q_{R_b,U_b,f_b,i_b}}
\vert_{R_a\cap R_b}}
\\
*+[r]{{}\quad\,\,\,\begin{subarray}{l}
\ts \bigl[i_a^*(
\bD_{\Crit(f_a)}(\PV_{U_a,f_a}^\bu))\\
\ts \ot_{\Z/2\Z}Q_{R_a,U_a,f_a,i_a}\bigr]\big\vert_{R_a\cap R_b}\bigr)\\
\ts \cong\bD_{R_a\cap R_b}\bigl([i_a^*(
\PV_{U_a,f_a}^\bu)\\
\ts \ot_{\Z/2\Z}Q_{R_a,U_a,f_a,i_a}]\vert_{R_a\cap
R_b}\bigr)\end{subarray}} &
*+[l]{\begin{subarray}{l}
\ts \bigl[i_b^*(
\bD_{\Crit(f_b)}(\PV_{U_b,f_b}^\bu))\\
\ts \ot_{\Z/2\Z}Q_{R_b,U_b,f_b,i_b}\bigr]\big\vert_{R_a\cap R_b}\bigr)\\
\ts \cong\bD_{R_a\cap R_b}\bigl([i_b^*(
\PV_{U_b,f_b}^\bu)\\
\ts \ot_{\Z/2\Z}Q_{R_b,U_b,f_b,i_b}]\vert_{R_a\cap
R_b}\bigr),\end{subarray}\quad{}} \ar@<4ex>[l]_{\bD_{R_a\cap
R_b}(\al_{ab})} }\!\!\!\!\!\!{}
\end{gathered}
\label{sm6eq20}\\
\begin{gathered}
\xymatrix@!0@C=270pt@R=50pt{
*+[r]{\begin{subarray}{l}\ts \bigl[i_a^*\bigl(\PV_{U_a,f_a}^\bu\bigr)
\ot_{\Z/2\Z}\\
\ts Q_{R_a,U_a,f_a,i_a}\bigr]\big\vert_{R_a\cap R_b}\end{subarray}}
\ar@<-2ex>[d]^{i_a^*(\tau_{U_a,f_a})\ot\id_{Q_{R_a,U_a,f_a,i_a}}\vert_{R_a\cap
R_b}} \ar@<2ex>[r]_{\al_{ab}} &
*+[l]{\begin{subarray}{l}\ts\bigl[i_b^*\bigl(\PV_{U_b,f_b}^\bu\bigr)
\ot_{\Z/2\Z}\\ \ts Q_{R_b,U_b,f_b,i_b}\bigr]\big\vert_{R_a\cap
R_b}\end{subarray}}
\ar@<2ex>[d]_{i_b^*(\tau_{U_b,f_b})\ot\id_{Q_{R_b,U_b,f_b,i_b}}\vert_{R_a\cap
R_b}}
\\
*+[r]{\begin{subarray}{l}\ts \bigl[i_a^*\bigl(\PV_{U_a,f_a}^\bu\bigr)
\ot_{\Z/2\Z}\\
\ts Q_{R_a,U_a,f_a,i_a}\bigr]\big\vert_{R_a\cap R_b}\end{subarray}}
\ar@<-2ex>[r]^{\al_{ab}} &
*+[l]{\begin{subarray}{l}\ts\bigl[i_b^*\bigl(\PV_{U_b,f_b}^\bu\bigr)
\ot_{\Z/2\Z}\\ \ts Q_{R_b,U_b,f_b,i_b}\bigr]\big\vert_{R_a\cap
R_b}.\end{subarray}} }
\end{gathered}
\label{sm6eq21}
\ea

When $a=b$ we can take $\Psi^d=\Phi^d$, so \eq{sm6eq17} gives
$\al_{aa}^d=\id$, and $\al_{aa}=\id$.

To prove \eq{sm6eq16}, let $a,b,c\in A$, and $x\in R_a\cap R_b\cap
R_c$. Applying Theorem \ref{sm6thm2} twice and composing the
embeddings, we can construct subcharts
$(R_a',U_a',f_a',i_a')\subseteq (R_a,U_a,f_a,i_a)$,
$(R_b',U_b',f_b',i_b')\subseteq (R_b,U_b,f_b,i_b)$,
$(R_c',U_c',f_c',i_c')\ab\subseteq (R_c,U_c,f_c,i_c)$ with $x\in
R_a'\cap R_b'\cap R_c'$, a critical chart $(S,V,g,j)$ on $(X,s)$,
and embeddings $\Phi:(R_a',U_a',f_a',i_a')\hookra(S,V,g,j)$,
$\Psi:(R_b',U_b',f_b',i_b')\hookra(S,V,g,j)$,
$\Up:(R_c',U_c',f_c',i_c')\hookra(S,V,g,j)$. Then the construction
of $\al_{ab}$ above yields
\begin{align*}
\al_{ab}\vert_{R_a'\cap R_b'\cap R_c'}&=\!
\bigl((i_b^*(\Th_\Psi^{-1})\!\ot\!\id)\!\ci\!
(\id\!\ot\La_\Psi)\!\ci\!(\id\!\ot\La_\Phi^{-1})\!\ci\!
(i_a^*(\Th_\Phi)\!\ot\!\id)\bigr)\vert_{R_a'\cap R_b'\cap R_c'},\\
\al_{bc}\vert_{R_a'\cap R_b'\cap R_c'}&=\!
\bigl((i_c^*(\Th_\Up^{-1})\!\ot\!\id)\!\ci\!
(\id\!\ot\La_\Up)\!\ci\!(\id\!\ot\La_\Psi^{-1})\!\ci\!
(i_b^*(\Th_\Psi)\!\ot\!\id)\bigr)\vert_{R_a'\cap R_b'\cap R_c'},\\
\al_{ac}\vert_{R_a'\cap R_b'\cap R_c'}&=\!
\bigl((i_c^*(\Th_\Up^{-1})\!\ot\!\id)\!\ci\!
(\id\!\ot\La_\Up)\!\ci\!(\id\!\ot\La_\Phi^{-1})\!\ci\!
(i_a^*(\Th_\Phi)\!\ot\!\id)\bigr)\vert_{R_a'\cap R_b'\cap R_c'},
\end{align*}
so that $\al_{bc}\vert_{R_a'\cap R_b'\cap R_c'}\ci
\al_{ab}\vert_{R_a'\cap R_b'\cap R_c'}=\al_{ac}\vert_{R_a'\cap
R_b'\cap R_c'}$. As we can cover $R_a\cap R_b\cap R_c$ by such
Zariski open $R_a'\cap R_b'\cap R_c'$, equation \eq{sm6eq16} follows
by Theorem~\ref{sm2thm3}(i).

The Zariski open cover version of Theorem \ref{sm2thm3}(ii) now
implies that there exists $P_{X,s}^\bu$ in $\Perv(X)$, unique up to
canonical isomorphism, with isomorphisms
\begin{equation*}
\om_{R_a,U_a,f_a,i_a}:P_{X,s}^\bu\vert_{R_a}\longra
i_a^*\bigl(\PV_{U_a,f_a}^\bu\bigr)\ot_{\Z/2\Z} Q_{R_a,U_a,f_a,i_a}
\end{equation*}
as in \eq{sm6eq6} for each $a\in A$, with
$\al_{ab}\ci\om_{R_a,U_a,f_a,i_a}\vert_{R_a\cap R_b}=
\om_{R_b,U_b,f_b,i_b}\vert_{R_a\cap R_b}$ for all $a,b\in A$. Also,
\eq{sm6eq7}--\eq{sm6eq8} with $(R_a,U_a,f_a,i_a)$ in place of
$(R,U,f,i)$ define isomorphisms $\Si_{X,s}\vert_{R_a}$,
$\Tau_{X,s}\vert_{R_a}$ for each $a\in A$. Equations
\eq{sm6eq20}--\eq{sm6eq21} imply that the prescribed values for
$\Si_{X,s}\vert_{R_a},\Tau_{X,s}\vert_{R_a}$ and
$\Si_{X,s}\vert_{R_b},\Tau_{X,s}\vert_{R_b}$ agree when restricted
to $R_a\cap R_b$ for all $a,b\in A$. Hence, Theorem \ref{sm2thm3}(i)
gives unique isomorphisms $\Si_{X,s},\Tau_{X,s}$ in \eq{sm6eq5} such
that \eq{sm6eq7}--\eq{sm6eq8} commute with $(R_a,U_a,f_a,i_a)$ in
place of $(R,U,f,i)$ for all~$a\in A$.

Suppose $\bigl\{(R_a,U_a,f_a,i_a):a\in A\bigr\}$ and $\bigl\{(\ti
R_a,\ti U_a,\ti f_a,\ti\imath_a):a\in\ti A\bigr\}$ are alternative
choices above, yielding $P_{X,s}^\bu,\Si_{X,s},\Tau_{X,s}$ and $\ti
P_{X,s}^\bu,\ti\Si_{X,s},\ti\Tau_{X,s}$. Then applying the same
construction to the family $\bigl\{(R_a,U_a,f_a,i_a):a\in
A\bigr\}\amalg\bigl\{(\ti R_a,\ti U_a,\ti f_a,\ti\imath_a):a\in\ti
A\bigr\}$ to get $\hat P_{X,s}^\bu$, we have canonical isomorphisms
$P_{X,s}^\bu\cong\hat P_{X,s}^\bu\cong\ti P_{X,s}^\bu$, which
identify $\Si_{X,s},\Tau_{X,s}$ with $\ti\Si_{X,s},\ti\Tau_{X,s}$.
Thus $P_{X,s}^\bu,\Si_{X,s},\Tau_{X,s}$ are independent of choices
up to canonical isomorphism.

Now fix $\bigl\{(R_a,U_a,f_a,i_a):a\in A\bigr\}$,
$P_{X,s}^\bu,\Si_{X,s},\Tau_{X,s}$ and $\om_{R_a,U_a,f_a,i_a}$ for
$a\in A$ above for the rest of the proof. Suppose $(R,U,f,i)$ is a
critical chart on $(X,s)$. Running the construction above with the
family $\bigl\{(R_a,U_a,f_a,i_a):a\in
A\bigr\}\amalg\bigl\{(R,U,f,i)\bigr\}$, we can suppose it yields the
same (not just isomorphic) $P_{X,s}^\bu,\Si_{X,s},\Tau_{X,s}$ and
$\om_{R_a,U_a,f_a,i_a}$, but it also yields a unique $\om_{R,U,f,i}$
in \eq{sm6eq6} which makes \eq{sm6eq7}--\eq{sm6eq8} commute. This
proves Theorem~\ref{sm6thm6}(i).

Let $\Phi:(R,U,f,i)\hookra(S,V,g,j)$ be an embedding of critical
charts on $(X,s)$. The definition of $\La_\Phi$ in Theorem
\ref{sm6thm6}(ii) is immediate. Run the construction above using the
family $\bigl\{(R_a,U_a,f_a,i_a):a\in A\bigr\}\amalg
\bigl\{(R,U,f,i),(S,V,g,j)\bigr\}$, and follow the definition of
$\al_{ab}$ with $(R,U,f,i),(S,V,g,j)$ in place of
$(R_a,U_a,f_a,i_a),\ab(R_b,U_b,f_b,i_b)$. We can take
$D_{ab}=\{d\}$, $(R_a^{\prime d},U_a^{\prime d},f_a^{\prime
d},i_a^{\prime d})=(R,U,f,i)$, $(R_b^{\prime d},\ab U_b^{\prime
d},\ab f_b^{\prime d},\ab i_b^{\prime
d})=(S^d,V^d,g^d,j^d)=(S,V,g,j)$, $\Phi^d=\Phi$ and $\Psi^d=\id_V$.
Then \eq{sm6eq17} gives
$\al_{ab}=\al_{ab}^d=(\id\ot\La_\Phi^{-1})\ci(i^*(\Th_\Phi)\ot\id)$.
Thus, $\al_{ab}\ci\om_{R_a,U_a,f_a,i_a}\vert_{R_a\cap R_b}=
\om_{R_b,U_b,f_b,i_b}\vert_{R_a\cap R_b}$ implies that \eq{sm6eq10}
commutes, proving Theorem~\ref{sm6thm6}(ii).

\subsection{$\cD$-modules and mixed Hodge modules}
\label{sm64}

Once again, the proof of Theorem \ref{sm6thm6} carries over to our
other contexts in \S\ref{sm26}--\S\ref{sm210} using the general
framework of \S\ref{sm25}, now also making use of the Stack Property
(x) for objects. For the case of mixed Hodge modules, we use Theorem
\ref{sm2thm8}(ii) to glue the
$i_a^*\bigl(\HV_{U_a,f_a}^\bu\bigr)\ot_{\Z/2\Z} Q_{R_a,U_a,f_a,i_a}$
on $R_a\subseteq X$ for $a\in A$ {\it with their natural strong
polarizations\/} \eq{sm2eq25}, which are preserved by the
isomorphisms $\al_{ab}$ in \S\ref{sm63} on overlaps~$R_a\cap R_b$.

\appendix

\section{Compatibility results, by J\"org Sch\"urmann}
\label{smA}

In the main body of the paper, when comparing results for mixed Hodge modules to those 
involving perverse sheaves, we rely on the compatibility
between duality and Thom--Sebastiani type isomorphisms of perverse sheaves and mixed
Hodge modules. These compatibility statements cannot easily be read off from the existing literature, so we provide proofs here. 

\begin{prop} If\/ $X$ is a $\C$-scheme and\/ $f:X\ra\C$ is regular, then Massey's natural isomorphisms from\/ {\rm\cite{Mass3}} quoted as Theorem {\rm\ref{sm2thm4}(iv)} coincide with the image under the realization functor of Saito's analogous isomorphisms {\rm\cite{Sait1}} between functors on mixed Hodge modules. There is also an analogous compatibility result for $X$ a complex analytic space equipped with an analytic function~$f$.
\label{smAprop1}
\end{prop}

\begin{proof} Massey's construction in \cite{Mass3} of the duality isomorphisms uses the definition of the vanishing cycle functor in terms of the local cohomology of suitable real half-spaces,
compare also \cite{Schu}. Using their notation, the compatibility comes down to compatibility of the
diagram
\e
\begin{gathered}
\xymatrix@=130pt@R=15pt{
*+[r]{\phi_f^p\ci\bD_X} \ar[r]_\sim \ar[d]^\cong & *+[l]{\bD_{X_0}\ci\phi_f^p} \\
*+[r]{(R\Ga_{\{\Re(f)\geq 0\}}(-)|_{X_0})\ci\bD_X} \ar[r]^\sim &
*+[l]{\bD_{X_0}\ci(R\Ga_{\{\Re(f)\leq 0\}}(-)|_{X_0}).\!} \ar[u]^\cong }
\end{gathered}
\label{smAeq1}
\e

Here the upper, respectively lower horizontal isomorphisms are the ones of
Saito, respectively Massey, and the vertical isomorphisms follow for example from
\cite[Lem.~1.3.2, p.~69]{Schu}. Saito deduces his duality isomorphism in \cite[Lem.~5.2.4, p.~965]{Sait1} from a pairing on nearby cycles induced by a pairing
$F \otimes G\ra a^!_X A$ for $F,G\in D^b_c(X)$, with $a_X:X\ra {\rm pt}$ the
constant map and $A\subset\C$ a coefficient field. But Massey's duality isomorphism can be also be induced from such a pairing fitting into a commutative diagram, with $L_0=\{\Re(f) = 0\}$
and $j:\{\Re(f) = 0, f \neq 0\}\ra L_0$ the open inclusion:
\e
\begin{gathered}
\xymatrix@C=140pt@R=15pt{
*+[r]{(R\Ga_{\{\Re(f)< 0\}}(F)|_{X_0}) \otimes (R\Ga_{\{\Re(f)> 0\}}(G)|_{X_0})} \ar[d] &
*+[l]{\psi_f(F)\otimes \psi_f(G)} \ar[d] \ar[l]^(0.265)\sim \\
*+[r]{(Rj_*j^*(a^!_{L_0}A|_{L_0}))|_{X_0}} \ar[d]^\cong & *+[l]{\psi_f(a^!_X A)} \ar[l] \ar[d] \\
*+[r]{(Rj_*j^*(a^!_{L_0}A))|_{X_0}[1]} \ar[r] & *+[l]{a_{X_0}^!A[2].\!} }
\end{gathered}
\label{smAeq2}
\e
Here the isomorphism $j^*(a^!_X A|_{L_0}) \cong j^*(a^!_{L_0}A)[1]$ comes from the fact
that $\Re(f)$ has no critical points (in a stratified sense) in $X\setminus X_0$, 
locally near $X_0$. But then the commutativity of \eq{smAeq2} implies by
\cite[Lem.~5.2.4, p.~965]{Sait1} the commutativity of \eq{smAeq1}, concluding the proof.
\end{proof}

A similar compatibility question arises for the Thom--Sebastiani isomorphism. Here the precise
statement is the following. 

\begin{prop} Let\/ $f_i: Y_i\ra\C$ be regular functions on smooth\/ $\C$-schemes, for $i=1,2$. Let\/ $f=f_1\boxplus f_2: Y_1\t Y_2\ra\C$ be as in Theorem {\rm\ref{sm2thm5}}. Then the isomorphism \eq{sm2eq8} of Massey\/ {\rm\cite{Mass1,Schu}} coincides with the image under the realization functor of Saito's analogous isomorphism \eq{sm2eq26} of\/ {\rm\cite{Sait5}} for mixed Hodge modules.
\label{smAprop2}
\end{prop}

\begin{proof} The Thom--Sebastiani isomorphism \eq{sm2eq26} is constructed by Saito \cite[Th.~2.6]{Sait5} based on the Verdier specialization \cite{Verd}. First, let $f: Y\ra \C$ be a regular function, with $X=f^{-1}(0)$ of codimension one, 
so that the normal cone $C_XY = X \t \C$ becomes a trivial line bundle with $f': C_XY\ra\C$ given by the projection. Let $p: D_XY \ra\C$ be the deformation to the normal cone with $C_XY = p^{-1}(0)\subset D_XY$, with $q : D_XY \ra Y$ the natural map. Then $f'$ extends to a function $g: D_XY\ra\C$, with $g=f/s$ on ${p\neq 0} = Y\t\C^∗$ for $s$ the usual coordinate on $\C$. For $F\in D^b_c(Y)$, we get a commutative diagram
\begin{equation*}
\xymatrix@C=120pt@R=15pt{
*+[r]{\phi_{f'}^p(sp_XF)} \ar[d]^\cong & *+[l]{\phi_f^p(F)} \ar[l]^\sim \ar[d]_\cong \\
*+[r]{R\Ga_{\{\Re(f')\geq 0\}}(sp_XF)|_{X}} & *+[l]{R\Ga_{\{\Re(f)\geq 0\}}(F)|_{X}.\!} \ar[l] }
\end{equation*}
Here the monodromical sheaf complex $sp_XF\in D^b_c(C_XY)_{\rm mon}$ is the
Verdier specialization of $F$ as in \cite{Verd, Sait5}. The upper horizontal
isomorphism is the one of \cite[Lem.~2.2]{Sait5}, whereas the vertical
isomorphisms are those of \cite[Lem.~1.3.2, p.~69]{Schu}. The lower
horizontal map is defined by the natural base change morphism
\begin{equation*}
R\Ga_{\{\Re(f')\geq 0\}}(sp_XF)|_{X} \longleftarrow \psi_p (R\Ga_{\{\Re(g)\geq 0\}}(q^*F))|_{X} 
\cong R\Ga_{\{\Re(f)\geq 0\}}(F)|_{X},
\end{equation*}
where the last isomorphism follows as in \cite[Lem.~1.3.3, p.~70-71]{Schu}. 

Consider now the situation in the proposition, with 
\begin{equation*}
f=f_1\boxplus f_2: Y=Y_1\t Y_2\ra\C,
\end{equation*}
also $X_i=f_i^{-1}(0)$ and $X=f^{-1}(0)$; finally let $\mu_{f_i}=R\Ga_{\{\Re(f_i)\geq 0\}}$ to shorten the notation. Let
\begin{equation*}
\pi : C_{X_1}Y_1\t C_{X_2}(Y_2) = (X_1\t X_2) \t \C^2 \ra (X_1\t X_2) \t \C\subset C_X(Y)
\end{equation*}
be the map induced by addition in the fibres. Then, for $F_i\in D^b_c(Y_i)$, one gets a commutative diagram
\begin{equation*}
\xymatrix@C=160pt@R=15pt{ *+[r]{\phi_f^p(F_1\boxtimes F_2)|_{X_1\t X_2}} \ar[d]^\cong & *+[l]{\phi^p_{f_i}(F_1)\otimes \phi^p_{f_2}(F_2)} \ar[l]^\sim \ar[d]_\cong \\
*+[r]{\mu_{f'}(\pi_*(sp_{X_1}F_1\boxtimes sp_{X_2}F_2))|_{X_1\t X_2}}  & *+[l]{\mu_{f'_1}(sp_{X_1}F_1)|_{X_1}\otimes\mu_{f'_2}(sp_{X_2}F_2)|_{X_2}} \ar[l] \\
*+[r]{\mu_f(F_1\boxtimes F_2)|_{X_1\t X_2}} \ar[u] & *+[l]{\mu_{f_1}(F_1)|_{X_1}\otimes \mu_{f_2}(F_2)|_{X_2}.\!} \ar[l]_\sim \ar[u]^\cong }
\end{equation*}
The upper horizontal and left vertical isomorphisms form the Thom--Sebas\-tiani
isomorphism \eq{sm2eq26} of \cite{Sait5}, whereas the lower horizontal isomorphism is the
Thom--Sebastiani isomorphism \eq{sm2eq8} of \cite{Mass1,Schu}. This concludes the proof. 
\end{proof}

\medskip

\noindent{\small\sc Address for Christopher Brav:

\noindent Faculty of Mathematics, Higher School of Economics, 7 Vavilova Str., Moscow, Russia

\noindent E-mail: {\tt chris.i.brav@gmail.com}.

\noindent Address for Vittoria Bussi:

\noindent ICTP, Strada Costiera 11, Trieste, Italy

\noindent E-mail: {\tt vbussi@ictp.it}.

\noindent Address for Dominic Joyce and Bal\'azs Szendr\H oi:

\noindent The Mathematical Institute, Woodstock Road, Oxford, UK

\noindent E-mails: {\tt joyce@maths.ox.ac.uk, szendroi@maths.ox.ac.uk}.

\noindent Address for J\"org Sch\"urmann:

\noindent Mathematische Institut, Universit\"at M\"unster, Einsteinstrasse 62,
48149 M\"unster, Germany.

\noindent E-mail: {\tt jschuerm@math.uni-muenster.de}.
}


\begin{thebibliography}{99}
\addcontentsline{toc}{section}{References}

\bibitem{AGV} V.I. Arnold, S.M. Gusein-Zade and A.N. Varchenko, {\it
Singularities of differentiable maps, Volume 1}, Monographs in Math.
82, Birkh\"auser, 1985.

\bibitem{Behr} K. Behrend, {\it Donaldson--Thomas type invariants
via microlocal geometry}, Ann. of Math. 170 (2009), 1307--1338.
math.AG/0507523.

\bibitem{BBS} K. Behrend, J. Bryan and B. Szendr\H oi, {\it Motivic
degree zero Donaldson--Thomas invariants}, Invent. Math. 192 (2013),
111--160. arXiv:0909.5088.

\bibitem{BeFa} K. Behrend and B. Fantechi, {\it Gerstenhaber and
Batalin--Vilkovisky structures on Lagrangian intersections}, pages
1--47 in {\it Algebra, arithmetic, and geometry}, Progr. Math. 269,
Birkh\"auser, Boston, MA, 2009.

\bibitem{BBD} A.A. Beilinson, J. Bernstein and P. Deligne, {\it
Faisceaux pervers}, Ast\'erisque 100, 1982.

\bibitem{BBBJ} O. Ben-Bassat, C. Brav, V. Bussi, and D. Joyce, {\it
A `Darboux Theorem' for shifted symplectic structures on derived
Artin stacks, with applications}, Geometry and Topology 19 (2015), 1287--1359. arXiv:1312.0090.

\bibitem{Bjor} J.-E. Bj\"ork, {\it Analytic $\cD$-modules and
applications}, Kluwer, Dordrecht, 1993.

\bibitem{Bore} A. Borel et al., {\it Algebraic $\cD$-modules},
Perspectives in Mathematics 2, Academic Press, 1987.

\bibitem{BBJ} C. Brav, V. Bussi and D. Joyce, {\it A Darboux
theorem for derived schemes with shifted symplectic structure},
arXiv:1305.6302, 2013.

\bibitem{Buss} V. Bussi, {\it Categorification of Lagrangian
intersections on complex symplectic manifolds using perverse sheaves
of vanishing cycles}, arXiv:1404.1329, 2014.

\bibitem{BJM} V. Bussi, D. Joyce and S. Meinhardt, {\it On motivic
vanishing cycles of critical loci}, arXiv:1305.6428, 2013.

\bibitem{Cout} S.C. Coutinho, {\it A primer of algebraic $\cD$-modules},
L.M.S. Student texts 33, Cambridge University Press, 1995.

\bibitem{DASc} A. D'Agnolo and P. Schapira, {\it Quantization of
complex Lagrangian submanifolds}, Advances in Math. 213 (2007),
358--379. math.AG/0506064.

\bibitem{Dimc} A. Dimca, {\it Sheaves in Topology}, Universitext,
Springer-Verlag, Berlin, 2004.

\bibitem{DiSz} A. Dimca and B. Szendr\H oi, {\it The Milnor fibre
of the Pfaffian and the Hilbert scheme of four points on $\C^3$},
Math. Res. Lett. 16 (2009), 1037--1055. arXiv:0904.2419.

\bibitem{DoTh} S.K. Donaldson and R.P. Thomas, {\it Gauge Theory in
Higher Dimensions}, Chapter 3 in S.A. Huggett, L.J. Mason, K.P. Tod,
S.T. Tsou and N.M.J. Woodhouse, editors, {\it The Geometric
Universe}, Oxford University Press, Oxford, 1998.

\bibitem{FrKi} E. Freitag and R. Kiehl, {\it Etale cohomology and
the Weil Conjecture}, Ergeb. der Math. und ihrer Grenzgebiete 13,
Springer-Verlag, 1988.

\bibitem{FOOO} K. Fukaya, Y.-G. Oh, H. Ohta and K. Ono,
{\it Lagrangian intersection Floer theory --- anomaly and
obstruction}, Parts I \& II. AMS/IP Studies in Advanced Mathematics,
46.1 \& 46.2, A.M.S./International Press, 2009.

\bibitem{GaHa} T. Gaffney and H. Hauser, {\it Characterizing
singularities of varieties and of mappings}, Invent. math. 81
(1985), 427--447.

\bibitem{GLM} G. Guibert, F. Loeser and M. Merle, {\it Iterated vanishing cycles, convolution,
and a motivic analogue of a conjecture of Steenbrink}, 
Duke Math. J. 132 (2006), 409--457.

\bibitem{HTT} R. Hotta, T. Tanisaki and K. Takeuchi, {\it
$\cD$-modules, perverse sheaves, and representation theory}, Progr.
Math. 236, Birkh\"auser, Boston, MA, 2008.

\bibitem{HuTh} D. Huybrechts and R.P. Thomas, {\it
Deformation-obstruction theory for complexes via Atiyah and
Kodaira--Spencer classes}, Math. Ann. 346 (2010), 545--569.
arXiv:0805.3527.

\bibitem{Joyc} D. Joyce, {\it A classical model for derived
critical loci}, to appear in Journal of Differential Geometry, 2015. arXiv:1304.4508.

\bibitem{JoSo} D. Joyce and Y. Song, {\it A theory of generalized
Donaldson--Thomas invariants}, Mem. Amer. Math. Soc. 217 (2012), no.
1020. arXiv:0810.5645.

\bibitem{Kash1} M. Kashiwara, {\it The Riemann--Hilbert problem for
holonomic systems}, Publ. RIMS 20 (1984), 319--365.

\bibitem{Kash2} M. Kashiwara, {\it $\cD$-modules and microlocal
calculus}, Trans. Math. Mono. 217, A.M.S., 2003.

\bibitem{KaSc1} M. Kashiwara and P. Schapira, {\it
Sheaves on manifolds}, Grundlehren der Math. Wiss. 292,
Springer-Verlag, Berlin, 1990.

\bibitem{KaSc2} M. Kashiwara and P. Schapira, {\it Constructibility
and duality for simple modules on symplectic manifolds}, Amer. J.
Math. 130 (2008), 207--237. math.QA/0512047.

\bibitem{KaSc3} M. Kashiwara and P. Schapira, {\it
Deformation quantization modules},\hfil\break Ast\'erisque 345,
2012.

\bibitem{KiWe} R. Kiehl and R. Weissauer, {\it Weil Conjectures,
perverse sheaves and l'adic Fourier transform}, Springer-Verlag,
2001.

\bibitem{KiLi} Y.-H. Kiem and J. Li, {\it Categorification of
Donaldson--Thomas invariants via perverse sheaves}, arXiv:1212.6444,
2012.

\bibitem{KoSo1} M. Kontsevich and Y. Soibelman, {\it Stability
structures, motivic Donaldson--Thomas invariants and cluster
transformations}, arXiv:0811.2435, 2008.

\bibitem{KoSo2} M. Kontsevich and Y. Soibelman, {\it Cohomological
Hall algebra, exponential Hodge structures and motivic
Donaldson--Thomas invariants}, Commun. Number Theory Phys. 5 (2011),
231--352. arXiv:1006.2706.

\bibitem{MaMe} P. Maisonobe and Z. Mebkhout, {\it Le th\'eor\`eme de comparaison 
pour les cycles \'evanescents}, pages 311--389 in {\it \'El\'ements de la th\'eorie des syst\`emes 
diff\'erentiels g\'eom\'etriques}, S\'emin. Congr., 8, Soc. Math. France, Paris, 2004. 

\bibitem{Mass1} D. Massey, {\it The Sebastiani--Thom isomorphism in
the derived category}, Compositio Math. 125 (2001), 353--362.
math.AG/9908101.

\bibitem{Mass2} D. Massey, {\it Notes on perverse sheaves and
vanishing cycles}, \hfil\break math.AG/9908107, 1999.

\bibitem{Mass3} D. Massey, {\it Natural commuting of vanishing
cycles and the Verdier dual}, arXiv:0908.2799, 2009.

\bibitem{MaYa} J. Mather and S.S. Yau, {\it Classification of
isolated hypersurface singularities by their moduli algebra},
Invent. Math. 69 (1982), 243--251.

\bibitem{MSS} L. Maxim, M. Saito and J. Sch\" urmann, {\it
Symmetric products of mixed Hodge modules}, J. Math. Pures Appl. 96
(2011), 462--483.

\bibitem{Nadl} D. Nadler, {\it Microlocal branes are constructible
sheaves}, Selecta Math. 15 (2009), 563--619. math.SG/0612399.

\bibitem{NaZa} D. Nadler and E. Zaslow, {\it Constructible sheaves
and the Fukaya category}, J. Amer. Math. Soc. 22 (2009), 233--286.
math.SG/0604379.

\bibitem{PTVV} T. Pantev, B. To\"en, M. Vaqui\'e and G. Vezzosi,
{\it Shifted symplectic structures}, Publ. Math. I.H.E.S. 117
(2013), 271--328. arXiv:1111.3209.

\bibitem{Riet} K. Rietsch, {\it An introduction to perverse
sheaves}, pages 391--429 in V. Dlab and C.M. Ringel, editors, {\it
Representations of finite dimensional algebras and related topics in
Lie theory and geometry}, Fields Inst. Commun. 40, A.M.S.,
Providence, RI, 2004. math.RT/0307349.

\bibitem{Sabb} C. Sabbah, {\it On a twisted de Rham complex, II},
arXiv:1012.3818, 2010.

\bibitem{Sait1} M. Saito, {\it Modules de Hodge polarisables},
Publ. RIMS 24 (1988), 849--995.

\bibitem{Sait2} M. Saito, {\it Duality for vanishing cycle
functors}, Publ. RIMS 25 (1989), 889--921.

\bibitem{Sait3} M. Saito, {\it Mixed Hodge Modules},
Publ. RIMS 26 (1990), 221--333.

\bibitem{Sait4} M. Saito, {\it $\cD$-modules on analytic spaces},
Publ. RIMS 27 (1991), 291--332.

\bibitem{Sait5} M. Saito, {\it Thom--Sebastiani Theorem for Hodge
Modules}, preprint, 2010.

\bibitem{Schu} J. Sch\"urmann, {\it Topology of singular spaces and
constructible sheaves}, Monografie Matematyczne 63, Birkh\"auser,
Basel, 2003.

\bibitem{Szen} B. Szendr\H oi, {\it Nekrasov's partition function
and refined Donaldson--Thomas theory: the rank one case}, SIGMA
(2012) 088, 16pp. arXiv:1210.5181.

\bibitem{Thom} R.P. Thomas, {\it A holomorphic Casson invariant for
Calabi--Yau $3$--folds, and bundles on $K3$ fibrations}, J. Diff.
Geom. 54 (2000), 367--438. \hfil\break math.AG/9806111.

\bibitem{Verd} J.-L. Verdier, {\it Sp\'ecialization de faisceaux et monodromie mod\'er\'ee},
Ast\-\'e\-ris\-que 101-102 (1983), 332--364. 

\end{thebibliography}
\end{document}